\documentclass[final]{siamltex}
\usepackage{a4,graphicx,latexsym,amsmath,amssymb,color,verbatim,mathrsfs,bbm}
\usepackage [latin1] {inputenc}
\usepackage{listings, algorithm, algorithmic}
\usepackage{xspace}
\usepackage{enumerate}
\usepackage{subfigure}
\usepackage[enableskew]{youngtab}
\usepackage{fancyhdr}
\usepackage[bf,small,hang]{caption}
\usepackage{url}
\usepackage {csquotes}
\usepackage{wrapfig}
\usepackage{paralist}
\usepackage{makerobust}
\usepackage{environ}

\numberwithin{equation}{section} 

\usepackage {tikz}
\usetikzlibrary{arrows} 
\usetikzlibrary {decorations,decorations.pathmorphing}

\newdimen\arrowsize
\pgfarrowsdeclare{my}{my}
{
  \arrowsize=0.2pt
  \advance\arrowsize by .5\pgflinewidth
  \pgfarrowsleftextend{-4\arrowsize-.5\pgflinewidth}
  \pgfarrowsrightextend{.5\pgflinewidth}
}
{
  \arrowsize=0.2pt
  \advance\arrowsize by .5\pgflinewidth
  \pgfsetdash{}{0pt} 
  \pgfsetroundjoin   
  \pgfsetroundcap    
  \pgfpathmoveto{\pgfpointorigin}
  \pgfpathlineto{\pgfpoint{-3\arrowsize}{4\arrowsize}}
  \pgfpathlineto{\pgfpoint{-3\arrowsize}{-4\arrowsize}}
  \pgfpathclose
  \pgfusepathqfill
}
\tikzset{rhrhombithickside/.style={ultra thick}}
\tikzset{rhrhombiarrow/.style={}}
\tikzset{rhrhombifinearrow/.style={line width=1pt}}
\tikzset{rhrhombipatharrow/.style={thin}}
\tikzset{rhrhombidraw/.style={ultra thin}}
\tikzset{rhrhombifill/.style={thick,fill=black!40}}
\newcommand{\FIGtriangulararray}{\begin{tikzpicture}
\draw[rhrhombidraw] (0.0pt,0.0pt) -- (-4.619pt,-8.0pt) -- (4.619pt,-8.0pt) -- cycle;
\draw[rhrhombidraw] (-4.619pt,-8.0pt) -- (-9.238pt,-16.0pt) -- (0.0pt,-16.0pt) -- cycle;
\draw[rhrhombidraw] (4.619pt,-8.0pt) -- (0.0pt,-16.0pt) -- (9.238pt,-16.0pt) -- cycle;
\draw[rhrhombidraw] (-9.238pt,-16.0pt) -- (-13.857pt,-24.0pt) -- (-4.619pt,-24.0pt) -- cycle;
\draw[rhrhombidraw] (0.0pt,-16.0pt) -- (-4.619pt,-24.0pt) -- (4.619pt,-24.0pt) -- cycle;
\draw[rhrhombidraw] (9.238pt,-16.0pt) -- (4.619pt,-24.0pt) -- (13.857pt,-24.0pt) -- cycle;
\draw[rhrhombidraw] (-13.857pt,-24.0pt) -- (-18.476pt,-32.0pt) -- (-9.238pt,-32.0pt) -- cycle;
\draw[rhrhombidraw] (-4.619pt,-24.0pt) -- (-9.238pt,-32.0pt) -- (0.0pt,-32.0pt) -- cycle;
\draw[rhrhombidraw] (4.619pt,-24.0pt) -- (0.0pt,-32.0pt) -- (9.238pt,-32.0pt) -- cycle;
\draw[rhrhombidraw] (13.857pt,-24.0pt) -- (9.238pt,-32.0pt) -- (18.476pt,-32.0pt) -- cycle;
\draw[rhrhombidraw] (-18.476pt,-32.0pt) -- (-23.095pt,-40.0pt) -- (-13.857pt,-40.0pt) -- cycle;
\draw[rhrhombidraw] (-9.238pt,-32.0pt) -- (-13.857pt,-40.0pt) -- (-4.619pt,-40.0pt) -- cycle;
\draw[rhrhombidraw] (0.0pt,-32.0pt) -- (-4.619pt,-40.0pt) -- (4.619pt,-40.0pt) -- cycle;
\draw[rhrhombidraw] (9.238pt,-32.0pt) -- (4.619pt,-40.0pt) -- (13.857pt,-40.0pt) -- cycle;
\draw[rhrhombidraw] (18.476pt,-32.0pt) -- (13.857pt,-40.0pt) -- (23.095pt,-40.0pt) -- cycle;
\end{tikzpicture}}
\newcommand{\FIGdualtriangulararray}{\begin{tikzpicture}
\draw[rhrhombidraw] (0.0pt,0.0pt) -- (-9.238pt,-16.0pt) -- (9.238pt,-16.0pt) -- cycle;
\draw (-4.619pt,-8.0pt) -- (0.0pt,-10.656pt);
\draw (4.619pt,-8.0pt) -- (0.0pt,-10.656pt);
\draw (0.0pt,-16.0pt) -- (0.0pt,-10.656pt);
\draw (0.0pt,-16.0pt) -- (0.0pt,-21.344pt);
\draw (-4.619pt,-24.0pt) -- (0.0pt,-21.344pt);
\draw (4.619pt,-24.0pt) -- (0.0pt,-21.344pt);
\fill[color=white] (-4.619pt,-8.0pt) circle (1.2pt);
\draw[rhrhombidraw] (-4.619pt,-8.0pt) circle (1.2pt);
\fill[color=white] (4.619pt,-8.0pt) circle (1.2pt);
\draw[rhrhombidraw] (4.619pt,-8.0pt) circle (1.2pt);
\fill[color=white] (0.0pt,-16.0pt) circle (1.2pt);
\draw[rhrhombidraw] (0.0pt,-16.0pt) circle (1.2pt);
\fill (0.0pt,-10.656pt) circle (1.5pt);
\fill (0.0pt,-21.344pt) circle (1.5pt);
\draw[rhrhombidraw] (-9.238pt,-16.0pt) -- (-18.476pt,-32.0pt) -- (0.0pt,-32.0pt) -- cycle;
\draw (-13.857pt,-24.0pt) -- (-9.238pt,-26.656pt);
\draw (-4.619pt,-24.0pt) -- (-9.238pt,-26.656pt);
\draw (-9.238pt,-32.0pt) -- (-9.238pt,-26.656pt);
\draw (-9.238pt,-32.0pt) -- (-9.238pt,-37.344pt);
\draw (-13.857pt,-40.0pt) -- (-9.238pt,-37.344pt);
\draw (-4.619pt,-40.0pt) -- (-9.238pt,-37.344pt);
\fill[color=white] (-13.857pt,-24.0pt) circle (1.2pt);
\draw[rhrhombidraw] (-13.857pt,-24.0pt) circle (1.2pt);
\fill[color=white] (-4.619pt,-24.0pt) circle (1.2pt);
\draw[rhrhombidraw] (-4.619pt,-24.0pt) circle (1.2pt);
\fill[color=white] (-9.238pt,-32.0pt) circle (1.2pt);
\draw[rhrhombidraw] (-9.238pt,-32.0pt) circle (1.2pt);
\fill (-9.238pt,-26.656pt) circle (1.5pt);
\fill (-9.238pt,-37.344pt) circle (1.5pt);
\draw[rhrhombidraw] (9.238pt,-16.0pt) -- (0.0pt,-32.0pt) -- (18.476pt,-32.0pt) -- cycle;
\draw (4.619pt,-24.0pt) -- (9.238pt,-26.656pt);
\draw (13.857pt,-24.0pt) -- (9.238pt,-26.656pt);
\draw (9.238pt,-32.0pt) -- (9.238pt,-26.656pt);
\draw (9.238pt,-32.0pt) -- (9.238pt,-37.344pt);
\draw (4.619pt,-40.0pt) -- (9.238pt,-37.344pt);
\draw (13.857pt,-40.0pt) -- (9.238pt,-37.344pt);
\fill[color=white] (4.619pt,-24.0pt) circle (1.2pt);
\draw[rhrhombidraw] (4.619pt,-24.0pt) circle (1.2pt);
\fill[color=white] (13.857pt,-24.0pt) circle (1.2pt);
\draw[rhrhombidraw] (13.857pt,-24.0pt) circle (1.2pt);
\fill[color=white] (9.238pt,-32.0pt) circle (1.2pt);
\draw[rhrhombidraw] (9.238pt,-32.0pt) circle (1.2pt);
\fill (9.238pt,-26.656pt) circle (1.5pt);
\fill (9.238pt,-37.344pt) circle (1.5pt);
\draw[rhrhombidraw] (-18.476pt,-32.0pt) -- (-27.714pt,-48.0pt) -- (-9.238pt,-48.0pt) -- cycle;
\draw (-23.095pt,-40.0pt) -- (-18.476pt,-42.656pt);
\draw (-13.857pt,-40.0pt) -- (-18.476pt,-42.656pt);
\draw (-18.476pt,-48.0pt) -- (-18.476pt,-42.656pt);
\draw (-18.476pt,-48.0pt) -- (-18.476pt,-53.344pt);
\draw (-23.095pt,-56.0pt) -- (-18.476pt,-53.344pt);
\draw (-13.857pt,-56.0pt) -- (-18.476pt,-53.344pt);
\fill[color=white] (-23.095pt,-40.0pt) circle (1.2pt);
\draw[rhrhombidraw] (-23.095pt,-40.0pt) circle (1.2pt);
\fill[color=white] (-13.857pt,-40.0pt) circle (1.2pt);
\draw[rhrhombidraw] (-13.857pt,-40.0pt) circle (1.2pt);
\fill[color=white] (-18.476pt,-48.0pt) circle (1.2pt);
\draw[rhrhombidraw] (-18.476pt,-48.0pt) circle (1.2pt);
\fill (-18.476pt,-42.656pt) circle (1.5pt);
\fill (-18.476pt,-53.344pt) circle (1.5pt);
\draw[rhrhombidraw] (0.0pt,-32.0pt) -- (-9.238pt,-48.0pt) -- (9.238pt,-48.0pt) -- cycle;
\draw (-4.619pt,-40.0pt) -- (0.0pt,-42.656pt);
\draw (4.619pt,-40.0pt) -- (0.0pt,-42.656pt);
\draw (0.0pt,-48.0pt) -- (0.0pt,-42.656pt);
\draw (0.0pt,-48.0pt) -- (0.0pt,-53.344pt);
\draw (-4.619pt,-56.0pt) -- (0.0pt,-53.344pt);
\draw (4.619pt,-56.0pt) -- (0.0pt,-53.344pt);
\fill[color=white] (-4.619pt,-40.0pt) circle (1.2pt);
\draw[rhrhombidraw] (-4.619pt,-40.0pt) circle (1.2pt);
\fill[color=white] (4.619pt,-40.0pt) circle (1.2pt);
\draw[rhrhombidraw] (4.619pt,-40.0pt) circle (1.2pt);
\fill[color=white] (0.0pt,-48.0pt) circle (1.2pt);
\draw[rhrhombidraw] (0.0pt,-48.0pt) circle (1.2pt);
\fill (0.0pt,-42.656pt) circle (1.5pt);
\fill (0.0pt,-53.344pt) circle (1.5pt);
\draw[rhrhombidraw] (18.476pt,-32.0pt) -- (9.238pt,-48.0pt) -- (27.714pt,-48.0pt) -- cycle;
\draw (13.857pt,-40.0pt) -- (18.476pt,-42.656pt);
\draw (23.095pt,-40.0pt) -- (18.476pt,-42.656pt);
\draw (18.476pt,-48.0pt) -- (18.476pt,-42.656pt);
\draw (18.476pt,-48.0pt) -- (18.476pt,-53.344pt);
\draw (13.857pt,-56.0pt) -- (18.476pt,-53.344pt);
\draw (23.095pt,-56.0pt) -- (18.476pt,-53.344pt);
\fill[color=white] (13.857pt,-40.0pt) circle (1.2pt);
\draw[rhrhombidraw] (13.857pt,-40.0pt) circle (1.2pt);
\fill[color=white] (23.095pt,-40.0pt) circle (1.2pt);
\draw[rhrhombidraw] (23.095pt,-40.0pt) circle (1.2pt);
\fill[color=white] (18.476pt,-48.0pt) circle (1.2pt);
\draw[rhrhombidraw] (18.476pt,-48.0pt) circle (1.2pt);
\fill (18.476pt,-42.656pt) circle (1.5pt);
\fill (18.476pt,-53.344pt) circle (1.5pt);
\draw[rhrhombidraw] (-27.714pt,-48.0pt) -- (-36.952pt,-64.0pt) -- (-18.476pt,-64.0pt) -- cycle;
\draw (-32.333pt,-56.0pt) -- (-27.714pt,-58.656pt);
\draw (-23.095pt,-56.0pt) -- (-27.714pt,-58.656pt);
\draw (-27.714pt,-64.0pt) -- (-27.714pt,-58.656pt);
\draw (-27.714pt,-64.0pt) -- (-27.714pt,-69.344pt);
\draw (-32.333pt,-72.0pt) -- (-27.714pt,-69.344pt);
\draw (-23.095pt,-72.0pt) -- (-27.714pt,-69.344pt);
\fill[color=white] (-32.333pt,-56.0pt) circle (1.2pt);
\draw[rhrhombidraw] (-32.333pt,-56.0pt) circle (1.2pt);
\fill[color=white] (-23.095pt,-56.0pt) circle (1.2pt);
\draw[rhrhombidraw] (-23.095pt,-56.0pt) circle (1.2pt);
\fill[color=white] (-27.714pt,-64.0pt) circle (1.2pt);
\draw[rhrhombidraw] (-27.714pt,-64.0pt) circle (1.2pt);
\fill (-27.714pt,-58.656pt) circle (1.5pt);
\fill (-27.714pt,-69.344pt) circle (1.5pt);
\draw[rhrhombidraw] (-9.238pt,-48.0pt) -- (-18.476pt,-64.0pt) -- (0.0pt,-64.0pt) -- cycle;
\draw (-13.857pt,-56.0pt) -- (-9.238pt,-58.656pt);
\draw (-4.619pt,-56.0pt) -- (-9.238pt,-58.656pt);
\draw (-9.238pt,-64.0pt) -- (-9.238pt,-58.656pt);
\draw (-9.238pt,-64.0pt) -- (-9.238pt,-69.344pt);
\draw (-13.857pt,-72.0pt) -- (-9.238pt,-69.344pt);
\draw (-4.619pt,-72.0pt) -- (-9.238pt,-69.344pt);
\fill[color=white] (-13.857pt,-56.0pt) circle (1.2pt);
\draw[rhrhombidraw] (-13.857pt,-56.0pt) circle (1.2pt);
\fill[color=white] (-4.619pt,-56.0pt) circle (1.2pt);
\draw[rhrhombidraw] (-4.619pt,-56.0pt) circle (1.2pt);
\fill[color=white] (-9.238pt,-64.0pt) circle (1.2pt);
\draw[rhrhombidraw] (-9.238pt,-64.0pt) circle (1.2pt);
\fill (-9.238pt,-58.656pt) circle (1.5pt);
\fill (-9.238pt,-69.344pt) circle (1.5pt);
\draw[rhrhombidraw] (9.238pt,-48.0pt) -- (0.0pt,-64.0pt) -- (18.476pt,-64.0pt) -- cycle;
\draw (4.619pt,-56.0pt) -- (9.238pt,-58.656pt);
\draw (13.857pt,-56.0pt) -- (9.238pt,-58.656pt);
\draw (9.238pt,-64.0pt) -- (9.238pt,-58.656pt);
\draw (9.238pt,-64.0pt) -- (9.238pt,-69.344pt);
\draw (4.619pt,-72.0pt) -- (9.238pt,-69.344pt);
\draw (13.857pt,-72.0pt) -- (9.238pt,-69.344pt);
\fill[color=white] (4.619pt,-56.0pt) circle (1.2pt);
\draw[rhrhombidraw] (4.619pt,-56.0pt) circle (1.2pt);
\fill[color=white] (13.857pt,-56.0pt) circle (1.2pt);
\draw[rhrhombidraw] (13.857pt,-56.0pt) circle (1.2pt);
\fill[color=white] (9.238pt,-64.0pt) circle (1.2pt);
\draw[rhrhombidraw] (9.238pt,-64.0pt) circle (1.2pt);
\fill (9.238pt,-58.656pt) circle (1.5pt);
\fill (9.238pt,-69.344pt) circle (1.5pt);
\draw[rhrhombidraw] (27.714pt,-48.0pt) -- (18.476pt,-64.0pt) -- (36.952pt,-64.0pt) -- cycle;
\draw (23.095pt,-56.0pt) -- (27.714pt,-58.656pt);
\draw (32.333pt,-56.0pt) -- (27.714pt,-58.656pt);
\draw (27.714pt,-64.0pt) -- (27.714pt,-58.656pt);
\draw (27.714pt,-64.0pt) -- (27.714pt,-69.344pt);
\draw (23.095pt,-72.0pt) -- (27.714pt,-69.344pt);
\draw (32.333pt,-72.0pt) -- (27.714pt,-69.344pt);
\fill[color=white] (23.095pt,-56.0pt) circle (1.2pt);
\draw[rhrhombidraw] (23.095pt,-56.0pt) circle (1.2pt);
\fill[color=white] (32.333pt,-56.0pt) circle (1.2pt);
\draw[rhrhombidraw] (32.333pt,-56.0pt) circle (1.2pt);
\fill[color=white] (27.714pt,-64.0pt) circle (1.2pt);
\draw[rhrhombidraw] (27.714pt,-64.0pt) circle (1.2pt);
\fill (27.714pt,-58.656pt) circle (1.5pt);
\fill (27.714pt,-69.344pt) circle (1.5pt);
\draw[rhrhombidraw] (-36.952pt,-64.0pt) -- (-46.19pt,-80.0pt) -- (-27.714pt,-80.0pt) -- cycle;
\draw (-41.571pt,-72.0pt) -- (-36.952pt,-74.656pt);
\draw (-32.333pt,-72.0pt) -- (-36.952pt,-74.656pt);
\draw (-36.952pt,-80.0pt) -- (-36.952pt,-74.656pt);
\fill[color=white] (-41.571pt,-72.0pt) circle (1.2pt);
\draw[rhrhombidraw] (-41.571pt,-72.0pt) circle (1.2pt);
\fill[color=white] (-32.333pt,-72.0pt) circle (1.2pt);
\draw[rhrhombidraw] (-32.333pt,-72.0pt) circle (1.2pt);
\fill[color=white] (-36.952pt,-80.0pt) circle (1.2pt);
\draw[rhrhombidraw] (-36.952pt,-80.0pt) circle (1.2pt);
\fill (-36.952pt,-74.656pt) circle (1.5pt);
\draw[rhrhombidraw] (-18.476pt,-64.0pt) -- (-27.714pt,-80.0pt) -- (-9.238pt,-80.0pt) -- cycle;
\draw (-23.095pt,-72.0pt) -- (-18.476pt,-74.656pt);
\draw (-13.857pt,-72.0pt) -- (-18.476pt,-74.656pt);
\draw (-18.476pt,-80.0pt) -- (-18.476pt,-74.656pt);
\fill[color=white] (-23.095pt,-72.0pt) circle (1.2pt);
\draw[rhrhombidraw] (-23.095pt,-72.0pt) circle (1.2pt);
\fill[color=white] (-13.857pt,-72.0pt) circle (1.2pt);
\draw[rhrhombidraw] (-13.857pt,-72.0pt) circle (1.2pt);
\fill[color=white] (-18.476pt,-80.0pt) circle (1.2pt);
\draw[rhrhombidraw] (-18.476pt,-80.0pt) circle (1.2pt);
\fill (-18.476pt,-74.656pt) circle (1.5pt);
\draw[rhrhombidraw] (0.0pt,-64.0pt) -- (-9.238pt,-80.0pt) -- (9.238pt,-80.0pt) -- cycle;
\draw (-4.619pt,-72.0pt) -- (0.0pt,-74.656pt);
\draw (4.619pt,-72.0pt) -- (0.0pt,-74.656pt);
\draw (0.0pt,-80.0pt) -- (0.0pt,-74.656pt);
\fill[color=white] (-4.619pt,-72.0pt) circle (1.2pt);
\draw[rhrhombidraw] (-4.619pt,-72.0pt) circle (1.2pt);
\fill[color=white] (4.619pt,-72.0pt) circle (1.2pt);
\draw[rhrhombidraw] (4.619pt,-72.0pt) circle (1.2pt);
\fill[color=white] (0.0pt,-80.0pt) circle (1.2pt);
\draw[rhrhombidraw] (0.0pt,-80.0pt) circle (1.2pt);
\fill (0.0pt,-74.656pt) circle (1.5pt);
\draw[rhrhombidraw] (18.476pt,-64.0pt) -- (9.238pt,-80.0pt) -- (27.714pt,-80.0pt) -- cycle;
\draw (13.857pt,-72.0pt) -- (18.476pt,-74.656pt);
\draw (23.095pt,-72.0pt) -- (18.476pt,-74.656pt);
\draw (18.476pt,-80.0pt) -- (18.476pt,-74.656pt);
\fill[color=white] (13.857pt,-72.0pt) circle (1.2pt);
\draw[rhrhombidraw] (13.857pt,-72.0pt) circle (1.2pt);
\fill[color=white] (23.095pt,-72.0pt) circle (1.2pt);
\draw[rhrhombidraw] (23.095pt,-72.0pt) circle (1.2pt);
\fill[color=white] (18.476pt,-80.0pt) circle (1.2pt);
\draw[rhrhombidraw] (18.476pt,-80.0pt) circle (1.2pt);
\fill (18.476pt,-74.656pt) circle (1.5pt);
\draw[rhrhombidraw] (36.952pt,-64.0pt) -- (27.714pt,-80.0pt) -- (46.19pt,-80.0pt) -- cycle;
\draw (32.333pt,-72.0pt) -- (36.952pt,-74.656pt);
\draw (41.571pt,-72.0pt) -- (36.952pt,-74.656pt);
\draw (36.952pt,-80.0pt) -- (36.952pt,-74.656pt);
\fill[color=white] (32.333pt,-72.0pt) circle (1.2pt);
\draw[rhrhombidraw] (32.333pt,-72.0pt) circle (1.2pt);
\fill[color=white] (41.571pt,-72.0pt) circle (1.2pt);
\draw[rhrhombidraw] (41.571pt,-72.0pt) circle (1.2pt);
\fill[color=white] (36.952pt,-80.0pt) circle (1.2pt);
\draw[rhrhombidraw] (36.952pt,-80.0pt) circle (1.2pt);
\fill (36.952pt,-74.656pt) circle (1.5pt);
\end{tikzpicture}}
\newcommand{\rhtikzrhombus}[1]{\smash{\raisebox{-2.5pt}{\scalebox{0.7}{\begin{tikzpicture}
  \draw[rhrhombidraw] (0.0pt,0.0pt) -- (-4.619pt,8.0pt) -- (0.0pt,16.0pt) -- (4.619pt,8.0pt) -- cycle;
  #1
\end{tikzpicture}}}}}

\newcommand{\rhc}{\rhtikzrhombus{}}
\newcommand{\rhdnorth}{\rhtikzrhombus{\fill (0.0pt,16.0pt) circle (2.3095pt);}}
\newcommand{\rhdsouth}{\rhtikzrhombus{\fill (0.0pt,0.0pt) circle (2.3095pt);}}
\newcommand{\rhdeast}{\rhtikzrhombus{\fill (4.619pt,8.0pt) circle (2.3095pt);}}
\newcommand{\rhdwest}{\rhtikzrhombus{\fill (-4.619pt,8.0pt) circle (2.3095pt);}}
\newcommand{\rhsc}{\rhtikzrhombus{\draw[rhrhombithickside] (4.619pt,8.0pt) -- (-4.619pt,8.0pt);}}
\newcommand{\rhsoul}{\rhtikzrhombus{\draw[rhrhombithickside] (-4.618pt,8.0pt) -- (0.0010pt,16.0pt);}}
\newcommand{\rhsour}{\rhtikzrhombus{\draw[rhrhombithickside] (0.0010pt,16.0pt) -- (4.62pt,8.0pt);}}
\newcommand{\rhsoll}{\rhtikzrhombus{\draw[rhrhombithickside] (-4.618pt,8.0pt) -- (0.0010pt,0.0pt);}}
\newcommand{\rhslr}{\rhtikzrhombus{\draw[rhrhombithickside] (4.62pt,8.0pt) -- (9.239pt,0.0pt);}}
\newcommand{\rhrhoul}{\rhtikzrhombus{\fill[rhrhombifill] (-4.618pt,8.0pt) -- (4.62pt,8.0pt) -- (0.0010pt,16.0pt) -- (-9.237pt,16.0pt) -- cycle;}}
\newcommand{\rhrhour}{\rhtikzrhombus{\fill[rhrhombifill] (0.0010pt,16.0pt) -- (-4.618pt,8.0pt) -- (4.62pt,8.0pt) -- (9.239pt,16.0pt) -- cycle;}}
\newcommand{\rhrholl}{\rhtikzrhombus{\fill[rhrhombifill] (-4.618pt,8.0pt) -- (-9.237pt,0.0pt) -- (0.0010pt,0.0pt) -- (4.62pt,8.0pt) -- cycle;}}
\newcommand{\rhrholr}{\rhtikzrhombus{\fill[rhrhombifill] (0.0010pt,0.0pt) -- (9.239pt,0.0pt) -- (4.62pt,8.0pt) -- (-4.618pt,8.0pt) -- cycle;}}
\newcommand{\rhrhul}{\rhtikzrhombus{\fill[rhrhombifill] (-9.237pt,16.0pt) -- (-13.856pt,8.0pt) -- (-4.618pt,8.0pt) -- (0.0010pt,16.0pt) -- cycle;}}
\newcommand{\rhrhll}{\rhtikzrhombus{\fill[rhrhombifill] (-9.237pt,0.0pt) -- (0.0010pt,0.0pt) -- (-4.618pt,8.0pt) -- (-13.856pt,8.0pt) -- cycle;}}
\newcommand{\rhrhpl}{\rhtikzrhombus{\fill[rhrhombifill] (-4.619pt,8.0pt) -- (-9.238pt,16.0pt) -- (-13.857pt,8.0pt) -- (-9.238pt,0.0pt) -- cycle;}}
\newcommand{\rhrhxul}{\rhtikzrhombus{\fill[rhrhombifill] (-13.856pt,8.0pt) -- (-4.618pt,8.0pt) -- (-9.237pt,16.0pt) -- (-18.475pt,16.0pt) -- cycle;}}
\newcommand{\rhrhxll}{\rhtikzrhombus{\fill[rhrhombifill] (-13.856pt,8.0pt) -- (-18.475pt,0.0pt) -- (-9.237pt,0.0pt) -- (-4.618pt,8.0pt) -- cycle;}}
\newcommand{\rhacM}{\rhtikzrhombus{\draw[rhrhombithickside] (4.619pt,8.0pt) -- (-4.619pt,8.0pt);
\draw[->,rhrhombiarrow] (0.0pt,4.0pt) -- (0.0pt,12.0pt);}}
\newcommand{\rhacW}{\rhtikzrhombus{\draw[rhrhombithickside] (4.619pt,8.0pt) -- (-4.619pt,8.0pt);
\draw[->,rhrhombiarrow] (0.0pt,12.0pt) -- (0.0pt,4.0pt);}}
\newcommand{\rhaoulM}{\rhtikzrhombus{\draw[rhrhombithickside] (-4.618pt,8.0pt) -- (0.0010pt,16.0pt);
\draw[->,rhrhombiarrow] (-5.773pt,14.0pt) -- (1.156pt,10.0pt);}}
\newcommand{\rhaoulW}{\rhtikzrhombus{\draw[rhrhombithickside] (-4.618pt,8.0pt) -- (0.0010pt,16.0pt);
\draw[->,rhrhombiarrow] (1.155pt,10.0pt) -- (-5.773pt,14.0pt);}}
\newcommand{\rhaourM}{\rhtikzrhombus{\draw[rhrhombithickside] (0.0010pt,16.0pt) -- (4.62pt,8.0pt);
\draw[->,rhrhombiarrow] (5.774pt,14.0pt) -- (-1.154pt,10.0pt);}}
\newcommand{\rhaourW}{\rhtikzrhombus{\draw[rhrhombithickside] (0.0010pt,16.0pt) -- (4.62pt,8.0pt);
\draw[->,rhrhombiarrow] (-1.154pt,10.0pt) -- (5.775pt,14.0pt);}}
\newcommand{\rhaollM}{\rhtikzrhombus{\draw[rhrhombithickside] (-4.618pt,8.0pt) -- (0.0010pt,0.0pt);
\draw[->,rhrhombiarrow] (1.155pt,6.0pt) -- (-5.773pt,2.0pt);}}
\newcommand{\rhaollW}{\rhtikzrhombus{\draw[rhrhombithickside] (-4.618pt,8.0pt) -- (0.0010pt,0.0pt);
\draw[->,rhrhombiarrow] (-5.773pt,2.0pt) -- (1.156pt,6.0pt);}}
\newcommand{\rhaolrM}{\rhtikzrhombus{\draw[rhrhombithickside] (0.0010pt,0.0pt) -- (4.62pt,8.0pt);
\draw[->,rhrhombiarrow] (-1.154pt,6.0pt) -- (5.775pt,2.0pt);}}
\newcommand{\rhaolrW}{\rhtikzrhombus{\draw[rhrhombithickside] (0.0010pt,0.0pt) -- (4.62pt,8.0pt);
\draw[->,rhrhombiarrow] (5.774pt,2.0pt) -- (-1.154pt,6.0pt);}}

\newcommand{\rhalrM}{\rhtikzrhombus{\draw[rhrhombithickside] (4.62pt,8.0pt) -- (9.239pt,0.0pt);
\draw[->,rhrhombiarrow] (10.393pt,6.0pt) -- (3.465pt,2.0pt);}}

\newcommand{\rhpcMl}{\rhtikzrhombus{\draw[rhrhombipatharrow,-my] (0.0pt,8.0pt) arc (0:60:4.619pt);}}
\newcommand{\rhpcMr}{\rhtikzrhombus{\draw[rhrhombipatharrow,-my] (0.0pt,8.0pt) arc (180:120:4.619pt);}}
\newcommand{\rhpcMll}{\rhtikzrhombus{\draw[rhrhombipatharrow,-my] (0.0pt,8.0pt) arc (0:60:4.619pt) arc (60:120:4.619pt);}}
\newcommand{\rhpcMrl}{\rhtikzrhombus{\draw[rhrhombipatharrow,-my] (0.0pt,8.0pt) arc (180:120:4.619pt) arc (300:360:4.619pt);}}
\newcommand{\rhpcWl}{\rhtikzrhombus{\draw[rhrhombipatharrow,-my] (0.0pt,8.0pt) arc (180:240:4.619pt);}}
\newcommand{\rhpcWr}{\rhtikzrhombus{\draw[rhrhombipatharrow,-my] (0.0pt,8.0pt) arc (0:-60:4.619pt);}}
\newcommand{\rhpcWrr}{\rhtikzrhombus{\draw[rhrhombipatharrow,-my] (0.0pt,8.0pt) arc (0:-60:4.619pt) arc (300:240:4.619pt);}}
\newcommand{\rhpcWrl}{\rhtikzrhombus{\draw[rhrhombipatharrow,-my] (0.0pt,8.0pt) arc (0:-60:4.619pt) arc (120:180:4.619pt);}}
\newcommand{\rhpoulMl}{\rhtikzrhombus{\draw[rhrhombipatharrow,-my] (-2.309pt,12.0pt) arc (240:300:4.619pt);}}
\newcommand{\rhpoulMr}{\rhtikzrhombus{\draw[rhrhombipatharrow,-my] (-2.309pt,12.0pt) arc (60:0:4.619pt);}}
\newcommand{\rhpoulMll}{\rhtikzrhombus{\draw[rhrhombipatharrow,-my] (-2.309pt,12.0pt) arc (240:300:4.619pt) arc (300:360:4.619pt);}}
\newcommand{\rhpoulMlr}{\rhtikzrhombus{\draw[rhrhombipatharrow,-my] (-2.309pt,12.0pt) arc (240:300:4.619pt) arc (120:60:4.619pt);}}
\newcommand{\rhpoulMrl}{\rhtikzrhombus{\draw[rhrhombipatharrow,-my] (-2.309pt,12.0pt) arc (60:0:4.619pt) arc (180:240:4.619pt);}}
\newcommand{\rhpoulMrr}{\rhtikzrhombus{\draw[rhrhombipatharrow,-my] (-2.309pt,12.0pt) arc (60:0:4.619pt) arc (0:-60:4.619pt);}}
\newcommand{\rhpoulMrrl}{\rhtikzrhombus{\draw[rhrhombipatharrow,-my] (-2.309pt,12.0pt) arc (60:0:4.619pt) arc (0:-60:4.619pt) arc (120:180:4.619pt);}}
\newcommand{\rhpourMl}{\rhtikzrhombus{\draw[rhrhombipatharrow,-my] (2.31pt,12.0pt) arc (120:180:4.619pt);}}
\newcommand{\rhpourMr}{\rhtikzrhombus{\draw[rhrhombipatharrow,-my] (2.31pt,12.0pt) arc (300:240:4.619pt);}}
\newcommand{\rhpourMll}{\rhtikzrhombus{\draw[rhrhombipatharrow,-my] (2.31pt,12.0pt) arc (120:180:4.619pt) arc (180:240:4.619pt);}}
\newcommand{\rhpourMlr}{\rhtikzrhombus{\draw[rhrhombipatharrow,-my] (2.31pt,12.0pt) arc (120:180:4.619pt) arc (0:-60:4.619pt);}}
\newcommand{\rhpollMl}{\rhtikzrhombus{\draw[rhrhombipatharrow,-my] (-2.309pt,4.0pt) arc (120:180:4.619pt);}}
\newcommand{\rhpollMr}{\rhtikzrhombus{\draw[rhrhombipatharrow,-my] (-2.309pt,4.0pt) arc (300:240:4.619pt);}}
\newcommand{\rhpollMrl}{\rhtikzrhombus{\draw[rhrhombipatharrow,-my] (-2.309pt,4.0pt) arc (300:240:4.619pt) arc (60:120:4.619pt);}}
\newcommand{\rhpollMrr}{\rhtikzrhombus{\draw[rhrhombipatharrow,-my] (-2.309pt,4.0pt) arc (300:240:4.619pt) arc (240:180:4.619pt);}}
\newcommand{\rhpollWl}{\rhtikzrhombus{\draw[rhrhombipatharrow,-my] (-2.309pt,4.0pt) arc (300:360:4.619pt);}}
\newcommand{\rhpollWr}{\rhtikzrhombus{\draw[rhrhombipatharrow,-my] (-2.309pt,4.0pt) arc (120:60:4.619pt);}}
\newcommand{\rhpollWll}{\rhtikzrhombus{\draw[rhrhombipatharrow,-my] (-2.309pt,4.0pt) arc (300:360:4.619pt) arc (0:60:4.619pt);}}
\newcommand{\rhpollWlr}{\rhtikzrhombus{\draw[rhrhombipatharrow,-my] (-2.309pt,4.0pt) arc (300:360:4.619pt) arc (180:120:4.619pt);}}
\newcommand{\rhpolrMl}{\rhtikzrhombus{\draw[rhrhombipatharrow,-my] (2.31pt,4.0pt) arc (240:300:4.619pt);}}
\newcommand{\rhpolrWl}{\rhtikzrhombus{\draw[rhrhombipatharrow,-my] (2.31pt,4.0pt) arc (60:120:4.619pt);}}
\newcommand{\rhpolrWr}{\rhtikzrhombus{\draw[rhrhombipatharrow,-my] (2.31pt,4.0pt) arc (240:180:4.619pt);}}
\newcommand{\rhpolrWll}{\rhtikzrhombus{\draw[rhrhombipatharrow,-my] (2.31pt,4.0pt) arc (60:120:4.619pt) arc (120:180:4.619pt);}}
\newcommand{\rhpolrWlr}{\rhtikzrhombus{\draw[rhrhombipatharrow,-my] (2.31pt,4.0pt) arc (60:120:4.619pt) arc (300:240:4.619pt);}}
\newcommand{\rhpolrWrl}{\rhtikzrhombus{\draw[rhrhombipatharrow,-my] (2.31pt,4.0pt) arc (240:180:4.619pt) arc (0:60:4.619pt);}}
\newcommand{\rhpolrWrr}{\rhtikzrhombus{\draw[rhrhombipatharrow,-my] (2.31pt,4.0pt) arc (240:180:4.619pt) arc (180:120:4.619pt);}}

\newcommand{\rhpulMll}{\rhtikzrhombus{\draw[rhrhombipatharrow,-my] (-6.928pt,12.0pt) arc (120:180:4.619pt) arc (180:240:4.619pt);}}
\newcommand{\rhpulWr}{\rhtikzrhombus{\draw[rhrhombipatharrow,-my] (-6.928pt,12.0pt) arc (120:60:4.619pt);}}
\newcommand{\rhpulWrl}{\rhtikzrhombus{\draw[rhrhombipatharrow,-my] (-6.928pt,12.0pt) arc (120:60:4.619pt) arc (240:300:4.619pt);}}
\newcommand{\rhpllMl}{\rhtikzrhombus{\draw[rhrhombipatharrow,-my] (-6.928pt,4.0pt) arc (240:300:4.619pt);}}
\newcommand{\rhpllMr}{\rhtikzrhombus{\draw[rhrhombipatharrow,-my] (-6.928pt,4.0pt) arc (60:0:4.619pt);}}
\newcommand{\rhpllMll}{\rhtikzrhombus{\draw[rhrhombipatharrow,-my] (-6.928pt,4.0pt) arc (240:300:4.619pt) arc (300:360:4.619pt);}}

\newcommand{\rhpplMrrl}{\rhtikzrhombus{\draw[rhrhombipatharrow,-my] (-9.238pt,8.0pt) arc (180:120:4.619pt) arc (120:60:4.619pt) arc (240:300:4.619pt);}}
\newcommand{\rhppulWl}{\rhtikzrhombus{\draw[rhrhombipatharrow,-my] (-4.619pt,16.0pt) arc (180:240:4.619pt);}}
\newcommand{\rhppulWr}{\rhtikzrhombus{\draw[rhrhombipatharrow,-my] (-4.619pt,16.0pt) arc (0:-60:4.619pt);}}
\newcommand{\rhppulWll}{\rhtikzrhombus{\draw[rhrhombipatharrow,-my] (-4.619pt,16.0pt) arc (180:240:4.619pt) arc (240:300:4.619pt);}}
\newcommand{\rhppulWlr}{\rhtikzrhombus{\draw[rhrhombipatharrow,-my] (-4.619pt,16.0pt) arc (180:240:4.619pt) arc (60:0:4.619pt);}}
\newcommand{\rhppulWrl}{\rhtikzrhombus{\draw[rhrhombipatharrow,-my] (-4.619pt,16.0pt) arc (0:-60:4.619pt) arc (120:180:4.619pt);}}
\newcommand{\rhppulWlrr}{\rhtikzrhombus{\draw[rhrhombipatharrow,-my] (-4.619pt,16.0pt) arc (180:240:4.619pt) arc (60:0:4.619pt) arc (0:-60:4.619pt);}}
\newcommand{\rhppulWrll}{\rhtikzrhombus{\draw[rhrhombipatharrow,-my] (-4.619pt,16.0pt) arc (0:-60:4.619pt) arc (120:180:4.619pt) arc (180:240:4.619pt);}}
\newcommand{\rhppulWrlr}{\rhtikzrhombus{\draw[rhrhombipatharrow,-my] (-4.619pt,16.0pt) arc (0:-60:4.619pt) arc (120:180:4.619pt) arc (0:-60:4.619pt);}}
\newcommand{\rhppllMr}{\rhtikzrhombus{\draw[rhrhombipatharrow,-my] (-4.619pt,0.0pt) arc (180:120:4.619pt);}}
\newcommand{\rhppllMrr}{\rhtikzrhombus{\draw[rhrhombipatharrow,-my] (-4.619pt,0.0pt) arc (180:120:4.619pt) arc (120:60:4.619pt);}}
\newcommand{\rhpxulMl}{\rhtikzrhombus{\draw[rhrhombipatharrow,-my] (-11.547pt,12.0pt) arc (240:300:4.619pt);}}

\newcommand{\rhpxulMrr}{\rhtikzrhombus{\draw[rhrhombipatharrow,-my] (-11.547pt,12.0pt) arc (60:0:4.619pt) arc (0:-60:4.619pt);}}

\newcommand{\rhpoulMlXoulMr}{\rhtikzrhombus{\draw[rhrhombipatharrow,-my] (-2.309pt,12.0pt) arc (240:300:4.619pt);\draw[rhrhombipatharrow,-my] (-2.309pt,12.0pt) arc (60:0:4.619pt);}}
\newcommand{\rhpoulMlXolrWl}{\rhtikzrhombus{\draw[rhrhombipatharrow,-my] (-2.309pt,12.0pt) arc (240:300:4.619pt);\draw[rhrhombipatharrow,-my] (2.31pt,4.0pt) arc (60:120:4.619pt);}}
\newcommand{\rhpoulMrXourMl}{\rhtikzrhombus{\draw[rhrhombipatharrow,-my] (-2.309pt,12.0pt) arc (60:0:4.619pt);\draw[rhrhombipatharrow,-my] (2.31pt,12.0pt) arc (120:180:4.619pt);}}
\newcommand{\rhpulWrXpulWl}{\rhtikzrhombus{\draw[rhrhombipatharrow,-my] (-6.928pt,12.0pt) arc (120:60:4.619pt);\draw[rhrhombipatharrow,-my] (-4.619pt,16.0pt) arc (180:240:4.619pt);}}

\date{\today}

\MakeRobustCommand\mbox
\MakeRobustCommand\hbox
\MakeRobustCommand\unhbox

\newcommand{\nopar}{}

\let\oldsection=\section
\renewcommand{\section}{\setcounter{theorem}{0}\oldsection}

\let\oldref=\ref
\newcommand{\newref}[1]{\textup{\oldref{#1}}}
\renewcommand{\ref}[1]{\newref{#1}}

\newcommand{\neweqref}[1]{(\textup{\oldref{#1}})}
\renewcommand{\eqref}[1]{\neweqref{#1}}

\newcommand{\N}{\ensuremath{\mathbb{N}}}
\newcommand{\Z}{\ensuremath{\mathbb{Z}}}
\newcommand{\R}{\ensuremath{\mathbb{R}}}
\newcommand{\C}{\ensuremath{\mathbb{C}}}

\newcommand{\gO}{\ensuremath{\mathcal{O}}}
\newcommand{\LRC}[3]{\ensuremath{c_{#1,#2}^{#3}}}

\newcommand{\res}{\ensuremath{R}}
\newcommand{\resf}[1]{\ensuremath{R_{#1}}}

\newcommand{\resfk}[2]{\ensuremath{R_{#1}^{#2}}}

\newcommand{\loz}{\ensuremath{\lozenge}}
\newcommand{\s}[2]{\sigma(#1,#2)}
\newcommand{\sgn}{{\mathrm{sgn}}}
\newcommand{\SUPP}{{\mathrm{supp}}}

\newcommand{\inflow}[3]{\delta(#2,\begin{tikzpicture}[scale=0.3,baseline=-.1cm]\draw[->] (0,0)--(1,0);\end{tikzpicture}#1, #3)}
\newcommand{\outflow}[3]{\delta(#2,#1\begin{tikzpicture}[scale=0.3,baseline=-.1cm]\draw[->] (0,0)--(1,0);\end{tikzpicture}, #3)}
 
\newcommand{\inom}[3]{\omega(#2,\begin{tikzpicture}[scale=0.3,baseline=-.1cm]\draw[->] (0,0)--(1,0);\end{tikzpicture}#1, #3)}
\newcommand{\outom}[3]{\omega(#2,#1\begin{tikzpicture}[scale=0.3,baseline=-.1cm]\draw[->] (0,0)--(1,0);\end{tikzpicture}, #3)}

\newcommand{\Pin}[2]{\cP_f (#1\begin{tikzpicture}[scale=0.3,baseline=-.1cm]\draw[->] (0,0)--(1,0);\end{tikzpicture}, #2)}
\newcommand{\Pout}[2]{\cP_f (\begin{tikzpicture}[scale=0.3,baseline=-.1cm]\draw[->] (0,0)--(1,0);\end{tikzpicture}#1, #2)}

\newcommand{\sP}{\bf \ensuremath{\#\P}}
\def\GL{\mathsf{GL}}
\def\CP{\sP}
\def\P{\mathrm{\bf P}}
\def\NP{\mathrm{\bf NP}}
\def\la{\lambda}

\newcommand{\pstart}[1]{\textsf{start}({#1})}
\newcommand{\pend}[1]{\textsf{end}({#1})}
\let\lparen=(\let\rparen=)

\DeclareMathOperator{\dist}{dist}

\def\True{\textbf{TRUE}\xspace}
\def\False{\textbf{FALSE}\xspace}
\newcommand{\p}{\ensuremath{\Psi}}

\newcommand\dash{\nobreakdash-\hspace{0pt}}
\newcommand\stturnpath{$s$\dash $t$\dash turnpath\xspace}
\newcommand\stturnpaths{$s$\dash $t$\dash turnpaths\xspace}

\NewEnviron{tikzpictured}{\smash{\scalebox{0.7}{\begin{tikzpicture}[baseline=4pt]\BODY\end{tikzpicture}}}\ }
\NewEnviron{tikzpicturednosmash}{{\scalebox{0.7}{\begin{tikzpicture}[baseline=4pt]\BODY\end{tikzpicture}}}\ }

\newcommand{\blsquare}{\scalebox{0.7}{\ensuremath{\blacksquare}}} 

\newcommand{\oF}{\overline{F}}
\newcommand{\iflow}{\mathsf{in}}
\newcommand{\oflow}{\mathsf{out}}
\newcommand{\cP}{\mathcal{P}}
\newcommand{\cS}{\mathcal{S}}

\newcommand{\inflo}{\mathsf{inflow}}
\newcommand{\outflo}{\mathsf{outflow}}
\newcommand{\Ke}{{K}}
\newcommand{\eins}{{\bf 1}}

\newtheorem{claim}[theorem]{Claim}

\newtheorem{observation}[theorem]{Observation}
\newtheorem{remark}[theorem]{Remark}
\newtheorem{example}[theorem]{Example}

\newenvironment{prooff}[1][\unskip]{{\em #1}.}{\endproof}

\renewcommand{\paragraph}[1]{

\textbf{#1.}}

\let\brleft=(
\let\brright=)

\title{Deciding Positivity of \mbox{Littlewood--Richardson Coefficients}}
\author{Peter B\"urgisser\thanks{Institute of Mathematics, University of Paderborn, D-33098 Paderborn, Germany, pbuerg@math.upb.de, partially supported by DFG grant BU 1371/3-2.} \and Christian Ikenmeyer\thanks{Institute of Mathematics, University of Paderborn, D-33098 Paderborn, Germany, ciken@math.upb.de, partially supported by DFG grant BU 1371/3-2.}}

\newcommand{\runningtime}{$\gO\big(\ell(\nu)^3 \log\nu_1\big)$}

\begin{document}
\sloppy

\maketitle
\begin{abstract} 
Starting with Knutson and Tao's hive model~\cite{knta:99}, 
we characterize the Littlewood--Richardson coefficient
$\LRC{\lambda}{\mu}{\nu}$ of given partitions $\la,\mu,\nu\in\N^n$
as the number of capacity achieving hive flows on the honeycomb graph.
Based on this, we design a polynomial time algorithm for deciding
$\LRC{\lambda}{\mu}{\nu} >0$. This algorithm is easy to state 
and takes $\gO\big(n^3\log\nu_1\big)$ arithmetic operations and comparisons.
We further show that the capacity achieving hive flows can be seen as 
the vertices of a connected graph, which leads to new structural insights into 
Littlewood--Richardson coefficients. 
\end{abstract}



%
%




\section{Introduction}

Let $\la,\mu,\nu\in\Z^n$ be nonincreasing $n$-tuples of integers. 
The {\em Littlewood--Richardson coefficient} $\LRC{\lambda}{\mu}{\nu}$
is defined as the multiplicity of the irreducible $\GL_n(\C)$-representation $V_\nu$ 
with dominant weight $\nu$ in the tensor product $V_\la\otimes V_\mu$.
These coefficients appear not only in representation theory and 
algebraic combinatorics, but also in topology and enumerative geometry 
(Schubert calculus): 
for instance, they 
determine the multiplication 
in the cohomology ring of the Grassmann varieties.
Littlewood--Richardson coefficients gained further prominence 
due to their role in the proof of Horn's conjecture 
\cite{hero:95,klya:98,knta:99,ktw:04} 
on the relation of the eigenvalues 
of a triple $A,B,C$ of Hermitian matrices satisfying $C=A+B$.
The latter problem is of relevance in perturbation and quantum information theory. 
We refer to Fulton~\cite{fult:00} for an excellent account of these more recent developments.

Different combinatorial characterizations of the Littlewood--Richardson coefficients are known. 
The classic Littlewood--Richardson rule (cf.~\cite{fult:97}) counts certain skew tableaux, 
while in Berenstein and Zelevinsky~\cite{bz:92}, 
the number of integer points of certain polytopes are counted.  
A beautiful characterization was given by Knutson and Tao~\cite{knta:99}, 
who characterized Littlewood--Richardson coefficients either as the number of honeycombs 
or hives with prescribed boundary conditions. 

The focus of this paper is on the complexity of computing the Littlewood--Richardson coefficient 
$\LRC{\lambda}{\mu}{\nu}$ on input $\la,\mu,\nu$. Without loss of generality we assume that 
the components of $\la,\mu,\nu$ are nonnegative integers and put 
$|\la| := \sum_i \la_i$. Moreover we write 
$\ell(\la)$ for the number of nonzero components of~$\la$.
Then $|\nu| = |\la| + |\mu|$ and $\nu_1\geq \max\{\la_1,\mu_1\}$ are necessary conditions for $\LRC{\lambda}{\mu}{\nu} > 0$. 
We think of $\la,\mu,\nu$ as encoded in binary 
and interpret 
$\sum_{i=1}^n (\log\la_i + \log\mu_i + \log\nu_i) \le 3 n\log\nu_1$
as a measure of the input size.
All the algorithms derived from the above mentioned characterizations of Littlewood--Richardson coefficients 
take exponential time in the worst case. 
Narayanan~\cite{nara:06} proved that this is unavoidable:
the computation of $\LRC{\la}{\mu}{\nu}$ is a $\CP$-complete problem.
Hence there does not exist a polynomial time algorithm
for computing $\LRC{\la}{\mu}{\nu}$ under the widely believed hypothesis $\P\ne\NP$.


\smallskip

\noindent {\bf Main results.} 
We first characterize $\LRC{\la}{\mu}{\nu}$ as the number of 
{\em capacity achieving hive flows} on the honeycomb graph~$G$, 
cf.\ Figures~\ref{fig:diagrams}--\ref{fig:degengraph}.  
Besides capacity constraints given by $\la,\mu,\nu$, these flows have to satisfy 
rhombus inequalities corresponding to the ones considered in~\cite{knta:99, buc:00}.  
We then develop a polynomial time algorithm  (Algorithm~\ref{alg:lrpcsa}) 
for deciding $\LRC{\la}{\mu}{\nu}>0$ 
with $\gO\big(n^3\log\nu_1\big)$ arithmetic operations and comparisons. 
This is basically a capacity scaling Ford-Fulkerson algorithm~\cite{ff:62} on well-chosen residual networks.
The algorithm is easy to state and implement: 
we encourage the reader to try out our Java applet at 
{\tt http://www-math.upb.de/agpb/flowapplet/flowapplet.html}.
We also show that the set of capacity achieving hive flows 
is the vertex set of a natural connected graph, which is  
relevant for efficiently enumerating these flows. 
In fact, our work is the basis of a follow-up paper~\cite{ike:12}, 
in which more algorithmic insights are obtained, notably an algorithm 
for deciding $\LRC{\la}{\mu}{\nu} \ge t$ in time polynomial in 
$n\log \nu_1$ and~$t$.
This implies that ``small'' Littlewood--Richardson coefficients 
can be efficiently computed. 
In \cite{ike:12} we also prove a conjecture on stretched
Littlewood--Richardson coefficients posed by King, Tollu, and Toumazet
in~\cite{ktt:04}.

\smallskip

\noindent {\bf Motivation and previous work.} 
Our investigations are motivated by Geometric Complexity Theory, 
an approach towards proving fundamental complexity lower bounds 
by means of algebraic geometry and representation theory, 
that was initiated by Mulmuley and Sohoni~\cite{gct1,gct2} 
(see \cite {mulmuley:11} for recent pointers to the literature). 
To the best of our knowledge, the existence of a polynomial time 
algorithm for deciding $\LRC{\la}{\mu}{\nu}>0$ was first pointed out in~\cite{deloera:06},
and one day later in~\cite{GCT3}.
Indeed, by the saturation property, 
$\LRC{\la}{\mu}{\nu}>0$ is equivalent to 
$\exists N\, \LRC {N\la}{N\mu}{N\nu}>0$, 
which can be rephrased as the feasibility problem of a certain rational polyhedron,
whose elements are called \emph{hives}.
Feasibility of rational polyhedra is a basic problem in linear programming,  
well-known to be solvable in polynomial time, cf.~\cite{grlos:93}.
More specifically, the existence of a hive with prescribed boundary 
conditions can be expressed as the feasibility of a system
$Ax \le b$ of linear inequalities, where the entries of the 
matrix~$A$ are in $\{-1,0,1\}$ and the components of the vector~$b$
are either zero or among the components of $\la,\mu,\nu$. 
For the format $M\times N$ of the matrix $A$ we have $M,N=\gO(n^2)$. 
The basic ellipsoid method (cf.~\cite[p.~80]{grlos:93}) 
for solving this feasibility problem takes 
$\gO\big( M N^3 \ell\big)$ arithmetic steps and comparisons, 
where $\ell= \gO(n^2 \log\nu_1)$ is the encoding length of the 
linear program.  This gives a bound of 
$\gO\big(n^{10}\log\nu_1\big)$, which is considerably worse 
than the bound $\gO\big(n^{3}\log\nu_1\big)$ proven for 
Algorithm~\ref{alg:lrpcsa} in this paper. 
We also note that standard interior point methods at least require 
$n^9\log (n \nu_1)$ arithmetic operations, cf.~\cite[Chap. 10]{bucu:13}. 

The starting point for the present work was a question in~\cite{GCT3} asking for a combinatorial algorithm for
deciding $\LRC{\la}{\mu}{\nu}>0$ in polynomial time, using ideas 
similar to the max-flow or weighted matching problems in combinatorial optimization. 

The algorithm in this paper is a considerable improvement over the one  
presented by the authors at FPSAC 2009~\cite{ike:08,bi:09}, 
both with regard to simplicity and running time.
The reason is that there, before each augmentation step, the flow had
to be substituted by a nondegenerate flow 
using a costly routine. 
(Nondegenerate meaning that small triangles and small rhombi are the only flatspaces, cf.~\cite{buc:00}
and Remark~\ref{re:comparison}.)
The present algorithm does not suffer from this deficiency anymore.

\smallskip

\noindent {\bf Outline of paper.}
Section~\ref{se:flowLRC} describes the setting and introduces the main terminology. 
We define the notion of hive flows on honeycomb graphs and associate with a triple 
$\la,\mu,\nu$ of partitions the polytope~$B(\la,\mu,\nu)$ of bounded hive flows, 
along with a linear function~$\delta$ measuring the overall throughput of a flow. 
A flow $f\in B(\la,\mu,\nu)$ is called \emph{capacity achieving} if $\delta(f) =|\nu|$. 
We denote by $P(\la,\mu,\nu)$ the polytope consisting of these flows and by 
$P(\la,\mu,\nu)_\Z$ the set of its integral points. 
It turns out that the Littlewood--Richardson coefficient $\LRC{\la}{\mu}{\nu}$ 
counts the elements of $P(\la,\mu,\nu)_\Z$
(Proposition~\ref{pro:flowdescription}). 
  
In Section~\ref{se:graph-str} we obtain some structural insights into the set of hive flows. 
We show that $P(\la,\mu,\nu)_\Z$ 
is the vertex set of a natural connected graph 
(\emph{Connectedness Theorem}~\ref{thm:connectedness}).
The connectedness 
immediately implies the property  
$\LRC{\la}{\mu}{\nu} = 1 \Rightarrow \forall N\,  \LRC{N\la}{N\mu}{N\nu} = 1$, 
that was conjectured by Fulton and proved in~\cite{ktw:04} 
(Corollary~\ref{thm:fultonconj}).
The connectedness of $P(\la,\mu,\nu)_\Z$ is also relevant for efficiently enumerating the points of~$P(\la,\mu,\nu)_\Z$
and for proving the implication
$\LRC{\la}{\mu}{\nu} = 2 \Rightarrow \forall N\,  \LRC{N\la}{N\mu}{N\nu} = N+1$,
which was conjectured by King, Tollu, and Toumazet~\cite{ktt:04}, cf.~\cite{ike:12}.

Proposition~\ref{pro:flowdescription} suggests to decide $\LRC \la \mu \nu > 0$ 
by optimizing the overall throughput function~$\delta$ 
on the polytope~$B(\la,\mu,\nu)$ of bounded hive flows. 
We imitate the basic Ford-Fulkerson idea and construct, 
for a given integral hive flow~$f$, 
a ``residual digraph'' $\resf f$, such that 
$f$~optimizes $\delta$ on $B(\la,\mu,\nu)$ iff $\resf f$ does not contain an $s$-$t$-path.
In Section~\ref{se:res} we define  the residual digraph $\resf f$ 
and study the  partition of the triangular graph into $f$\dash flatspaces. 
We present and analyze a first max-flow algorithm for deciding $\LRC \la \mu \nu > 0$ (Algorithm~\ref{alg:lrpa}).
The proof of correctness of this algorithm requires an in-depth understanding of the properties of hives 
and it has two main ingredients. 
The {\em Shortest Path Theorem}~\ref{thm:shopath} states that the rhombus inequalities are not violated after 
augmenting the current flow~$f$ by a shortest path in the residual network~$\resf f$.
This is remarkable since, unlike in the usual max-flow situation, the polytopes of hive flows are not integral, 
cf.~\cite{buc:00}. 
The other ingredient, needed for the optimality criterion, 
is the {\em Rerouting Theorem}~\ref{thm:rerouting}, which 
tells us how to replace an augmenting flow direction~$d$ 
by a flow in the residual network without changing the overall throughput.
This amounts to a rerouting of~$d$ 
along the borders of the flatspaces of the current flow~$f$
(cf.\ Figure~\ref{fig:innerandborder}).  
Here the main difficulty is the analysis of the degenerate situation of large flatspaces, 
a topic not pursued in detail in the previous papers~\cite{knta:99,buc:00}. 

In Section~\ref{sec:polytime} we state and analyze 
our polynomial time Algorithm~\ref{alg:lrpcsa} 
for deciding the positivity of Littlewood--Richardson coefficients.   

The remainder of the paper is devoted to the combinatorially quite intricate proofs 
of the Rerouting Theorem, the Shortest Path Theorem, and the Connectedness Theorem.

\section{Flow description of LR coefficients}
\label{se:flowLRC}


\subsection{Flows on digraphs}
\label{se:flows}

We fix some terminology regarding flows on directed graphs, compare~\cite{amo:93}.
Let $D$ be a digraph with vertex set $V(D)$ and edge set $E(D)$. 
We assume that $s,t \in V(D)$ are two different distinguished vertices,
called source and target, respectively.   
Let $e_-:=u$ denote the vertex where the edge~$e$ starts and 
$e_+:=v$ the vertex where $e$ ends. 
The {\em inflow} and {\em outflow} 
of a map $f\colon E(D)\rightarrow \R$ 
at a vertex $v\in V(D)$ are defined as 
$$
\inflo(v,f) := \sum\limits_{e_+=v} f(e),\quad 
\outflo(v,f) := \sum\limits_{e_-=v} f(e),
$$
respectively. 
A {\em flow on $D$} is defined as a map
$f\colon E(D)\rightarrow \R$ that satisfies 
Kirchhoff's conservation laws: 
$\inflo(v,f) = \outflo(v,f)$ for all 
$v \in V(D) \setminus \{s,t\}$. 

The set of flows on $D$ is a vector space that we denote by $F(D)$.
A flow is called {\em integral} if it takes only integer values and 
we denote by $F(D)_\Z$ the group of integral flows on $D$.
The quantity
$\delta(f) := \sum_{e_- = s} f(e) - \sum_{e_+ = s} f(e)$  
is called the \emph{overall throughput} of the flow~$f$.

By a {\em walk~$p$} in $D$ we understand a sequence 
$x_0,\ldots,x_\ell$ of vertices of~$D$ such that 
$(x_{i-1},x_i) \in E$ for all $1\le i\le\ell$. 
A {\em path~$p$} in $D$ is defined as a walk such that 
the vertices $x_0,\ldots,x_\ell$ are pairwise distinct.
We will say that  $x_0,\ldots,x_\ell$ are the {\em vertices used by~$p$}.  
The path~$p$ is called an 
{\em $s$-$t$-path} if $x_0=s$ and $x_\ell=t$; 
$p$ is called a {\em $t$-$s$-path} if $x_0=t$ and $x_\ell=s$. 
A sequence $x_0,\ldots,x_\ell$ of vertices of $D$ is called a 
{\em cycle~$c$} if $x_0,\ldots,x_{\ell-1}$ are pairwise distinct,
$x_{\ell}=x_0$, and  $(x_{i-1},x_i) \in E$ 
for all $1\le i\le\ell$. 
Again we say that $x_0,\ldots,x_\ell$ are the vertices used by~$c$. 
We call $c$ a {\em proper cycle} if $c$ does not use $s$ or $t$.
It will be sometimes useful to identify a path or a cycle with the set of its 
edges $\{ (x_0,x_1),\ldots,(x_{\ell-1},x_\ell)\}$.
Since the starting vertex~$x_0$ of a cycle is not relevant, this does not harm. 
By a {\em complete path~$p$ in~$D$} we understand an 
$s$-$t$-path, $t$-$s$-path, or a cycle in~$D$. 
(It is not excluded that the cycle passes through $s$ or $t$.)

A complete path~$p$ in $D$ defines a flow~$f$ on $D$ by setting 
$f(e) := 1$ if $e\in p$ and $f(e):=0$ otherwise. 
It will be convenient to denote this flow with~$p$ as well. 
We note that $\delta(p)=1$ for an $s$-$t$-path~$p$,
$\delta(p)=-1$ for a $t$-$s$-path~$p$,
and $\delta(c)=0$ for a cycle~$c$. 

A flow is called {\em nonnegative} if $f(e)\ge 0$ for all edges $e\in E$.
We call 
$\SUPP(f):=\{e\in E(D) \mid f(e) \ne 0\}$ the {\em support} of~$f$.

An important method for analyzing flows is the fact that they can 
be decomposed into paths and cycles~\cite{amo:93}.

\begin{lemma}\label{le:flow-decomp}
For any nonnegative flow $f\in F(D)$
there exists a family $p_1,\ldots,p_m$ 
of complete paths in~$D$ contained in $\SUPP(f)$, 
and positive real numbers
$\alpha_1,\ldots,\alpha_m$ such that 
$f= \sum_{i=1}^m \alpha_i p_i$ 
Moreover, if the flow $f$ is integral, 
then the $\alpha_i$ may be assumed to be integers.~\endproof
\end{lemma}

We will study flows in two rather different situations. 
The residual digraph $R$ introduced in Section~\ref{se:res}
has the property that it never contains an edge $(u,v)$ 
and its reverse edge $(v,u)$. Only nonnegative flows on $R$ 
will be of interest. 

On the other hand, we also need to look at flows 
on digraphs resulting from a undirected graph $G$ 
by replacing each of its undirected edges $\{u,v\}$ by 
the directed edge $e=(u,v)$ and its reverse  $-e := (v,u)$.
We shall denote the resulting digraph also by $G$. 
To a flow~$f$ on~$G$ we assign its 
{\em reduced representative} $\tilde{f}$ defined by 
$\tilde{f}(e) := f(e) - \min\{f(e),f(-e)\}$.
Hence $\tilde{f}(e) = f(e) -f(-e) \ge 0$ and $\tilde{f}(-e) =0$ 
if $f(e) \ge f(-e)$.
It will be convenient to interpret 
$f$ and $\tilde{f}$ as manifestations of the same flow.
Formally, we consider the linear subspace 
$N(G) := \{f \in \R^{E(G)} \mid \forall e \in E(G): f(e)=f(-e)\}$
of ``null flows'' and the factor space 
\begin{equation}\label{eq:def-oFG}
 \oF(G) := F(G) / N(G).
\end{equation}
We call the elements of $\oF(G)$ {\em flow classes on $G$} 
(or simply flows) and denote them by the same symbols as for flows. 
No confusion should arise from this abuse of notation in the context at hand. 
We usually identify flow classes with their reduced representative. 
We note that the overall throughput function factors to a linear function
$\delta\colon \oF(G) \to \R$. 
A flow class is called {\em integral} if its reduced representative is integral 
and we denote by $\oF(G)_\Z$ the group of integral flow classes on $G$.

We remark that in the literature on flows, the subtle distinction between flows and their classes 
is not relevant, as the goal usually is to optimize the throughput of a flow
subject to certain capacity constraints. But in the context of LR coefficients, 
we are interested in {\em counting} the number of capacity achieving flow classes, 
so that this distinction is necessary. 

\subsection{Flows on the honeycomb graph $G$}
\label{se:honey}

We start with a triangular array of vertices, $n+1$ on each side, as seen in Figure~\ref{fig:triangulararray}.
The resulting planar graph $\Delta$ shall be called the {\em triangular graph} with parameter~$n$, 
we denote its vertex set with $V(\Delta)$ and its edge set with $E(\Delta)$. 
A triangle consisting of three edges in~$\Delta$ is called a \emph{hive triangle}.
Note that there are two types of hive triangles: upright and downright oriented ones. 
A \emph{rhombus} is defined to be the union of an upright and a downright hive triangle which share a common side.
In contrast to the usual geometric definition of the term \emph{rhombus}
we use this term here in this very restricted sense only.
Note that the angles at the corners of a rhombus are either acute of $60^\circ$ or obtuse of $120^\circ$.
Two distinct rhombi are called \emph{overlapping} if they share a hive triangle.
\begin{figure}[h] 
  \begin{center}
      \subfigure[The triangular graph $\Delta$ for $n=5$.]
        {\scalebox{2}{\FIGtriangulararray}
        \nopar\label{fig:triangulararray}} \hspace{0.5cm}
      \subfigure[The honeycomb graph $G$. Source, target and its incident edges are omitted.]
        {\scalebox{1}{\FIGdualtriangulararray}
        \nopar\label{fig:dualtriangulararray}} \hspace{0.5cm}
      \subfigure[The throughput capacities.] 
        {
\scalebox{0.78}{
\begin{tikzpicture}[scale=5]\draw[rhrhombidraw] (0.0pt,0.0pt) -- (12.5pt,21.65pt) -- (25.0pt,0.0pt) -- cycle;\draw[rhrhombidraw] (2.5pt,4.33pt) -- (22.5pt,4.33pt) ;\draw[rhrhombidraw] (5.0pt,8.66pt) -- (20.0pt,8.66pt) ;\draw[rhrhombidraw] (7.5pt,12.99pt) -- (17.5pt,12.99pt) ;\draw[rhrhombidraw] (10.0pt,17.32pt) -- (15.0pt,17.32pt) ;\draw[rhrhombidraw] (15.0pt,17.32pt) -- (5.0pt,0.0pt) ;\draw[rhrhombidraw] (17.5pt,12.99pt) -- (10.0pt,0.0pt) ;\draw[rhrhombidraw] (20.0pt,8.66pt) -- (15.0pt,0.0pt) ;\draw[rhrhombidraw] (22.5pt,4.33pt) -- (20.0pt,0.0pt) ;\draw[rhrhombidraw] (20.0pt,0.0pt) -- (10.0pt,17.32pt) ;\draw[rhrhombidraw] (15.0pt,0.0pt) -- (7.5pt,12.99pt) ;\draw[rhrhombidraw] (10.0pt,0.0pt) -- (5.0pt,8.66pt) ;\draw[rhrhombidraw] (5.0pt,0.0pt) -- (2.5pt,4.33pt) ;\draw[rhrhombithickside] (12.5pt,21.65pt) -- (15.0pt,17.32pt);\draw[->,rhrhombiarrow] (15.915pt,20.735pt) -- (11.585pt,18.235pt);\draw[rhrhombithickside] (15.0pt,17.32pt) -- (17.5pt,12.99pt);\draw[->,rhrhombiarrow] (18.415pt,16.405pt) -- (14.085pt,13.905pt);\node at (17.5pt,21.65pt) {{$\leq \la_1$}};\node at (20.0pt,17.32pt) {{$\leq \la_2$}};\draw[rhrhombithickside] (17.5pt,12.99pt) -- (20.0pt,8.66pt);\draw[->,rhrhombiarrow] (20.915pt,12.075pt) -- (16.585pt,9.575pt);\draw[rhrhombithickside] (20.0pt,8.66pt) -- (22.5pt,4.33pt);\draw[->,rhrhombiarrow] (23.415pt,7.745pt) -- (19.085pt,5.245pt);\draw[rhrhombithickside] (22.5pt,4.33pt) -- (25.0pt,0.0pt);\draw[->,rhrhombiarrow] (25.915pt,3.415pt) -- (21.585pt,0.915pt);\draw[rhrhombithickside] (25.0pt,0.0pt) -- (20.0pt,0.0pt);\draw[->,rhrhombiarrow] (22.5pt,-2.5pt) -- (22.5pt,2.5pt);\draw[rhrhombithickside] (20.0pt,0.0pt) -- (15.0pt,0.0pt);\draw[->,rhrhombiarrow] (17.5pt,-2.5pt) -- (17.5pt,2.5pt);\draw[rhrhombithickside] (15.0pt,0.0pt) -- (10.0pt,0.0pt);\draw[->,rhrhombiarrow] (12.5pt,-2.5pt) -- (12.5pt,2.5pt);\draw[rhrhombithickside] (10.0pt,0.0pt) -- (5.0pt,0.0pt);\draw[->,rhrhombiarrow] (7.5pt,-2.5pt) -- (7.5pt,2.5pt);\draw[rhrhombithickside] (5.0pt,0.0pt) -- (0.0pt,0.0pt);\draw[->,rhrhombiarrow] (2.5pt,-2.5pt) -- (2.5pt,2.5pt);\node at (22.5pt,12.99pt) {{$\leq \la_3$}};\node at (25.0pt,8.66pt) {{$\leq \la_4$}};\node at (27.5pt,4.33pt) {{$\leq \la_5$}};\node at (7.5pt,21.65pt) {{$\leq \nu_1$}};\node at (5.0pt,17.32pt) {{$\leq \nu_2$}};\node at (2.5pt,12.99pt) {{$\leq \nu_3$}};\node at (0.0pt,8.66pt) {{$\leq \nu_4$}};\node at (-2.5pt,4.33pt) {{$\leq \nu_5$}};\node at (22.5pt,-4.33pt) {\rotatebox{45}{$\leq \mu_1$}};\node at (17.5pt,-4.33pt) {\rotatebox{45}{$\leq \mu_2$}};\node at (12.5pt,-4.33pt) {\rotatebox{45}{$\leq \mu_3$}};\node at (7.5pt,-4.33pt) {\rotatebox{45}{$\leq \mu_4$}};\node at (2.5pt,-4.33pt) {\rotatebox{45}{$\leq \mu_5$}};\draw[rhrhombithickside] (12.5pt,21.65pt) -- (10.0pt,17.32pt);\draw[->,rhrhombiarrow] (13.415pt,18.235pt) -- (9.085pt,20.735pt);\draw[rhrhombithickside] (10.0pt,17.32pt) -- (7.5pt,12.99pt);\draw[->,rhrhombiarrow] (10.915pt,13.905pt) -- (6.585pt,16.405pt);\draw[rhrhombithickside] (7.5pt,12.99pt) -- (5.0pt,8.66pt);\draw[->,rhrhombiarrow] (8.415pt,9.575pt) -- (4.085pt,12.075pt);\draw[rhrhombithickside] (5.0pt,8.66pt) -- (2.5pt,4.33pt);\draw[->,rhrhombiarrow] (5.915pt,5.245pt) -- (1.585pt,7.745pt);\draw[rhrhombithickside] (2.5pt,4.33pt) -- (0.0pt,0.0pt);\draw[->,rhrhombiarrow] (3.415pt,0.915pt) -- (-0.915pt,3.415pt);\end{tikzpicture}
}
        \nopar\label{fig:triangulararraywithflow}}
    \caption{Graph constructions.} 
    \nopar\label{fig:diagrams}
  \end{center}
\end{figure}
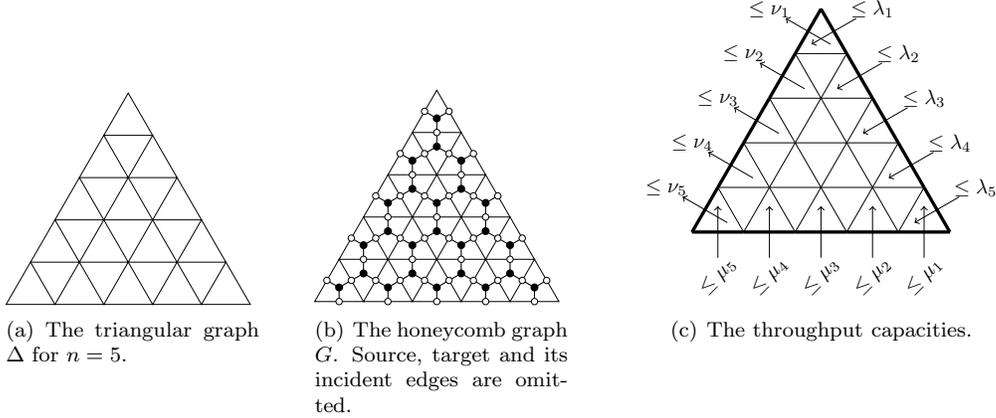

To realize the dual graph of $\Delta$, as in~\cite{buc:00}, 
we introduce a black vertex in the middle of each hive triangle
and a white vertex on each hive triangle side, see Figure~\ref{fig:dualtriangulararray}.
Moreover, in each hive triangle~$T$, we introduce edges connecting 
the three white vertices of $T$ with the black vertex. 
Additionally (not depicted in Figure~\ref{fig:dualtriangulararray}),  
we introduce a source vertex~$s$ and a target vertex~$t$. 
The source $s$ is connected by an edge with each white vertex~$v$ 
on the right or on the bottom border of $\Delta$,
and the target~$t$ is connected by an edge with each white vertex~$v$ 
on the left border of $\Delta$.
Clearly, the resulting (undirected) graph~$G$ is bipartite and planar. 
We shall call $G$ the {\em honeycomb graph} with parameter~$n$. 

We study now the vector space $\oF(G)$ of flow classes on $G$ 
introduced in Section~\ref{se:flows}. Recall that for this, 
we have to replace each edge of $G$ by the corresponding two directed edges.  
Correspondingly, we will consider $G$ as a directed graph.
Any complete path~$p$ in the digraph $G$ defines a flow 
and thus a flow class on $G$, that we denote by~$p$ as well. 
According to Lemma~\ref{le:flow-decomp} we can write 
each flow class $f\in\oF(G)$ as a nonnegative linear combination of complete paths. 
(Note that the reduced representative of any flow on~$G$ is nonnegative.) 


In order to characterize the flow class $f\in \oF(G)$ in a concise way, we introduce the 
notion of the throughput of $f$ through edges of $\Delta$. 
For each edge $k\in E(\Delta)$, 
there is exactly one upright hive triangle having $k$ as a side: 
let $e_k\in E(G)$ denote the directed edge in this triangle pointing from the white vertex on~$k$ 
towards the black vertex in this upright triangle. 
Then we call $\delta(k,f) :=  f(e_k)- f(-e_k)$ the {\em throughput of $f$ through $k$}, 
which is clearly independent of the choice of the representative.
As for the choice of sign: this should be interpreted as the total flow of~$f$ 
{\em going into the upright hive triangle} through~$k$. 
Note that $\oF(G)\to\R, f\mapsto \delta(k,f)$ is a linear form.

It is obvious that a flow class~$f$ on $G$ is completely determined 
by the throughput function $\delta\colon E(\Delta)\to\R, k\mapsto\delta(k,f)$. 
Furthermore, Kirchhoff's conservation laws 
translate to the closedness condition
\begin{equation}\label{eq:closedness}
\delta(k_1,f) + \delta(k_2,f) + \delta(k_3,f) =0
\end{equation}
holding for each hive triangle (upright or downright) 
with sides denoted by $k_1,k_2,k_3$. 
So we see that the vector space $\oF(G)$ of flow classes on~ $G$ can be 
identified with the subvector space $Z\subseteq\R^{E(\Delta)}$ 
consisting of the functions $\delta$ satisfying~\eqref{eq:closedness} 
for all hive triangles. 
Moreover, under this identification, 
integral flow classes~$f$ correspond to functions 
in the subgroup $Z_\Z$ consisting of 
functions $\delta$ taking integer values. 


By adding up \eqref{eq:closedness} for all upright hive triangles
and subtracting \eqref{eq:closedness} for all downright hive triangles,
taking into account the cancelling of throughputs on all 
inner sides~$k$, we see that the sum of $\delta(k,f)$ over 
all border edges~$k$ of $\Delta$ vanishes. 
Therefore, we can express the overall throughput $\delta(f)$ as 
\begin{equation}\label{eq:overall-thrput}
 \delta(f) \ =\  \sum_{k\in E_r \cup E_b} \delta(k,f) \ =\ -\sum_{k'\in E_\ell} \delta(k',f) ,
\end{equation}
where $E_\ell$, $E_r$, and $E_b$ denotes the set of edges of $\Delta$ 
on the left side, right side, and bottom side, respectively.

%

The flow classes on $G$ can be characterized in yet another way. 
Let $x_0$ be the top vertex of~$\Delta$
and define the vector space $H$ of functions $h\colon V(\Delta)\to\R$ 
satisfying $h(x_0)=0$. We denote by $H_\Z$ the subgroup of
functions~$h\in H$ taking integer values.

For a moment, think of the edges~$k$ of $\Delta$ as oriented  
such that all upright hive triangles get clockwise oriented. 
Consider the linear map 
$\partial\colon H \to \R^{E(\Delta)}, h \mapsto \delta$ defined by 
$\delta(k) = h(k_+) - h(k_-)$, 
where $k$ points from $k_-$ to $k_+$. 
Then it is obvious that $\partial$ is injective, and it is 
straightforward to check that $\mathrm{im}\partial\subseteq Z$. 
In order to show equality, suppose $\delta\in Z$. 
For a vertex $x\in V(\Delta)$, choose a directed path~$p$ 
(in the sense of the above chosen orientations)
from the top vertex~$x_0$ to $x$. 
The closedness condition~\eqref{eq:closedness} 
easily implies that the sum
$h(x) :=\sum_{k\subseteq p} \delta(k)$
is independent of the choice of~$p$. 
It follows that $\partial(h)=\delta$. 
 
So we have a linear isomorphism $\partial\colon H\to Z$, which  
induces an isomorphism $H_\Z \to Z_\Z$. 

\begin{remark}
The reader familiar with basic algebraic topology will recognize 
$\partial$ as a coboundary map of the simplicial complex provided
by $\Delta$, and hence $\mathrm{im}\partial = Z$ as a consequence of the 
fact that the triangle underlying $\Delta$ is simply connected. 
\end{remark}



\subsection{Hives and hive flows}



Following~\cite{knta:99,buc:00} we define a {\em hive} on $\Delta$
as a function $h\in H$ 
such that for all rhombi~$\varrho$, 
the sum of the values of~$h$ at the two obtuse vertices of~$\varrho$  
is greater than or equal to the sum of the values of~$h$ at the two acute vertices 
of $\varrho$. In pictorial notation, 
\begin{equation}\label{eq:hive-inequ}
 \sigma(\rhc,h) := h\big(\rhdwest\big)+h\big(\rhdeast\big) - h\big(\rhdnorth\big) - h\big(\rhdsouth\big) \ \ge\  0
\end{equation}
where $\rhdwest, \rhdeast, \rhdnorth, \rhdsouth\in V(\Delta)$ denote the corner vertices of $\rhc$.
We call the $\sigma(\varrho,h)$ the {\em slack} of the rhombus $\varrho$ with respect to the hive~$h$. 


If one interprets $h(v)$ as the height of a point over~$v\in V(\Delta)$ and interpolates these points linearly over 
each hive triangle of $\Delta$ one gets a continuous function $h\colon\Delta\to\R$.
(Here the triangle $\Delta$ is to be interpreted as a convex subset of $\R^2$.)
Then the conditions \eqref{eq:hive-inequ} mean that 
$h$~is a concave function. 
The function $h$ is linear over a rhombus $\varrho$ iff $\s \varrho h =0$, 
in which case we call the rhombus {\em $\varrho$ $h$-flat}.

\begin{lemma}\label{le:bbound}
For a hive $h\in H$ and $x\in V(\Delta)$ we have $\min_{\partial\Delta} h \ \le\   h(x) \ \le\ n\max_{\partial\Delta} h$,  
where $\partial\Delta$ denotes the boundary of the convex set $\Delta\subseteq\R^2$. 
\end{lemma}

\begin{proof} 
Let $x(m,i)$ denote the vertex of $\Delta$ in the $m$th line parallel to the ground side (counting from the top) 
and on the $i$th side parallel to the left side (counting from the left), for $0\le i \le m\le n$. 
So $x(0,0)$ is the top vertex and $h(x(0,0))=0$ for $h\in H$. 
Put $a:=h(x(1,0))$ and $b:=h(x(1,1))$. 

Since $h$ is a concave function, its subgraph 
$S:=\{ (x,y) \in \Delta\times\R \mid y \le h(x)\}$ is convex. 
Hence $S$ is bounded from above by the plane spanned by $((0,0),0)$, $((1,0),a)$, and $((1,1),b)$.
This plane's height at $x(m,i)$ is $a m + (b-a) i$.
This implies that 
$h(x(m,i)) \le a m + (b-a) i$ 
for all $h\in H$. 
Therefore,
$h(x(m,i)) \le m\max\{a,b\} $, 
proving the upper bound. 

The lower bound follows easily from the convexity of $S$.
\qquad\end{proof}


In this paper, it will be extremely helpful to have some graphical way 
of describing rhombi and throughputs.
We shall denote a rhombus $\varrho$ of $\Delta$ by the pictogram $\rhc$, 
even though $\varrho$ may lie in any of the three positions 
``$\rhc$'', ``\ \rotatebox{120}{$\rhc$}'' or ``\raisebox{0.2cm}{\rotatebox{240}{$\rhc$}}\ \ ''
obtained by rotating with a multiple of $60^\circ$. 
Let $\rhsc$ denote the edge~$k$ of~$\Delta$ given by the \emph{diagonal} of $\varrho$ 
connecting its two obtuse angles.
Then we denote by $\rhacM (f) := \delta(k,f)$ the throughput of $f$ through~$k$ 
(going into the upright hive triangle). 
Similarly, we define the throughput 
$ \rhacW (f) := - \delta(k,f)$. 
The advantage of this notation is that 
if the throughput is positive, then the flow goes 
in the direction of the arrow.
For instance, using the symbolic notation, 
we note the following consequence of the flow conservation laws: 
\begin{equation}\label{eq:fl-conv}
 \rhaollW(f) + \rhaolrW(f) \ =\  \rhaoulW(f) + \rhaourW(f) .
\end{equation}


If $f$ is the flow corresponding to the hive $h\in H$
under the isomorphisms $H\simeq Z\simeq \oF(G)$, 
then \eqref{eq:hive-inequ} 
and the definition of the coboundary map~$\partial$ imply that
$$
 \sigma(\rhc,h)  \ =\ 
  \Big(h\big(\rhdwest\big)  - h\big(\rhdnorth\big) \Big) + 
  \Big( h\big(\rhdeast\big) - h\big(\rhdsouth\big) \Big)
 \ =\ 
 \rhaoulW(f) +  \rhaolrM(f) .
$$ 
We define now the slack of a rhombus with respect to a flow~$f$ 
as the slack with respect to the corresponding hive~$h$. 

\begin{definition}\label{def:slack}
The \emph{slack} of the rhombus $\rhc$ with respect to $f\in \oF(G)$ is defined as 
$$
\s \rhc f \ :=\ \rhaoulW(f) + \rhaolrM(f).
$$
The rhombus $\rhc$ is called \emph{$f$\dash flat} if $\s \rhc f$ = 0.
\end{definition}

It is clear that $\oF(G)\to\R, f\mapsto \s \varrho f$ is a linear form.
Note also that by~\eqref{eq:fl-conv},
the slack can be written in varous different ways: 
$$
 \s \rhc f \ =\ \rhaoulW(f) +  \rhaolrM(f) \ = \  
 \rhaoulW(f) -  \rhaolrW(f) \ =\  \rhaollW(f) - \rhaourW(f) 
 \ = \ \rhaollW(f) + \rhaourM(f) .
$$



\begin{definition}\label{def:hiveflow}
A flow $f \in \oF(G)$ is called a \emph{hive flow} iff $\s \varrho f \geq 0$ for all rhombi~$\varrho$ in $\Delta$.
\end{definition}

By definition, the hives correspond to the hive flows under the isomorphism $H\simeq \oF(G)$.
Note that the set of hive flows is a cone in $\oF(G)$. 
Figure~\ref{fig:degengraph} provides an example of a hive flow.
We encourage the reader to verify the slack inequalities there to get 
some idea of the nature of these constraints. 
\begin{figure}[h]
\begin{center}
\scalebox{1.180}{
\begin{tikzpicture}\draw[rhrhombidraw] (45.0pt,77.94pt) -- (-90.0pt,-155.88pt) ;\draw[rhrhombidraw] (-135.0pt,-129.9pt) -- (-120.0pt,-155.88pt) ;\draw[rhrhombidraw] (-120.0pt,-103.92pt) -- (-90.0pt,-155.88pt) ;\draw[rhrhombidraw] (-105.0pt,-77.94pt) -- (-60.0pt,-155.88pt) ;\draw[rhrhombidraw] (-90.0pt,-51.96pt) -- (-30.0pt,-155.88pt) ;\draw[rhrhombidraw] (-75.0pt,-25.98pt) -- (0.0pt,-155.88pt) ;\draw[rhrhombidraw] (-60.0pt,0.0pt) -- (30.0pt,-155.88pt) ;\draw[rhrhombidraw] (-45.0pt,25.98pt) -- (60.0pt,-155.88pt) ;\draw[rhrhombidraw] (-30.0pt,51.96pt) -- (90.0pt,-155.88pt) ;\draw[rhrhombidraw] (-15.0pt,77.94pt) -- (120.0pt,-155.88pt) ;\draw[rhrhombidraw] (0.0pt,103.92pt) -- (150.0pt,-155.88pt) ;\draw[rhrhombidraw] (150.0pt,-103.92pt) -- (120.0pt,-155.88pt) ;\draw[rhrhombidraw] (165.0pt,-129.9pt) -- (150.0pt,-155.88pt) ;\draw[rhrhombidraw] (135.0pt,-77.94pt) -- (90.0pt,-155.88pt) ;\draw[rhrhombidraw] (120.0pt,-51.96pt) -- (60.0pt,-155.88pt) ;\draw[rhrhombidraw] (105.0pt,-25.98pt) -- (30.0pt,-155.88pt) ;\draw[rhrhombidraw] (90.0pt,0.0pt) -- (0.0pt,-155.88pt) ;\draw[rhrhombidraw] (75.0pt,25.98pt) -- (-30.0pt,-155.88pt) ;\draw[rhrhombidraw] (60.0pt,51.96pt) -- (-60.0pt,-155.88pt) ;\draw[rhrhombidraw] (30.0pt,103.92pt) -- (-120.0pt,-155.88pt) ;\draw[rhrhombidraw] (165.0pt,-129.9pt) -- (-135.0pt,-129.9pt) ;\draw[rhrhombidraw] (150.0pt,-103.92pt) -- (-120.0pt,-103.92pt) ;\draw[rhrhombidraw] (135.0pt,-77.94pt) -- (-105.0pt,-77.94pt) ;\draw[rhrhombidraw] (120.0pt,-51.96pt) -- (-90.0pt,-51.96pt) ;\draw[rhrhombidraw] (105.0pt,-25.98pt) -- (-75.0pt,-25.98pt) ;\draw[rhrhombidraw] (90.0pt,0.0pt) -- (-60.0pt,0.0pt) ;\draw[rhrhombidraw] (60.0pt,51.96pt) -- (-30.0pt,51.96pt) ;\draw[rhrhombidraw] (45.0pt,77.94pt) -- (-15.0pt,77.94pt) ;\draw[rhrhombidraw] (30.0pt,103.92pt) -- (0.0pt,103.92pt) ;\draw[rhrhombithickside] (0.0pt,103.92pt) -- (75.0pt,-25.98pt);\draw[rhrhombithickside] (75.0pt,25.98pt) -- (-45.0pt,25.98pt);\draw[rhrhombithickside] (75.0pt,25.98pt) -- (15.0pt,-77.94pt);\draw[rhrhombithickside] (-45.0pt,25.98pt) -- (60.0pt,-155.88pt);\draw[rhrhombithickside] (-60.0pt,0.0pt) -- (-15.0pt,-77.94pt);\draw[rhrhombithickside] (-60.0pt,-155.88pt) -- (-15.0pt,-77.94pt);\draw[rhrhombithickside] (-15.0pt,-77.94pt) -- (45.0pt,-77.94pt);\draw[rhrhombithickside] (45.0pt,-77.94pt) -- (90.0pt,0.0pt);\draw[rhrhombithickside] (105.0pt,-25.98pt) -- (75.0pt,-25.98pt);\draw[rhrhombithickside] (45.0pt,-77.94pt) -- (90.0pt,-155.88pt);\draw[rhrhombithickside] (120.0pt,-155.88pt) -- (150.0pt,-103.92pt);\draw[rhrhombidraw] (15.0pt,129.9pt) -- (-150.0pt,-155.88pt) -- (180.0pt,-155.88pt) -- cycle;\fill (15.0pt,129.9pt) circle (0.4pt);\fill (0.0pt,103.92pt) circle (0.4pt);\fill (30.0pt,103.92pt) circle (0.4pt);\fill (-15.0pt,77.94pt) circle (0.4pt);\fill (15.0pt,77.94pt) circle (0.4pt);\fill (45.0pt,77.94pt) circle (0.4pt);\fill (-30.0pt,51.96pt) circle (0.4pt);\fill (0.0pt,51.96pt) circle (0.4pt);\fill (30.0pt,51.96pt) circle (0.4pt);\fill (60.0pt,51.96pt) circle (0.4pt);\fill (-45.0pt,25.98pt) circle (0.4pt);\fill (-15.0pt,25.98pt) circle (0.4pt);\fill (15.0pt,25.98pt) circle (0.4pt);\fill (45.0pt,25.98pt) circle (0.4pt);\fill (75.0pt,25.98pt) circle (0.4pt);\fill (-60.0pt,0.0pt) circle (0.4pt);\fill (-30.0pt,0.0pt) circle (0.4pt);\fill (0.0pt,0.0pt) circle (0.4pt);\fill (30.0pt,0.0pt) circle (0.4pt);\fill (60.0pt,0.0pt) circle (0.4pt);\fill (90.0pt,0.0pt) circle (0.4pt);\fill (-75.0pt,-25.98pt) circle (0.4pt);\fill (-45.0pt,-25.98pt) circle (0.4pt);\fill (-15.0pt,-25.98pt) circle (0.4pt);\fill (15.0pt,-25.98pt) circle (0.4pt);\fill (45.0pt,-25.98pt) circle (0.4pt);\fill (75.0pt,-25.98pt) circle (0.4pt);\fill (105.0pt,-25.98pt) circle (0.4pt);\fill (-90.0pt,-51.96pt) circle (0.4pt);\fill (-60.0pt,-51.96pt) circle (0.4pt);\fill (-30.0pt,-51.96pt) circle (0.4pt);\fill (0.0pt,-51.96pt) circle (0.4pt);\fill (30.0pt,-51.96pt) circle (0.4pt);\fill (60.0pt,-51.96pt) circle (0.4pt);\fill (90.0pt,-51.96pt) circle (0.4pt);\fill (120.0pt,-51.96pt) circle (0.4pt);\fill (-105.0pt,-77.94pt) circle (0.4pt);\fill (-75.0pt,-77.94pt) circle (0.4pt);\fill (-45.0pt,-77.94pt) circle (0.4pt);\fill (-15.0pt,-77.94pt) circle (0.4pt);\fill (15.0pt,-77.94pt) circle (0.4pt);\fill (45.0pt,-77.94pt) circle (0.4pt);\fill (75.0pt,-77.94pt) circle (0.4pt);\fill (105.0pt,-77.94pt) circle (0.4pt);\fill (135.0pt,-77.94pt) circle (0.4pt);\fill (-120.0pt,-103.92pt) circle (0.4pt);\fill (-90.0pt,-103.92pt) circle (0.4pt);\fill (-60.0pt,-103.92pt) circle (0.4pt);\fill (-30.0pt,-103.92pt) circle (0.4pt);\fill (0.0pt,-103.92pt) circle (0.4pt);\fill (30.0pt,-103.92pt) circle (0.4pt);\fill (60.0pt,-103.92pt) circle (0.4pt);\fill (90.0pt,-103.92pt) circle (0.4pt);\fill (120.0pt,-103.92pt) circle (0.4pt);\fill (150.0pt,-103.92pt) circle (0.4pt);\fill (-135.0pt,-129.9pt) circle (0.4pt);\fill (-105.0pt,-129.9pt) circle (0.4pt);\fill (-75.0pt,-129.9pt) circle (0.4pt);\fill (-45.0pt,-129.9pt) circle (0.4pt);\fill (-15.0pt,-129.9pt) circle (0.4pt);\fill (15.0pt,-129.9pt) circle (0.4pt);\fill (45.0pt,-129.9pt) circle (0.4pt);\fill (75.0pt,-129.9pt) circle (0.4pt);\fill (105.0pt,-129.9pt) circle (0.4pt);\fill (135.0pt,-129.9pt) circle (0.4pt);\fill (165.0pt,-129.9pt) circle (0.4pt);\fill (-150.0pt,-155.88pt) circle (0.4pt);\fill (-120.0pt,-155.88pt) circle (0.4pt);\fill (-90.0pt,-155.88pt) circle (0.4pt);\fill (-60.0pt,-155.88pt) circle (0.4pt);\fill (-30.0pt,-155.88pt) circle (0.4pt);\fill (0.0pt,-155.88pt) circle (0.4pt);\fill (30.0pt,-155.88pt) circle (0.4pt);\fill (60.0pt,-155.88pt) circle (0.4pt);\fill (90.0pt,-155.88pt) circle (0.4pt);\fill (120.0pt,-155.88pt) circle (0.4pt);\fill (150.0pt,-155.88pt) circle (0.4pt);\fill (180.0pt,-155.88pt) circle (0.4pt);\draw[butt cap-latex,line width=1pt,black] (27.696pt,119.91pt) -- (14.706pt,112.41pt);\fill[white] (22.5pt,116.91pt) circle (2.5pt);\node at (22.5pt,116.91pt) {\tiny 5};\draw[butt cap-latex,line width=1pt,black] (42.696pt,93.93pt) -- (29.706pt,86.43pt);\fill[white] (37.5pt,90.93pt) circle (2.5pt);\node at (37.5pt,90.93pt) {\tiny 5};\draw[butt cap-latex,line width=1pt,black] (57.696pt,67.95pt) -- (44.706pt,60.45pt);\fill[white] (52.5pt,64.95pt) circle (2.5pt);\node at (52.5pt,64.95pt) {\tiny 5};\draw[butt cap-latex,line width=1pt,black] (72.696pt,41.97pt) -- (59.706pt,34.47pt);\fill[white] (67.5pt,38.97pt) circle (2.5pt);\node at (67.5pt,38.97pt) {\tiny 5};\draw[butt cap-latex,line width=1pt,black] (87.696pt,15.99pt) -- (74.706pt,8.49pt);\fill[white] (82.5pt,12.99pt) circle (2.5pt);\node at (82.5pt,12.99pt) {\tiny 3};\draw[butt cap-latex,line width=1pt,black] (102.696pt,-9.99pt) -- (89.706pt,-17.49pt);\fill[white] (97.5pt,-12.99pt) circle (2.5pt);\node at (97.5pt,-12.99pt) {\tiny 2};\draw[butt cap-latex,line width=1pt,black] (117.696pt,-35.97pt) -- (104.706pt,-43.47pt);\fill[white] (112.5pt,-38.97pt) circle (2.5pt);\node at (112.5pt,-38.97pt) {\tiny 1};\draw[butt cap-latex,line width=1pt,black] (132.696pt,-61.95pt) -- (119.706pt,-69.45pt);\fill[white] (127.5pt,-64.95pt) circle (2.5pt);\node at (127.5pt,-64.95pt) {\tiny 1};\draw[butt cap-latex,line width=1pt,black] (147.696pt,-87.93pt) -- (134.706pt,-95.43pt);\fill[white] (142.5pt,-90.93pt) circle (2.5pt);\node at (142.5pt,-90.93pt) {\tiny 1};\draw[butt cap-latex,line width=1pt,black] (12.696pt,93.93pt) -- (-0.294pt,86.43pt);\fill[white] (7.5pt,90.93pt) circle (2.5pt);\node at (7.5pt,90.93pt) {\tiny 5};\draw[butt cap-latex,line width=1pt,black] (27.696pt,67.95pt) -- (14.706pt,60.45pt);\fill[white] (22.5pt,64.95pt) circle (2.5pt);\node at (22.5pt,64.95pt) {\tiny 5};\draw[butt cap-latex,line width=1pt,black] (42.696pt,41.97pt) -- (29.706pt,34.47pt);\fill[white] (37.5pt,38.97pt) circle (2.5pt);\node at (37.5pt,38.97pt) {\tiny 5};\draw[butt cap-latex,line width=1pt,black] (57.696pt,15.99pt) -- (44.706pt,8.49pt);\fill[white] (52.5pt,12.99pt) circle (2.5pt);\node at (52.5pt,12.99pt) {\tiny 4};\draw[butt cap-latex,line width=1pt,black] (72.696pt,-9.99pt) -- (59.706pt,-17.49pt);\fill[white] (67.5pt,-12.99pt) circle (2.5pt);\node at (67.5pt,-12.99pt) {\tiny 3};\draw[butt cap-latex,line width=1pt,black] (87.696pt,-35.97pt) -- (74.706pt,-43.47pt);\fill[white] (82.5pt,-38.97pt) circle (2.5pt);\node at (82.5pt,-38.97pt) {\tiny 1};\draw[butt cap-latex,line width=1pt,black] (102.696pt,-61.95pt) -- (89.706pt,-69.45pt);\fill[white] (97.5pt,-64.95pt) circle (2.5pt);\node at (97.5pt,-64.95pt) {\tiny 1};\draw[butt cap-latex,line width=1pt,black] (117.696pt,-87.93pt) -- (104.706pt,-95.43pt);\fill[white] (112.5pt,-90.93pt) circle (2.5pt);\node at (112.5pt,-90.93pt) {\tiny 1};\draw[butt cap-latex,line width=1pt,black] (132.696pt,-113.91pt) -- (119.706pt,-121.41pt);\fill[white] (127.5pt,-116.91pt) circle (2.5pt);\node at (127.5pt,-116.91pt) {\tiny 1};\draw[butt cap-latex,line width=1pt,black] (-2.304pt,67.95pt) -- (-15.294pt,60.45pt);\fill[white] (-7.5pt,64.95pt) circle (2.5pt);\node at (-7.5pt,64.95pt) {\tiny 5};\draw[butt cap-latex,line width=1pt,black] (12.696pt,41.97pt) -- (-0.294pt,34.47pt);\fill[white] (7.5pt,38.97pt) circle (2.5pt);\node at (7.5pt,38.97pt) {\tiny 5};\draw[butt cap-latex,line width=1pt,black] (27.696pt,15.99pt) -- (14.706pt,8.49pt);\fill[white] (22.5pt,12.99pt) circle (2.5pt);\node at (22.5pt,12.99pt) {\tiny 4};\draw[butt cap-latex,line width=1pt,black] (42.696pt,-9.99pt) -- (29.706pt,-17.49pt);\fill[white] (37.5pt,-12.99pt) circle (2.5pt);\node at (37.5pt,-12.99pt) {\tiny 4};\draw[butt cap-latex,line width=1pt,black] (57.696pt,-35.97pt) -- (44.706pt,-43.47pt);\fill[white] (52.5pt,-38.97pt) circle (2.5pt);\node at (52.5pt,-38.97pt) {\tiny 3};\draw[butt cap-latex,line width=1pt,black] (72.696pt,-61.95pt) -- (59.706pt,-69.45pt);\fill[white] (67.5pt,-64.95pt) circle (2.5pt);\node at (67.5pt,-64.95pt) {\tiny 1};\draw[butt cap-latex,line width=1pt,black] (87.696pt,-87.93pt) -- (74.706pt,-95.43pt);\fill[white] (82.5pt,-90.93pt) circle (2.5pt);\node at (82.5pt,-90.93pt) {\tiny 1};\draw[butt cap-latex,line width=1pt,black] (102.696pt,-113.91pt) -- (89.706pt,-121.41pt);\fill[white] (97.5pt,-116.91pt) circle (2.5pt);\node at (97.5pt,-116.91pt) {\tiny 1};\draw[butt cap-latex,line width=1pt,black] (117.696pt,-139.89pt) -- (104.706pt,-147.39pt);\fill[white] (112.5pt,-142.89pt) circle (2.5pt);\node at (112.5pt,-142.89pt) {\tiny 1};\draw[butt cap-latex,line width=1pt,black] (-17.304pt,41.97pt) -- (-30.294pt,34.47pt);\fill[white] (-22.5pt,38.97pt) circle (2.5pt);\node at (-22.5pt,38.97pt) {\tiny 5};\draw[butt cap-latex,line width=1pt,black] (-2.304pt,15.99pt) -- (-15.294pt,8.49pt);\fill[white] (-7.5pt,12.99pt) circle (2.5pt);\node at (-7.5pt,12.99pt) {\tiny 4};\draw[butt cap-latex,line width=1pt,black] (12.696pt,-9.99pt) -- (-0.294pt,-17.49pt);\fill[white] (7.5pt,-12.99pt) circle (2.5pt);\node at (7.5pt,-12.99pt) {\tiny 4};\draw[butt cap-latex,line width=1pt,black] (27.696pt,-35.97pt) -- (14.706pt,-43.47pt);\fill[white] (22.5pt,-38.97pt) circle (2.5pt);\node at (22.5pt,-38.97pt) {\tiny 4};\draw[butt cap-latex,line width=1pt,black] (42.696pt,-61.95pt) -- (29.706pt,-69.45pt);\fill[white] (37.5pt,-64.95pt) circle (2.5pt);\node at (37.5pt,-64.95pt) {\tiny 3};\draw[butt cap-latex,line width=1pt,black] (57.696pt,-87.93pt) -- (44.706pt,-95.43pt);\fill[white] (52.5pt,-90.93pt) circle (2.5pt);\node at (52.5pt,-90.93pt) {\tiny 1};\draw[butt cap-latex,line width=1pt,black] (72.696pt,-113.91pt) -- (59.706pt,-121.41pt);\fill[white] (67.5pt,-116.91pt) circle (2.5pt);\node at (67.5pt,-116.91pt) {\tiny 1};\draw[butt cap-latex,line width=1pt,black] (87.696pt,-139.89pt) -- (74.706pt,-147.39pt);\fill[white] (82.5pt,-142.89pt) circle (2.5pt);\node at (82.5pt,-142.89pt) {\tiny 1};\draw[butt cap-latex,line width=1pt,black] (-32.304pt,15.99pt) -- (-45.294pt,8.49pt);\fill[white] (-37.5pt,12.99pt) circle (2.5pt);\node at (-37.5pt,12.99pt) {\tiny 4};\draw[butt cap-latex,line width=1pt,black] (-17.304pt,-9.99pt) -- (-30.294pt,-17.49pt);\fill[white] (-22.5pt,-12.99pt) circle (2.5pt);\node at (-22.5pt,-12.99pt) {\tiny 4};\draw[butt cap-latex,line width=1pt,black] (-2.304pt,-35.97pt) -- (-15.294pt,-43.47pt);\fill[white] (-7.5pt,-38.97pt) circle (2.5pt);\node at (-7.5pt,-38.97pt) {\tiny 4};\draw[butt cap-latex,line width=1pt,black] (12.696pt,-61.95pt) -- (-0.294pt,-69.45pt);\fill[white] (7.5pt,-64.95pt) circle (2.5pt);\node at (7.5pt,-64.95pt) {\tiny 4};\draw[butt cap-latex,line width=1pt,black] (27.696pt,-87.93pt) -- (14.706pt,-95.43pt);\fill[white] (22.5pt,-90.93pt) circle (2.5pt);\node at (22.5pt,-90.93pt) {\tiny 1};\draw[butt cap-latex,line width=1pt,black] (42.696pt,-113.91pt) -- (29.706pt,-121.41pt);\fill[white] (37.5pt,-116.91pt) circle (2.5pt);\node at (37.5pt,-116.91pt) {\tiny 1};\draw[butt cap-latex,line width=1pt,black] (57.696pt,-139.89pt) -- (44.706pt,-147.39pt);\fill[white] (52.5pt,-142.89pt) circle (2.5pt);\node at (52.5pt,-142.89pt) {\tiny 1};\draw[butt cap-latex,line width=1pt,black] (-47.304pt,-9.99pt) -- (-60.294pt,-17.49pt);\fill[white] (-52.5pt,-12.99pt) circle (2.5pt);\node at (-52.5pt,-12.99pt) {\tiny 4};\draw[butt cap-latex,line width=1pt,black] (-32.304pt,-35.97pt) -- (-45.294pt,-43.47pt);\fill[white] (-37.5pt,-38.97pt) circle (2.5pt);\node at (-37.5pt,-38.97pt) {\tiny 4};\draw[butt cap-latex,line width=1pt,black] (-17.304pt,-61.95pt) -- (-30.294pt,-69.45pt);\fill[white] (-22.5pt,-64.95pt) circle (2.5pt);\node at (-22.5pt,-64.95pt) {\tiny 4};\draw[butt cap-latex,line width=1pt,black] (-2.304pt,-87.93pt) -- (-15.294pt,-95.43pt);\fill[white] (-7.5pt,-90.93pt) circle (2.5pt);\node at (-7.5pt,-90.93pt) {\tiny 1};\draw[butt cap-latex,line width=1pt,black] (12.696pt,-113.91pt) -- (-0.294pt,-121.41pt);\fill[white] (7.5pt,-116.91pt) circle (2.5pt);\node at (7.5pt,-116.91pt) {\tiny 1};\draw[butt cap-latex,line width=1pt,black] (27.696pt,-139.89pt) -- (14.706pt,-147.39pt);\fill[white] (22.5pt,-142.89pt) circle (2.5pt);\node at (22.5pt,-142.89pt) {\tiny 1};\draw[butt cap-latex,line width=1pt,black] (-62.304pt,-35.97pt) -- (-75.294pt,-43.47pt);\fill[white] (-67.5pt,-38.97pt) circle (2.5pt);\node at (-67.5pt,-38.97pt) {\tiny 4};\draw[butt cap-latex,line width=1pt,black] (-47.304pt,-61.95pt) -- (-60.294pt,-69.45pt);\fill[white] (-52.5pt,-64.95pt) circle (2.5pt);\node at (-52.5pt,-64.95pt) {\tiny 4};\draw[butt cap-latex,line width=1pt,black] (-32.304pt,-87.93pt) -- (-45.294pt,-95.43pt);\fill[white] (-37.5pt,-90.93pt) circle (2.5pt);\node at (-37.5pt,-90.93pt) {\tiny 4};\draw[butt cap-latex,line width=1pt,black] (-77.304pt,-61.95pt) -- (-90.294pt,-69.45pt);\fill[white] (-82.5pt,-64.95pt) circle (2.5pt);\node at (-82.5pt,-64.95pt) {\tiny 4};\draw[butt cap-latex,line width=1pt,black] (-62.304pt,-87.93pt) -- (-75.294pt,-95.43pt);\fill[white] (-67.5pt,-90.93pt) circle (2.5pt);\node at (-67.5pt,-90.93pt) {\tiny 4};\draw[butt cap-latex,line width=1pt,black] (-47.304pt,-113.91pt) -- (-60.294pt,-121.41pt);\fill[white] (-52.5pt,-116.91pt) circle (2.5pt);\node at (-52.5pt,-116.91pt) {\tiny 4};\draw[butt cap-latex,line width=1pt,black] (-92.304pt,-87.93pt) -- (-105.294pt,-95.43pt);\fill[white] (-97.5pt,-90.93pt) circle (2.5pt);\node at (-97.5pt,-90.93pt) {\tiny 4};\draw[butt cap-latex,line width=1pt,black] (-77.304pt,-113.91pt) -- (-90.294pt,-121.41pt);\fill[white] (-82.5pt,-116.91pt) circle (2.5pt);\node at (-82.5pt,-116.91pt) {\tiny 4};\draw[butt cap-latex,line width=1pt,black] (-62.304pt,-139.89pt) -- (-75.294pt,-147.39pt);\fill[white] (-67.5pt,-142.89pt) circle (2.5pt);\node at (-67.5pt,-142.89pt) {\tiny 4};\draw[butt cap-latex,line width=1pt,black] (-107.304pt,-113.91pt) -- (-120.294pt,-121.41pt);\fill[white] (-112.5pt,-116.91pt) circle (2.5pt);\node at (-112.5pt,-116.91pt) {\tiny 4};\draw[butt cap-latex,line width=1pt,black] (-92.304pt,-139.89pt) -- (-105.294pt,-147.39pt);\fill[white] (-97.5pt,-142.89pt) circle (2.5pt);\node at (-97.5pt,-142.89pt) {\tiny 4};\draw[butt cap-latex,line width=1pt,black] (-122.304pt,-139.89pt) -- (-135.294pt,-147.39pt);\fill[white] (-127.5pt,-142.89pt) circle (2.5pt);\node at (-127.5pt,-142.89pt) {\tiny 4};\draw[butt cap-latex,line width=1pt,black] (-2.304pt,-139.89pt) -- (-15.294pt,-147.39pt);\fill[white] (-7.5pt,-142.89pt) circle (2.5pt);\node at (-7.5pt,-142.89pt) {\tiny 1};\draw[butt cap-latex,line width=1pt,black] (-17.304pt,-113.91pt) -- (-30.294pt,-121.41pt);\fill[white] (-22.5pt,-116.91pt) circle (2.5pt);\node at (-22.5pt,-116.91pt) {\tiny 1};\draw[butt cap-latex,line width=1pt,black] (-32.304pt,-139.89pt) -- (-45.294pt,-147.39pt);\fill[white] (-37.5pt,-142.89pt) circle (2.5pt);\node at (-37.5pt,-142.89pt) {\tiny 1};\draw[butt cap-latex,line width=1pt,black] (-2.304pt,87.93pt) -- (-15.294pt,95.43pt);\fill[white] (-7.5pt,90.93pt) circle (2.5pt);\node at (-7.5pt,90.93pt) {\tiny 9};\draw[butt cap-latex,line width=1pt,black] (-17.304pt,61.95pt) -- (-30.294pt,69.45pt);\fill[white] (-22.5pt,64.95pt) circle (2.5pt);\node at (-22.5pt,64.95pt) {\tiny 9};\draw[butt cap-latex,line width=1pt,black] (-32.304pt,35.97pt) -- (-45.294pt,43.47pt);\fill[white] (-37.5pt,38.97pt) circle (2.5pt);\node at (-37.5pt,38.97pt) {\tiny 9};\draw[butt cap-latex,line width=1pt,black] (-47.304pt,9.99pt) -- (-60.294pt,17.49pt);\fill[white] (-52.5pt,12.99pt) circle (2.5pt);\node at (-52.5pt,12.99pt) {\tiny 7};\draw[butt cap-latex,line width=1pt,black] (-62.304pt,-15.99pt) -- (-75.294pt,-8.49pt);\fill[white] (-67.5pt,-12.99pt) circle (2.5pt);\node at (-67.5pt,-12.99pt) {\tiny 4};\draw[butt cap-latex,line width=1pt,black] (-77.304pt,-41.97pt) -- (-90.294pt,-34.47pt);\fill[white] (-82.5pt,-38.97pt) circle (2.5pt);\node at (-82.5pt,-38.97pt) {\tiny 4};\draw[butt cap-latex,line width=1pt,black] (-92.304pt,-67.95pt) -- (-105.294pt,-60.45pt);\fill[white] (-97.5pt,-64.95pt) circle (2.5pt);\node at (-97.5pt,-64.95pt) {\tiny 4};\draw[butt cap-latex,line width=1pt,black] (-107.304pt,-93.93pt) -- (-120.294pt,-86.43pt);\fill[white] (-112.5pt,-90.93pt) circle (2.5pt);\node at (-112.5pt,-90.93pt) {\tiny 4};\draw[butt cap-latex,line width=1pt,black] (-122.304pt,-119.91pt) -- (-135.294pt,-112.41pt);\fill[white] (-127.5pt,-116.91pt) circle (2.5pt);\node at (-127.5pt,-116.91pt) {\tiny 4};\draw[butt cap-latex,line width=1pt,black] (-137.304pt,-145.89pt) -- (-150.294pt,-138.39pt);\fill[white] (-142.5pt,-142.89pt) circle (2.5pt);\node at (-142.5pt,-142.89pt) {\tiny 4};\draw[butt cap-latex,line width=1pt,black] (12.696pt,61.95pt) -- (-0.294pt,69.45pt);\fill[white] (7.5pt,64.95pt) circle (2.5pt);\node at (7.5pt,64.95pt) {\tiny 9};\draw[butt cap-latex,line width=1pt,black] (-2.304pt,35.97pt) -- (-15.294pt,43.47pt);\fill[white] (-7.5pt,38.97pt) circle (2.5pt);\node at (-7.5pt,38.97pt) {\tiny 9};\draw[butt cap-latex,line width=1pt,black] (-17.304pt,9.99pt) -- (-30.294pt,17.49pt);\fill[white] (-22.5pt,12.99pt) circle (2.5pt);\node at (-22.5pt,12.99pt) {\tiny 8};\draw[butt cap-latex,line width=1pt,black] (-32.304pt,-15.99pt) -- (-45.294pt,-8.49pt);\fill[white] (-37.5pt,-12.99pt) circle (2.5pt);\node at (-37.5pt,-12.99pt) {\tiny 7};\draw[butt cap-latex,line width=1pt,black] (-47.304pt,-41.97pt) -- (-60.294pt,-34.47pt);\fill[white] (-52.5pt,-38.97pt) circle (2.5pt);\node at (-52.5pt,-38.97pt) {\tiny 4};\draw[butt cap-latex,line width=1pt,black] (-62.304pt,-67.95pt) -- (-75.294pt,-60.45pt);\fill[white] (-67.5pt,-64.95pt) circle (2.5pt);\node at (-67.5pt,-64.95pt) {\tiny 4};\draw[butt cap-latex,line width=1pt,black] (-77.304pt,-93.93pt) -- (-90.294pt,-86.43pt);\fill[white] (-82.5pt,-90.93pt) circle (2.5pt);\node at (-82.5pt,-90.93pt) {\tiny 4};\draw[butt cap-latex,line width=1pt,black] (-92.304pt,-119.91pt) -- (-105.294pt,-112.41pt);\fill[white] (-97.5pt,-116.91pt) circle (2.5pt);\node at (-97.5pt,-116.91pt) {\tiny 4};\draw[butt cap-latex,line width=1pt,black] (-107.304pt,-145.89pt) -- (-120.294pt,-138.39pt);\fill[white] (-112.5pt,-142.89pt) circle (2.5pt);\node at (-112.5pt,-142.89pt) {\tiny 4};\draw[butt cap-latex,line width=1pt,black] (27.696pt,35.97pt) -- (14.706pt,43.47pt);\fill[white] (22.5pt,38.97pt) circle (2.5pt);\node at (22.5pt,38.97pt) {\tiny 9};\draw[butt cap-latex,line width=1pt,black] (12.696pt,9.99pt) -- (-0.294pt,17.49pt);\fill[white] (7.5pt,12.99pt) circle (2.5pt);\node at (7.5pt,12.99pt) {\tiny 8};\draw[butt cap-latex,line width=1pt,black] (-2.304pt,-15.99pt) -- (-15.294pt,-8.49pt);\fill[white] (-7.5pt,-12.99pt) circle (2.5pt);\node at (-7.5pt,-12.99pt) {\tiny 8};\draw[butt cap-latex,line width=1pt,black] (42.696pt,9.99pt) -- (29.706pt,17.49pt);\fill[white] (37.5pt,12.99pt) circle (2.5pt);\node at (37.5pt,12.99pt) {\tiny 8};\draw[butt cap-latex,line width=1pt,black] (27.696pt,-15.99pt) -- (14.706pt,-8.49pt);\fill[white] (22.5pt,-12.99pt) circle (2.5pt);\node at (22.5pt,-12.99pt) {\tiny 8};\draw[butt cap-latex,line width=1pt,black] (12.696pt,-41.97pt) -- (-0.294pt,-34.47pt);\fill[white] (7.5pt,-38.97pt) circle (2.5pt);\node at (7.5pt,-38.97pt) {\tiny 8};\draw[butt cap-latex,line width=1pt,black] (57.696pt,-15.99pt) -- (44.706pt,-8.49pt);\fill[white] (52.5pt,-12.99pt) circle (2.5pt);\node at (52.5pt,-12.99pt) {\tiny 8};\draw[butt cap-latex,line width=1pt,black] (42.696pt,-41.97pt) -- (29.706pt,-34.47pt);\fill[white] (37.5pt,-38.97pt) circle (2.5pt);\node at (37.5pt,-38.97pt) {\tiny 8};\draw[butt cap-latex,line width=1pt,black] (27.696pt,-67.95pt) -- (14.706pt,-60.45pt);\fill[white] (22.5pt,-64.95pt) circle (2.5pt);\node at (22.5pt,-64.95pt) {\tiny 8};\draw[butt cap-latex,line width=1pt,black] (72.696pt,-41.97pt) -- (59.706pt,-34.47pt);\fill[white] (67.5pt,-38.97pt) circle (2.5pt);\node at (67.5pt,-38.97pt) {\tiny 8};\draw[butt cap-latex,line width=1pt,black] (57.696pt,-67.95pt) -- (44.706pt,-60.45pt);\fill[white] (52.5pt,-64.95pt) circle (2.5pt);\node at (52.5pt,-64.95pt) {\tiny 8};\draw[butt cap-latex,line width=1pt,black] (102.696pt,-41.97pt) -- (89.706pt,-34.47pt);\fill[white] (97.5pt,-38.97pt) circle (2.5pt);\node at (97.5pt,-38.97pt) {\tiny 8};\draw[butt cap-latex,line width=1pt,black] (87.696pt,-67.95pt) -- (74.706pt,-60.45pt);\fill[white] (82.5pt,-64.95pt) circle (2.5pt);\node at (82.5pt,-64.95pt) {\tiny 8};\draw[butt cap-latex,line width=1pt,black] (72.696pt,-93.93pt) -- (59.706pt,-86.43pt);\fill[white] (67.5pt,-90.93pt) circle (2.5pt);\node at (67.5pt,-90.93pt) {\tiny 8};\draw[butt cap-latex,line width=1pt,black] (117.696pt,-67.95pt) -- (104.706pt,-60.45pt);\fill[white] (112.5pt,-64.95pt) circle (2.5pt);\node at (112.5pt,-64.95pt) {\tiny 8};\draw[butt cap-latex,line width=1pt,black] (102.696pt,-93.93pt) -- (89.706pt,-86.43pt);\fill[white] (97.5pt,-90.93pt) circle (2.5pt);\node at (97.5pt,-90.93pt) {\tiny 8};\draw[butt cap-latex,line width=1pt,black] (87.696pt,-119.91pt) -- (74.706pt,-112.41pt);\fill[white] (82.5pt,-116.91pt) circle (2.5pt);\node at (82.5pt,-116.91pt) {\tiny 8};\draw[butt cap-latex,line width=1pt,black] (132.696pt,-93.93pt) -- (119.706pt,-86.43pt);\fill[white] (127.5pt,-90.93pt) circle (2.5pt);\node at (127.5pt,-90.93pt) {\tiny 8};\draw[butt cap-latex,line width=1pt,black] (117.696pt,-119.91pt) -- (104.706pt,-112.41pt);\fill[white] (112.5pt,-116.91pt) circle (2.5pt);\node at (112.5pt,-116.91pt) {\tiny 8};\draw[butt cap-latex,line width=1pt,black] (102.696pt,-145.89pt) -- (89.706pt,-138.39pt);\fill[white] (97.5pt,-142.89pt) circle (2.5pt);\node at (97.5pt,-142.89pt) {\tiny 8};\draw[butt cap-latex,line width=1pt,black] (147.696pt,-119.91pt) -- (134.706pt,-112.41pt);\fill[white] (142.5pt,-116.91pt) circle (2.5pt);\node at (142.5pt,-116.91pt) {\tiny 8};\draw[butt cap-latex,line width=1pt,black] (132.696pt,-145.89pt) -- (119.706pt,-138.39pt);\fill[white] (127.5pt,-142.89pt) circle (2.5pt);\node at (127.5pt,-142.89pt) {\tiny 8};\draw[butt cap-latex,line width=1pt,black] (162.696pt,-145.89pt) -- (149.706pt,-138.39pt);\fill[white] (157.5pt,-142.89pt) circle (2.5pt);\node at (157.5pt,-142.89pt) {\tiny 8};\draw[butt cap-latex,line width=1pt,black] (-17.304pt,-41.97pt) -- (-30.294pt,-34.47pt);\fill[white] (-22.5pt,-38.97pt) circle (2.5pt);\node at (-22.5pt,-38.97pt) {\tiny 7};\draw[butt cap-latex,line width=1pt,black] (-2.304pt,-67.95pt) -- (-15.294pt,-60.45pt);\fill[white] (-7.5pt,-64.95pt) circle (2.5pt);\node at (-7.5pt,-64.95pt) {\tiny 7};\draw[butt cap-latex,line width=1pt,black] (-32.304pt,-67.95pt) -- (-45.294pt,-60.45pt);\fill[white] (-37.5pt,-64.95pt) circle (2.5pt);\node at (-37.5pt,-64.95pt) {\tiny 4};\draw[butt cap-latex,line width=1pt,black] (-47.304pt,-93.93pt) -- (-60.294pt,-86.43pt);\fill[white] (-52.5pt,-90.93pt) circle (2.5pt);\node at (-52.5pt,-90.93pt) {\tiny 4};\draw[butt cap-latex,line width=1pt,black] (-62.304pt,-119.91pt) -- (-75.294pt,-112.41pt);\fill[white] (-67.5pt,-116.91pt) circle (2.5pt);\node at (-67.5pt,-116.91pt) {\tiny 4};\draw[butt cap-latex,line width=1pt,black] (-77.304pt,-145.89pt) -- (-90.294pt,-138.39pt);\fill[white] (-82.5pt,-142.89pt) circle (2.5pt);\node at (-82.5pt,-142.89pt) {\tiny 4};\draw[butt cap-latex,line width=1pt,black] (-17.304pt,-93.93pt) -- (-30.294pt,-86.43pt);\fill[white] (-22.5pt,-90.93pt) circle (2.5pt);\node at (-22.5pt,-90.93pt) {\tiny 4};\draw[butt cap-latex,line width=1pt,black] (-32.304pt,-119.91pt) -- (-45.294pt,-112.41pt);\fill[white] (-37.5pt,-116.91pt) circle (2.5pt);\node at (-37.5pt,-116.91pt) {\tiny 4};\draw[butt cap-latex,line width=1pt,black] (-47.304pt,-145.89pt) -- (-60.294pt,-138.39pt);\fill[white] (-52.5pt,-142.89pt) circle (2.5pt);\node at (-52.5pt,-142.89pt) {\tiny 4};\draw[butt cap-latex,line width=1pt,black] (12.696pt,-93.93pt) -- (-0.294pt,-86.43pt);\fill[white] (7.5pt,-90.93pt) circle (2.5pt);\node at (7.5pt,-90.93pt) {\tiny 4};\draw[butt cap-latex,line width=1pt,black] (-2.304pt,-119.91pt) -- (-15.294pt,-112.41pt);\fill[white] (-7.5pt,-116.91pt) circle (2.5pt);\node at (-7.5pt,-116.91pt) {\tiny 4};\draw[butt cap-latex,line width=1pt,black] (-17.304pt,-145.89pt) -- (-30.294pt,-138.39pt);\fill[white] (-22.5pt,-142.89pt) circle (2.5pt);\node at (-22.5pt,-142.89pt) {\tiny 4};\draw[butt cap-latex,line width=1pt,black] (27.696pt,-119.91pt) -- (14.706pt,-112.41pt);\fill[white] (22.5pt,-116.91pt) circle (2.5pt);\node at (22.5pt,-116.91pt) {\tiny 4};\draw[butt cap-latex,line width=1pt,black] (12.696pt,-145.89pt) -- (-0.294pt,-138.39pt);\fill[white] (7.5pt,-142.89pt) circle (2.5pt);\node at (7.5pt,-142.89pt) {\tiny 4};\draw[butt cap-latex,line width=1pt,black] (42.696pt,-145.89pt) -- (29.706pt,-138.39pt);\fill[white] (37.5pt,-142.89pt) circle (2.5pt);\node at (37.5pt,-142.89pt) {\tiny 4};\draw[butt cap-latex,line width=1pt,black] (42.696pt,-93.93pt) -- (29.706pt,-86.43pt);\fill[white] (37.5pt,-90.93pt) circle (2.5pt);\node at (37.5pt,-90.93pt) {\tiny 6};\draw[butt cap-latex,line width=1pt,black] (57.696pt,-119.91pt) -- (44.706pt,-112.41pt);\fill[white] (52.5pt,-116.91pt) circle (2.5pt);\node at (52.5pt,-116.91pt) {\tiny 6};\draw[butt cap-latex,line width=1pt,black] (72.696pt,-145.89pt) -- (59.706pt,-138.39pt);\fill[white] (67.5pt,-142.89pt) circle (2.5pt);\node at (67.5pt,-142.89pt) {\tiny 6};\draw[butt cap-latex,line width=1pt,black] (87.696pt,-15.99pt) -- (74.706pt,-8.49pt);\fill[white] (82.5pt,-12.99pt) circle (2.5pt);\node at (82.5pt,-12.99pt) {\tiny 9};\draw[butt cap-latex,line width=1pt,black] (72.696pt,9.99pt) -- (59.706pt,17.49pt);\fill[white] (67.5pt,12.99pt) circle (2.5pt);\node at (67.5pt,12.99pt) {\tiny 9};\draw[butt cap-latex,line width=1pt,black] (15.0pt,97.92pt) -- (15.0pt,112.92pt);\fill[white] (15.0pt,103.92pt) circle (2.5pt);\node at (15.0pt,103.92pt) {\tiny 5};\draw[butt cap-latex,line width=1pt,black] (30.0pt,71.94pt) -- (30.0pt,86.94pt);\fill[white] (30.0pt,77.94pt) circle (2.5pt);\node at (30.0pt,77.94pt) {\tiny 5};\draw[butt cap-latex,line width=1pt,black] (45.0pt,45.96pt) -- (45.0pt,60.96pt);\fill[white] (45.0pt,51.96pt) circle (2.5pt);\node at (45.0pt,51.96pt) {\tiny 5};\draw[butt cap-latex,line width=1pt,black] (60.0pt,19.98pt) -- (60.0pt,34.98pt);\fill[white] (60.0pt,25.98pt) circle (2.5pt);\node at (60.0pt,25.98pt) {\tiny 5};\draw[butt cap-latex,line width=1pt,black] (0.0pt,71.94pt) -- (0.0pt,86.94pt);\fill[white] (0.0pt,77.94pt) circle (2.5pt);\node at (0.0pt,77.94pt) {\tiny 4};\draw[butt cap-latex,line width=1pt,black] (15.0pt,45.96pt) -- (15.0pt,60.96pt);\fill[white] (15.0pt,51.96pt) circle (2.5pt);\node at (15.0pt,51.96pt) {\tiny 4};\draw[butt cap-latex,line width=1pt,black] (-15.0pt,45.96pt) -- (-15.0pt,60.96pt);\fill[white] (-15.0pt,51.96pt) circle (2.5pt);\node at (-15.0pt,51.96pt) {\tiny 4};\draw[butt cap-latex,line width=1pt,black] (30.0pt,19.98pt) -- (30.0pt,34.98pt);\fill[white] (30.0pt,25.98pt) circle (2.5pt);\node at (30.0pt,25.98pt) {\tiny 4};\draw[butt cap-latex,line width=1pt,black] (0.0pt,19.98pt) -- (0.0pt,34.98pt);\fill[white] (0.0pt,25.98pt) circle (2.5pt);\node at (0.0pt,25.98pt) {\tiny 4};\draw[butt cap-latex,line width=1pt,black] (-30.0pt,19.98pt) -- (-30.0pt,34.98pt);\fill[white] (-30.0pt,25.98pt) circle (2.5pt);\node at (-30.0pt,25.98pt) {\tiny 4};\draw[butt cap-latex,line width=1pt,black] (-15.0pt,-6.0pt) -- (-15.0pt,9.0pt);\fill[white] (-15.0pt,0.0pt) circle (2.5pt);\node at (-15.0pt,0.0pt) {\tiny 4};\draw[butt cap-latex,line width=1pt,black] (15.0pt,-6.0pt) -- (15.0pt,9.0pt);\fill[white] (15.0pt,0.0pt) circle (2.5pt);\node at (15.0pt,0.0pt) {\tiny 4};\draw[butt cap-latex,line width=1pt,black] (45.0pt,-6.0pt) -- (45.0pt,9.0pt);\fill[white] (45.0pt,0.0pt) circle (2.5pt);\node at (45.0pt,0.0pt) {\tiny 4};\draw[butt cap-latex,line width=1pt,black] (30.0pt,-31.98pt) -- (30.0pt,-16.98pt);\fill[white] (30.0pt,-25.98pt) circle (2.5pt);\node at (30.0pt,-25.98pt) {\tiny 4};\draw[butt cap-latex,line width=1pt,black] (0.0pt,-31.98pt) -- (0.0pt,-16.98pt);\fill[white] (0.0pt,-25.98pt) circle (2.5pt);\node at (0.0pt,-25.98pt) {\tiny 4};\draw[butt cap-latex,line width=1pt,black] (15.0pt,-57.96pt) -- (15.0pt,-42.96pt);\fill[white] (15.0pt,-51.96pt) circle (2.5pt);\node at (15.0pt,-51.96pt) {\tiny 4};\draw[butt cap-latex,line width=1pt,black] (-45.0pt,-6.0pt) -- (-45.0pt,9.0pt);\fill[white] (-45.0pt,0.0pt) circle (2.5pt);\node at (-45.0pt,0.0pt) {\tiny 3};\draw[butt cap-latex,line width=1pt,black] (-30.0pt,-31.98pt) -- (-30.0pt,-16.98pt);\fill[white] (-30.0pt,-25.98pt) circle (2.5pt);\node at (-30.0pt,-25.98pt) {\tiny 3};\draw[butt cap-latex,line width=1pt,black] (-15.0pt,-57.96pt) -- (-15.0pt,-42.96pt);\fill[white] (-15.0pt,-51.96pt) circle (2.5pt);\node at (-15.0pt,-51.96pt) {\tiny 3};\draw[butt cap-latex,line width=1pt,black] (0.0pt,-83.94pt) -- (0.0pt,-68.94pt);\fill[white] (0.0pt,-77.94pt) circle (2.5pt);\node at (0.0pt,-77.94pt) {\tiny 3};\draw[butt cap-latex,line width=1pt,black] (15.0pt,-109.92pt) -- (15.0pt,-94.92pt);\fill[white] (15.0pt,-103.92pt) circle (2.5pt);\node at (15.0pt,-103.92pt) {\tiny 3};\draw[butt cap-latex,line width=1pt,black] (-15.0pt,-109.92pt) -- (-15.0pt,-94.92pt);\fill[white] (-15.0pt,-103.92pt) circle (2.5pt);\node at (-15.0pt,-103.92pt) {\tiny 3};\draw[butt cap-latex,line width=1pt,black] (30.0pt,-135.9pt) -- (30.0pt,-120.9pt);\fill[white] (30.0pt,-129.9pt) circle (2.5pt);\node at (30.0pt,-129.9pt) {\tiny 3};\draw[butt cap-latex,line width=1pt,black] (0.0pt,-135.9pt) -- (0.0pt,-120.9pt);\fill[white] (0.0pt,-129.9pt) circle (2.5pt);\node at (0.0pt,-129.9pt) {\tiny 3};\draw[butt cap-latex,line width=1pt,black] (-30.0pt,-135.9pt) -- (-30.0pt,-120.9pt);\fill[white] (-30.0pt,-129.9pt) circle (2.5pt);\node at (-30.0pt,-129.9pt) {\tiny 3};\draw[butt cap-latex,line width=1pt,black] (45.0pt,-161.88pt) -- (45.0pt,-146.88pt);\fill[white] (45.0pt,-155.88pt) circle (2.5pt);\node at (45.0pt,-155.88pt) {\tiny 3};\draw[butt cap-latex,line width=1pt,black] (15.0pt,-161.88pt) -- (15.0pt,-146.88pt);\fill[white] (15.0pt,-155.88pt) circle (2.5pt);\node at (15.0pt,-155.88pt) {\tiny 3};\draw[butt cap-latex,line width=1pt,black] (-15.0pt,-161.88pt) -- (-15.0pt,-146.88pt);\fill[white] (-15.0pt,-155.88pt) circle (2.5pt);\node at (-15.0pt,-155.88pt) {\tiny 3};\draw[butt cap-latex,line width=1pt,black] (-45.0pt,-161.88pt) -- (-45.0pt,-146.88pt);\fill[white] (-45.0pt,-155.88pt) circle (2.5pt);\node at (-45.0pt,-155.88pt) {\tiny 3};\draw[butt cap-latex,line width=1pt,black] (75.0pt,-161.88pt) -- (75.0pt,-146.88pt);\fill[white] (75.0pt,-155.88pt) circle (2.5pt);\node at (75.0pt,-155.88pt) {\tiny 5};\draw[butt cap-latex,line width=1pt,black] (60.0pt,-135.9pt) -- (60.0pt,-120.9pt);\fill[white] (60.0pt,-129.9pt) circle (2.5pt);\node at (60.0pt,-129.9pt) {\tiny 5};\draw[butt cap-latex,line width=1pt,black] (45.0pt,-109.92pt) -- (45.0pt,-94.92pt);\fill[white] (45.0pt,-103.92pt) circle (2.5pt);\node at (45.0pt,-103.92pt) {\tiny 5};\draw[butt cap-latex,line width=1pt,black] (30.0pt,-83.94pt) -- (30.0pt,-68.94pt);\fill[white] (30.0pt,-77.94pt) circle (2.5pt);\node at (30.0pt,-77.94pt) {\tiny 5};\draw[butt cap-latex,line width=1pt,black] (45.0pt,-57.96pt) -- (45.0pt,-42.96pt);\fill[white] (45.0pt,-51.96pt) circle (2.5pt);\node at (45.0pt,-51.96pt) {\tiny 5};\draw[butt cap-latex,line width=1pt,black] (60.0pt,-31.98pt) -- (60.0pt,-16.98pt);\fill[white] (60.0pt,-25.98pt) circle (2.5pt);\node at (60.0pt,-25.98pt) {\tiny 5};\draw[butt cap-latex,line width=1pt,black] (75.0pt,-6.0pt) -- (75.0pt,9.0pt);\fill[white] (75.0pt,0.0pt) circle (2.5pt);\node at (75.0pt,0.0pt) {\tiny 6};\draw[butt cap-latex,line width=1pt,black] (90.0pt,-31.98pt) -- (90.0pt,-16.98pt);\fill[white] (90.0pt,-25.98pt) circle (2.5pt);\node at (90.0pt,-25.98pt) {\tiny 7};\draw[butt cap-latex,line width=1pt,black] (105.0pt,-57.96pt) -- (105.0pt,-42.96pt);\fill[white] (105.0pt,-51.96pt) circle (2.5pt);\node at (105.0pt,-51.96pt) {\tiny 7};\draw[butt cap-latex,line width=1pt,black] (75.0pt,-57.96pt) -- (75.0pt,-42.96pt);\fill[white] (75.0pt,-51.96pt) circle (2.5pt);\node at (75.0pt,-51.96pt) {\tiny 7};\draw[butt cap-latex,line width=1pt,black] (120.0pt,-83.94pt) -- (120.0pt,-68.94pt);\fill[white] (120.0pt,-77.94pt) circle (2.5pt);\node at (120.0pt,-77.94pt) {\tiny 7};\draw[butt cap-latex,line width=1pt,black] (90.0pt,-83.94pt) -- (90.0pt,-68.94pt);\fill[white] (90.0pt,-77.94pt) circle (2.5pt);\node at (90.0pt,-77.94pt) {\tiny 7};\draw[butt cap-latex,line width=1pt,black] (60.0pt,-83.94pt) -- (60.0pt,-68.94pt);\fill[white] (60.0pt,-77.94pt) circle (2.5pt);\node at (60.0pt,-77.94pt) {\tiny 7};\draw[butt cap-latex,line width=1pt,black] (135.0pt,-109.92pt) -- (135.0pt,-94.92pt);\fill[white] (135.0pt,-103.92pt) circle (2.5pt);\node at (135.0pt,-103.92pt) {\tiny 7};\draw[butt cap-latex,line width=1pt,black] (105.0pt,-109.92pt) -- (105.0pt,-94.92pt);\fill[white] (105.0pt,-103.92pt) circle (2.5pt);\node at (105.0pt,-103.92pt) {\tiny 7};\draw[butt cap-latex,line width=1pt,black] (75.0pt,-109.92pt) -- (75.0pt,-94.92pt);\fill[white] (75.0pt,-103.92pt) circle (2.5pt);\node at (75.0pt,-103.92pt) {\tiny 7};\draw[butt cap-latex,line width=1pt,black] (120.0pt,-135.9pt) -- (120.0pt,-120.9pt);\fill[white] (120.0pt,-129.9pt) circle (2.5pt);\node at (120.0pt,-129.9pt) {\tiny 7};\draw[butt cap-latex,line width=1pt,black] (90.0pt,-135.9pt) -- (90.0pt,-120.9pt);\fill[white] (90.0pt,-129.9pt) circle (2.5pt);\node at (90.0pt,-129.9pt) {\tiny 7};\draw[butt cap-latex,line width=1pt,black] (105.0pt,-161.88pt) -- (105.0pt,-146.88pt);\fill[white] (105.0pt,-155.88pt) circle (2.5pt);\node at (105.0pt,-155.88pt) {\tiny 7};\draw[butt cap-latex,line width=1pt,black] (135.0pt,-161.88pt) -- (135.0pt,-146.88pt);\fill[white] (135.0pt,-155.88pt) circle (2.5pt);\node at (135.0pt,-155.88pt) {\tiny 8};\draw[butt cap-latex,line width=1pt,black] (165.0pt,-161.88pt) -- (165.0pt,-146.88pt);\fill[white] (165.0pt,-155.88pt) circle (2.5pt);\node at (165.0pt,-155.88pt) {\tiny 8};\draw[butt cap-latex,line width=1pt,black] (150.0pt,-135.9pt) -- (150.0pt,-120.9pt);\fill[white] (150.0pt,-129.9pt) circle (2.5pt);\node at (150.0pt,-129.9pt) {\tiny 8};\draw[butt cap-latex,line width=1pt,black] (57.696pt,35.97pt) -- (44.706pt,43.47pt);\fill[white] (52.5pt,38.97pt) circle (2.5pt);\node at (52.5pt,38.97pt) {\tiny {1\!0}};\draw[butt cap-latex,line width=1pt,black] (42.696pt,61.95pt) -- (29.706pt,69.45pt);\fill[white] (37.5pt,64.95pt) circle (2.5pt);\node at (37.5pt,64.95pt) {\tiny {1\!0}};\draw[butt cap-latex,line width=1pt,black] (27.696pt,87.93pt) -- (14.706pt,95.43pt);\fill[white] (22.5pt,90.93pt) circle (2.5pt);\node at (22.5pt,90.93pt) {\tiny {1\!0}};\draw[butt cap-latex,line width=1pt,black] (12.696pt,113.91pt) -- (-0.294pt,121.41pt);\fill[white] (7.5pt,116.91pt) circle (2.5pt);\node at (7.5pt,116.91pt) {\tiny {1\!0}};\end{tikzpicture}
}
    \caption{A hive flow for $n=11$, 
                   $\la =  (5,5,5,5,3,2,1,1,1,0,0)$, 
                   $\mu=(8,8,7,5,3,3,3,3,0,0,0)$, 
                  $\nu=(10,9,9,9,7,4,4,4,4,4,4)$, 
                   and its corresponding partition of $\Delta$ into flatspaces. 
                   The numbers give the throughputs through edges of $\Delta$ in the directions of the arrows. 
                   The properties of Lemma~\protect\ref{lem:flatspaceprop} and Lemma~\ref{obs:borderentranceexit} are readily verified.}
    \nopar\label{fig:degengraph}
\end{center}
\end{figure}
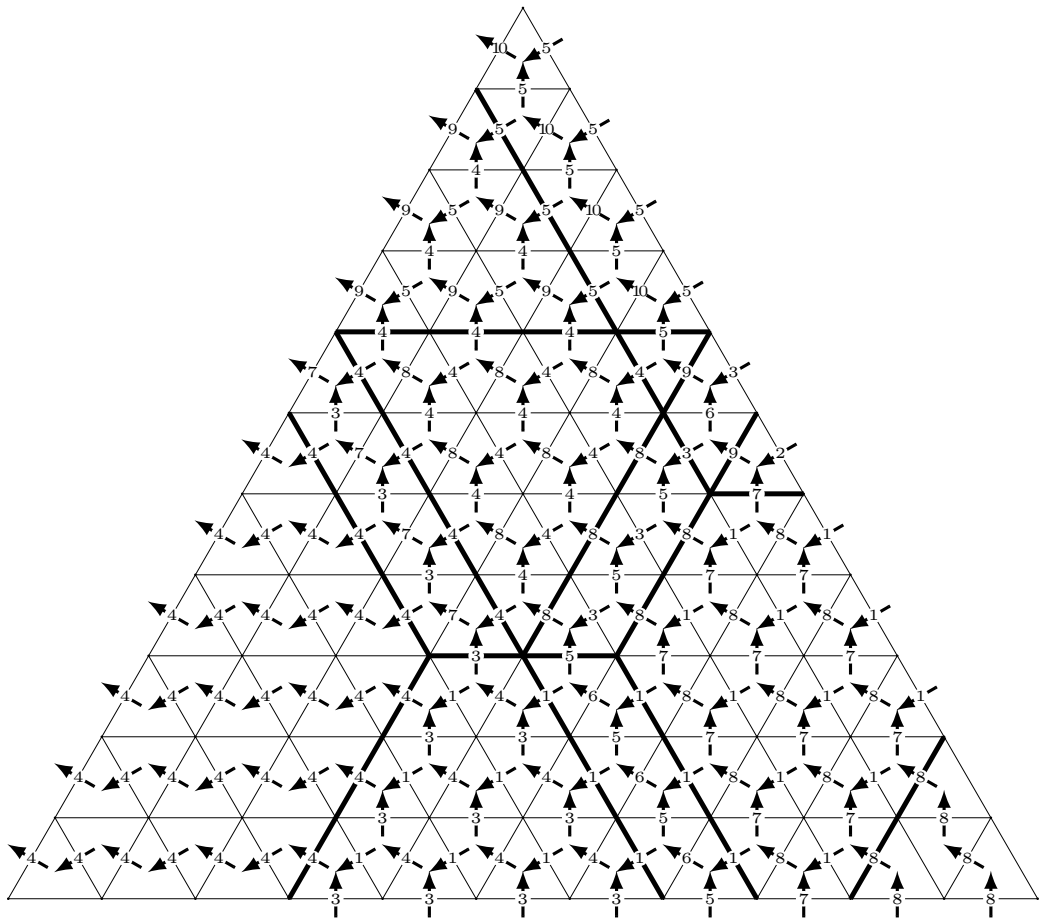


We formulate now capacity constraints for the throughputs of hive flows 
at the border of~$\Delta$, depending on a chosen triple $\la,\mu,\nu\in\N^n$ of partitions 
satisfying $|\nu|=|\la|+|\mu|$.
Hereby, we treat the left border of $\Delta$ differently from the 
right and bottom border of $\Delta$ with regard to orientations. 
To the $i$th border edge~$k$ of~$\Delta$ on the right border of~$\Delta$, counted from top to bottom, 
we assign the \emph{throughput capacity} $b(k):=\la_i$, see Figure~\ref{fig:triangulararraywithflow}.
Further, we set $b(k):= \mu_i$ for the $i$th edge~$k$ on the bottom border of~$\Delta$, counted from right to left.
Finally, we set $b(k'):= \nu_i$ for the $i$th edge~$k'$ on the left border of~$\Delta$, counted from top to bottom.
Recall that $\delta(k,f)$ denotes the throughput of a flow $f$ into $\Delta$, 
while $-\delta(k',f)$ denotes the throughput of~$f$ out of $\Delta$.



\begin{definition}\label{def:hiveflowpolytope}
Let $\la,\mu,\nu\in\N^n$ be a triple of partitions satisfying $|\nu|=|\la|+|\mu|$.
The \emph{polytope of bounded hive flows} $B:=B(\la,\mu,\nu) \subseteq \oF(G)$ 
is defined to be the set of hive flows $f\in\oF(G)$ satisfying 
$$
 0\le\delta(k,f) \leq b(k) \ \mbox{ and }\  
 0\le-\delta(k',f) \leq b(k') 
$$
for all border edges~$k$ on the right or bottom border of $\Delta$, 
and for all border edges~$k'$ on the left border of $\Delta$.
The \emph{polytope of capacity achieving hive flows} $P:=P(\la,\mu,\nu)$ 
consists of those $f\in B(\la,\mu,\nu)$ for which  
$\delta(k,f) = b(k)$ and $-\delta(k',f) = b(k')$ 
for all $k$ and $k'$ as above.
We also set $B_\Z := B \cap \oF(G)_\Z$ and $P_\Z := P \cap \oF(G)_\Z$.
\end{definition}

Lemma~\ref{le:bbound} and the isomorphism $\oF(G)_\Z\simeq H_\Z$ 
imply that $B$ is bounded and thus $B$ and~$P$ are indeed polytopes. 


We note that by~\eqref{eq:overall-thrput}, we have 
$\delta(f) \le |\nu|$ for any $f\in  B(\la,\mu,\nu)$.
Moreover, $f\in B(\la,\mu,\nu)$ is capacity achieving iff $\delta(f) = |\nu|$.

%
%
Knutson and Tao~\cite{knta:99} (see also \cite{buc:00}) characterized 
the Littlewood--Richardson coefficient $\LRC \la \mu \nu$ as the number 
of integral hives taking fixed values on the border vertices of~$\Delta$, 
prescribed by the partitions $\la,\mu,\nu$. 
Their description via the isomorphism $\oF(G)_\Z \simeq H_\Z$ immediately translates 
to the following fundamental result.

\begin{proposition}
\label{pro:flowdescription}
The Littlewood--Richardson coefficient $\LRC \la \mu \nu$ equals the number of capacity achieving integral hive flows,
i.e., $\LRC \la \mu \nu = |P(\la,\mu,\nu)_\Z|$. \endproof
\end{proposition}

To advocate the advantage of the flow interpretation of Littlewood--Richardson coefficients,
we show in the next section that $P_\Z := P(\la,\mu,\nu)_\Z$ 
can be interpreted as the set of vertices
of a graph in a natural way.  
This will be important for searching and enumerating~$P_\Z$
in an efficient way. 
Our investigations will be purely structural though. 
We leave the (more complicated) algorithmic aspects 
of searching to the forthcoming paper~\cite{ike:12}. 


%


\section{Properties of hive flows} 
\label{se:graph-str}

We recall that any complete path~$p$ defines a flow on $G$, denoted 
by the same symbol. In order to describe the slack of a rhombus with 
respect to $p$, we introduce some further terminology. 


\begin{definition}\label{def:turn}
A \emph{turn} is defined to be a path in $G$ of length~2 that lies inside $\Delta$, 
starts at a white vertex and ends with a different white vertex, 
see Figure~\ref{fig:dualtriangulararray}. 
\end{definition}

Note that there are six turns in each hive triangle. 
We shall denote turns pictorially by $\rhpoulMl$, $\rhpoulMr$ etc.\  
with the obvious interpretation. 
Similarly, $\rhpourMll$ and $\rhpoulMrl$ stand for a path consisting of four edges. 

In order to describe the different ways a complete path~$p$ may pass a rhombus~$\varrho$, 
we consider the following sets of paths in $\varrho$. 

\begin{definition}\label{def:slack-contr}
The sets of paths, interpreted as subsets of $E(G)$,  
$$
 \p_+(\rhc) := \{\rhpourMr, \rhpollWr, \rhpourMll, \rhpollWll\},\quad
 \p_-(\rhc) := \{\rhpoulMl, \rhpolrWl, \rhpoulMrr, \rhpolrWrr\},\mbox{ and } 
 \p_0(\rhc) := \{\rhpoulMrl,\rhpourMlr,\rhpollWlr,\rhpolrWrl\}
$$
are called the sets of  of 
\emph{positive},
\emph{negative}, and 
\emph{neutral slack contributions} of 
the rhombus $\rhc$, respectively. 
\end{definition}

For later use the reader should remember that the turns in $\p_+(\rhc)$ 
at the acute angles are clockwise, while the concatenations of two turns 
at the obtuse angles are counterclockwise. 

The verification of the following is immediate using Definition~\ref{def:slack} of 
the slack. 

\begin{observation}
Let $p$ be a complete path in $G$ and $E_\varrho$ be the set of edges of $G$ contained in a rhombus~$\varrho$. 
Then $p\cap E_\varrho$ is either empty, or it is a union of one or two slack contributions~$q$.
The slack $\s \varrho p$ is obtained by adding $1$, $0$, or $-1$ 
over the contributions~$q$ contained in $p$, according to whether~$q$ is   
is  positive, negative, or neutral.
\end{observation}

We remark that the only situations, 
in which $p\cap E_\varrho$ is a union of two slack contributions,
is when $p$ uses both counterclockwise turns $\rhpoulMl$ and $ \rhpolrWl$ at acute angles, 
or both clockwise turns $\rhpourMr$ and $\rhpollWr$ at acute angles, in which case $\s \varrho p= -2$ or $\s \varrho p = 2$, 
respectively.
It is not possible that $c$ uses both $\rhpoulMl$ and $ \rhpollWr$ since otherwise, 
due to the planarity of $\Delta$, $c$~would have to intersect itself.

\subsection{The support of flows on $G$}
\label{se:antipodal}

Recall the definition of the support $\SUPP(d)$ of a flow class $d\in\oF(G)$.
By the definition, $\SUPP(d)$ cannot contain an edge and its reverse. 
We note the following:
$$
\big( \rhpcMl \subseteq \SUPP(d) \ \text{ or } \ \rhpcMr \subseteq
\SUPP(d)\big) \quad \Longleftrightarrow \quad \rhacM(d) > 0 .
$$


Recall from Definition~\ref{def:slack-contr} 
the sets $\Psi_+(\varrho)$, $\Psi_-(\varrho)$, and $\Psi_0(\varrho)$ of 
positive, negative, and neutral slack contributions of a rhombus $\varrho$, respectively, 
interpreted as sets of directed edges of $G$.
We assign to any slack contribution $p\in \Psi_+(\varrho)\cup\Psi_-(\varrho)$ 
of a rhombus $\varrho$ 
its {\em antipodal contribution} $p'\in \Psi_+(\varrho)\cup\Psi_-(\varrho)$, 
which is defined by reversing $p$ and then applying a rotation of $180^\circ$.
For instance, $\rhpolrWl$ is the antipodal contribution of $\rhpourMr$ and 
$\rhpoulMrr$ 
is the antipodal contribution of $\rhpourMll$.
Clearly, $p\mapsto p'$ is an involution. 

The following lemma on antipodal contributions will be of great use.

\begin{lemma}
\label{cla:negimpliespos}
Let $d\in\oF(G)$ such that $\s \varrho d \geq 0$ for a rhombus~$\varrho$.
If $p\subseteq\SUPP(d)$ for a negative slack contribution~$p$ of $\varrho$, 
then $p'\subseteq\SUPP(d)$ for its antipodal contribution~$p'$. 
\end{lemma}

\begin{proof}
1. Suppose that $\rhpoulMl \subseteq \SUPP(d)$, which means $\delta_1 := \rhaoulM(d) > 0$ and $\delta_2 := \rhaourW(d) > 0$.
Since $\delta_3 := \rhaolrM(d) - \rhaoulM(d) = \s \rhc d \geq 0$ we get $\rhaolrM(d) = \delta_1+\delta_3 > 0$.
Moreover, 
$\rhaollW(d)=\rhaolrM(d) - \rhacW(d) = (\delta_1+\delta_3)-(\delta_1-\delta_2) = \delta_3+\delta_2 > 0$.
Altogether, $\rhpollWr \subseteq \SUPP(d)$.

2. Suppose that $\rhpoulMrr \subseteq \SUPP(d)$, which means
$\delta_1 := \rhaoulM(d) > 0$, $\delta_2 := \rhaollM(d) >0$, and $\delta_3 := \rhacW(d) >0$.
Hence $\rhaourM(d) = \delta_3-\delta_1$.
We have $\delta_4 := \rhaolrM(d) - \rhaoulM(d) = \s \rhc d \geq 0$ and thus $\rhaolrM(d) = \delta_1+\delta_4 > 0$.
Therefore $\delta_3=\delta_2+(\delta_1+\delta_4)$ and thus 
$\rhaourM(d) = \delta_3-\delta_1 = (\delta_2+\delta_1+\delta_4)-\delta_1=\delta_2+\delta_4 > 0$.
Altogether, $\rhpourMll \subseteq\SUPP(d)$.
\qquad\end{proof}

Applying Lemma~\ref{cla:negimpliespos} successively  
can provide important information about the support of a flow class $d$. 
This is stated in the following lemma on ``flow propagation''. 

It will be convenient to use symbols like $\rhrhoul$, $\rhrholl$ etc., which stand for the rhombi in the positions relative to $\rhc$ 
as indicated by the shaded regions.

\begin{lemma}
\label{lem:flowpropagation}
Given $d \in \oF(G)$ such that $\s \rhc d \geq 0$ and $\rhpoulMl \subseteq \SUPP(d)$.
Then $\rhpollWr \subseteq \SUPP(d)$.
If additionally $\s \rhrholl d \geq 0$, then $\rhpllMl \subseteq \SUPP(d)$.
Similarly, if additionally $\s \rhrholr d \geq 0$, then $\rhpolrMl \subseteq \SUPP(d)$.
\end{lemma}

For an example on how Lemma~\ref{lem:flowpropagation} can be used, see Figure~\ref{fig:flowpropagation}.

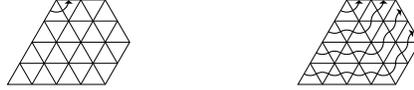
\begin{figure}[h]
\begin{center}
\begin{tikzpicture}\draw[rhrhombidraw] (0.0pt,0.0pt) -- (36.952pt,0.0pt) -- (46.19pt,16.0002pt) -- (36.952pt,32.0004pt) -- (18.476pt,32.0004pt) -- cycle;\draw[thin,-my] (16.1665pt,28.0004pt) arc (240:300:4.619pt) arc (300:360:4.619pt);\draw[rhrhombidraw] (18.476pt,32.0004pt) -- (36.952pt,0.0pt) ;\draw[rhrhombidraw] (27.714pt,32.0004pt) -- (41.571pt,8.0001pt) ;\draw[rhrhombidraw] (13.857pt,24.0003pt) -- (27.714pt,0.0pt) ;\draw[rhrhombidraw] (9.238pt,16.0002pt) -- (18.476pt,0.0pt) ;\draw[rhrhombidraw] (4.619pt,8.0001pt) -- (9.238pt,0.0pt) ;\draw[rhrhombidraw] (4.619pt,8.0001pt) -- (41.571pt,8.0001pt) ;\draw[rhrhombidraw] (9.238pt,16.0002pt) -- (46.19pt,16.0002pt) ;\draw[rhrhombidraw] (13.857pt,24.0003pt) -- (41.571pt,24.0003pt) ;\draw[rhrhombidraw] (41.571pt,24.0003pt) -- (27.714pt,0.0pt) ;\draw[rhrhombidraw] (36.952pt,32.0004pt) -- (18.476pt,0.0pt) ;\draw[rhrhombidraw] (27.714pt,32.0004pt) -- (9.238pt,0.0pt) ;\end{tikzpicture}
\hspace{2cm}
\begin{tikzpicture}\draw[rhrhombidraw] (0.0pt,0.0pt) -- (36.952pt,0.0pt) -- (46.19pt,16.0002pt) -- (36.952pt,32.0004pt) -- (18.476pt,32.0004pt) -- cycle;\draw[thin,-my] (16.1665pt,28.0004pt) arc (240:300:4.619pt) arc (300:360:4.619pt);\draw[thin,-my] (11.5475pt,20.0003pt) arc (240:300:4.619pt) arc (120:60:4.619pt) arc (240:300:4.619pt) arc (300:360:4.619pt) arc (180:120:4.619pt) arc (300:360:4.619pt);\draw[thin,-my] (6.9285pt,12.0002pt) arc (240:300:4.619pt) arc (120:60:4.619pt) arc (240:300:4.619pt) arc (120:60:4.619pt) arc (240:300:4.619pt) arc (300:360:4.619pt) arc (180:120:4.619pt) arc (300:360:4.619pt) arc (180:120:4.619pt);\draw[thin,-my] (2.3095pt,4.0001pt) arc (240:300:4.619pt) arc (120:60:4.619pt) arc (240:300:4.619pt) arc (120:60:4.619pt) arc (240:300:4.619pt) arc (120:60:4.619pt) arc (240:300:4.619pt) arc (300:360:4.619pt) arc (180:120:4.619pt) arc (300:360:4.619pt) arc (180:120:4.619pt);\draw[rhrhombidraw] (4.619pt,8.0001pt) -- (9.238pt,0.0pt) ;\draw[rhrhombidraw] (9.238pt,16.0002pt) -- (18.476pt,0.0pt) ;\draw[rhrhombidraw] (13.857pt,24.0003pt) -- (27.714pt,0.0pt) ;\draw[rhrhombidraw] (18.476pt,32.0004pt) -- (36.952pt,0.0pt) ;\draw[rhrhombidraw] (27.714pt,32.0004pt) -- (41.571pt,8.0001pt) ;\draw[rhrhombidraw] (4.619pt,8.0001pt) -- (41.571pt,8.0001pt) ;\draw[rhrhombidraw] (9.238pt,16.0002pt) -- (46.19pt,16.0002pt) ;\draw[rhrhombidraw] (13.857pt,24.0003pt) -- (41.571pt,24.0003pt) ;\draw[rhrhombidraw] (27.714pt,32.0004pt) -- (9.238pt,0.0pt) ;\draw[rhrhombidraw] (36.952pt,32.0004pt) -- (18.476pt,0.0pt) ;\draw[rhrhombidraw] (41.571pt,24.0003pt) -- (27.714pt,0.0pt) ;\end{tikzpicture}
    \caption{The rhombi of the pentagon have nonnegative slack with respect to the flow~$d$. If the turns in the left picture are in $\SUPP(d)$, then, 
      by applying Lemma~\protect\ref{lem:flowpropagation} several times, we see that all the turns in the right picture are in $\SUPP(d)$.}
    \nopar\label{fig:flowpropagation}
\end{center}
\end{figure}

\begin{prooff}[Proof of Lemma~\textup{\ref{lem:flowpropagation}}]
The first assertion is a direct application of Lemma~\ref{cla:negimpliespos}.
Suppose that $\s \rhrholl d \geq 0$.
Since $\rhpollWr \subseteq \SUPP(d)$,
flow conservation implies that $\rhpllMl \subseteq \SUPP(d)$ or $\rhppllMr \subseteq \SUPP(d)$.
We want to show $\rhpllMl \subseteq \SUPP(d)$.
If $\rhppllMr \subseteq \SUPP(d)$, then $\rhppllMrr \subseteq \SUPP(d)$ and $\rhppllMrr$ is a negative contribution in~$\rhrholl$.
Hence by Lemma~\ref{cla:negimpliespos}, we have $\rhpllMll \subseteq \SUPP(d)$.
The other assertion is proved analogously.
\qquad\end{prooff}

\subsection{The graph of capacity achieving integral hive flows}
\label{se:Graph-CA-hive-flows}

Fix $\la,\mu,\nu$ and recall the polytopes $B$ and $P$ from Definition~\ref{def:hiveflowpolytope}.
We show now that $P_\Z$ can be naturally seen as the vertex set of a graph. 

\begin{definition}
We say that $f,g\in P_\Z$ are {\em neighbours} iff $g-f$ is a cycle in $G$.
The resulting graph with the set of vertices $P_\Z$ is also denoted by $P_\Z$. 
\end{definition}

The neighbour relation is clearly symmetric. 
We also remark that a cycle of the form $g-f$ 
must be proper, i.e., it neither uses the source or target. 
The reason is that the flow $g-f$ vanishes on the edges touching the 
border of $\Delta$, as $f$ and $g$ are both capacity achieving. 

For an explicit characterization of the neighbour relation 
we need the following concepts.


\begin{definition}\label{def:f-secure}
Let $f\in B$ and $c$ be a proper cycle in $G$.

1. We call a rhombus $\varrho$ 
{\em nearly $f$-flat} iff $\s\varrho f = 1$. 

2. $c$ is called {\em $f$-hive preserving} iff 
$c$ does not use negative contributions in $f$-flat rhombi. 

3. $c$ is called {\em $f$-secure} 
iff $c$ is $f$-hive preserving and 
$c$ does not use both counterclockwise turns at acute angles in 
nearly $f$-flat rhombi (\rhpoulMlXolrWl).
\end{definition}

We remark that $c$ is $f$-hive preserving iff 
$f + \varepsilon c \in B$ for sufficiently small $\varepsilon>0$. 


\begin{proposition}\label{pro:char-f-secure}
Assume $f\in P_\Z$. 
If $g\in P_\Z$ is a neighbour of $f$, then $g-f$ is an $f$-secure cycle.
Conversely, if $c$ is an $f$-secure cycle, then $f+c\in P_\Z$ is a neighbour of~$f$. 
\end{proposition}

\begin{proof}
Assume that $f$ and $g$ are neighbours in $P_\Z$, so 
$c:=g-f$ is a proper cycle in~$G$. 
Hence $\s \varrho {f+c} \ge 0$ for each rhombus~$\varrho$.
This implies that $\s \varrho c \ge 0$ for each $f$-flat rhombus, 
that is, $c$ is $f$-hive preserving. 
Moreover, if $\s \varrho f = 1$, then $\s \varrho c \ge -1$.
Hence $c$ is $f$-secure.
The argument can be reversed.
\qquad\end{proof}


Apparently, the symmetry of the neighbour relation in $P_\Z$ does not seem to be obvious 
from the characterization in Proposition~\ref{pro:char-f-secure}. 

Before continuing, we state a useful observation.
The union of two overlapping rhombi $\varrho_1$ and $\varrho_2$ forms a trapezoid.
Glueing together two such trapezoids $(\varrho_1,\varrho_2)$ and $(\varrho'_1,\varrho'_2)$ at their longer side,
we get a hexagon. 
The verification of the following {\em hexagon equality} is
straightforward and left to the reader: 
$\s {\varrho_1} f + \s {\varrho_2} f = \s {\varrho'_1} f + \s {\varrho'_2} f$ 
for any flow $f\in \oF(G)$. In pictorial notation, the hexagon equality can be succinctly expressed as 
\begin{equation}
\label{cla:BZ}
\s {
\begin{tikzpictured}\fill[rhrhombifill] (9.238pt,0.0pt) -- (4.619pt,8.0001pt) -- (9.238pt,16.0002pt) -- (13.857pt,8.0001pt) -- cycle;\fill (0.0pt,0.0pt) circle (0.4pt);\fill (9.238pt,0.0pt) circle (0.4pt);\fill (4.619pt,8.0001pt) circle (0.4pt);\fill (-4.619pt,8.0001pt) circle (0.4pt);\fill (13.857pt,8.0001pt) circle (0.4pt);\fill (0.0pt,16.0002pt) circle (0.4pt);\fill (9.238pt,16.0002pt) circle (0.4pt);\draw[rhrhombidraw] (0.0pt,16.0002pt) -- (9.238pt,0.0pt) ;\end{tikzpictured}
} f + \s {
\begin{tikzpictured}\fill[rhrhombifill] (13.857pt,8.0001pt) -- (9.238pt,16.0002pt) -- (0.0pt,16.0002pt) -- (4.619pt,8.0001pt) -- cycle;\fill (0.0pt,0.0pt) circle (0.4pt);\fill (9.238pt,0.0pt) circle (0.4pt);\fill (4.619pt,8.0001pt) circle (0.4pt);\fill (-4.619pt,8.0001pt) circle (0.4pt);\fill (13.857pt,8.0001pt) circle (0.4pt);\fill (0.0pt,16.0002pt) circle (0.4pt);\fill (9.238pt,16.0002pt) circle (0.4pt);\draw[rhrhombidraw] (9.238pt,0.0pt) -- (0.0pt,16.0002pt) ;\end{tikzpictured}
} f = \s {
\begin{tikzpictured}\fill[rhrhombifill] (0.0pt,16.0002pt) -- (-4.619pt,8.0001pt) -- (0.0pt,0.0pt) -- (4.619pt,8.0001pt) -- cycle;\fill (0.0pt,0.0pt) circle (0.4pt);\fill (9.238pt,0.0pt) circle (0.4pt);\fill (4.619pt,8.0001pt) circle (0.4pt);\fill (-4.619pt,8.0001pt) circle (0.4pt);\fill (13.857pt,8.0001pt) circle (0.4pt);\fill (0.0pt,16.0002pt) circle (0.4pt);\fill (9.238pt,16.0002pt) circle (0.4pt);\draw[rhrhombidraw] (0.0pt,16.0002pt) -- (9.238pt,0.0pt) ;\end{tikzpictured}
} f + \s {
\begin{tikzpictured}\fill[rhrhombifill] (-4.619pt,8.0001pt) -- (0.0pt,0.0pt) -- (9.238pt,0.0pt) -- (4.619pt,8.0001pt) -- cycle;\fill (0.0pt,0.0pt) circle (0.4pt);\fill (9.238pt,0.0pt) circle (0.4pt);\fill (4.619pt,8.0001pt) circle (0.4pt);\fill (-4.619pt,8.0001pt) circle (0.4pt);\fill (13.857pt,8.0001pt) circle (0.4pt);\fill (0.0pt,16.0002pt) circle (0.4pt);\fill (9.238pt,16.0002pt) circle (0.4pt);\draw[rhrhombidraw] (0.0pt,16.0002pt) -- (9.238pt,0.0pt) ;\end{tikzpictured}
} f,
\end{equation}

As an immediate consequence we obtain the following. 

\begin{corollary}
\label{cor:BZ2}
For all hive flows $f$, if
\begin{tikzpictured}\fill[rhrhombifill] (9.238pt,0.0pt) -- (4.619pt,8.0001pt) -- (9.238pt,16.0002pt) -- (13.857pt,8.0001pt) -- cycle;\fill (0.0pt,0.0pt) circle (0.4pt);\fill (9.238pt,0.0pt) circle (0.4pt);\fill (4.619pt,8.0001pt) circle (0.4pt);\fill (-4.619pt,8.0001pt) circle (0.4pt);\fill (13.857pt,8.0001pt) circle (0.4pt);\fill (0.0pt,16.0002pt) circle (0.4pt);\fill (9.238pt,16.0002pt) circle (0.4pt);\draw[rhrhombidraw] (0.0pt,16.0002pt) -- (9.238pt,0.0pt) ;\end{tikzpictured}
and
\begin{tikzpictured}\fill[rhrhombifill] (13.857pt,8.0001pt) -- (9.238pt,16.0002pt) -- (0.0pt,16.0002pt) -- (4.619pt,8.0001pt) -- cycle;\fill (0.0pt,0.0pt) circle (0.4pt);\fill (9.238pt,0.0pt) circle (0.4pt);\fill (4.619pt,8.0001pt) circle (0.4pt);\fill (-4.619pt,8.0001pt) circle (0.4pt);\fill (13.857pt,8.0001pt) circle (0.4pt);\fill (0.0pt,16.0002pt) circle (0.4pt);\fill (9.238pt,16.0002pt) circle (0.4pt);\draw[rhrhombidraw] (9.238pt,0.0pt) -- (0.0pt,16.0002pt) ;\end{tikzpictured}
are $f$\dash flat, then also
\begin{tikzpictured}\fill[rhrhombifill] (0.0pt,16.0002pt) -- (-4.619pt,8.0001pt) -- (0.0pt,0.0pt) -- (4.619pt,8.0001pt) -- cycle;\fill (0.0pt,0.0pt) circle (0.4pt);\fill (9.238pt,0.0pt) circle (0.4pt);\fill (4.619pt,8.0001pt) circle (0.4pt);\fill (-4.619pt,8.0001pt) circle (0.4pt);\fill (13.857pt,8.0001pt) circle (0.4pt);\fill (0.0pt,16.0002pt) circle (0.4pt);\fill (9.238pt,16.0002pt) circle (0.4pt);\draw[rhrhombidraw] (0.0pt,16.0002pt) -- (9.238pt,0.0pt) ;\end{tikzpictured}
and
\begin{tikzpictured}\fill[rhrhombifill] (-4.619pt,8.0001pt) -- (0.0pt,0.0pt) -- (9.238pt,0.0pt) -- (4.619pt,8.0001pt) -- cycle;\fill (0.0pt,0.0pt) circle (0.4pt);\fill (9.238pt,0.0pt) circle (0.4pt);\fill (4.619pt,8.0001pt) circle (0.4pt);\fill (-4.619pt,8.0001pt) circle (0.4pt);\fill (13.857pt,8.0001pt) circle (0.4pt);\fill (0.0pt,16.0002pt) circle (0.4pt);\fill (9.238pt,16.0002pt) circle (0.4pt);\draw[rhrhombidraw] (0.0pt,16.0002pt) -- (9.238pt,0.0pt) ;\end{tikzpictured}
are $f$\dash flat. \endproof
\end{corollary}

\begin{remark}
The slacks of rhombi are exactly the numbers in Berenstein-Zelevinsky triangles \cite{pv:05} and 
the hexagon equality~\eqref{cla:BZ} is just the condition for their validity.
\end{remark}

The next result tells us how $f$-secure cycles may arise. 

\begin{theorem}\label{pro:shortestplanarcycle}
Let $f\in B_\Z$ and $c$ be an $f$\dash hive preserving cycle in $G$ of minimal length. 
Then $c$ is $f$-secure. 
\end{theorem}

\begin{proof}
We argue by contradiction. Suppose that $c$ is an $f$-hive preserving cycle in $G$ of minimal length, 
but not $f$-secure. So there is a nearly $f$-flat rhombus $\rhc$ in which $c$ uses both turns $\rhpoulMl$ and $\rhpolrWl$. 
Let us call such rhombi {\em bad}.

Since $c$ has minimal length, it cannot be rerouted via $\rhpoulMrr$.
Hence we cannot be in the following case, in which $c$~can easily be rerouted via $\rhpoulMrr$:
\begin{equation*}
\mbox{\Big($\rhrhoul$ is not $f$\dash flat or $c$ uses $\rhppulWll$\Big) \quad and\quad
\Big($\rhrholl$ is not $f$\dash flat or $c$ uses $\rhpcWrr$\Big)}.
\end{equation*}
In the remaining two cases
\begin{equation}\label{eq:A-B}
\mbox{(A) $\rhrhoul$ is $f$\dash flat and $c$ uses $\rhpulWrl$\quad or\quad
(B) $\rhrholl$ is $f$\dash flat and $c$ uses $\rhpolrWlr$}
\end{equation}
the $f$\dash hive preserving cycle~$c$ cannot be rerouted via $\rhpoulMrr$ by Definition~\ref{def:f-secure}(2).

Let us assume that we are in the situation (A). 
So we have the bad, nearly $f$-flat rhombus $\rhc$ 
and the shaded $f$-flat rhombus.
The hexagon equality~\eqref{cla:BZ} implies that
either $\rhrhpl$ is $f$\dash flat and $\rhrhll$ is nearly $f$\dash flat,
or $\rhrhpl$ is nearly $f$\dash flat and $\rhrhll$ is $f$\dash flat.
These two possibilities are indicated on the left and right side of the following picture, respectively,
where the shaded rhombi are $f$\dash flat and  
the diagonals of nearly $f$\dash flat rhombi are drawn thick.
Further, parts of~$c$ which run in $f$-flat rhombi, are drawn with straight arrows:
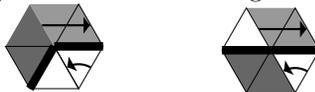
\begin{center}
\scalebox{2}{
\begin{tikzpicture}\fill[thick,fill=black!60] (-13.857pt,8.0001pt) -- (-18.476pt,0.0pt) -- (-13.857pt,-8.0001pt) -- (-9.238pt,0.0pt) -- cycle;\fill[rhrhombifill] (-9.238pt,0.0pt) -- (0.0pt,0.0pt) -- (-4.619pt,8.0001pt) -- (-13.857pt,8.0001pt) -- cycle;\draw[rhrhombidraw] (-4.619pt,-8.0001pt) -- (0.0pt,0.0pt) -- (-4.619pt,8.0001pt) -- (-9.238pt,0.0pt) -- cycle;\draw[rhrhombidraw] (-4.619pt,-8.0001pt) -- (-9.238pt,0.0pt) -- (-18.476pt,0.0pt) -- (-13.857pt,-8.0001pt) -- cycle;\draw[rhrhombithickside] (0.0pt,0.0pt) -- (-9.238pt,0.0pt);\draw[rhrhombithickside] (-9.238pt,0.0pt) -- (-13.857pt,-8.0001pt);\draw[thin,-my] (-11.5475pt,4.0001pt) --  (-2.3095pt,4.0001pt);;\draw[thin,-my] (-2.3095pt,-4.0001pt) arc (60:120:4.619pt);\end{tikzpicture}
\hspace{0.5cm}
\begin{tikzpicture}\fill[thick,fill=black!60] (-18.476pt,0.0pt) -- (-9.238pt,0.0pt) -- (-4.619pt,-8.0001pt) -- (-13.857pt,-8.0001pt) -- cycle;\fill[rhrhombifill] (-9.238pt,0.0pt) -- (0.0pt,0.0pt) -- (-4.619pt,8.0001pt) -- (-13.857pt,8.0001pt) -- cycle;\draw[rhrhombidraw] (-4.619pt,-8.0001pt) -- (0.0pt,0.0pt) -- (-4.619pt,8.0001pt) -- (-9.238pt,0.0pt) -- cycle;\draw[rhrhombithickside] (0.0pt,0.0pt) -- (-9.238pt,0.0pt);\draw[thin,-my] (-11.5475pt,4.0001pt) --  (-2.3095pt,4.0001pt);;\draw[thin,-my] (-2.3095pt,-4.0001pt) arc (60:120:4.619pt);\draw[rhrhombidraw] (-13.857pt,8.0001pt) -- (-18.476pt,0.0pt) -- (-13.857pt,-8.0001pt) -- (-9.238pt,0.0pt) -- cycle;\draw[rhrhombithickside] (-9.238pt,0.0pt) -- (-18.476pt,0.0pt);\end{tikzpicture}
}
\end{center}
The fact that $c$ uses no negative contributions in $f$-flat rhombi (and no vertex of $G$ twice)  
forces $c$ to run exactly as depicted in the following picture:
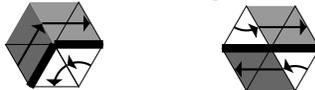
\begin{center}
\scalebox{2}{
\begin{tikzpicture}\fill[thick,fill=black!60] (-13.857pt,8.0001pt) -- (-18.476pt,0.0pt) -- (-13.857pt,-8.0001pt) -- (-9.238pt,0.0pt) -- cycle;\fill[rhrhombifill] (-9.238pt,0.0pt) -- (0.0pt,0.0pt) -- (-4.619pt,8.0001pt) -- (-13.857pt,8.0001pt) -- cycle;\draw[rhrhombidraw] (-4.619pt,-8.0001pt) -- (0.0pt,0.0pt) -- (-4.619pt,8.0001pt) -- (-9.238pt,0.0pt) -- cycle;\draw[rhrhombidraw] (-4.619pt,-8.0001pt) -- (-9.238pt,0.0pt) -- (-18.476pt,0.0pt) -- (-13.857pt,-8.0001pt) -- cycle;\draw[rhrhombithickside] (0.0pt,0.0pt) -- (-9.238pt,0.0pt);\draw[rhrhombithickside] (-9.238pt,0.0pt) -- (-13.857pt,-8.0001pt);\draw[thin,-my] (-11.5475pt,4.0001pt) --  (-2.3095pt,4.0001pt);;\draw[thin,-my] (-2.3095pt,-4.0001pt) arc (60:120:4.619pt);\draw[thin,-my] (-16.1665pt,-4.0001pt) --  (-11.5475pt,4.0001pt);;\draw[thin,-my] (-6.9285pt,-4.0001pt) arc (120:180:4.619pt);\end{tikzpicture}
\hspace{0.5cm}
\begin{tikzpicture}\fill[thick,fill=black!60] (-18.476pt,0.0pt) -- (-9.238pt,0.0pt) -- (-4.619pt,-8.0001pt) -- (-13.857pt,-8.0001pt) -- cycle;\fill[rhrhombifill] (-9.238pt,0.0pt) -- (0.0pt,0.0pt) -- (-4.619pt,8.0001pt) -- (-13.857pt,8.0001pt) -- cycle;\draw[rhrhombidraw] (-4.619pt,-8.0001pt) -- (0.0pt,0.0pt) -- (-4.619pt,8.0001pt) -- (-9.238pt,0.0pt) -- cycle;\draw[rhrhombithickside] (0.0pt,0.0pt) -- (-9.238pt,0.0pt);\draw[thin,-my] (-11.5475pt,4.0001pt) --  (-2.3095pt,4.0001pt);;\draw[thin,-my] (-2.3095pt,-4.0001pt) arc (60:120:4.619pt);\draw[rhrhombidraw] (-13.857pt,8.0001pt) -- (-18.476pt,0.0pt) -- (-13.857pt,-8.0001pt) -- (-9.238pt,0.0pt) -- cycle;\draw[rhrhombithickside] (-9.238pt,0.0pt) -- (-18.476pt,0.0pt);\draw[thin,-my] (-6.9285pt,-4.0001pt) --  (-16.1665pt,-4.0001pt);;\draw[thin,-my] (-16.1665pt,4.0001pt) arc (240:300:4.619pt);\end{tikzpicture}
}
\end{center}
Hence the second nearly $f$\dash flat rhombus in the left and right picture, respectively, is bad as well.

So we see that the diagonal $\rhsc$ of the bad rhombus $\rhc$ shares a vertex $\rhdwest$ with the diagonal of another
bad rhombus and that both diagonals either lie on the same line or include an angle of $120^\circ$.
By symmetry, the same conclusion can be drawn in the case (B). 

By induction, this implies that there is a region bounded by diagonals of bad rhombi.
This is impossible, because $c$ would have to run both inside and outside of this region.
\qquad\end{proof}

The following is an important insight into the structure of $P_\Z$. 
We postpone the proof to Section~\ref{sec:connectedness}.

\begin{theorem}[Connectedness Theorem]
\label{thm:connectedness}
The graph $P_\Z$ is connected.
\end{theorem}


As an application of our insights, we obtain the following characterization of 
multiplicity freeness. Recall that 
$\LRC \la \mu \nu = |P(\la,\mu,\nu)_\Z|$. 

\begin{proposition}\label{pro:char-mf} 
Suppose that $f\in P(\la,\mu,\nu)_\Z$. 
Then we have $\LRC \la \mu \nu > 1$ iff there exists an $f$\dash hive preserving cycle in $G$. 
\end{proposition}

\begin{proof}
If there exists an $f$-hive preserving cycle in $G$, then there is also one of minimal length, 
call it $c$. 
Theorem~\ref{pro:shortestplanarcycle} implies that $c$ is $f$-secure. 
Proposition~\ref{pro:char-f-secure} tells us that $f+c\in P_\Z$. 
It follows that $|P_\Z| \ge 2$. 

Conversely, assume that $|P_\Z| \ge 2$. 
Since $f\in P_\Z$ and $P_\Z$ is connected 
by Theorem~\ref{thm:connectedness}, 
there exists a neighbour $g\in P_\Z$ of $f$. 
Proposition~\ref{pro:char-f-secure} tells us that
$g-f$ is an $f$-secure cycle. 
\qquad\end{proof}

A proof of Fulton's conjecture, first shown in~\cite{ktw:04} by different methods, 
is obtained as an easy consequence.

\begin{corollary}\label{thm:fultonconj}
If $\LRC {\la} {\mu} {\nu} = 1$, then  
$\LRC {N\la} {N\mu} {N\nu} = 1$ for all $N\ge 1$. 
\end{corollary}

\begin{proof}
By definition, 
$c$ is an $f$-hive preserving cycle in $G$ iff
$c$ is an $Nf$-hive preserving cycle in $G$. 
Now apply Proposition~\ref{pro:char-mf}.  
\qquad\end{proof}
 

The characterization in Proposition~\ref{pro:char-mf} 
points to a way of algorithmically deciding whether $\LRC \la \mu \nu > 1$.
However, it is not obvious how to efficiently search for $f$-hive preserving cycles 
in the graph~$G$.  
For this, and even for the simpler task of deciding $\LRC \la \mu \nu > 0$, 
we have to construct suitable ``residual digraphs'', which brings us to the 
topic of the next section. 
More details on the complexity of testing $\LRC \la \mu \nu > 1$ 
can be found in \cite{ike:12}.


\section{The residual digraph $\resf f$}\label{se:res}

Proposition~\ref{pro:flowdescription} suggests to decide $\LRC \la \mu \nu > 0$ 
by solving the problem of optimizing the linear (overall throughput) function $\delta$ 
on the polytope~$B=B(\la,\mu,\nu)$ of bounded hive flows. 
In fact, we will show later that the optimum is always obtained at an integral flow.
Maximizing a certain linear functional on the hive polytope and showing that the 
maximum is attained at an integer point is the basic idea in~\cite{knta:99, buc:00}. 
However, they do not present any algorithmic result.

We follow a Ford-Fulkerson approach and try to construct for a given integral hive flow $f\in B_\Z$ 
a digraph $\resf f$, such that adding an $s$\dash $t$\dash path $p$ in $\resf f$ to $f$ 
leads to a bounded hive flow. 
We have to guarantee that $f+p$ does not lead to negative slacks of rhombi 
so that $f+p$ is a hive flow. 
On the other hand, we want to make sure that $f$ is optimal, 
when there is no $s$\dash $t$\dash path in $\resf f$.


\subsection{Turnpaths and turncycles} 
\label{sec:tp-tc}

The intuition is to consider paths in $G$  
in which each node remembers its predecessor. 
This can be formally achieved by studying paths 
in an auxiliary digraph that we define next. 

Recall that a turn is a path in $G$ of length~2 that lies inside $\Delta$, 
starts at a white vertex and ends with a different white vertex. 

\begin{definition}\label{def:turnedge}
A {\em turnedge} is an ordered pair of turns that can be concatenated to a path in $G$.
\end{definition}

Note that a turnedge defines a path in $G$ of length~$4$. 
We write turnedges pictorially like
$\rhpoulMrr:=\big(\rhpoulMr,\rhpcWr\big)$ etc. 
We construct now the auxiliary digraph $\res$.

\begin{definition}\label{def:res}
The {\em digraph $\res$} has as vertices the turns, 
henceforth called \emph{turnvertices}, 
and the source and target of $G$.
The edges of $\res$ are the turnedges and the following additional edges: 
the digraph~$\res$ contains an edge $(s,\vartheta)$ from the source $s$ to any turnvertex~$\vartheta$
starting at the right or bottom border of~$\Delta$. 
Vice versa, for any turnvertex~$\vartheta'$ pointing at the right or bottom border of $\Delta$, 
there is a turnedge $(\vartheta',s)$ in $\res$. 
Similarly, for the target~$t$, 
there are edges~$(\vartheta,t)$ for each turnvertex~$\vartheta$ pointing at the left border of~$\Delta$
and vice versa, there are edges~$(t,\vartheta')$ for each turnvertex~$\vartheta'$ starting at the left border of~$\Delta$.
\end{definition}

The reader should check that $\res$~never contains an edge and its reverse. 
In fact, the digraph $\res$ is rather complicated, 
for instance one can show that it is not planar for $n\ge 2$.

We have a well-defined notion of flows on $\res$ as 
$\res$ is a digraph with two distinguished vertices $s$ and $t$.
We assign now to a flow~$f$ on $R$ a flow class $\tilde{f}$ on $G$
by defining the corresponding throughput map 
$E(\Delta) \to \R, k\mapsto \delta(k,\tilde{f})$ as follows.
An edge~$k$ of $\Delta$ lies in exactly one upright hive triangle
\begin{tikzpicture}\draw[rhrhombidraw] (0.0pt,0.0pt) -- (9.238pt,0.0pt) -- (4.619pt,8.0001pt) -- cycle;\end{tikzpicture}. 
Let 
\begin{tikzpicture}\draw[rhrhombidraw] (0.0pt,0.0pt) -- (9.238pt,0.0pt) -- (4.619pt,8.0001pt) -- cycle;\draw[thin,-my] (2.3095pt,4.0001pt) arc (60:0:4.619pt);\end{tikzpicture}
and 
\begin{tikzpicture}\draw[rhrhombidraw] (0.0pt,0.0pt) -- (9.238pt,0.0pt) -- (4.619pt,8.0001pt) -- cycle;\draw[thin,-my] (6.9285pt,4.0001pt) arc (120:180:4.619pt);\end{tikzpicture}
denote the two turns in 
\begin{tikzpicture}\draw[rhrhombidraw] (0.0pt,0.0pt) -- (9.238pt,0.0pt) -- (4.619pt,8.0001pt) -- cycle;\end{tikzpicture}
pointing towards the white vertex on~$k$. Further, let 
\begin{tikzpicture}\draw[rhrhombidraw] (0.0pt,0.0pt) -- (9.238pt,0.0pt) -- (4.619pt,8.0001pt) -- cycle;\draw[thin,-my] (4.619pt,0.0pt) arc (0:60:4.619pt);\end{tikzpicture}
and 
\begin{tikzpicture}\draw[rhrhombidraw] (0.0pt,0.0pt) -- (9.238pt,0.0pt) -- (4.619pt,8.0001pt) -- cycle;\draw[thin,-my] (4.619pt,0.0pt) arc (180:120:4.619pt);\end{tikzpicture}
denote the turns obtained when reversing 
\begin{tikzpicture}\draw[rhrhombidraw] (0.0pt,0.0pt) -- (9.238pt,0.0pt) -- (4.619pt,8.0001pt) -- cycle;\draw[thin,-my] (2.3095pt,4.0001pt) arc (60:0:4.619pt);\end{tikzpicture}
and
\begin{tikzpicture}\draw[rhrhombidraw] (0.0pt,0.0pt) -- (9.238pt,0.0pt) -- (4.619pt,8.0001pt) -- cycle;\draw[thin,-my] (6.9285pt,4.0001pt) arc (120:180:4.619pt);\end{tikzpicture}
. 
We define 
\begin{equation}\label{eq:def-delta-pi}
 \delta(k,\tilde{f}) \ =\ \rhacM (f) \ :=\ 
\inflo(
\begin{tikzpicture}\draw[rhrhombidraw] (0.0pt,0.0pt) -- (9.238pt,0.0pt) -- (4.619pt,8.0001pt) -- cycle;\draw[thin,-my] (4.619pt,0.0pt) arc (0:60:4.619pt);\end{tikzpicture}
 ,f) +\inflo(
\begin{tikzpicture}\draw[rhrhombidraw] (0.0pt,0.0pt) -- (9.238pt,0.0pt) -- (4.619pt,8.0001pt) -- cycle;\draw[thin,-my] (4.619pt,0.0pt) arc (180:120:4.619pt);\end{tikzpicture}
 ,f) 
 -\outflo(
\begin{tikzpicture}\draw[rhrhombidraw] (0.0pt,0.0pt) -- (9.238pt,0.0pt) -- (4.619pt,8.0001pt) -- cycle;\draw[thin,-my] (2.3095pt,4.0001pt) arc (60:0:4.619pt);\end{tikzpicture}
 ,f) - \outflo(
\begin{tikzpicture}\draw[rhrhombidraw] (0.0pt,0.0pt) -- (9.238pt,0.0pt) -- (4.619pt,8.0001pt) -- cycle;\draw[thin,-my] (6.9285pt,4.0001pt) arc (120:180:4.619pt);\end{tikzpicture}
 ,f).
\end{equation}
More explicitly,
$$
\rhacM (f) =  f(\rhpollWll) + f(\rhpollWlr) + f(\rhpolrWrr) + f(\rhpolrWrl) 
                    - f(\rhpoulMrr) - f(\rhpoulMrl) - f(\rhpourMll) - f(\rhpourMlr) .
$$

From~\eqref{eq:def-delta-pi} it is straightforward to check that the 
closedness condition~\eqref{eq:closedness} 
is satisfied in the hive triangle
\begin{tikzpicture}\draw[rhrhombidraw] (0.0pt,0.0pt) -- (9.238pt,0.0pt) -- (4.619pt,8.0001pt) -- cycle;\end{tikzpicture}. 

Therefore, the flow class $\tilde{f}\in\oF(G)$ 
is well defined by~\eqref{eq:def-delta-pi}. 
So we have defined the linear map 
\begin{equation}\label{eq:defpi}
 \pi\colon F(\res) \to \oF(G), f\mapsto\tilde{f} 
\end{equation}
that moreover maps integral flows to integral flows. 

Let us stress that we are only interested in the 
{\em cone $\Ke(\res)$ of nonnegative flows} on $\res$. 
We define the {\em slack} $\s \varrho f$ of rhombus $\varrho$ 
with respect to a flow $f\in \Ke(\res)$ 
by $\s \varrho f : = \s \varrho {\pi(f)}$. 
Similarly, we define the {\em throughput} 
$\rhacM (f) := \rhacM (\pi(f))$ of~$f$ through an edge~$k$, 
and we call $\delta(f) := \delta(\pi(f))$ 
the {\em overall throughput} $\delta(f)$ 
of $f\in \Ke(\res)$.
 
For the sake of clarity, paths and cycles in $\res$ 
shall be called {\em turnpaths} and {\em turncycles}. 
Correspondingly, we have the notions of 
$s$\dash $t$\dash turnpaths, $t$\dash $s$\dash turnpaths.
By a {\em complete turnpaths} we understand 
an $s$\dash $t$\dash turnpaths, a $t$\dash $s$\dash turnpaths, 
or a turncycle (which may pass through $s$ or $t$). 
A complete turnpath~$p$ defines a flow on $\res$,
again denoted by~$p$,  
by putting the flow value of 1 on each turnedge used.

\begin{example}
The flow $\pi(p)$ induced by a complete turnpath~$p$ in $\res$
is not necessarily given by a complete path on $G$.  
E.g., it is possible that $p$ uses both turnedges 
$\rhpoulMrl$ and $\rhpourMlr$
(which do not share a turnvertex).
Then $\rhacW(\pi(p)) = 2$, while for a complete path $q$ of $G$ 
we always have $\rhacW(q) \in \{-1,0,1\}$. 
\end{example}

If $p$ is a complete turnpath in~$\res$ and $x$ is a turnvertex or turnedge,  
then it will be convenient to write $\eins_p(x)=1$ if $x$ occurs~$p$, 
and $\eins_p(x)=0$ otherwise.

We reconsider now Definition~\ref{def:slack-contr} and interpret the sets 
$\p_+(\rhc)$, $\p_-(\rhc)$, and $\p_0(\rhc)$ of slack contributions of a rhombus $\rhc$ ---instead of as subsets of $E(G)$--- 
as subsets of $V(\res)\cup E(\res)$. Note that these sets are pairwise disjoint. 


\begin{lemma}
\label{lem:slackcalculation}
Let $p$ be a complete turnpath in~$\res$. 
Then we have for any rhombus $\varrho$, 
$$ 
 \s \varrho p = \sum_{x \in \p_+(\varrho)} \eins_p(x) - \sum_{x \in \p_-(\varrho)} \eins_p(x) .
$$
Moreover, $\s \varrho p \in \{-4,-3,\ldots,3,4\}$. \endproof
\end{lemma}

\begin{proof}
For each rhombus $\rhc$ we have in pictorial notation 
$\eins_p(\rhpoulMr)+\eins_p(\rhpourMl)=\eins_p(\rhpcWl)+\eins_p(\rhpcWr)$. Moreover,
\begin{align*}
\s \rhc p &= \rhaoulW(p) + \rhaolrM(p) \\
&= \big( \eins_p(\rhpcMl) + \eins_p(\rhpourMr) - \eins_p(\rhpoulMl) - \eins_p(\rhpoulMr) \big) + \big( \eins_p(\rhpcWl) + \eins_p(\rhpollWr) - \eins_p(\rhpolrWr) - \eins_p(\rhpolrWl) \big) \\
&= \big( \eins_p(\rhpcMl) - \eins_p(\rhpolrWr) \big) + \eins_p(\rhpourMr) - \eins_p(\rhpoulMl) + \big( \eins_p(\rhpcWl) - \eins_p(\rhpoulMr) \big) + \eins_p(\rhpollWr) - \eins_p(\rhpolrWl) \\
&= \big( \eins_p(\rhpollWll) + \eins_p(\rhpolrWrl) - \eins_p(\rhpolrWrl) - \eins_p(\rhpolrWrr) \big) + \eins_p(\rhpourMr) - \eins_p(\rhpoulMl)\\
&\hspace{1cm} + \big( \eins_p(\rhpourMll) + \eins_p(\rhpoulMrl) - \eins_p(\rhpoulMrl) - \eins_p(\rhpoulMrr) \big) + \eins_p(\rhpollWr) - \eins_p(\rhpolrWl)\\
&= \big( \eins_p(\rhpollWll)  - \eins_p(\rhpolrWrr) \big) + \eins_p(\rhpourMr) - \eins_p(\rhpoulMl) + \big( \eins_p(\rhpourMll) - \eins_p(\rhpoulMrr) \big) + \eins_p(\rhpollWr) - \eins_p(\rhpolrWl)\\
&= \sum_{x \in \p_+(\loz)} \eins_p(x) - \sum_{x \in \p_-(\loz)} \eins_p(x). 
\end{align*}
The assertion on the possible values of $\s \varrho p$ follows immediately. 
\qquad\end{proof}

\begin{example}
There are complete turnpaths~$p$ and $q$, 
that use all the turnvertices in 
$\p_+(\varrho)$ and $\p_-(\varrho)$, respectively, 
resulting in the slacks $\s \varrho p = 4$ and $\s \varrho q = -4$. 
\end{example}


We construct now the digraph $R_f$ from $R$ by deleting the negative slack contributions in $f$-flat rhombi, 
and removing all edges of $R$ crossing capacity achieving edges of $\Delta$ at the border of $\Delta$. 
Recall from Definition~\ref{def:hiveflowpolytope} 
the definition of the throughput capacities $b(k)$ of the border edges of $\Delta$, 
given by $\la,\mu,\nu$. 

%

\begin{definition}
\label{def:resf} 
Let $f \in B(\la,\mu,\nu)$.
The {\em residual digraph} $\resf f:=\resf f(\la,\mu,\nu)$ is obtained from $\res$
by deleting the turnvertices and turnedges 
in $\p_-(\varrho)$ in $f$-flat rhombi $\varrho$. 
Moreover, for all edges~$k$ on the right and bottom border of $\Delta$ 
satisfying $\delta(k,f) = b(k)$,  we delete all four edges of $\res$ crossing~$k$. 
Similarly,  for all edges~$k'$ on the left border of $\Delta$ 
satisfying $-\delta(k',f) = b(k')$,  we delete all four edges of $\res$ crossing~$k'$. 
%
Let $\cP_f$ denote the set of {\em complete turnpaths in $\resf f$}.
\end{definition}


The following is an immediate consequence of the construction of~$\resf f$ 
and Lemma~\ref{lem:slackcalculation}. 

\begin{lemma}\label{le:spos}
We have $\s \varrho p \ge 0$ for any $p\in\cP_f$ 
and any $f$\dash flat rhombus $\varrho$.\endproof
\end{lemma}

We denote by $\Ke(\resf f)$ the cone of nonnegative flows on $\resf f$. 
As $\resf f$ is a subgraph of $\res$, 
a nonnegative flow $f\in\Ke(\resf f)$ 
can be interpreted as a flow $f\in\Ke(\res)$ with value zero 
on the turnedges not present in~$\resf f$.

Let $f \in B_\Z$ and $p$ be an \stturnpath in $\resf f$. 
Is $f+\pi(p) \in B_\Z$? 

By construction of $\resf f$, if $p$ crosses the border edge~$k$, 
then $\delta(k,f)  < b(k)$ if $k$ is on the right or bottom border of $\Delta$.
Similarly, $-\delta(k',f) < b(k')$ if $k'$ is on the left border of~$\Delta$.  
Thus the flow $f+\pi(p)$ does not violate the border capacity constraints.

In order to see whether $f+\pi(p)$ is a hive flow, 
we note that if $\varrho$ is an $f$\dash flat rhombus, then 
$\s \varrho {f +\pi(p)} = \s \varrho f + \s \varrho {\pi(p)} = \s \varrho p \ge 0$ by 
Lemma~\ref{le:spos}.
However, for rhombi~$\varrho$ that are not $f$\dash flat,  
it may be that 
$\s \varrho f + \s \varrho {\pi(p)} < 0$. 
Fortunately, it turns out that if $p$ is an \stturnpath of minimal length, then this cannot happen!

The proof of the following result is astonishingly delicate and postponed to Section~\ref{sec:shortestpaththeorem}.


\begin{theorem}[Shortest Path Theorem]\label{thm:shopath}
Let $f\in B_\Z$ and let $p$ be a shortest \stturnpath in~$\resf f$. Then $f+\pi(p) \in B_\Z$.
\end{theorem}

To investigate in a more general context to what extent the hive conditions are preserved when adding a flow 
$d\in \oF(G)$ to $f\in B$, we make the following definition, extending Definition~\ref{def:f-secure}. 

\begin{definition}\label{def:hive-preserv}
For a hive flow $f \in B$, a flow $d \in \oF(G)$ is called \emph{$f$\dash hive preserving} if 
$f + \varepsilon d \in B$ for sufficiently small $\varepsilon>0$. 
\end{definition}

We note that the set of $f$\dash hive preserving flows forms a cone~$C_f$, 
which was called ``cone of feasible directions'' in~\cite{bi:09}.  

\begin{lemma}\label{lem:direction}
Let $f \in B$ and $d'\in \Ke(\resf f)$. Then $\pi(d')$ is $f$\dash hive preserving.
\end{lemma}

\begin{proof}
According to Lemma~\ref{le:flow-decomp}, 
there are complete turnpaths~$p_1,\ldots,p_m\in\cP_f$ 
and $\alpha_1,\ldots,\alpha_m\ge 0$ 
such that 
$d' = \sum_{i=1}^m\alpha_i p_i $. 
Lemma~\ref{le:spos} tells us that $\s \varrho {p_i} \geq 0$ 
if $\varrho$ is $f$-flat. 

By construction of $\resf f$, if~$p_i$ crosses a border edge~$k$ on the right or bottom side of $\Delta$, 
then $\delta(k,f) < b(k)$. 
This implies that 
$\delta(k,f+\varepsilon d') = \delta(k,f) + \varepsilon\sum_i \alpha_i \delta(k,p_i) < b(k)$ 
for sufficiently small $\varepsilon >0$.
The argument is analogous for the left border edges. 

We show now that $f+\varepsilon \pi(d')$ is a hive flow for sufficiently small $\varepsilon>0$.
By the linearity of the slack, 
this means to show that for all rhombi $\varrho$, we have 
$\s \varrho f + \varepsilon \sum_i \alpha_i \s \varrho {p_i} \geq 0$ 
for sufficiently small $\varepsilon>0$. 
In the case $\s \varrho f >0$, this is obvious.
On the other hand, if $\s \varrho f =0$, this follows from $\s \varrho {p_i} \geq 0$. 
\qquad\end{proof}

\subsection{Flatspaces}
\label{sec:flatspaces}

Our goal here is to get a detailed understanding of how turnpaths in $\resf f$ behave. 
For this,we first have to recall the concept of flatspaces from~\cite{knta:99,buc:00}.
In the following we fix $f\in B$.

By a {\em convex set $L$ in the triangular graph $\Delta$} we shall understand a union of hive triangles, which is convex.
It is obvious that the angles at the corners of a convex set $L$ are either acute of $60^\circ$ or obtuse of $120^\circ$.


We call two hive triangles \emph{adjacent} if they share a side and form an $f$-flat rhombus.
This defines a graph whose vertices are the hive triangles.
An {\em$f$-flatspace} is defined to be a connected component of this graph,
see Figure~\ref{fig:degengraph}.
We simply write {\em flatspace} if it is clear, which flow $f$ is meant.
Also, we will identify flatspaces with the union of their hive triangles.

The following was observed in~\cite{buc:00}. 

\begin{lemma}
\label{lem:flatspaceprop}
\begin{remunerate}
\item $f$\dash flatspaces are convex sets.
\item A side of an $f$-flatspace is either on the border of $\Delta$, or it is also a side of a neighbouring flatspace.
\item There are exactly five types of convex sets: triangles, parallelograms, trapezoids, pentagons and hexagons.
\end{remunerate}
\end{lemma}

\begin{proof}
The first and second claim follow from Corollary~\ref{cor:BZ2}.
The third claim is just the enumeration of convex shapes on the triangular grid.
\qquad\end{proof}

We show next that turnpaths in $\resf f$  
can move in $f$\dash flatspaces only in a very limited way: 
namely along the border 
in counterclockwise direction, 
and they can enter and leave the flatspace only through a few distinguished edges that we define next
(cf.\ Figure~\ref{fig:innerandborder}).


\begin{figure}[h]
\begin{center}
\scalebox{2}{
\begin{tikzpicture}\draw[rhrhombidraw] (4.619pt,40.0005pt) -- (-46.19pt,-48.0006pt) -- (55.428pt,-48.0006pt) -- cycle;\draw[rhrhombidraw] (46.19pt,-32.0004pt) -- (36.952pt,-48.0006pt) ;\draw[rhrhombidraw] (-18.476pt,-48.0006pt) -- (-4.619pt,-24.0003pt) -- (-18.476pt,0.0pt) ;\draw[rhrhombidraw] (-4.619pt,-24.0003pt) -- (4.619pt,-24.0003pt) -- (-13.857pt,8.0001pt) ;\draw[rhrhombidraw] (18.476pt,-48.0006pt) -- (4.619pt,-24.0003pt) -- (13.857pt,-24.0003pt) -- (27.714pt,-48.0006pt) ;\draw[rhrhombidraw] (13.857pt,-24.0003pt) -- (23.095pt,-8.0001pt) -- (0.0pt,32.0004pt) ;\draw[rhrhombidraw] (32.333pt,-8.0001pt) -- (23.095pt,-8.0001pt) -- (27.714pt,0.0pt) ;\draw[rhrhombidraw] (4.619pt,-24.0003pt) -- (23.095pt,8.0001pt) -- (-13.857pt,8.0001pt) ;\fill (4.619pt,40.0005pt) circle (0.4pt);\fill (0.0pt,32.0004pt) circle (0.4pt);\fill (9.238pt,32.0004pt) circle (0.4pt);\fill (-4.619pt,24.0003pt) circle (0.4pt);\fill (4.619pt,24.0003pt) circle (0.4pt);\fill (13.857pt,24.0003pt) circle (0.4pt);\fill (-9.238pt,16.0002pt) circle (0.4pt);\fill (0.0pt,16.0002pt) circle (0.4pt);\fill (9.238pt,16.0002pt) circle (0.4pt);\fill (18.476pt,16.0002pt) circle (0.4pt);\fill (-13.857pt,8.0001pt) circle (0.4pt);\fill (-4.619pt,8.0001pt) circle (0.4pt);\fill (4.619pt,8.0001pt) circle (0.4pt);\fill (13.857pt,8.0001pt) circle (0.4pt);\fill (23.095pt,8.0001pt) circle (0.4pt);\fill (-18.476pt,0.0pt) circle (0.4pt);\fill (-9.238pt,0.0pt) circle (0.4pt);\fill (0.0pt,0.0pt) circle (0.4pt);\fill (9.238pt,0.0pt) circle (0.4pt);\fill (18.476pt,0.0pt) circle (0.4pt);\fill (27.714pt,0.0pt) circle (0.4pt);\fill (-23.095pt,-8.0001pt) circle (0.4pt);\fill (-13.857pt,-8.0001pt) circle (0.4pt);\fill (-4.619pt,-8.0001pt) circle (0.4pt);\fill (4.619pt,-8.0001pt) circle (0.4pt);\fill (13.857pt,-8.0001pt) circle (0.4pt);\fill (23.095pt,-8.0001pt) circle (0.4pt);\fill (32.333pt,-8.0001pt) circle (0.4pt);\fill (-27.714pt,-16.0002pt) circle (0.4pt);\fill (-18.476pt,-16.0002pt) circle (0.4pt);\fill (-9.238pt,-16.0002pt) circle (0.4pt);\fill (0.0pt,-16.0002pt) circle (0.4pt);\fill (9.238pt,-16.0002pt) circle (0.4pt);\fill (18.476pt,-16.0002pt) circle (0.4pt);\fill (27.714pt,-16.0002pt) circle (0.4pt);\fill (36.952pt,-16.0002pt) circle (0.4pt);\fill (-32.333pt,-24.0003pt) circle (0.4pt);\fill (-23.095pt,-24.0003pt) circle (0.4pt);\fill (-13.857pt,-24.0003pt) circle (0.4pt);\fill (-4.619pt,-24.0003pt) circle (0.4pt);\fill (4.619pt,-24.0003pt) circle (0.4pt);\fill (13.857pt,-24.0003pt) circle (0.4pt);\fill (23.095pt,-24.0003pt) circle (0.4pt);\fill (32.333pt,-24.0003pt) circle (0.4pt);\fill (41.571pt,-24.0003pt) circle (0.4pt);\fill (-36.952pt,-32.0004pt) circle (0.4pt);\fill (-27.714pt,-32.0004pt) circle (0.4pt);\fill (-18.476pt,-32.0004pt) circle (0.4pt);\fill (-9.238pt,-32.0004pt) circle (0.4pt);\fill (0.0pt,-32.0004pt) circle (0.4pt);\fill (9.238pt,-32.0004pt) circle (0.4pt);\fill (18.476pt,-32.0004pt) circle (0.4pt);\fill (27.714pt,-32.0004pt) circle (0.4pt);\fill (36.952pt,-32.0004pt) circle (0.4pt);\fill (46.19pt,-32.0004pt) circle (0.4pt);\fill (-41.571pt,-40.0005pt) circle (0.4pt);\fill (-32.333pt,-40.0005pt) circle (0.4pt);\fill (-23.095pt,-40.0005pt) circle (0.4pt);\fill (-13.857pt,-40.0005pt) circle (0.4pt);\fill (-4.619pt,-40.0005pt) circle (0.4pt);\fill (4.619pt,-40.0005pt) circle (0.4pt);\fill (13.857pt,-40.0005pt) circle (0.4pt);\fill (23.095pt,-40.0005pt) circle (0.4pt);\fill (32.333pt,-40.0005pt) circle (0.4pt);\fill (41.571pt,-40.0005pt) circle (0.4pt);\fill (50.809pt,-40.0005pt) circle (0.4pt);\fill (-46.19pt,-48.0006pt) circle (0.4pt);\fill (-36.952pt,-48.0006pt) circle (0.4pt);\fill (-27.714pt,-48.0006pt) circle (0.4pt);\fill (-18.476pt,-48.0006pt) circle (0.4pt);\fill (-9.238pt,-48.0006pt) circle (0.4pt);\fill (0.0pt,-48.0006pt) circle (0.4pt);\fill (9.238pt,-48.0006pt) circle (0.4pt);\fill (18.476pt,-48.0006pt) circle (0.4pt);\fill (27.714pt,-48.0006pt) circle (0.4pt);\fill (36.952pt,-48.0006pt) circle (0.4pt);\fill (46.19pt,-48.0006pt) circle (0.4pt);\fill (55.428pt,-48.0006pt) circle (0.4pt);\fill[thick,fill=black!60] (-32.333pt,-40.0005pt) -- (-23.095pt,-40.0005pt) -- (-13.857pt,-24.0003pt) -- (-18.476pt,-16.0002pt) -- cycle;\fill[thick,fill=black!60] (27.714pt,-16.0002pt) -- (23.095pt,-24.0003pt) -- (32.333pt,-40.0005pt) -- (36.952pt,-32.0004pt) -- cycle;\fill[thick,fill=black!60] (0.0pt,0.0pt) -- (4.619pt,-8.0001pt) -- (9.238pt,0.0pt) -- cycle;\fill[thick,fill=black!60] (0.0pt,-32.0004pt) -- (-4.619pt,-40.0005pt) -- (4.619pt,-40.0005pt) -- cycle;\draw[thin,-my] (32.333pt,-48.0006pt) arc (180:120:4.619pt);\draw[thin,-my] (34.6425pt,-44.0006pt) arc (300:360:4.619pt);\draw[thin,-my] (36.952pt,-40.0005pt) arc (180:120:4.619pt);\draw[thin,-my] (39.2615pt,-36.0005pt) arc (300:360:4.619pt);\draw[thin,-my] (41.571pt,-32.0004pt) arc (0:60:4.619pt);\draw[thin,-my] (39.2615pt,-28.0004pt) arc (240:180:4.619pt);\draw[thin,-my] (36.952pt,-24.0003pt) arc (0:60:4.619pt);\draw[thin,-my] (34.6425pt,-20.0003pt) arc (240:180:4.619pt);\draw[thin,-my] (32.333pt,-16.0002pt) arc (0:60:4.619pt);\draw[thin,-my] (30.0235pt,-12.0002pt) arc (60:120:4.619pt);\draw[thin,-my] (25.4045pt,-12.0002pt) arc (300:240:4.619pt);\draw[thin,-my] (20.7855pt,-12.0002pt) --  (16.1665pt,-4.0001pt);;\draw[thin,-my] (16.1665pt,-12.0002pt) --  (20.7855pt,-4.0001pt);;\draw[thin,-my] (16.1665pt,-4.0001pt) arc (240:180:4.619pt);\draw[thin,-my] (13.857pt,0.0pt) arc (0:60:4.619pt);\draw[thin,-my] (11.5475pt,4.0001pt) arc (60:120:4.619pt);\draw[thin,-my] (6.9285pt,4.0001pt) arc (300:240:4.619pt);\draw[thin,-my] (2.3095pt,4.0001pt) arc (60:120:4.619pt);\draw[thin,-my] (-2.3095pt,4.0001pt) arc (300:240:4.619pt);\draw[thin,-my] (-6.9285pt,4.0001pt) arc (60:120:4.619pt);\draw[thin,-my] (-11.5475pt,4.0001pt) arc (120:180:4.619pt);\draw[thin,-my] (-13.857pt,0.0pt) arc (180:240:4.619pt);\draw[thin,-my] (-11.5475pt,-4.0001pt) arc (60:0:4.619pt);\draw[thin,-my] (-9.238pt,-8.0001pt) arc (180:240:4.619pt);\draw[thin,-my] (-6.9285pt,-12.0002pt) arc (60:0:4.619pt);\draw[thin,-my] (-4.619pt,-16.0002pt) arc (180:240:4.619pt);\draw[thin,-my] (-4.619pt,-16.0002pt) arc (180:240:4.619pt);\draw[thin,-my] (-2.3095pt,-20.0003pt) arc (240:300:4.619pt);\draw[thin,-my] (2.3095pt,-20.0003pt) arc (120:60:4.619pt);\draw[thin,-my] (6.9285pt,-20.0003pt) arc (240:300:4.619pt);\draw[thin,-my] (11.5475pt,-20.0003pt) arc (300:360:4.619pt);\draw[thin,-my] (13.857pt,-16.0002pt) arc (180:120:4.619pt);\draw[thin,-my] (20.7855pt,-4.0001pt) arc (300:360:4.619pt);\draw[thin,-my] (23.095pt,0.0pt) arc (0:60:4.619pt);\draw[thin,-my] (20.7855pt,4.0001pt) arc (240:180:4.619pt);\draw[thin,-my] (18.476pt,8.0001pt) arc (0:60:4.619pt);\draw[thin,-my] (16.1665pt,12.0002pt) arc (240:180:4.619pt);\draw[thin,-my] (13.857pt,16.0002pt) arc (0:60:4.619pt);\draw[thin,-my] (11.5475pt,20.0003pt) arc (240:180:4.619pt);\draw[thin,-my] (9.238pt,24.0003pt) arc (0:60:4.619pt);\draw[thin,-my] (6.9285pt,28.0004pt) arc (240:180:4.619pt);\draw[thin,-my] (4.619pt,32.0004pt) arc (0:60:4.619pt);\end{tikzpicture}
}
 \caption{The same situation as in Figure~\protect\ref{fig:degengraph}. The inner triangles are shaded while the border triangles are white.
An exemplary turnpath in $\resf f$ is also depicted, where the two straight arrows represent crossing turnedges.}
    \nopar\label{fig:innerandborder}
\end{center}
\end{figure}

Let $L$ be a convex set and $a$ be one of its sides.
For $k\in E(\Delta)$ we write $k\subseteq a$ to express that $k$ is contained in $a$. 
Let $k_1,\ldots,k_r$ denote the edges contained in~$a$ in clockwise order. 
We call $a_{\to L} := k_1$ the \emph{$L$-entrance edge} of $a$ and 
$a_{L\to} := k_r$ the \emph{$L$-exit edge} of $a$. 
We may have $r=1$ in which case the entrance and exit edges coincide.
Note that if $M$ is a convex set adjacent to $L$, 
sharing with it the joint side $a$,
then the $M$-entrance edge of $a$ is at the same time the $L$-exit edge of $L$,
that is, $a_{\to M} = a_{L\to}$.  

The hive triangles in a convex set~$L$ either touch the border of $L$ or lie inside~$L$.  
Correspondingly, we will speak about {\em border triangles} and 
{\em inner triangles} of $L$. 
Recall from Definition~\ref{def:resf} 
the set $\cP_f$ of complete turnpaths in $\resf f$. 

\begin{proposition}
\label{cla:outerccdir}
Let $p\in\cP_f$ and $L$ be an $f$\dash flatspace. Then: 
\begin{remunerate}
\item $p$ can enter~$L$ only by crossing entrance edges of $L$.
Similarly, $p$ can leave $L$ only by crossing exit edges of $L$.

\item $p$ uses only turnvertices in border triangles of $L$ and traverses the border of~$L$ 
in counterclockwise direction. 
\end{remunerate}
\end{proposition} 

\begin{samepage}
\begin{proof}
We call turnvertices, which lie in $L$ and start at entrance edges of~$L$, 
{\em entrance turnvertices}. 
Diagonals of non\dash$f$\dash flat rhombi shall be called \emph{dividers}.
\begin{remunerate}
\item[(1)] If $p$ enters $L$ with a counterclockwise turn
\begin{tikzpicture}\draw[rhrhombidraw] (0.0pt,0.0pt) -- (9.238pt,0.0pt) -- cycle;\draw[thin,-my] (4.619pt,0.0pt) arc (0:60:4.619pt);\end{tikzpicture}
, then
\begin{tikzpicture}\draw[rhrhombidraw] (0.0pt,0.0pt) -- (9.238pt,0.0pt) -- cycle;\draw[thin,-my] (4.619pt,0.0pt) arc (0:60:4.619pt);\draw[rhrhombithickside] (9.238pt,0.0pt) -- (4.619pt,8.0001pt);\end{tikzpicture}
must be a divider. Hence 
\begin{tikzpicture}\draw[rhrhombidraw] (0.0pt,0.0pt) -- (9.238pt,0.0pt) -- cycle;\draw[thin,-my] (4.619pt,0.0pt) arc (0:60:4.619pt);\end{tikzpicture}
is an entrance turnvertex.
\item[(2)] If $p$ enters $L$ with a clockwise and a counterclockwise turn
\begin{tikzpicture}\draw[rhrhombidraw] (0.0pt,0.0pt) -- (9.238pt,0.0pt) -- cycle;\draw[thin,-my] (4.619pt,0.0pt) arc (180:120:4.619pt) arc (300:360:4.619pt);\end{tikzpicture}
, then 
\begin{tikzpicture}\draw[rhrhombidraw] (0.0pt,0.0pt) -- (9.238pt,0.0pt) -- cycle;\draw[thin,-my] (4.619pt,0.0pt) arc (180:120:4.619pt) arc (300:360:4.619pt);\draw[rhrhombithickside] (9.238pt,0.0pt) -- (13.857pt,8.0001pt);\end{tikzpicture}
is a divider and hence
\begin{tikzpicture}\draw[rhrhombidraw] (0.0pt,0.0pt) -- (9.238pt,0.0pt) -- cycle;\draw[thin,-my] (4.619pt,0.0pt) arc (180:120:4.619pt);\end{tikzpicture}
is an entrance turnvertex.
\item[(3)] If $p$ enters $L$ with two clockwise turns
\begin{tikzpicture}\draw[rhrhombidraw] (0.0pt,0.0pt) -- (9.238pt,0.0pt) -- cycle;\draw[thin,-my] (4.619pt,0.0pt) arc (180:120:4.619pt) arc (120:60:4.619pt);\end{tikzpicture}%
, then
\begin{tikzpicture}\draw[rhrhombidraw] (0.0pt,0.0pt) -- (9.238pt,0.0pt) -- cycle;\draw[thin,-my] (4.619pt,0.0pt) arc (180:120:4.619pt) arc (120:60:4.619pt);\draw[rhrhombithickside] (9.238pt,0.0pt) -- (4.619pt,8.0001pt);\end{tikzpicture}
is a divider and hence
\begin{tikzpicture}\draw[rhrhombidraw] (0.0pt,0.0pt) -- (9.238pt,0.0pt) -- cycle;\draw[thin,-my] (4.619pt,0.0pt) arc (180:120:4.619pt);\end{tikzpicture}
is an entrance turnvertex.
\end{remunerate}
Analogous arguments hold for exits with the situations
\begin{tikzpicture}\draw[rhrhombidraw] (4.619pt,8.0001pt) -- (-4.619pt,8.0001pt) -- cycle;\draw[thin,-my] (-2.3095pt,4.0001pt) arc (300:360:4.619pt);\draw[rhrhombithickside] (0.0pt,0.0pt) -- (4.619pt,8.0001pt);\end{tikzpicture}%
, 
\begin{tikzpicture}\draw[rhrhombidraw] (4.619pt,8.0001pt) -- (-4.619pt,8.0001pt) -- cycle;\draw[thin,-my] (4.619pt,0.0pt) arc (0:60:4.619pt) arc (240:180:4.619pt);\draw[rhrhombithickside] (4.619pt,8.0001pt) -- (9.238pt,0.0pt);\end{tikzpicture}
 and 
\begin{tikzpicture}\draw[rhrhombidraw] (4.619pt,8.0001pt) -- (-4.619pt,8.0001pt) -- cycle;\draw[thin,-my] (6.9285pt,4.0001pt) arc (300:240:4.619pt) arc (240:180:4.619pt);\draw[rhrhombithickside] (0.0pt,0.0pt) -- (4.619pt,8.0001pt);\end{tikzpicture}%
. This proves the first assertion.

We now show the second assertion. 
Consider an inner triangle
\begin{tikzpicture}\draw[rhrhombidraw] (0.0pt,0.0pt) -- (9.238pt,0.0pt) -- (4.619pt,8.0001pt) -- cycle;\end{tikzpicture}%
.
All rhombi in the shaded area
\begin{tikzpicturednosmash}\fill[thick,fill=black!20] (-4.619pt,24.0003pt) -- (4.619pt,24.0003pt) -- (13.857pt,8.0001pt) -- (9.238pt,0.0pt) -- (-9.238pt,0.0pt) -- (-13.857pt,8.0001pt) -- cycle;\draw[rhrhombidraw] (0.0pt,16.0002pt) -- (-4.619pt,8.0001pt) -- (4.619pt,8.0001pt) -- cycle;\end{tikzpicturednosmash}
are $f$\dash flat.
By the definition of $\resf f$, 
the counterclockwise turnvertices
\begin{tikzpicture}\draw[rhrhombidraw] (0.0pt,16.0002pt) -- (-4.619pt,8.0001pt) -- (4.619pt,8.0001pt) -- cycle;\draw[thin,-my] (-2.3095pt,12.0002pt) arc (240:300:4.619pt);\end{tikzpicture}%
,
\begin{tikzpicture}\draw[rhrhombidraw] (0.0pt,16.0002pt) -- (-4.619pt,8.0001pt) -- (4.619pt,8.0001pt) -- cycle;\draw[thin,-my] (0.0pt,8.0001pt) arc (0:60:4.619pt);\end{tikzpicture}
and
\begin{tikzpicture}\draw[rhrhombidraw] (0.0pt,16.0002pt) -- (-4.619pt,8.0001pt) -- (4.619pt,8.0001pt) -- cycle;\draw[thin,-my] (2.3095pt,12.0002pt) arc (120:180:4.619pt);\end{tikzpicture}
are not vertices of $\resf f$. For the same reason, the clockwise turnvertices 
\begin{tikzpicture}\draw[rhrhombidraw] (0.0pt,16.0002pt) -- (-4.619pt,8.0001pt) -- (4.619pt,8.0001pt) -- cycle;\draw[thin,-my] (-2.3095pt,12.0002pt) arc (60:0:4.619pt);\end{tikzpicture}%
,
\begin{tikzpicture}\draw[rhrhombidraw] (0.0pt,16.0002pt) -- (-4.619pt,8.0001pt) -- (4.619pt,8.0001pt) -- cycle;\draw[thin,-my] (0.0pt,8.0001pt) arc (180:120:4.619pt);\end{tikzpicture}
and
\begin{tikzpicture}\draw[rhrhombidraw] (0.0pt,16.0002pt) -- (-4.619pt,8.0001pt) -- (4.619pt,8.0001pt) -- cycle;\draw[thin,-my] (2.3095pt,12.0002pt) arc (300:240:4.619pt);\end{tikzpicture}
have no incident turnedge in $\resf f$.
This shows that turnpaths and turncycles in $\resf f$ can only use turnvertices in border triangles. 

Finally, the fact that a counterclockwise turn
{
\begin{tikzpicture}\draw[thin,-my] (-2.3095pt,12.0002pt) arc (240:300:4.619pt);\fill[white] (-4.619pt,8.0001pt) circle (0.1pt);\fill[white] (4.619pt,8.0001pt) circle (0.1pt);\fill[white] (0.0pt,16.0002pt) circle (0.1pt);\end{tikzpicture}
}
in $p$ implies that
{
\begin{tikzpicture}\draw[thin,-my] (-2.3095pt,12.0002pt) arc (240:300:4.619pt);\draw[rhrhombithickside] (-4.619pt,8.0001pt) -- (4.619pt,8.0001pt);\fill[white] (0.0pt,16.0002pt) circle (0.1pt);\end{tikzpicture}
}
is a divider, shows that $p$ traverses the border triangles of $f$\dash flatspaces in counterclockwise direction.
\qquad\end{proof}
\end{samepage}

Let $d\in\oF(G)$ be a flow and $k\in E(G)$ be an edge 
lying at the border of a convex set~$L$. 
If the hive triangle in~$L$ having the side~$k$ is upright, 
we define $\inflow L {k} d := \delta(k,f)$, 
otherwise we set $\inflow L {k} d := -\delta(k,f)$. 
We call $\inflow L {k} d$ the {\em throughput of $d$ into $L$} 
through~$k$. 
It will be convenient to call
$\outflow L {k} d := -\inflow L {k} d$ the {\em throughput of $d$ out of $L$} through~$k$. 

Note that if the convex sets $L$ and $M$ are adjacent, sharing an edge~$k$, then 
$\inflow {M}  {k} d = \outflow L  {k} d$. 

For some of the following properties of throughputs compare Figure~\ref{fig:degengraph}. 

\begin{lemma}
\label{obs:borderentranceexit}
Let $L$ be a convex set  contained in an $f$\dash flatspace 
and $a$ be a side of $L$. Further, let $k_1,\ldots,k_r\in E(\Delta)$ be the edges contained in $a$ 
in clockwise order. 
Then $\inflow L {k_1} f = \ldots =\inflow L {k_r} f$. 
Moreover, if $d\in\oF(G)$ is $f$\dash hive preserving, 
then $\inflow L {k_1} d \geq \ldots \ge \inflow L {k_r} d$.
\end{lemma}

\begin{proof}
It is sufficient to show this for adjacent edges $k_1 = \rhsour$ and $k_2 = \rhslr$,
where the rhombi $\rhc$ and $\rhrholr$ are $f$\dash flat.
Since $0 = \s \rhc f = \rhaourM(f)+\rhaollW(f)$ and $0 = \s \rhrholr f = \rhaollW(f)+\rhalrM(f)$,
it follows $\rhaourM(f) = \rhalrM(f)$.

We have $\s\rhc d \ge 0$ and $\s \rhrholr d \ge 0$ 
as $d$ is $f$-hive preserving and $\rhc$ and $\rhrholr$ are $f$-flat.
The second statement follows now similarly as before. 
\qquad\end{proof}

\begin{observation}
\label{obs:bd-cap}
 Let $L$ be an $f$\dash flatspace with a side $a$ lying on the left border of $\Delta$.
Then the maximum of the capacities $b(k)$ (cf.\ Definition~\ref{def:hiveflowpolytope}) 
over all edges $k\subseteq a$ is attained at the exit edge of~$a$.
An analogous statement holds for the right and bottom border and entrance edges. 
\end{observation}

\begin{proof}
This follows directly from the fact that $\nu_1 \geq \ldots\geq \nu_n$
and the definition of the throughput capacities $b(k)$ of the border edges~$k$ of $\Delta$, 
cf.\ Figure~\ref{fig:triangulararraywithflow}. Similarly for $\la$ and~$\mu$.
\qquad\end{proof}


It will be important to decompose the throughput 
$\inflow L k d$ into its positive and negative part.
Recall that $\inflow L k d = - \outflow L k d$.

\begin{definition}\label{def:def-inflow}
Let $d\in\oF(G)$, $L$ be an $f$\dash flatspace, and 
$k\in E(\Delta)$ be an edge at the border of~$L$. 
The  {\em $L$-inflow of $d$ through~$k$} and
{\em $L$-outflow of $d$  through~$k$} are defined as
$$
 \inom L k d := \max\{\inflow L {k} d,0\},\quad 
 \outom L k d := \max\{\outflow L {k} d,0\} .
$$
Further, for a side~$a$ of $L$, we define the 
{\em $L$-inflow of $d$ through~$a$} and the 
{\em $L$-outflow of $d$ through~$a$} by
\begin{equation*}
 \inom L a d \ :=\ \sum_{k\subseteq a} \inom L k d ,\quad 
 \outom L a d \ :=\ \sum_{k\subseteq a} \outom L k  d .
\end{equation*}
We write $\inom \Delta a d :=\inom L a d$ and 
$\outom \Delta a d :=\outom L a d$
if the side~$a$ is on the border of $\Delta$. 
\end{definition}



Note that 
$\inflow L k d = \inom L k d - \outom L k d$.  
Further, if $L$ and $M$ are adjacent convex sets sharing a side~$a$, then
$\inom L k d = \outom M k d$ for $k\subseteq a$ 
and hence 
\begin{equation}\label{eq:in=out}
\inom L a d = \outom M a d .
\end{equation}

The partition of $\Delta$ into $f$\dash flatspaces leads to 
a partition of the border of $\Delta$. 
Let $\cS_f$ denote the set of sides of $f$\dash flatspaces that lie 
on the right or bottom border of $\Delta$.

\begin{lemma}\label{le:deltacomp}
For $f\in B$ and $d\in\oF(G)$ we have 
$$
 \delta(d) = \sum_{a\in\cS_f} \big(\inom \Delta a d  - \outom \Delta a d \big) .
$$
\end{lemma}

\begin{proof}
By the definition of the overall throughput, and since $s$ is connected in $G$ 
only to the vertices on the 
right or bottom border of $\Delta$, we have
$$
 \delta(d) \ =\ \sum_{e_-=s} d(e) - \sum_{e_+=s} d(e) \ =\  
 \sum_{k} \inflow \Delta k d ,
$$
where the right-hand sum is over all edges~$k\in E(\Delta)$ on the right or bottom border of $\Delta$. 
Recall that 
$\inflow {\Delta} k d = \inom {\Delta} k d - \outom  {\Delta} k d$ 
By Definition~\ref{def:def-inflow},
$$
 \sum_k \inom \Delta k d \ =\ 
 \sum_{a\in\cS_f} \sum_{k\subseteq a} \inom \Delta k d \ =\ 
 \sum_{a\in\cS_f} \inom \Delta a d .
$$
Similarly,
$\sum_k \outom \Delta k d =\sum_{a\in\cS_f} \outom \Delta a d$
and the assertion follows.
\qquad\end{proof}

\subsection{The Rerouting Theorem}\label{se:RRT}

We fix $f\in B$. Recall the set $\cP_f$ of complete turnpaths in $\resf f$
from Definition~\ref{def:resf}.
Let $\cP_{st}$,  $\cP_{ts}$, and $\cP_c$ denote the sets of 
$s$-$t$-turnpaths, $t$-$s$-turnpaths, and turncycles in $\resf f$, respectively.
Then we have the disjoint decomposition
$\cP_f=\cP_{st}\cup\cP_{ts}\cup\cP_{c}$. 
Note that every $p\in\cP_{st}$ enters $\Delta$ through exactly one edge
on the right or bottom side of~$\Delta$, 
and leaves $\Delta$ through exactly one edge on the left side of~$\Delta$
(otherwise $s$ or $t$ would be used more than once). 
Similarly, every $p\in\cP_{ts}$ enters $\Delta$ through exactly one edge
on the left  side of~$\Delta$ and leaves $\Delta$ through 
the right or bottom side of~$\Delta$. 
The reader should also note that turncycles $p\in\cP_c$ may 
pass through $s$ or $t$ (or both of them).

\begin{definition}\label{def:weight-fam}
A {\em weighted family $\varphi$ of complete turnpaths} in $\resf f$ 
is defined as a map $\varphi\colon\cP_f\to \R_{\ge 0}$. 
If $\varphi$ takes values in~$\N$, 
we call $\varphi$ a {\em multiset of complete turnpaths} in $\resf f$. 
In this case, we interpret $\varphi(p)$ as the multiplicity 
with which~$p$ occurs in the multiset~$\varphi$. 
\end{definition}

A weighted family $\varphi$ of complete turnpaths in $\resf f$ 
defines the nonnegative flow 
$\sum_{p\in\cP_f} \varphi(p) p$ in~$\resf f$. 
On the other hand, by Lemma~\ref{le:flow-decomp}, 
any nonnegative flow 
$d'\in F_+(\resf f)$ can be written in this form.

The flow $d':=\sum_p \varphi(p) p$ on $\resf f$ defined by 
the weighted family~$\varphi$ satisfies 
\begin{equation}\label{eq:gaga}
 \delta(d') \ =\  \sum_{p\in\cP_{st}}\varphi(p) - \sum_{p\in\cP_{ts}}\varphi(p) .
\end{equation}

To motivate the next definition, recall from Proposition~\ref{cla:outerccdir} 
that a complete turnpath $p\in\cP_f$ can enter an $f$-flatspace~$L$ only through an 
entrance edge $a_{\to L}$ of a side~$a$ of $L$, and leave only through an exit edge  $a_{L\to}$. 

\begin{definition}\label{def:sw-ew} 
Let $a$ be a side of an $f$-flatspace~$L$. 
We denote by 
$\Pout L a$ the set of $p\in\cP_f$ that enter~$L$ through the 
edge $a_{\to L}$. The set 
$\Pin L a$ denotes the set of $p\in\cP_f$ that exit~$L$ 
through the edge $a_{L\to }$.
For a weighted family $\varphi$ of complete turnpaths 
and an $f$-flatspace $L$
we define the {\em entrance weight} and the 
{\em exit weight} of a side $a$ of $L$ as follows:
\begin{equation*}\label{eq:in-out-ms} 
\inom L a \varphi \ := \sum_{p \in\cP_f(\to L,a)} \varphi(p),\quad
\outom L a \varphi \ := \sum_{p\in\cP_f(L\to,a)} \varphi(p) . 
\end{equation*} 
If the side~$a$ is on the border of $\Delta$ we write  
$\inom \Delta a \varphi :=\inom L a \varphi$, 
$\outom \Delta a \varphi :=\outom L a \varphi$, and 
$\Pout \Delta a := \Pout L a$. 
\end{definition}



The following remarkable result tells us that for 
any $f$-hive preserving flow~$d\in\oF(G)$,  
there is a weighted family $\varphi$ of complete turnpaths 
such that the inflows and outflows of $d$ through the sides~$a$ of the $f$-flatspaces
are given by the entrance weight and exit weight of~$\varphi$ through~$a$, respectively. 

\begin{theorem} [Rerouting Theorem]
\label{thm:rerouting}
Let $f \in B$ and $d \in \oF(G)$ be $f$\dash hive preserving.
Then there exists a weighed family $\varphi$ of complete turnpaths in $\resf f$
such that 
$\inom L a d = \inom L a \varphi$ and 
$\outom L a d = \outom L a \varphi$
for all $f$\dash flatspaces~$L$ and all sides $a$ of $L$. 
If $d$ is integral, then we may assume that $\varphi$ is a multiset. 
\end{theorem}

Let us draw an immediate consequence. 

\begin{corollary}\label{cor:RRT}
Under the assumptions of Theorem~\ref{thm:rerouting}, 
the nonnegative flow
$d'  := \sum_{p\in\cP_f} \varphi(p) p$
on $\resf f$ satisfies $\delta(d') =\delta(d)$. 
\end{corollary}

\begin{proof} 
Recall that $\cS_f$ denotes the set of sides of $f$\dash flatspaces that lie 
on the right or bottom border of $\Delta$.
We have 
$$
 \sum_{p\in\cP_{st}} \varphi(p) + \sum_{p\in\cP_{c}\atop s\in p} \varphi(p) \ =\ 
 \sum_{a\in\cS_f} \sum_{p \in\cP_f(\to \Delta, a)} \varphi(p) \ =\ 
\sum_{a\in\cS_f} \inom \Delta a \varphi \ =\ \sum_{a\in\cS_f}\inom \Delta a d,
$$ 
where we have used Theorem~\ref{thm:rerouting} 
for the last equality. 
Similarly, 
$$
 \sum_{p\in\cP_{ts}} \varphi(p) + \sum_{p\in\cP_{c}\atop s\in p} \varphi(p) \ =\ 
 \sum_{a\in\cS_f} \sum_{p \in\cP_f(\Delta\to, a)} \varphi(p) \ =\ 
 \sum_{a\in\cS_f} \outom \Delta a d .
$$
Subtracting and using~\eqref{eq:gaga} we get 
$$
 \delta(d') \ =\  
 \sum_{p\in\cP_{st}}\varphi(p) - \sum_{p\in\cP_{ts}}\varphi(p) \ =\ 
  \sum_{a\in\cS_f}\inom \Delta a d - \sum_{a\in\cS_f}\outom \Delta a d \ =\ \delta(d), 
$$
where we have used Lemma~\ref{le:deltacomp} for the last equality. 
\qquad\end{proof}

The proof of the Rerouting Theorem is postponed to Section~\ref{sec:reroutingtheorem}. 
The rough idea of the proof is to define a notion of canonical turnpaths within a convex set~$L$, 
that specializes to the complete turnpaths in $\resf f$ restricted to $L$, 
in case $L$ is an $f$-flatspace. 
We shall successively cut $L$ into convex subsets by straight lines 
and recursively build up the required canonical turnpaths by operations 
of concatenation and straightening.

\begin{remark}
\label{re:comparison}
For given $f \in B$, 
let $C_f$ denote the cone of $f$-hive preserving flows on~$G$. 
Lemma~\ref{lem:direction} states that 
$\pi(\Ke(\resf f)) \subseteq C_f$.
Using Proposition~\ref{cla:outerccdir}, it is easy to see that this inclusion may be strict
for some $f \in B$. 
On the other hand, one can show that 
$\pi(\Ke(\resf f)) = C_f$ 
if the hive triangles and rhombi are the only $f$\dash flatspaces. 
Hive flows~$f$ satisfying the latter property were called \emph{shattered} in~\cite{ike:08, bi:09}.
We note that if $f$ is shattered, then
the Rerouting Theorem is not needed, 
and hence the optimality criterion stated in Proposition~\ref{lem:optimalitytest} below
is much easier to prove. 
\end{remark}


\subsection{A first max-flow algorithm}

We have already said that our goal is to maximize the 
overall throughput function $\delta$ on the polytope~$B$ of bounded hive flows.
In order to implement this idea, we need a criterion that tells us when $f\in B$ is optimal.

\begin{proposition}[Optimality Criterion]
\label{lem:optimalitytest}
Let $f\in B$. Then 
$\delta(f) =\max_{g\in B}\delta(g)$ iff 
there exists no \stturnpath in~$\resf f$.
\end{proposition}

\begin{proof}
We call $f \in B$ \emph{optimal} iff $\delta(f)=\max_{g\in B}\delta(g)$.

If $p$ is an \stturnpath in $\resf f$, then by Lemma~\ref{lem:direction}, we have
$f+\varepsilon \pi(p) \in B$ for some $\varepsilon>0$. 
Since $\delta(f+\varepsilon p) = \delta(f) + \varepsilon > \delta(f)$, the flow $f$ is not optimal.

Now suppose that $f$ is not optimal and let $g \in B$ such that $\delta(g)>\delta(f)$. 
Clearly, $d:=g - f$ is $f$-hive preserving and satisfies $\delta(d)>0$.
Let $\varphi$ be the weighted family of complete turnpaths corresponding to~$d$  
as provided by the Rerouting Theorem~\ref{thm:rerouting}, and put
$d':=\sum_p \varphi(p) p$. 
Corollary~\ref{cor:RRT} shows that 
$\delta(d') = \delta(d)>0$ 
and \eqref{eq:gaga} implies that there exists an \stturnpath in~$\resf f$.
\qquad\end{proof}

Consider the following Algorithm~\ref{alg:lrpa}  for deciding positivity of LR 
coefficients. 

\begin{algorithm}
\caption{}\nopar \label{alg:lrpa}
\begin{algorithmic}[1]
\REQUIRE partitions $\la$, $\mu$, $\nu$ with $|\nu|=|\la|+|\mu|$.
\ENSURE \True, if $\LRC \la \mu \nu > 0$. \False otherwise.
    \STATE $f \leftarrow 0$.
    \WHILE{there is a shortest \stturnpath $p$ in $\resf f$}
          \STATE $f \leftarrow f + \pi(p)$.\nopar\label{alg:step}
    \ENDWHILE \nopar \label{alg:test}
    \RETURN whether $\delta(f) = |\nu|$.
\end{algorithmic}
\end{algorithm}

\begin{theorem}
\label{thm:correct}
 Algorithm~\ref{alg:lrpa} returns whether $\LRC \la \mu \nu~>~0$.
\end{theorem}


\begin{proof}
Clearly $f$ stays integral during the run of Algorithm~\ref{alg:lrpa}.
The Shortest Path Theorem~\ref{thm:shopath} ensures that 
during the run of Algorithm~\ref{alg:lrpa} 
we always have $f \in B_\Z$.
If Algorithm~\ref{alg:lrpa} returns \True, 
then we know that the final value of $f$ is an integral and capacity
achieving hive flow in $B$. 
Hence Proposition~\ref{pro:flowdescription} implies $\LRC \la \mu \nu > 0$.

On the other hand, if Algorithm~\ref{alg:lrpa} returns \False, we have $\delta(f)<|\nu|$
and according to Proposition~\ref{lem:optimalitytest}, the flow~$f$ 
has the maximum value of $\delta$ among all flows in $B$. 
Hence there is no  capacity achieving flow in $B$ and 
Proposition~\ref{pro:flowdescription} implies that $\LRC \la \mu \nu = 0$.
\qquad\end{proof}

We note the following important integrality property. 

\begin{corollary}\label{cor:findoptimal}
For all $\la,\mu,\nu$, the overall throughput function $\delta$ attains the maximal value 
on $B(\la,\mu,\nu)$ at an integer flow.
\end{corollary}

\begin{proof}
In the last line executed by Algorithm~\ref{alg:lrpa}, 
there exists no \stturnpath in~$\resf f$.
Hence, by Proposition~\ref{lem:optimalitytest}, the integral flow $f$ has the maximal value on $B$. 
\qquad\end{proof}


As an application of the foregoing, we deduce here the saturation property of the Littlewood--Richardson coefficients, 
which was first shown in \cite{knta:99}. 

\begin{corollary}
\label{thm:sat}
$\LRC {N\la} {N\mu} {N\nu} > 0$ for some $N\ge 1$ implies $\LRC \la \mu \nu > 0$.
\end{corollary}


\begin{proof} 
If $\LRC {N\la} {N\mu} {N\nu} > 0$, then 
there exists an integral capacity achieving hive flow $f \in B(N\la,N\mu,N\nu)$, 
by Proposition~\ref{pro:flowdescription} .
Hence $\frac f N \in B(\la,\mu,\nu)$ satisfies $\delta(\frac f N)=|\nu|$
and maximizes $\delta$ on $B(\la,\mu,\nu)$.
Even though $\frac f N$ may not be integral, 
Corollary~\ref{cor:findoptimal} implies that there exists 
an integral optimal hive flow $\tilde f \in B(\la,\mu,\nu)$ such 
that $\delta(\tilde f)=|\nu|$. 
Hence $\LRC \la \mu \nu > 0$ by Proposition~\ref{pro:flowdescription}. 
\qquad\end{proof}

\section{A polynomial time algorithm}
\label{sec:polytime}

In this section we use the capacity scaling approach (see, e.g., \cite[ch.~7.3]{amo:93}) 
to turn Algorithm~\ref{alg:lrpa} into a polynomial-time algorithm.
During this method, $f \in B$ stays \emph{$2^\ell$-integral}, 
for $\ell\in\N$, 
which means that all flow values are an integral multiple of $2^\ell$.
The incrementation step in Algorithm~\ref{alg:lrpa}, line~\ref{alg:step}, 
is replaced by adding $2^\ell \pi(p)$.
Further, $\ell$ is decreased in the course of the algorithm.
So our algorithm at first does big increments which over time decrease.

To implement this idea, we will search for a shortest \stturnpath in 
the subgraph $\resfk f \ell$ of $\resf f$ defined next.  
By construction we will have $\resfk f 0 = \resf f$.
Recall that the polytope~$B=B(\la,\mu,\nu)$ has the 
border capacity constraints as in Definition~\ref{def:hiveflowpolytope}.

\begin{definition}
\label{def:resfk}
Let $\ell \in \N$ and let $f \in B$ be $2^\ell$\dash integral.
The digraph $\resfk f \ell$ is obtained from~$\resf f$ by deleting 
all turnedges crossing an edge~$k$ on the right or bottom border 
of $\Delta$ satisfying 
$\delta(k,f)+2^\ell > b(k)$, 
and by deleting all turnedges crossing an edge~$k'$ on the left border 
of $\Delta$ satisfying 
$-\delta(k',f)+2^\ell > b(k')$.
\end{definition}


Algorithm~\ref{alg:lrpcsa} stated below is now fully specified.

\begin{algorithm}[H]
\caption{}\nopar \label{alg:lrpcsa}
\begin{algorithmic}[1]
\REQUIRE partitions $\la$, $\mu$, $\nu$ with $|\nu|=|\la|+|\mu|$ and $\nu_1 \geq \max{\la_1,\mu_1}$.
\ENSURE \True, if $\LRC \la \mu \nu > 0$. \False otherwise.
    \STATE $f \leftarrow 0$.
    \FOR{$\ell$ from $\lceil\log\nu_1\rceil$ down to $0$}
        \WHILE{there is a shortest \stturnpath $p$ in $\resfk f \ell$}
              \STATE $f \leftarrow f + 2^\ell \pi(p)$.
        \ENDWHILE
    \ENDFOR
    \RETURN whether $\delta(f) = |\nu|$.
\end{algorithmic}
\end{algorithm}

It is clear that $f$ stays $2^\ell$\dash integral during the run of Algorithm~\ref{alg:lrpcsa}.

\begin{claim}
 During the run of Algorithm~\ref{alg:lrpcsa}, the flow $f$ always is in $B$.
\end{claim}

\begin{proof}
Given a $2^\ell$\dash integral hive flow $f \in B=B(\la,\mu,\nu)$.
First we note that the set of \stturnpaths on $\resfk f \ell$ equals
the set of \stturnpaths on 
$\resf{\tilde{f}}(\tilde{\la},\tilde{\mu},\tilde{\nu})$, 
where
$\tilde{f} := f/2^\ell$, 
$\tilde{\la} := \lfloor\frac \la {2^\ell}\rfloor$, 
$\tilde{\mu} := \lfloor\frac \mu {2^\ell}\rfloor$, 
$\tilde{\nu} := \lfloor\frac \nu {2^\ell}\rfloor$,  
and division and rounding of partitions is defined componentwise.
Let $p$ be a shortest \stturnpath on $\resfk f \ell$ and hence
also a shortest \stturnpath on 
$\resf{\tilde{f}}(\tilde{\la},\tilde{\mu},\tilde{\nu})$.
According to the Shortest Path Theorem~\ref{thm:shopath} we have 
$\frac f {2^\ell}+\pi(p) \in B(\tilde{\la},\tilde{\mu},\tilde{\nu})$.  
Therefore, we obtain 
$f+2^\ell \pi(p) \in B(2^\ell \tilde{\la},2^\ell \tilde{\mu},2^\ell \tilde{\nu}) \subseteq B(\la,\mu,\nu)=B$.
\qquad\end{proof}

The last iteration of the for-loop of Algorithm~\ref{alg:lrpcsa} (where $\ell=0$) operates like Algorithm~\ref{alg:lrpa} 
and hence Theorem~\ref{thm:correct} 
implies that Algorithm~\ref{alg:lrpcsa} works according to its specification.

For the analysis of the running time we need the following auxiliary result, 
relying on the Rerouting Theorem~\ref{thm:rerouting}.

\begin{lemma}
\label{lem:scaling}
Let $f \in B_\Z$ be $2^\ell$-integral and $\ell \in \N$ be such that $\resfk f \ell$ has no \stturnpath.
Then $\delta_\text{max} - \delta(f) < 3n 2^\ell$, where 
$\delta_\text{max} := \max_{ g \in B} \delta(g)$.
\end{lemma}

\begin{proof}
Let $g\in B$ with $\delta(g) = \delta_\text{max}$ and put 
$d:=g-f \in\oF(G)$. Hence 
$\delta_\text{max} - \delta(f) = \delta(d)$. 
Let $\varphi$ be the family of complete weighted turnpaths
corresponding to $d$ as provided by the Rerouting Theorem~\ref{thm:rerouting}.
We decompose the set 
$\cP_f=\cP_{st}\cup\cP_{ts}\cup\cP_{c}$ 
of complete turnpaths in~$\resf f$
into the sets $\cP_{st}$,  $\cP_{ts}$, and $\cP_c$ of 
$s$\dash $t$\dash turnpaths, $t$\dash $s$\dash turnpaths, and turncycles, respectively.
Then the flow $d':=\sum_p \varphi(p) p$ on $\resf f$ defined by 
$\varphi$ satisfies by~\eqref{eq:gaga} and Corollary~\ref{cor:RRT}, 
\begin{equation}\tag{i}\label{eq:starstar}
 \delta(d) \ =\  \delta(d') \ =\  
  \sum_{p\in\cP_{st}}\varphi(p) - \sum_{p\in\cP_{ts}}\varphi(p) 
  \ \le\ \sum_{p\in\cP_{st}}\varphi(p) .
\end{equation}

A turnpath $p\in\cP_{st}$ enters~$\Delta$ exactly once 
(through the right or bottom border) 
and leaves~$\Delta$ exactly once (through the left border). 
For an edge~$k$ on the right or bottom border of~$\Delta$,  
let $\cP_f(k)$ denote the set of $p\in\cP_f$ that enter $\Delta$ 
through~$k$. 
Further, for an edge~$k'$ on the left border of~$\Delta$, 
let $\cP_{f}(k')$ denote the set of $p\in\cP_{st}$ that leave $\Delta$ 
through~$k'$. 

We call an edge $k$ on the right or bottom border of $\Delta$ \emph{small}, 
if $\delta(k,f)+2^\ell > b(k)$. 
Let $\mathscr{E}$ denote the set of these edges. 
Note that for $k\in \mathscr{E}$ we have 
\begin{equation}\tag{ii}\label{eq:star}
\delta(k,d) \ = \  \delta(k,g) - \delta(k,f) \ \le\ 
 b(k) - \delta(k,f) \ <\ 2^\ell.  
\end{equation}
Similarly, we call an  edge $k' \in E(\Delta)$ 
on the left border of~$\Delta$ \emph{small}, 
if $-\delta(k',f)+2^\ell > b(k')$ and denote the set of these edges by~$\mathscr{E}'$. 
Border edges that are not small are called \emph{big}.

The point is that an \stturnpath $p\in\cP_{st}$ in $\resf f$ 
that crosses two big edges 
is also an \stturnpath in $\resfk f \ell$.
Hence, by our assumption, there are no 
\stturnpaths in $\resf f$ that cross two big edges.
We conclude that for all $p\in\cP_{st}$, there exists $k\in \mathscr{E}\cup \mathscr{E}'$ such that 
$p \in\cP_f(k)$. 
Therefore, 
\begin{equation}\tag{iii}\label{eq:sumsum}
 \sum_{p\in\cP_{st}} \varphi(p) \ \le\ 
 \sum_{k\in \mathscr{E}} \sum_{p\in \cP_{f} (k)} \varphi(p) + \sum_{k'\in \mathscr{E}'} \sum_{p\in \cP_{f} (k')} \varphi(p) .
\end{equation}To bound the right-hand sums, suppose first that $k\in \mathscr{E}$. 
By Proposition~\ref{cla:outerccdir}, 
$p\in \cP_{st} (k)$ implies that 
$k$ is the entrance edge of the side~$a$ of an $f$-flatspace $L$,
in which case $k=a_{\to L}$. 
We have $\cP_f(k) = \Pout L a$ and hence, 
by Definition~\ref{def:sw-ew},
$$
 \sum_{p\in \cP_{f} (k)}  \varphi(p) \ =\  
 \inom L a \varphi  = \inom L a d ,
$$
where the last equality is guaranteed by the Rerouting Theorem~\ref{thm:rerouting}. 

Lemma~\ref{obs:borderentranceexit} and Observation~\ref{obs:bd-cap} 
imply that, since $k$ is a small edge, 
all the other edges contained in $a$ are small as well.
Note that $\inflow L {\tilde k} d < 2^\ell$ implies $\inom L {\tilde k} d < 2^\ell$ for all edges $\tilde k \subseteq a$ by Definition~\ref{def:def-inflow}.
Therefore we can use~\eqref{eq:star} to deduce
$$
 \inom L a d \ =\  \sum_{\tilde{k}\subseteq a} \inom L {\tilde{k}} d \ <\ |a| \, 2^\ell ,
$$ 
where $|a|$ denotes the number of edges of $\Delta$ contained in $a$. 
Summarizing, we conclude for $k\in \mathscr{E}$, 
\begin{equation*}
 \sum_{p\in \cP_f(k)}  \varphi(p) \ <\  |a|\, 2^\ell ,
\end{equation*}
where $a$ is the side of the $f$-flatspace, in which $k$ lies.

The same bound holds for $k'\in \mathscr{E}'$ by an analogous argument. 
Combining these bounds with~\eqref{eq:starstar} and~\eqref{eq:sumsum} 
we obtain 
$$
 \delta(d) = \delta(d') \ \le\ \sum_{p\in\cP_{st}} \varphi(p) \ <\  3n 2^\ell 
$$
since there are $3n$ edges on the border of $\Delta$. 
\qquad\end{proof}

\begin{theorem}\label{thm:main}
 Algorithm~\ref{alg:lrpcsa} decides the positivity of the Littlewood--Richardson coefficient $\LRC \la \mu \nu$ with 
\runningtime arithmetic operations and comparisons.
\end{theorem}

\begin{proof}
Again let $\delta_\text{max} := \max_{g\in B} \delta(g)$.
After ending the while-loop for the value~$\ell$, there is no 
\stturnpath in $\resfk f \ell$ and hence 
$\delta_\text{max} - \delta(f) < 3n 2^\ell = 6n\, 2^{\ell-1}$.
Hence in the next iteration of the while-loop, for the value $\ell-1$, 
at most $6n$ \stturnpaths can be found. 
Moreover, note that the initial value of $\ell$ is so large,  
that in the first iteration of the while\dash loop at most one \stturnpath can be found.

So Algorithm~\ref{alg:lrpcsa} finds at most $6n\lceil\log\nu_1\rceil$ many \stturnpaths 
and searches at most $\log\nu_1$ many times for an \stturnpath without finding one.
Note that searching for a shortest \stturnpath requires at most time $\gO(n^2)$ using breadth-first-search, 
since there are $\gO(n^2)$ turnvertices and turnedges.
Hence we get a total running time of $\gO\big(n^3 \log\nu_1\big)$.
\qquad\end{proof}

\section{Proof of the Rerouting Theorem}
\label{sec:reroutingtheorem}

We will easily derive the Rerouting Theorem~\ref{thm:rerouting} 
by a glueing argument from the Canonical Turnpath Theorem~\ref{thm:CTT} 
below.
The latter is a general result for hive flows on convex sets in 
the triangular graph. In the next subsection we introduce 
the necessary terminology to state the Canonical Turnpath Theorem, 
which is then proved by induction in the remainder of this section, 
considering separately the different possible shapes of 
convex sets. 

\subsection{Canonical turnpaths in convex sets}

Let $L$ be a convex set in the triangular graph $\Delta$. 
We define the graph~$G_L$ as the induced subgraph of 
the honeycomb graph~$G$ obtained by restricting to 
the set of vertices lying in~$L$ (including the vertices on the 
border of $L$, but omitting $s$ and $t$). 
A {\em flow on $G_L$} is defined as a map 
$E(G_L)\to\R$ satisfying the flow conservation laws 
at all vertices of $G_L$ that do not lie at the border of $L$. 
The vector space  $\oF(G_L)$ of {\em flow classes on $G_L$}
is defined as in~\eqref{eq:def-oFG} by factoring out the 
null flows.
As in Definition~\ref{def:hiveflow}, we define a
 {\em hive flow~$f$ on $L$} as a flow class in $\oF(G_L)$ 
satisfying $\s\varrho f \ge 0$ for all rhombi $\varrho$ lying in $L$. 
Similarly, we define the notion of a 
{\em flow on $\res$ (or $\resf f$) restricted to $L$}, 
by restricting to the subgraph induced by the 
turnvertices lying in $L$. 



For a fixed convex set $L$ we are going to define a set $\cP_L$ 
of distinguished turnpaths $p$ in~$L$ starting at 
some entrance edge $a_{\to L}$ and ending at some exit edge $b_{L\to}$. 
The goal is to achieve that $p$ is a turnpath in $\resf f$ 
whenever $L$ is an $f$-flatspace. 
Proposition~\ref{cla:outerccdir} provides 
the guiding principle for making the right definition.

Let $r \ge 3$ and $a_1,\ldots,a_r$ be a sequence of successive sides of~$L$ 
in counterclockwise order, where $a_2,\ldots,a_{r-1}$ are different. 
Further, we assume that the angles between 
$a_{i-1}$ and $a_i$ are obtuse for $i=3,\ldots,r-1$.  
We then form a unique turnpath~$p$ moving within the border triangles of $L$ 
from $(a_1)_{\to L}$ to $(a_r)_{L\to}$ in counterclockwise direction, 
cf.\ Figure~\ref{fig:def-can-path}. 
\begin{figure}[h]
\begin{center}
\scalebox{2}{
\begin{tikzpicture}\draw[thin,-my] (2.3095pt,4.0001pt) arc (240:300:4.619pt);\draw[thin,-my] (6.9285pt,4.0001pt) arc (120:60:4.619pt);\draw[thin,-my] (11.5475pt,4.0001pt) arc (240:300:4.619pt);\draw[thin,-my] (16.1665pt,4.0001pt) arc (120:60:4.619pt);\draw[thin,-my] (20.7855pt,4.0001pt) arc (240:300:4.619pt);\draw[thin,-my] (25.4045pt,4.0001pt) arc (120:60:4.619pt);\draw[thin,-my] (30.0235pt,4.0001pt) arc (240:300:4.619pt);\draw[thin,-my] (34.6425pt,4.0001pt) arc (300:360:4.619pt);\draw[thin,-my] (36.952pt,8.0001pt) arc (180:120:4.619pt);\draw[thin,-my] (39.2615pt,12.0002pt) arc (300:360:4.619pt);\draw[thin,-my] (41.571pt,16.0002pt) arc (180:120:4.619pt);\draw[thin,-my] (43.8805pt,20.0003pt) arc (300:360:4.619pt);\fill (13.857pt,24.0003pt) circle (0.4pt);\fill (9.238pt,16.0002pt) circle (0.4pt);\fill (4.619pt,8.0001pt) circle (0.4pt);\fill (0.0pt,0.0pt) circle (0.4pt);\fill (9.238pt,0.0pt) circle (0.4pt);\fill (13.857pt,8.0001pt) circle (0.4pt);\fill (18.476pt,16.0002pt) circle (0.4pt);\fill (23.095pt,24.0003pt) circle (0.4pt);\fill (32.333pt,24.0003pt) circle (0.4pt);\fill (27.714pt,16.0002pt) circle (0.4pt);\fill (23.095pt,8.0001pt) circle (0.4pt);\fill (18.476pt,0.0pt) circle (0.4pt);\fill (27.714pt,0.0pt) circle (0.4pt);\fill (32.333pt,8.0001pt) circle (0.4pt);\fill (36.952pt,16.0002pt) circle (0.4pt);\fill (41.571pt,24.0003pt) circle (0.4pt);\fill (50.809pt,24.0003pt) circle (0.4pt);\fill (46.19pt,16.0002pt) circle (0.4pt);\fill (41.571pt,8.0001pt) circle (0.4pt);\fill (36.952pt,0.0pt) circle (0.4pt);\draw[rhrhombidraw] (-13.857pt,24.0003pt) -- (0.0pt,0.0pt) -- (36.952pt,0.0pt) -- (50.809pt,24.0003pt) -- cycle;\draw[rhrhombidraw] (-55.428pt,0.0pt) -- (-41.571pt,24.0003pt) -- (-78.523pt,24.0003pt) -- (-92.38pt,0.0pt) -- cycle;\fill (-78.523pt,24.0003pt) circle (0.4pt);\fill (-69.285pt,24.0003pt) circle (0.4pt);\fill (-60.047pt,24.0003pt) circle (0.4pt);\fill (-50.809pt,24.0003pt) circle (0.4pt);\fill (-41.571pt,24.0003pt) circle (0.4pt);\fill (-83.142pt,16.0002pt) circle (0.4pt);\fill (-73.904pt,16.0002pt) circle (0.4pt);\fill (-64.666pt,16.0002pt) circle (0.4pt);\fill (-55.428pt,16.0002pt) circle (0.4pt);\fill (-46.19pt,16.0002pt) circle (0.4pt);\fill (-87.761pt,8.0001pt) circle (0.4pt);\fill (-78.523pt,8.0001pt) circle (0.4pt);\fill (-69.285pt,8.0001pt) circle (0.4pt);\fill (-60.047pt,8.0001pt) circle (0.4pt);\fill (-50.809pt,8.0001pt) circle (0.4pt);\fill (-92.38pt,0.0pt) circle (0.4pt);\fill (-83.142pt,0.0pt) circle (0.4pt);\fill (-73.904pt,0.0pt) circle (0.4pt);\fill (-64.666pt,0.0pt) circle (0.4pt);\fill (-55.428pt,0.0pt) circle (0.4pt);\fill (4.619pt,24.0003pt) circle (0.4pt);\fill (0.0pt,16.0002pt) circle (0.4pt);\fill (-4.619pt,8.0001pt) circle (0.4pt);\fill (-4.619pt,24.0003pt) circle (0.4pt);\fill (-9.238pt,16.0002pt) circle (0.4pt);\fill (-13.857pt,24.0003pt) circle (0.4pt);\draw[thin,-my] (-2.3095pt,4.0001pt) arc (120:60:4.619pt);\draw[thin,-my] (-90.0705pt,4.0001pt) arc (240:300:4.619pt);\draw[thin,-my] (-85.4515pt,4.0001pt) arc (120:60:4.619pt);\draw[thin,-my] (-80.8325pt,4.0001pt) arc (240:300:4.619pt);\draw[thin,-my] (-76.2135pt,4.0001pt) arc (120:60:4.619pt);\draw[thin,-my] (-71.5945pt,4.0001pt) arc (240:300:4.619pt);\draw[thin,-my] (-66.9755pt,4.0001pt) arc (120:60:4.619pt);\draw[thin,-my] (-62.3565pt,4.0001pt) arc (240:300:4.619pt);\draw[thin,-my] (-57.7375pt,4.0001pt) arc (300:360:4.619pt);\draw[thin,-my] (-55.428pt,8.0001pt) arc (180:120:4.619pt);\draw[thin,-my] (-53.1185pt,12.0002pt) arc (300:360:4.619pt);\draw[thin,-my] (-50.809pt,16.0002pt) arc (180:120:4.619pt);\draw[thin,-my] (-48.4995pt,20.0003pt) arc (300:360:4.619pt);\end{tikzpicture}
}
    \caption{Canonical turnpaths in a parallelogram and in a trapezoid. 
                    The left hand turnpath starts with a counterclockwise turn (acute angle) and the righthand turnpath 
                    starts with a clockwise turn (obtuse angle).}
    \nopar\label{fig:def-can-path}
\end{center}
\end{figure}
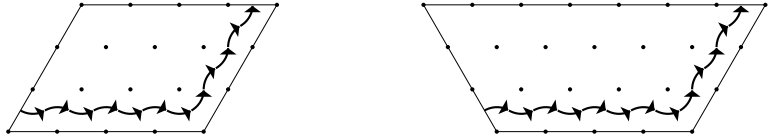 
The turnpath~$p$ alternatively takes clockwise and counterclockwise turns, except at the 
(obtuse) angles of~$L$ between $a_{i-1}$ and $a_i$ (for $i=3,\ldots,i-1$), 
where $p$ takes two consecutive counterclockwise turns to go around. 
If $a_1,a_2$ form an acute angle, then $p$ starts with a counterclockwise turn, 
otherwise $p$ starts with a clockwise turn. 
Moreover, if $a_{r-1},a_r$ form an acute angle, then $p$ ends with a counterclockwise turn, 
otherwise $p$ ends with a clockwise turn.
We call the resulting turnpath a {\em canonical turnpath} of $L$. 
We shall also consider the turnpaths consisting of a single clockwise turnvertex 
at an acute angle as a canonical turnpath of $L$.

\begin{definition}\label{def:PL}
The symbol $\cP_L$ denotes the set of all canonical turnpaths of
the convex set~$L$.
For $p\in\cP_L$, we denote by $\pstart{p}=(a_1)_{\to L}$ the edge of~$\Delta$ from which 
$p$~starts and by $\pend{p}=(a_r)_{L\to}$ the edge of~$\Delta$ where $p$ ends.
\end{definition}

\begin{example}\label{ex:can-tp}
A triangle has exactly six canonical turnpaths, 
cf.\ Figure~\ref{fig:tp-triangle}. 
A parallelogram has exactly eight canonical turnpaths, 
cf.\ Figure~\ref{fig:tp-parallelogram}. 
In particular, this holds true for rhombi. 
A trapezoid has exactly nine canonical turnpaths, 
cf.\ Figure~\ref{fig:tp-trapezoid}. 
A pentagon has exactly 16 canonical turnpaths,  
cf.\ Figure~\ref{fig:tp-pentagon}. 
A hexagon has six canonical turnpaths up to rotations, 
which makes a total of 36 turnpaths, see 
Figure~\ref{fig:tp-hexagon}. 
\end{example}

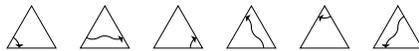
\begin{figure}[h]
\begin{center}
\begin{tikzpicture}\draw[rhrhombidraw] (-110.856pt,0.0pt) -- (-101.618pt,16.0002pt) -- (-92.38pt,0.0pt) -- cycle;\draw[rhrhombidraw] (-83.142pt,0.0pt) -- (-73.904pt,16.0002pt) -- (-64.666pt,0.0pt) -- cycle;\draw[rhrhombidraw] (-55.428pt,0.0pt) -- (-46.19pt,16.0002pt) -- (-36.952pt,0.0pt) -- cycle;\draw[rhrhombidraw] (-27.714pt,0.0pt) -- (-18.476pt,16.0002pt) -- (-9.238pt,0.0pt) -- cycle;\draw[rhrhombidraw] (0.0pt,0.0pt) -- (9.238pt,16.0002pt) -- (18.476pt,0.0pt) -- cycle;\draw[rhrhombidraw] (27.714pt,0.0pt) -- (36.952pt,16.0002pt) -- (46.19pt,0.0pt) -- cycle;\draw[thin,-my] (-108.5465pt,4.0001pt) arc (60:0:4.619pt);\draw[thin,-my] (-80.8325pt,4.0001pt) arc (240:300:4.619pt) arc (120:60:4.619pt) arc (240:300:4.619pt);\draw[thin,-my] (-41.571pt,0.0pt) arc (180:120:4.619pt);\draw[thin,-my] (-13.857pt,0.0pt) arc (0:60:4.619pt) arc (240:180:4.619pt) arc (0:60:4.619pt);\draw[thin,-my] (11.5475pt,12.0002pt) arc (300:240:4.619pt);\draw[thin,-my] (39.2615pt,12.0002pt) arc (120:180:4.619pt) arc (0:-60:4.619pt) arc (120:180:4.619pt);\end{tikzpicture}
    \caption{The six canonical turnpaths in a triangle.}
    \nopar\label{fig:tp-triangle}
\end{center}
\end{figure}

\begin{figure}[h]
\begin{center}
\begin{tikzpicture}\draw[rhrhombidraw] (-120.094pt,0.0pt) -- (-92.38pt,0.0pt) -- (-83.142pt,16.0002pt) -- (-110.856pt,16.0002pt) -- cycle;\draw[rhrhombidraw] (-83.142pt,0.0pt) -- (-55.428pt,0.0pt) -- (-46.19pt,16.0002pt) -- (-73.904pt,16.0002pt) -- cycle;\draw[rhrhombidraw] (-46.19pt,0.0pt) -- (-18.476pt,0.0pt) -- (-9.238pt,16.0002pt) -- (-36.952pt,16.0002pt) -- cycle;\draw[rhrhombidraw] (-9.238pt,0.0pt) -- (18.476pt,0.0pt) -- (27.714pt,16.0002pt) -- (0.0pt,16.0002pt) -- cycle;\draw[rhrhombidraw] (27.714pt,0.0pt) -- (55.428pt,0.0pt) -- (64.666pt,16.0002pt) -- (36.952pt,16.0002pt) -- cycle;\draw[rhrhombidraw] (64.666pt,0.0pt) -- (92.38pt,0.0pt) -- (101.618pt,16.0002pt) -- (73.904pt,16.0002pt) -- cycle;\draw[rhrhombidraw] (101.618pt,0.0pt) -- (129.332pt,0.0pt) -- (138.57pt,16.0002pt) -- (110.856pt,16.0002pt) -- cycle;\draw[rhrhombidraw] (138.57pt,0.0pt) -- (166.284pt,0.0pt) -- (175.522pt,16.0002pt) -- (147.808pt,16.0002pt) -- cycle;\draw[thin,-my] (-117.7845pt,4.0001pt) arc (60:0:4.619pt);\draw[thin,-my] (-80.8325pt,4.0001pt) arc (240:300:4.619pt) arc (120:60:4.619pt) arc (240:300:4.619pt) arc (120:60:4.619pt) arc (240:300:4.619pt) arc (120:60:4.619pt);\draw[thin,-my] (-43.8805pt,4.0001pt) arc (240:300:4.619pt) arc (120:60:4.619pt) arc (240:300:4.619pt) arc (120:60:4.619pt) arc (240:300:4.619pt) arc (300:360:4.619pt) arc (180:120:4.619pt) arc (300:360:4.619pt);\draw[thin,-my] (13.857pt,0.0pt) arc (180:120:4.619pt) arc (300:360:4.619pt) arc (180:120:4.619pt) arc (300:360:4.619pt);\draw[thin,-my] (62.3565pt,12.0002pt) arc (240:180:4.619pt);\draw[thin,-my] (99.3085pt,12.0002pt) arc (60:120:4.619pt) arc (300:240:4.619pt) arc (60:120:4.619pt) arc (300:240:4.619pt) arc (60:120:4.619pt) arc (300:240:4.619pt);\draw[thin,-my] (136.2605pt,12.0002pt) arc (60:120:4.619pt) arc (300:240:4.619pt) arc (60:120:4.619pt) arc (300:240:4.619pt) arc (60:120:4.619pt) arc (120:180:4.619pt) arc (0:-60:4.619pt) arc (120:180:4.619pt);\draw[thin,-my] (152.427pt,16.0002pt) arc (0:-60:4.619pt) arc (120:180:4.619pt) arc (0:-60:4.619pt) arc (120:180:4.619pt);\end{tikzpicture}
\\
\begin{tikzpicture}\draw[rhrhombidraw] (-110.856pt,0.0pt) -- (-106.237pt,8.0001pt) -- (-96.999pt,8.0001pt) -- (-101.618pt,0.0pt) -- cycle;\draw[rhrhombidraw] (-92.38pt,0.0pt) -- (-87.761pt,8.0001pt) -- (-78.523pt,8.0001pt) -- (-83.142pt,0.0pt) -- cycle;\draw[rhrhombidraw] (-73.904pt,0.0pt) -- (-69.285pt,8.0001pt) -- (-60.047pt,8.0001pt) -- (-64.666pt,0.0pt) -- cycle;\draw[rhrhombidraw] (-55.428pt,0.0pt) -- (-50.809pt,8.0001pt) -- (-41.571pt,8.0001pt) -- (-46.19pt,0.0pt) -- cycle;\draw[rhrhombidraw] (-36.952pt,0.0pt) -- (-32.333pt,8.0001pt) -- (-23.095pt,8.0001pt) -- (-27.714pt,0.0pt) -- cycle;\draw[rhrhombidraw] (-18.476pt,0.0pt) -- (-13.857pt,8.0001pt) -- (-4.619pt,8.0001pt) -- (-9.238pt,0.0pt) -- cycle;\draw[rhrhombidraw] (0.0pt,0.0pt) -- (4.619pt,8.0001pt) -- (13.857pt,8.0001pt) -- (9.238pt,0.0pt) -- cycle;\draw[rhrhombidraw] (18.476pt,0.0pt) -- (23.095pt,8.0001pt) -- (32.333pt,8.0001pt) -- (27.714pt,0.0pt) -- cycle;\draw[thin,-my] (-108.5465pt,4.0001pt) arc (60:0:4.619pt);\draw[thin,-my] (-90.0705pt,4.0001pt) arc (240:300:4.619pt) arc (120:60:4.619pt);\draw[thin,-my] (-71.5945pt,4.0001pt) arc (240:300:4.619pt) arc (300:360:4.619pt);\draw[thin,-my] (-50.809pt,0.0pt) arc (180:120:4.619pt) arc (300:360:4.619pt);\draw[thin,-my] (-25.4045pt,4.0001pt) arc (240:180:4.619pt);\draw[thin,-my] (-6.9285pt,4.0001pt) arc (60:120:4.619pt) arc (300:240:4.619pt);\draw[thin,-my] (11.5475pt,4.0001pt) arc (60:120:4.619pt) arc (120:180:4.619pt);\draw[thin,-my] (27.714pt,8.0001pt) arc (0:-60:4.619pt) arc (120:180:4.619pt);\end{tikzpicture}
    \caption{Top row: The eight canonical turnpaths in a parallelogram. Bottom row:
    The same list in the special case of a rhombus.
    }
    \nopar\label{fig:tp-parallelogram}
\end{center}
\end{figure}
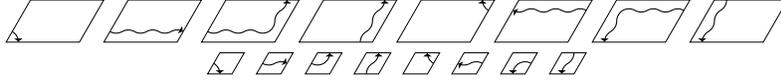

\begin{figure}[h]
\begin{center}
\scalebox{0.965}{
\begin{tikzpicture}\draw[rhrhombidraw] (-120.094pt,0.0pt) -- (-110.856pt,16.0002pt) -- (-92.38pt,16.0002pt) -- (-83.142pt,0.0pt) -- cycle;\draw[rhrhombidraw] (-73.904pt,0.0pt) -- (-64.666pt,16.0002pt) -- (-46.19pt,16.0002pt) -- (-36.952pt,0.0pt) -- cycle;\draw[rhrhombidraw] (-27.714pt,0.0pt) -- (-18.476pt,16.0002pt) -- (0.0pt,16.0002pt) -- (9.238pt,0.0pt) -- cycle;\draw[rhrhombidraw] (18.476pt,0.0pt) -- (27.714pt,16.0002pt) -- (46.19pt,16.0002pt) -- (55.428pt,0.0pt) -- cycle;\draw[rhrhombidraw] (64.666pt,0.0pt) -- (73.904pt,16.0002pt) -- (92.38pt,16.0002pt) -- (101.618pt,0.0pt) -- cycle;\draw[rhrhombidraw] (110.856pt,0.0pt) -- (120.094pt,16.0002pt) -- (138.57pt,16.0002pt) -- (147.808pt,0.0pt) -- cycle;\draw[rhrhombidraw] (157.046pt,0.0pt) -- (166.284pt,16.0002pt) -- (184.76pt,16.0002pt) -- (193.998pt,0.0pt) -- cycle;\draw[rhrhombidraw] (203.236pt,0.0pt) -- (212.474pt,16.0002pt) -- (230.95pt,16.0002pt) -- (240.188pt,0.0pt) -- cycle;\draw[rhrhombidraw] (249.426pt,0.0pt) -- (258.664pt,16.0002pt) -- (277.14pt,16.0002pt) -- (286.378pt,0.0pt) -- cycle;\draw[thin,-my] (-117.7845pt,4.0001pt) arc (60:0:4.619pt);\draw[thin,-my] (-71.5945pt,4.0001pt) arc (240:300:4.619pt) arc (120:60:4.619pt) arc (240:300:4.619pt) arc (120:60:4.619pt) arc (240:300:4.619pt) arc (120:60:4.619pt) arc (240:300:4.619pt);\draw[thin,-my] (4.619pt,0.0pt) arc (180:120:4.619pt);\draw[thin,-my] (50.809pt,0.0pt) arc (0:60:4.619pt) arc (240:180:4.619pt) arc (0:60:4.619pt) arc (240:180:4.619pt);\draw[thin,-my] (96.999pt,0.0pt) arc (0:60:4.619pt) arc (240:180:4.619pt) arc (0:60:4.619pt) arc (60:120:4.619pt) arc (300:240:4.619pt) arc (60:120:4.619pt) arc (300:240:4.619pt);\draw[thin,-my] (143.189pt,0.0pt) arc (0:60:4.619pt) arc (240:180:4.619pt) arc (0:60:4.619pt) arc (60:120:4.619pt) arc (300:240:4.619pt) arc (60:120:4.619pt) arc (120:180:4.619pt) arc (0:-60:4.619pt) arc (120:180:4.619pt);\draw[thin,-my] (187.0695pt,12.0002pt) arc (300:240:4.619pt) arc (60:120:4.619pt) arc (300:240:4.619pt) arc (60:120:4.619pt) arc (300:240:4.619pt);\draw[thin,-my] (233.2595pt,12.0002pt) arc (300:240:4.619pt) arc (60:120:4.619pt) arc (300:240:4.619pt) arc (60:120:4.619pt) arc (120:180:4.619pt) arc (0:-60:4.619pt) arc (120:180:4.619pt);\draw[thin,-my] (263.283pt,16.0002pt) arc (0:-60:4.619pt) arc (120:180:4.619pt) arc (0:-60:4.619pt) arc (120:180:4.619pt);\end{tikzpicture}
}
    \caption{The nine canonical turnpaths in a trapezoid.}
    \nopar\label{fig:tp-trapezoid}
\end{center}
\end{figure}
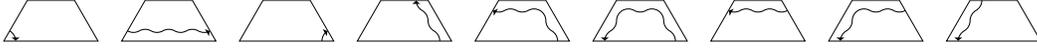

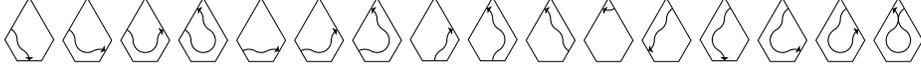
\begin{figure}[h]
\begin{center}
\begin{tikzpicture}\draw[rhrhombidraw] (-4.619pt,8.0001pt) -- (0.0pt,0.0pt) -- (9.238pt,0.0pt) -- (13.857pt,8.0001pt) -- (4.619pt,24.0003pt) -- cycle;\draw[thin,-my] (-2.3095pt,12.0002pt) arc (60:0:4.619pt) arc (180:240:4.619pt) arc (60:0:4.619pt);\end{tikzpicture}
\begin{tikzpicture}\draw[rhrhombidraw] (-4.619pt,8.0001pt) -- (0.0pt,0.0pt) -- (9.238pt,0.0pt) -- (13.857pt,8.0001pt) -- (4.619pt,24.0003pt) -- cycle;\draw[thin,-my] (-2.3095pt,12.0002pt) arc (60:0:4.619pt) arc (180:240:4.619pt) arc (240:300:4.619pt) arc (120:60:4.619pt);\end{tikzpicture}
\begin{tikzpicture}\draw[rhrhombidraw] (-4.619pt,8.0001pt) -- (0.0pt,0.0pt) -- (9.238pt,0.0pt) -- (13.857pt,8.0001pt) -- (4.619pt,24.0003pt) -- cycle;\draw[thin,-my] (-2.3095pt,12.0002pt) arc (60:0:4.619pt) arc (180:240:4.619pt) arc (240:300:4.619pt) arc (300:360:4.619pt) arc (180:120:4.619pt);\end{tikzpicture}
\begin{tikzpicture}\draw[rhrhombidraw] (-4.619pt,8.0001pt) -- (0.0pt,0.0pt) -- (9.238pt,0.0pt) -- (13.857pt,8.0001pt) -- (4.619pt,24.0003pt) -- cycle;\draw[thin,-my] (-2.3095pt,12.0002pt) arc (60:0:4.619pt) arc (180:240:4.619pt) arc (240:300:4.619pt) arc (300:360:4.619pt) arc (0:60:4.619pt) arc (240:180:4.619pt) arc (0:60:4.619pt);\end{tikzpicture}
\begin{tikzpicture}\draw[rhrhombidraw] (-4.619pt,8.0001pt) -- (0.0pt,0.0pt) -- (9.238pt,0.0pt) -- (13.857pt,8.0001pt) -- (4.619pt,24.0003pt) -- cycle;\draw[thin,-my] (-2.3095pt,4.0001pt) arc (120:60:4.619pt) arc (240:300:4.619pt) arc (120:60:4.619pt);\end{tikzpicture}
\begin{tikzpicture}\draw[rhrhombidraw] (-4.619pt,8.0001pt) -- (0.0pt,0.0pt) -- (9.238pt,0.0pt) -- (13.857pt,8.0001pt) -- (4.619pt,24.0003pt) -- cycle;\draw[thin,-my] (-2.3095pt,4.0001pt) arc (120:60:4.619pt) arc (240:300:4.619pt) arc (300:360:4.619pt) arc (180:120:4.619pt);\end{tikzpicture}
\begin{tikzpicture}\draw[rhrhombidraw] (-4.619pt,8.0001pt) -- (0.0pt,0.0pt) -- (9.238pt,0.0pt) -- (13.857pt,8.0001pt) -- (4.619pt,24.0003pt) -- cycle;\draw[thin,-my] (-2.3095pt,4.0001pt) arc (120:60:4.619pt) arc (240:300:4.619pt) arc (300:360:4.619pt) arc (0:60:4.619pt) arc (240:180:4.619pt) arc (0:60:4.619pt);\end{tikzpicture}
\begin{tikzpicture}\draw[rhrhombidraw] (-4.619pt,8.0001pt) -- (0.0pt,0.0pt) -- (9.238pt,0.0pt) -- (13.857pt,8.0001pt) -- (4.619pt,24.0003pt) -- cycle;\draw[thin,-my] (4.619pt,0.0pt) arc (180:120:4.619pt) arc (300:360:4.619pt) arc (180:120:4.619pt);\end{tikzpicture}
\begin{tikzpicture}\draw[rhrhombidraw] (-4.619pt,8.0001pt) -- (0.0pt,0.0pt) -- (9.238pt,0.0pt) -- (13.857pt,8.0001pt) -- (4.619pt,24.0003pt) -- cycle;\draw[thin,-my] (4.619pt,0.0pt) arc (180:120:4.619pt) arc (300:360:4.619pt) arc (0:60:4.619pt) arc (240:180:4.619pt) arc (0:60:4.619pt);\end{tikzpicture}
\begin{tikzpicture}\draw[rhrhombidraw] (-4.619pt,8.0001pt) -- (0.0pt,0.0pt) -- (9.238pt,0.0pt) -- (13.857pt,8.0001pt) -- (4.619pt,24.0003pt) -- cycle;\draw[thin,-my] (11.5475pt,4.0001pt) arc (240:180:4.619pt) arc (0:60:4.619pt) arc (240:180:4.619pt) arc (0:60:4.619pt);\end{tikzpicture}
\begin{tikzpicture}\draw[rhrhombidraw] (-4.619pt,8.0001pt) -- (0.0pt,0.0pt) -- (9.238pt,0.0pt) -- (13.857pt,8.0001pt) -- (4.619pt,24.0003pt) -- cycle;\draw[thin,-my] (6.9285pt,20.0003pt) arc (300:240:4.619pt);\end{tikzpicture}
\begin{tikzpicture}\draw[rhrhombidraw] (-4.619pt,8.0001pt) -- (0.0pt,0.0pt) -- (9.238pt,0.0pt) -- (13.857pt,8.0001pt) -- (4.619pt,24.0003pt) -- cycle;\draw[thin,-my] (6.9285pt,20.0003pt) arc (120:180:4.619pt) arc (0:-60:4.619pt) arc (120:180:4.619pt) arc (0:-60:4.619pt);\end{tikzpicture}
\begin{tikzpicture}\draw[rhrhombidraw] (-4.619pt,8.0001pt) -- (0.0pt,0.0pt) -- (9.238pt,0.0pt) -- (13.857pt,8.0001pt) -- (4.619pt,24.0003pt) -- cycle;\draw[thin,-my] (6.9285pt,20.0003pt) arc (120:180:4.619pt) arc (0:-60:4.619pt) arc (120:180:4.619pt) arc (180:240:4.619pt) arc (60:0:4.619pt);\end{tikzpicture}
\begin{tikzpicture}\draw[rhrhombidraw] (-4.619pt,8.0001pt) -- (0.0pt,0.0pt) -- (9.238pt,0.0pt) -- (13.857pt,8.0001pt) -- (4.619pt,24.0003pt) -- cycle;\draw[thin,-my] (6.9285pt,20.0003pt) arc (120:180:4.619pt) arc (0:-60:4.619pt) arc (120:180:4.619pt) arc (180:240:4.619pt) arc (240:300:4.619pt) arc (120:60:4.619pt);\end{tikzpicture}
\begin{tikzpicture}\draw[rhrhombidraw] (-4.619pt,8.0001pt) -- (0.0pt,0.0pt) -- (9.238pt,0.0pt) -- (13.857pt,8.0001pt) -- (4.619pt,24.0003pt) -- cycle;\draw[thin,-my] (6.9285pt,20.0003pt) arc (120:180:4.619pt) arc (0:-60:4.619pt) arc (120:180:4.619pt) arc (180:240:4.619pt) arc (240:300:4.619pt) arc (300:360:4.619pt) arc (180:120:4.619pt);\end{tikzpicture}
\begin{tikzpicture}\draw[rhrhombidraw] (-4.619pt,8.0001pt) -- (0.0pt,0.0pt) -- (9.238pt,0.0pt) -- (13.857pt,8.0001pt) -- (4.619pt,24.0003pt) -- cycle;\draw[thin,-my] (6.9285pt,20.0003pt) arc (120:180:4.619pt) arc (0:-60:4.619pt) arc (120:180:4.619pt) arc (180:240:4.619pt) arc (240:300:4.619pt) arc (300:360:4.619pt) arc (0:60:4.619pt) arc (240:180:4.619pt) arc (0:60:4.619pt);\end{tikzpicture}
    \caption{The 16 canonical turnpaths in a pentagon.}
    \nopar\label{fig:tp-pentagon}
\end{center}
\end{figure}

\begin{figure}[h]
\begin{center}
\begin{tikzpicture}\draw[rhrhombidraw] (-101.618pt,0.0pt) -- (-83.142pt,0.0pt) -- (-73.904pt,16.0002pt) -- (-83.142pt,32.0004pt) -- (-101.618pt,32.0004pt) -- (-110.856pt,16.0002pt) -- cycle;\draw[rhrhombidraw] (-64.666pt,16.0002pt) -- (-55.428pt,0.0pt) -- (-36.952pt,0.0pt) -- (-27.714pt,16.0002pt) -- (-36.952pt,32.0004pt) -- (-55.428pt,32.0004pt) -- cycle;\draw[rhrhombidraw] (-18.476pt,16.0002pt) -- (-9.238pt,0.0pt) -- (9.238pt,0.0pt) -- (18.476pt,16.0002pt) -- (9.238pt,32.0004pt) -- (-9.238pt,32.0004pt) -- cycle;\draw[rhrhombidraw] (27.714pt,16.0002pt) -- (36.952pt,0.0pt) -- (55.428pt,0.0pt) -- (64.666pt,16.0002pt) -- (55.428pt,32.0004pt) -- (36.952pt,32.0004pt) -- cycle;\draw[rhrhombidraw] (73.904pt,16.0002pt) -- (83.142pt,0.0pt) -- (101.618pt,0.0pt) -- (110.856pt,16.0002pt) -- (101.618pt,32.0004pt) -- (83.142pt,32.0004pt) -- cycle;\draw[rhrhombidraw] (120.094pt,16.0002pt) -- (129.332pt,0.0pt) -- (147.808pt,0.0pt) -- (157.046pt,16.0002pt) -- (147.808pt,32.0004pt) -- (129.332pt,32.0004pt) -- cycle;\draw[thin,-my] (-103.9275pt,4.0001pt) arc (120:60:4.619pt) arc (240:300:4.619pt) arc (120:60:4.619pt) arc (240:300:4.619pt) arc (120:60:4.619pt);\draw[thin,-my] (-57.7375pt,4.0001pt) arc (120:60:4.619pt) arc (240:300:4.619pt) arc (120:60:4.619pt) arc (240:300:4.619pt) arc (300:360:4.619pt) arc (180:120:4.619pt) arc (300:360:4.619pt) arc (180:120:4.619pt);\draw[thin,-my] (-11.5475pt,4.0001pt) arc (120:60:4.619pt) arc (240:300:4.619pt) arc (120:60:4.619pt) arc (240:300:4.619pt) arc (300:360:4.619pt) arc (180:120:4.619pt) arc (300:360:4.619pt) arc (0:60:4.619pt) arc (240:180:4.619pt) arc (0:60:4.619pt) arc (240:180:4.619pt);\draw[thin,-my] (34.6425pt,4.0001pt) arc (120:60:4.619pt) arc (240:300:4.619pt) arc (120:60:4.619pt) arc (240:300:4.619pt) arc (300:360:4.619pt) arc (180:120:4.619pt) arc (300:360:4.619pt) arc (0:60:4.619pt) arc (240:180:4.619pt) arc (0:60:4.619pt) arc (60:120:4.619pt) arc (300:240:4.619pt) arc (60:120:4.619pt) arc (300:240:4.619pt);\draw[thin,-my] (80.8325pt,4.0001pt) arc (120:60:4.619pt) arc (240:300:4.619pt) arc (120:60:4.619pt) arc (240:300:4.619pt) arc (300:360:4.619pt) arc (180:120:4.619pt) arc (300:360:4.619pt) arc (0:60:4.619pt) arc (240:180:4.619pt) arc (0:60:4.619pt) arc (60:120:4.619pt) arc (300:240:4.619pt) arc (60:120:4.619pt) arc (120:180:4.619pt) arc (0:-60:4.619pt) arc (120:180:4.619pt) arc (0:-60:4.619pt);\draw[thin,-my] (127.0225pt,4.0001pt) arc (120:60:4.619pt) arc (240:300:4.619pt) arc (120:60:4.619pt) arc (240:300:4.619pt) arc (300:360:4.619pt) arc (180:120:4.619pt) arc (300:360:4.619pt) arc (0:60:4.619pt) arc (240:180:4.619pt) arc (0:60:4.619pt) arc (60:120:4.619pt) arc (300:240:4.619pt) arc (60:120:4.619pt) arc (120:180:4.619pt) arc (0:-60:4.619pt) arc (120:180:4.619pt) arc (180:240:4.619pt) arc (60:0:4.619pt) arc (180:240:4.619pt) arc (60:0:4.619pt);\end{tikzpicture}
    \caption{The six canonical turnpaths in a hexagon starting from a fixed side.}
    \nopar\label{fig:tp-hexagon}
\end{center}
\end{figure}
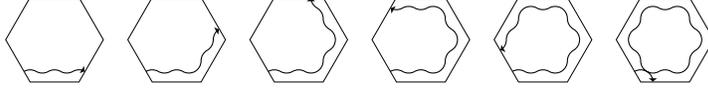


\begin{lemma}\label{le:canpa}
Let $f$ be a hive flow and $L$ be one of its $f$\dash flatspaces. 
Then any canonical turnpath $p\in \cP_L$ of $L$ is a turnpath in $\resf f$. 
\end{lemma}

\begin{proof}
We need to ensure that $p$ uses no negative contributions in $f$\dash flat rhombi.
But $p$~does not use any pair of successive clockwise turnvertices at all.
And whenever $p$ uses a counterclockwise turnvertex $\rhpoulMl$, then $\rhsc$ is at the border of an $f$\dash flatspace.
\qquad\end{proof}



%


We have to extend some of the notions introduced in Section~\ref{se:RRT}.

\begin{definition}\label{def:enter-exit-L}
A {\em multiset~$\varphi$ of canonical turnpaths} in a convex set~$L$ 
is defined as a map $\varphi\colon\cP_L\to \N_{\ge 0}$. 
Let $a$ be a side of $L$. 
The number of turnpaths of~$\varphi$ starting from $a_{\to L}$ and ending at $a_{L\to}$, respectively, is denoted by  
\begin{equation*}
\inom L a \varphi \ := \sum_{\pstart{p}= a_{\to L}}\varphi(p),\quad
\outom L  a \varphi \ := \sum_{\pend{p}= a_{L\to}} \varphi(p) .
\vspace{-0.5cm}
\end{equation*}
\end{definition}


Note that 
the weighted sum 
$\sum_{p\in\cP_L} \varphi(p) p$
defines a nonnegative flow~$d'_L$ on~$\res$ restricted to~$L$. 
Moreover, if $L$ is an $f$\dash flatspace, then 
$d'_L$~is a flow on $\resf f$ restricted to~$L$, 
as a consequence of Lemma~\ref{le:canpa}.

\begin{definition}\label{def:compatible}
Let $d$ be a hive flow on a convex set~$L$. 
A multiset~$\varphi$ of canonical turnpaths on~$L$ 
is called {\em compatible with $d$}, if 
$\inom L a \varphi = \inom L a d$ and $\outom L a \varphi = \outom L a d$ 
for all edges~$a$ of $L$.
\end{definition}

The key result, amenable to an inductive proof along~$L$,  
is the following.

\begin{theorem}[Canonical Turnpath Theorem]\label{thm:CTT}
Let $L$ be a convex set and $d$ be an integral hive flow on $L$. 
Then there exists a multiset~$\varphi$ of canonical turnpaths 
on~$L$ which is {\em compatible with $d$}.
\end{theorem}


\begin{lemma}
The Canonical Turnpath Theorem~\ref{thm:CTT} implies the 
Rerouting Theorem~\ref{thm:rerouting}. 
\end{lemma}

\begin{proof} 
We first note that it suffices to prove the Rerouting Theorem~\ref{thm:rerouting} 
for an integral flow $d\in\oF(G)$.  
Indeed, then it trivially must hold for a rational $d$.  
A standard continuity argument then shows the assertion for a real~$d$. 

So let $f\in B$ and $d\in\oF(G)$ be integral and $f$-hive preserving. 
Theorem~\ref{thm:CTT} applied to every $f$\dash flatspace~$L$ and
the hive flow $d$ restricted to $L$ yields a multiset~$\varphi_L$ of 
canonical turnpaths on $L$ compatible with $d$ restricted to $L$.
By Lemma~\ref{le:canpa}, canonical turnpaths of $L$ are in $\resf f$.

Suppose that $L$ and $M$ are adjacent $f$-flatspaces sharing the side~$a$. 
Then we have, using Theorem~\ref{thm:CTT} and~\eqref{eq:in=out}, 
\begin{equation}\label{eq:comprel}
 \inom L a {\varphi_L} = \inom L a d = \outom M a d = \outom M a {\varphi_M} .
\end{equation}
We set up an arbitrary bijection between the turnpaths $p_M$ in $\varphi_{M}$ 
ending at $a_{M\to}$ and the turnpaths $p_L$ in $\varphi_{L}$ 
starting from $a_{\to L} =a_{M\to}$ and concatenate these turnpaths correspondingly. 
It is essential to note that the additional turnedges used for joining $p_M$ and $p_L$ 
lie in~$\resf f$, since the rhombus with diagonal $a_{\to L }$ is not $f$\dash flat!

Similarly, we have
$\inom M a {\varphi_M} = \outom L a {\varphi_L}$ and we concatenate the turnpaths in~$\varphi_{L}$ 
ending at $a_{L\to}$ with the turnpaths in $\varphi_{M}$ 
starting from $a_{\to M} = a_{L\to}$ correspondingly. 

Doing so for all sides~$a$ shared by different $f$\dash flatspaces, 
we obtain a multiset of turncycles in $\resf f$ 
and a multiset of turnpaths in $\resf f$ going from a side of $\Delta$ to a side of $\Delta$. 
These turnpaths can be extended to complete turnpaths. 
Altogether, we obtain a multiset $\varphi$ of complete turnpaths in~$\resf f$. 

Then we have $\inom L a \varphi = \inom L a {\varphi_L}$ and 
$\outom L a \varphi = \outom L a {\varphi_L}$ 
for any side~$a$ of an $f$-flatspace~$L$. Hence,
$\inom L a \varphi = \inom L a d$ and 
$\outom L a \varphi = \outom L a d$ by~\eqref{eq:comprel}. 
So the multiset~$\varphi$ is as required.  
\qquad\end{proof}



In the subsequent sections we shall prove Theorem~\ref{thm:CTT} 
for the five possible shapes of $L$. Although the arguments are quite similar for 
the different shapes, there are subtle differences. 
We begin with the case of parallelograms. 


\subsection{Parallelograms}


By the {\em size} of a convex set~$L$ we understand the number 
of hive triangles contained in~$L$. 
We will prove the Canonical Turnpath Theorem~\ref{thm:CTT}
for parallelograms~$L$ by induction on the size of~$L$. 
The induction start is provided by the following lemma.



\begin{lemma}\label{le:CTT-rhombus}
The assertion of the Canonical Turnpath Theorem~\ref{thm:CTT} 
is true if $L=\varrho$ is a rhombus. 
More specifically, if $d$ is an integral hive flow on $\varrho$, 
then there is a multiset $\varphi$ of canonical turnpaths compatible with $d$,  
such that for all $p\in\Psi_+(\varrho)$ occuring in $\varphi$ we have 
$p\subseteq\SUPP(d)$.
\end{lemma}
 
\begin{proof}
The canonical turnpaths in a rhombus~$\varrho$ are exactly the 
eight contributions in $\Psi_+(\varrho)\cup\Psi_0(\varrho)$,
see Figure~\ref{fig:tp-parallelogram}. 
 

Given an integral hive flow~$d$ on $\varrho$.
If $p\subseteq\SUPP(d)$ for some $p\in\Psi_-(\varrho)$, 
then  $p'\subseteq\SUPP(d)$ by Lemma~\ref{cla:negimpliespos} on antipodal contributions. 
Since $\s \varrho {p+p'} =0$ it follows that $d-(p+p')$ is a hive flow. 
So we can successively subtract flows of the form $p+p'$ from $d$ 
to arrive at a flow decomposition
$d= \sum_i m_i (p_i +p'_i) + h$, 
where $m_i\in\N$, 
$h$ is a hive flow on~$\varrho$, 
and $p\not\subseteq\SUPP(h)$ 
for all $p\in \Psi_-(\varrho)$. 
Moreover, $\SUPP(h)\subseteq\SUPP(d)$ by construction. 
 
It is straighforward to check that~$h$ 
must be a nonnegative integer linear combination of 
turnpaths $p\in\Psi_+(\varrho)\cup\Psi_0(\varrho)$ such that 
$p\subseteq\SUPP(h)$.

Now we replace the sums $p_i+p'_i$ by sums $n_i+n'_i$ 
of two neutral slack contributions as follows:
we replace 
$\rhpourMr +\rhpolrWl$ by $\rhpourMlr+\rhpolrWrl$, 
we replace 
$\rhpourMll+\rhpoulMrr$ by $\rhpourMlr+\rhpoulMrl$, 
and similarly in the situations rotated by $180^\circ$. 
Since this exchange does not alter the number of turnpaths
entering and leaving a side of $\varrho$,
this leads to a multiset of canonical turnpaths of $\varrho$ 
satisfing the desired requirements. 
%
%
%
\qquad\end{proof}
 

The induction step will be based on the following result on
\emph{straightening} canonical turnpaths.

\begin{proposition}\label{pro:straight-parallelogram}
Let $L$ be a parallelogram cut into two parallelograms $L_1$ and $L_2$ 
by a straight line parallel to one of the sides of~$L$. 
Further let $p$ be a turnpath going from the side~$a$ of~$L$ to 
the side~$b$ of~$L$ such that~$p$ is either a canonical turnpath of $L_1$, 
or $p$ is obtained by concatenating a canonical turnpath~$p_1$ of $L_1$ 
with a canonical turnpath~$p_2$ of $L_2$. 
Then $p$ can be \emph{straightened}, that is, 
there exists a canonical turnpath of $L$ going from $a_{\to L}$ to $b_{L \to}$.
\end{proposition}

\begin{proof}
It suffices to check the various cases.
Recall the possible canonical turnpaths in a parallelogram from Figure~\ref{fig:tp-parallelogram}. 
Figure~\ref{fig:A} shows how to treat 
the four possible canonical turnpaths of $L_1$ going from a side of $L$ 
to a side of $L$. 
Note that only in two of these four cases, the turnpath has to be changed 
(by ``stretching'' or moving parallely). 
Figure~\ref{fig:B} shows how to treat 
the six possible cases of a concatenation~$p$ of a turnpath in $L_1$ 
with one in $L_2$. Only in two of the six cases, 
the turnpath has to be changed (by ``shrinking''). 
\qquad\end{proof}

\begin{figure}[h]
\begin{center}
    \subfigure[The modifications of the four canonical turnpaths of $L_1$
going from a side of $L$ to a side of~$L$.]
      {\hspace{2cm}%
\begin{tikzpicture}\draw[rhrhombidraw] (-110.856pt,0.0pt) -- (-106.237pt,8.0001pt) -- (-87.761pt,8.0001pt) -- (-92.38pt,0.0pt) -- cycle;\draw[rhrhombidraw] (-96.999pt,8.0001pt) -- (-101.618pt,0.0pt) ;\draw[thin,-my] (-101.618pt,8.0001pt) arc (0:-60:4.619pt) arc (120:180:4.619pt);\draw[decorate,decoration={snake,amplitude=.6mm,segment length=1.7mm, post length=.4mm},thin,-my] (-85.4515pt,4.0001pt) --  (-76.2135pt,4.0001pt);;\draw[rhrhombidraw] (-73.904pt,0.0pt) -- (-69.285pt,8.0001pt) -- (-50.809pt,8.0001pt) -- (-55.428pt,0.0pt) -- cycle;\draw[thin,-my] (-64.666pt,8.0001pt) arc (0:-60:4.619pt) arc (120:180:4.619pt);\draw[rhrhombidraw] (-110.856pt,-16.0002pt) -- (-106.237pt,-8.0001pt) -- (-87.761pt,-8.0001pt) -- (-92.38pt,-16.0002pt) -- cycle;\draw[rhrhombidraw] (-110.856pt,-32.0004pt) -- (-106.237pt,-24.0003pt) -- (-87.761pt,-24.0003pt) -- (-92.38pt,-32.0004pt) -- cycle;\draw[rhrhombidraw] (-110.856pt,-48.0006pt) -- (-106.237pt,-40.0005pt) -- (-87.761pt,-40.0005pt) -- (-92.38pt,-48.0006pt) -- cycle;\draw[rhrhombidraw] (-73.904pt,-16.0002pt) -- (-69.285pt,-8.0001pt) -- (-50.809pt,-8.0001pt) -- (-55.428pt,-16.0002pt) -- cycle;\draw[rhrhombidraw] (-73.904pt,-32.0004pt) -- (-69.285pt,-24.0003pt) -- (-50.809pt,-24.0003pt) -- (-55.428pt,-32.0004pt) -- cycle;\draw[rhrhombidraw] (-73.904pt,-48.0006pt) -- (-69.285pt,-40.0005pt) -- (-50.809pt,-40.0005pt) -- (-55.428pt,-48.0006pt) -- cycle;\draw[decorate,decoration={snake,amplitude=.6mm,segment length=1.7mm, post length=.4mm},thin,-my] (-85.4515pt,-12.0002pt) --  (-76.2135pt,-12.0002pt);;\draw[decorate,decoration={snake,amplitude=.6mm,segment length=1.7mm, post length=.4mm},thin,-my] (-85.4515pt,-28.0004pt) --  (-76.2135pt,-28.0004pt);;\draw[decorate,decoration={snake,amplitude=.6mm,segment length=1.7mm, post length=.4mm},thin,-my] (-85.4515pt,-44.0006pt) --  (-76.2135pt,-44.0006pt);;\draw[rhrhombidraw] (-96.999pt,-8.0001pt) -- (-101.618pt,-16.0002pt) ;\draw[rhrhombidraw] (-96.999pt,-24.0003pt) -- (-101.618pt,-32.0004pt) ;\draw[rhrhombidraw] (-96.999pt,-40.0005pt) -- (-101.618pt,-48.0006pt) ;\draw[thin,-my] (-108.5465pt,-12.0002pt) arc (60:0:4.619pt);\draw[thin,-my] (-108.5465pt,-28.0004pt) arc (240:300:4.619pt) arc (300:360:4.619pt);\draw[thin,-my] (-106.237pt,-48.0006pt) arc (180:120:4.619pt) arc (300:360:4.619pt);\draw[thin,-my] (-60.047pt,-48.0006pt) arc (180:120:4.619pt) arc (300:360:4.619pt);\draw[thin,-my] (-71.5945pt,-28.0004pt) arc (240:300:4.619pt) arc (120:60:4.619pt) arc (240:300:4.619pt) arc (300:360:4.619pt);\draw[thin,-my] (-71.5945pt,-12.0002pt) arc (60:0:4.619pt);\end{tikzpicture}%
      \hspace{2cm}
      \nopar\label{fig:A}} \hspace{1cm}
    \subfigure[The modifications of the six possible cases of a
concatenation~$p$ of a turnpath in $L_1$
                       with one in $L_2$.]
      {\hspace{2cm}%
\begin{tikzpicture}\draw[rhrhombidraw] (-110.856pt,0.0pt) -- (-106.237pt,8.0001pt) -- (-87.761pt,8.0001pt) -- (-92.38pt,0.0pt) -- cycle;\draw[rhrhombidraw] (-96.999pt,8.0001pt) -- (-101.618pt,0.0pt) ;\draw[decorate,decoration={snake,amplitude=.6mm,segment length=1.7mm, post length=.4mm},thin,-my] (-85.4515pt,4.0001pt) --  (-76.2135pt,4.0001pt);;\draw[rhrhombidraw] (-73.904pt,0.0pt) -- (-69.285pt,8.0001pt) -- (-50.809pt,8.0001pt) -- (-55.428pt,0.0pt) -- cycle;\draw[rhrhombidraw] (-110.856pt,-16.0002pt) -- (-106.237pt,-8.0001pt) -- (-87.761pt,-8.0001pt) -- (-92.38pt,-16.0002pt) -- cycle;\draw[rhrhombidraw] (-110.856pt,-32.0004pt) -- (-106.237pt,-24.0003pt) -- (-87.761pt,-24.0003pt) -- (-92.38pt,-32.0004pt) -- cycle;\draw[rhrhombidraw] (-110.856pt,-48.0006pt) -- (-106.237pt,-40.0005pt) -- (-87.761pt,-40.0005pt) -- (-92.38pt,-48.0006pt) -- cycle;\draw[rhrhombidraw] (-73.904pt,-16.0002pt) -- (-69.285pt,-8.0001pt) -- (-50.809pt,-8.0001pt) -- (-55.428pt,-16.0002pt) -- cycle;\draw[rhrhombidraw] (-73.904pt,-32.0004pt) -- (-69.285pt,-24.0003pt) -- (-50.809pt,-24.0003pt) -- (-55.428pt,-32.0004pt) -- cycle;\draw[rhrhombidraw] (-73.904pt,-48.0006pt) -- (-69.285pt,-40.0005pt) -- (-50.809pt,-40.0005pt) -- (-55.428pt,-48.0006pt) -- cycle;\draw[decorate,decoration={snake,amplitude=.6mm,segment length=1.7mm, post length=.4mm},thin,-my] (-85.4515pt,-12.0002pt) --  (-76.2135pt,-12.0002pt);;\draw[decorate,decoration={snake,amplitude=.6mm,segment length=1.7mm, post length=.4mm},thin,-my] (-85.4515pt,-28.0004pt) --  (-76.2135pt,-28.0004pt);;\draw[decorate,decoration={snake,amplitude=.6mm,segment length=1.7mm, post length=.4mm},thin,-my] (-85.4515pt,-44.0006pt) --  (-76.2135pt,-44.0006pt);;\draw[rhrhombidraw] (-96.999pt,-8.0001pt) -- (-101.618pt,-16.0002pt) ;\draw[rhrhombidraw] (-96.999pt,-24.0003pt) -- (-101.618pt,-32.0004pt) ;\draw[rhrhombidraw] (-96.999pt,-40.0005pt) -- (-101.618pt,-48.0006pt) ;\draw[rhrhombidraw] (-110.856pt,-64.0009pt) -- (-106.237pt,-56.0008pt) -- (-87.761pt,-56.0008pt) -- (-92.38pt,-64.0009pt) -- cycle;\draw[rhrhombidraw] (-110.856pt,-80.0011pt) -- (-106.237pt,-72.001pt) -- (-87.761pt,-72.001pt) -- (-92.38pt,-80.0011pt) -- cycle;\draw[rhrhombidraw] (-73.904pt,-64.0009pt) -- (-69.285pt,-56.0008pt) -- (-50.809pt,-56.0008pt) -- (-55.428pt,-64.0009pt) -- cycle;\draw[rhrhombidraw] (-73.904pt,-80.0011pt) -- (-69.285pt,-72.001pt) -- (-50.809pt,-72.001pt) -- (-55.428pt,-80.0011pt) -- cycle;\draw[rhrhombidraw] (-96.999pt,-56.0008pt) -- (-101.618pt,-64.0009pt) ;\draw[rhrhombidraw] (-96.999pt,-72.001pt) -- (-101.618pt,-80.0011pt) ;\draw[decorate,decoration={snake,amplitude=.6mm,segment length=1.7mm, post length=.4mm},thin,-my] (-85.4515pt,-60.0008pt) --  (-76.2135pt,-60.0008pt);;\draw[decorate,decoration={snake,amplitude=.6mm,segment length=1.7mm, post length=.4mm},thin,-my] (-85.4515pt,-76.001pt) --  (-76.2135pt,-76.001pt);;\draw[thin,-my] (-108.5465pt,4.0001pt) arc (240:300:4.619pt) arc (120:60:4.619pt);\draw[thin,-my] (-99.3085pt,4.0001pt) arc (60:0:4.619pt);\draw[thin,-my] (-71.5945pt,4.0001pt) arc (60:0:4.619pt);\draw[thin,-my] (-108.5465pt,-12.0002pt) arc (240:300:4.619pt) arc (120:60:4.619pt);\draw[thin,-my] (-99.3085pt,-12.0002pt) arc (240:300:4.619pt) arc (120:60:4.619pt);\draw[thin,-my] (-71.5945pt,-12.0002pt) arc (240:300:4.619pt) arc (120:60:4.619pt) arc (240:300:4.619pt) arc (120:60:4.619pt);\draw[thin,-my] (-108.5465pt,-28.0004pt) arc (240:300:4.619pt) arc (120:60:4.619pt);\draw[thin,-my] (-99.3085pt,-28.0004pt) arc (240:300:4.619pt) arc (300:360:4.619pt);\draw[thin,-my] (-71.5945pt,-28.0004pt) arc (240:300:4.619pt) arc (120:60:4.619pt) arc (240:300:4.619pt) arc (300:360:4.619pt);\draw[thin,-my] (-90.0705pt,-44.0006pt) arc (60:120:4.619pt) arc (300:240:4.619pt);\draw[thin,-my] (-99.3085pt,-44.0006pt) arc (240:180:4.619pt);\draw[thin,-my] (-53.1185pt,-44.0006pt) arc (240:180:4.619pt);\draw[thin,-my] (-90.0705pt,-60.0008pt) arc (60:120:4.619pt) arc (300:240:4.619pt);\draw[thin,-my] (-99.3085pt,-60.0008pt) arc (60:120:4.619pt) arc (300:240:4.619pt);\draw[thin,-my] (-53.1185pt,-60.0008pt) arc (60:120:4.619pt) arc (300:240:4.619pt) arc (60:120:4.619pt) arc (300:240:4.619pt);\draw[thin,-my] (-90.0705pt,-76.001pt) arc (60:120:4.619pt) arc (300:240:4.619pt);\draw[thin,-my] (-99.3085pt,-76.001pt) arc (60:120:4.619pt) arc (120:180:4.619pt);\draw[thin,-my] (-53.1185pt,-76.001pt) arc (60:120:4.619pt) arc (300:240:4.619pt) arc (60:120:4.619pt) arc (120:180:4.619pt);\end{tikzpicture}%
      \hspace{2cm}
      \nopar\label{fig:B}}
  \caption{Illustration of the proof of Proposition~\ref{pro:straight-parallelogram}.}
\end{center}
\end{figure}
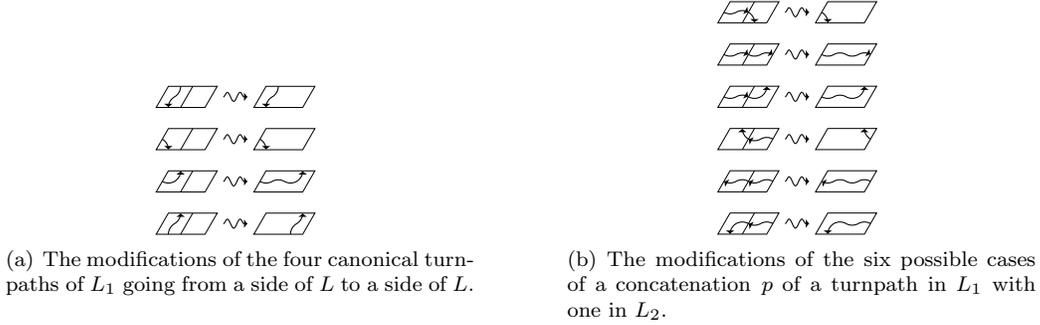


\begin{proposition}\label{pro:CTT-parallelogram}
The assertion of the Canonical Turnpath Theorem~\ref{thm:CTT} is true 
if $L$ is a parallelogram.
\end{proposition}

\begin{proof}
We proceed by induction on the size of~$L$. 
The induction start is provided by Lemma~\ref{le:CTT-rhombus}.
We suppose now that $L$ has size greater than two and we cut~$L$ 
into two parallelograms $L_1$ and~$L_2$ 
by a straight line parallel to one the sides of~$L$. 
The induction hypothesis yields the existence of 
multisets $\varphi_i$ of canonical turnpaths of~$L_i$ 
compatible with~$d$ restricted to~$L_i$, for~$i=1,2$. 

Let $a$ denote the side shared by $L_1$ and $L_2$
and note that $a_{L_1\to} = a_{\to L_2}$. 
The reader should note that it is not possible for a canonical 
turnpath in $L_i$ to start and end at~$a$, 
see Figure~\ref{fig:tp-parallelogram}.
Using Definition~\ref{def:enter-exit-L}, 
the induction hypothesis, 
and \eqref{eq:in=out}, we get 
$$
    \outom {L_1} a {\varphi_1} 
 = \outom {L_1} a d
 = \inom {L_2} a d 
 = \inom {L_2} a {\varphi_2}.
$$
This means that the number of turnpaths~$p_1$ in $\varphi_1$ 
ending at $a_{L_1\to}$ equals the number of turnpaths~$p_2$ in $\varphi_2$
starting at $a_{\to L_2}$. 
It is therefore possible to set up a bijection between the set of turnpaths~$p_1$ in $\varphi_1$ 
ending at $a_{L_1\to}$ 
with the set of turnpaths~$p_2$ in $\varphi_2$ starting at $a_{\to L_2}$, 
and to concatenate each~$p_1$ with its partner~$p_2$ to obtain a turnpath~$q$ in $L$ 
starting from a side of $L$ and ending at a side of $L$. 
However, the turnpath~$q$ may not be canonical for $L$. 
But now we use Proposition~\ref{pro:straight-parallelogram} 
to replace $q$ by a canonical turnpath of $L$ starting and ending 
at the same sides of $L$ as $q$ does. 

Similarly,  $a_{L_2\to} = a_{\to L_1}$ and we get 
$\outom {L_2} a {\varphi_2} = \inom {L_1} a {\varphi_1}$. 
As before, we can match and concatenate
the turnpaths~$p_2$ in $\varphi_2$ ending at $a_{L_2\to}$ 
with the turnpaths~$p_1$ in $\varphi_1$ starting at $a_{\to L_1}$. 
Again, we use Proposition~\ref{pro:straight-parallelogram} 
to replace the resulting turnpaths by canonical turnpaths of $L$ 
without changing the starting and ending side.

We also apply Proposition~\ref{pro:straight-parallelogram} 
to the turnpaths in $\varphi_1 $ and $\varphi_2$ going from 
a side of $L$ to a side of $L$. 

After performing these procedures, we obtain a multiset~$\varphi$ 
of canonical turnpaths of $L$.
Let~$b$ be the side of $L_1$ parallel to~$a$.
Then we have by construction
$$
\inom L {b} \varphi = \inom {L_1} {b} {\varphi_1} = 
\inom {L_1} {b} d = \inom L {b} d .
$$
Similarly for $b$ being the side of $L_2$ parallel to~$a$.
Now let~$b$ be a side of $L$ cut by $a$ into line segments $b_1$ and $b_2$.
Then we have
$$
\inom L {b} \varphi = \inom {L_1} {b_1} \varphi + \inom {L_2} {b_2} \varphi =
\inom {L_1} {b_1} d + \inom {L_2} {b_2} d = \inom L {b} d.
$$
It follows that $\varphi$ is compatible with~$d$.
\qquad\end{proof}

\subsection{Trapezoids, pentagons and hexagons}

We first treat the case of trapezoids. 
Again, the strategy is to proceed by induction, 
cutting the trapezoid into smaller trapezoids or parallelograms. 
But now, unlike the case of parallelograms before, 
the cutting has to be done in a certain way in order to ensure the
straightening of canonical turnpaths. 
The following result identifies the critical cases to be avoided. 
The straightforward proof is similar to the one of 
Proposition~\ref{pro:straight-parallelogram}
and left to the reader, who should consult 
Figures~\ref{fig:tp-parallelogram}--\ref{fig:tp-trapezoid}
for the possible canonical turnpaths in a parallelogram 
or a trapezoid, respectively. 

The {\em height} of a convex set~$L$ is defined as the number of its edges on its shortest side.

\begin{proposition}\label{pro:straight-trapezoid}
Let $L$ be a trapezoid cut into convex sets $L_1$ and $L_2$ by a straight line~$a$.
Further let $p$ be a turnpath going from the side~$b$ of~$L$ to 
the side~$c$ of~$L$ such that $p$~is either a canonical turnpath of $L_1$, 
a canonical turnpath of $L_2$, 
or $p$~is obtained by concatenating canonical turnpaths of $L_1$ 
with canonical turnpaths $L_2$ (in any order).  

1. If $a$ is parallel to the longest side of $L$ so that 
$L_1$ and $L_2$ are trapezoids, then $p$ can be straightened, i.e., 
there exists a canonical turnpath of $L$ going from $b_{\to L}$ to $c_{L \to}$.

2. Suppose that $L$ has the height~$1$ and that $L$ is cut by $a$ 
into a trapezoid (or a triangle) and a rhombus (there are two possibilities to do so).
Then $p$ can be straightened unless in the four critical cases 
depicted in Figure~\ref{fig:critical-trapezoid}(a)-(b).
\endproof
\end{proposition}

\begin{figure}[h]
\begin{center}
\subfigure[]{\nopar\label{fig:tra-uno}
\begin{tikzpicture}\draw[rhrhombidraw] (23.095pt,8.0001pt) -- (41.571pt,8.0001pt) -- (46.19pt,0.0pt) -- (27.714pt,0.0pt) ;\draw[rhrhombidraw] (32.333pt,8.0001pt) -- (36.952pt,0.0pt) ;\fill (24.9426pt,0.0pt) circle (0.4pt);\fill (23.095pt,0.0pt) circle (0.4pt);\fill (21.2474pt,0.0pt) circle (0.4pt);;\fill (20.3236pt,8.0001pt) circle (0.4pt);\fill (18.476pt,8.0001pt) circle (0.4pt);\fill (16.6284pt,8.0001pt) circle (0.4pt);;\draw[rhrhombidraw] (13.857pt,8.0001pt) -- (-4.619pt,8.0001pt) -- (-9.238pt,0.0pt) -- (18.476pt,0.0pt) ;\draw[rhrhombidraw] (23.095pt,-8.0001pt) -- (41.571pt,-8.0001pt) -- (46.19pt,-16.0002pt) -- (27.714pt,-16.0002pt) ;\draw[rhrhombidraw] (13.857pt,-8.0001pt) -- (-4.619pt,-8.0001pt) -- (-9.238pt,-16.0002pt) -- (18.476pt,-16.0002pt) ;\fill (20.3236pt,-8.0001pt) circle (0.4pt);\fill (18.476pt,-8.0001pt) circle (0.4pt);\fill (16.6284pt,-8.0001pt) circle (0.4pt);;\fill (24.9426pt,-16.0002pt) circle (0.4pt);\fill (23.095pt,-16.0002pt) circle (0.4pt);\fill (21.2474pt,-16.0002pt) circle (0.4pt);;\draw[rhrhombidraw] (0.0pt,-16.0002pt) -- (4.619pt,-8.0001pt) ;\draw[thin,-my] (36.952pt,8.0001pt) arc (180:240:4.619pt) arc (240:300:4.619pt);\draw[thin,-my] (-6.9285pt,-12.0002pt) arc (240:300:4.619pt) arc (300:360:4.619pt);\end{tikzpicture}
}
\hspace{0.5cm}
\subfigure[]{\nopar\label{fig:tra-due}
\begin{tikzpicture}\draw[rhrhombidraw] (23.095pt,8.0001pt) -- (41.571pt,8.0001pt) -- (46.19pt,0.0pt) -- (27.714pt,0.0pt) ;\draw[rhrhombidraw] (32.333pt,8.0001pt) -- (36.952pt,0.0pt) ;\fill (24.9426pt,0.0pt) circle (0.4pt);\fill (23.095pt,0.0pt) circle (0.4pt);\fill (21.2474pt,0.0pt) circle (0.4pt);;\fill (20.3236pt,8.0001pt) circle (0.4pt);\fill (18.476pt,8.0001pt) circle (0.4pt);\fill (16.6284pt,8.0001pt) circle (0.4pt);;\draw[rhrhombidraw] (13.857pt,8.0001pt) -- (-4.619pt,8.0001pt) -- (-9.238pt,0.0pt) -- (18.476pt,0.0pt) ;\draw[rhrhombidraw] (23.095pt,-8.0001pt) -- (41.571pt,-8.0001pt) -- (46.19pt,-16.0002pt) -- (27.714pt,-16.0002pt) ;\draw[rhrhombidraw] (13.857pt,-8.0001pt) -- (-4.619pt,-8.0001pt) -- (-9.238pt,-16.0002pt) -- (18.476pt,-16.0002pt) ;\fill (20.3236pt,-8.0001pt) circle (0.4pt);\fill (18.476pt,-8.0001pt) circle (0.4pt);\fill (16.6284pt,-8.0001pt) circle (0.4pt);;\fill (24.9426pt,-16.0002pt) circle (0.4pt);\fill (23.095pt,-16.0002pt) circle (0.4pt);\fill (21.2474pt,-16.0002pt) circle (0.4pt);;\draw[rhrhombidraw] (0.0pt,-16.0002pt) -- (4.619pt,-8.0001pt) ;\draw[thin,-my] (36.952pt,8.0001pt) arc (0:-60:4.619pt);\draw[thin,-my] (34.6425pt,4.0001pt) --  (-6.9285pt,4.0001pt);;\draw[thin,-my] (43.8805pt,-12.0002pt) --  (2.3095pt,-12.0002pt);;\draw[thin,-my] (2.3095pt,-12.0002pt) arc (240:180:4.619pt);\end{tikzpicture}
}
\hspace{0.5cm}
\subfigure[]{\nopar\label{fig:tra-tre}
\begin{tikzpicture}\draw[rhrhombidraw] (23.095pt,8.0001pt) -- (41.571pt,8.0001pt) -- (46.19pt,0.0pt) -- (27.714pt,0.0pt) ;\draw[rhrhombidraw] (32.333pt,8.0001pt) -- (36.952pt,0.0pt) ;\fill (24.9426pt,0.0pt) circle (0.4pt);\fill (23.095pt,0.0pt) circle (0.4pt);\fill (21.2474pt,0.0pt) circle (0.4pt);;\fill (20.3236pt,8.0001pt) circle (0.4pt);\fill (18.476pt,8.0001pt) circle (0.4pt);\fill (16.6284pt,8.0001pt) circle (0.4pt);;\draw[rhrhombidraw] (13.857pt,8.0001pt) -- (-4.619pt,8.0001pt) -- (-9.238pt,0.0pt) -- (18.476pt,0.0pt) ;\draw[rhrhombidraw] (23.095pt,-8.0001pt) -- (41.571pt,-8.0001pt) -- (46.19pt,-16.0002pt) -- (27.714pt,-16.0002pt) ;\draw[rhrhombidraw] (13.857pt,-8.0001pt) -- (-4.619pt,-8.0001pt) -- (-9.238pt,-16.0002pt) -- (18.476pt,-16.0002pt) ;\fill (20.3236pt,-8.0001pt) circle (0.4pt);\fill (18.476pt,-8.0001pt) circle (0.4pt);\fill (16.6284pt,-8.0001pt) circle (0.4pt);;\fill (24.9426pt,-16.0002pt) circle (0.4pt);\fill (23.095pt,-16.0002pt) circle (0.4pt);\fill (21.2474pt,-16.0002pt) circle (0.4pt);;\draw[rhrhombidraw] (0.0pt,-16.0002pt) -- (4.619pt,-8.0001pt) ;\draw[rhrhombithickside] (-4.619pt,8.0001pt) -- (4.619pt,8.0001pt);\draw[->,rhrhombiarrow] (0.0pt,12.6191pt) -- (0.0pt,3.3811pt);\draw[rhrhombithickside] (4.619pt,8.0001pt) -- (13.857pt,8.0001pt);\draw[->,rhrhombiarrow] (9.238pt,12.6191pt) -- (9.238pt,3.3811pt);\draw[rhrhombithickside] (23.095pt,8.0001pt) -- (32.333pt,8.0001pt);\draw[->,rhrhombiarrow] (27.714pt,12.6191pt) -- (27.714pt,3.3811pt);\draw[rhrhombithickside] (32.333pt,8.0001pt) -- (41.571pt,8.0001pt);\draw[->,rhrhombiarrow] (36.952pt,12.6191pt) -- (36.952pt,3.3811pt);\draw[rhrhombithickside] (41.571pt,-8.0001pt) -- (32.333pt,-8.0001pt);\draw[->,rhrhombiarrow] (36.952pt,-12.6191pt) -- (36.952pt,-3.3811pt);\draw[rhrhombithickside] (32.333pt,-8.0001pt) -- (23.095pt,-8.0001pt);\draw[->,rhrhombiarrow] (27.714pt,-12.6191pt) -- (27.714pt,-3.3811pt);\draw[rhrhombithickside] (13.857pt,-8.0001pt) -- (4.619pt,-8.0001pt);\draw[->,rhrhombiarrow] (9.238pt,-12.6191pt) -- (9.238pt,-3.3811pt);\draw[rhrhombithickside] (4.619pt,-8.0001pt) -- (-4.619pt,-8.0001pt);\draw[->,rhrhombiarrow] (0.0pt,-12.6191pt) -- (0.0pt,-3.3811pt);\end{tikzpicture}
}\caption{A trapezoid of height 1 cut into a trapezoid and a parallelogram 
               with the four critical cases of a  turnpath $p$ that cannot be straightened.}
    \nopar\label{fig:critical-trapezoid}
\end{center}
\end{figure}
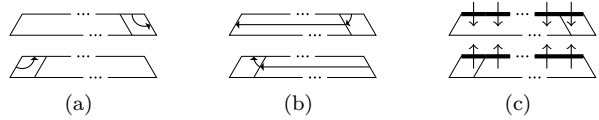

\begin{proposition}\label{pro:CTT-trapezoid}
The assertion of the Canonical Turnpath Theorem~\ref{thm:CTT} is true 
if $L$ is a trapezoid. 
\end{proposition}

\begin{proof}
We make induction on the size of~$L$.
The case where $L$ is a hive triangle, 
which we consider a degenerate trapezoid, is trivial. 
So suppose that $L$ has size at least three and let $d$ be an integral hive flow on $L$.

(A) If $L$ has height greater than 1, then we cut $L$ into two 
trapezoids $L_1$ and $L_2$ by a straight line 
parallel to the longest side of $L$ 
and apply the induction hypothesis to $L_1$ and $L_2$ to obtain multisets $\varphi_i$ 
of canonical turnpaths of~$L_i$ compatible with $d$ restricted to~$L_i$, 
for $i=1,2$.
Proceeding as in the proof of Proposition~\ref{pro:CTT-parallelogram} 
and using Proposition~\ref{pro:straight-trapezoid}(1),   
we construct from $\varphi_1,\varphi_2$ 
a multiset $\varphi$ of canonical turnpaths of $L$ 
satisfying 
$\inom L a \varphi=\inom L a d$ and $\outom L a \varphi=\outom L a d$ for all sides~$a$ of $L$.

(B) Now suppose that $L$ has height~$1$. There are two possibilities 
{\bf t} and {\bf b} of cutting~$L$ 
by a straight line into a trapezoid and a parallelogram, as depicted 
in the top and bottom row of Figure~\ref{fig:critical-trapezoid}. 

Choose the {\bf t} version of cutting and apply the induction hypothesis 
and Proposition~\ref{pro:CTT-parallelogram} to $L_1$ and $L_2$, 
respectively, 
to obtain multisets $\varphi_i$ 
of canonical turnpaths of~$L_i$. 
Then we apply the straightening of 
Proposition~\ref{pro:straight-trapezoid}(2) as before, 
which succeeds unless we are 
in one of the critical cases as depicted in the top row of 
Figure~\ref{fig:critical-trapezoid}.
For instance, assume that 
$L_2$ is a rhombus and the turnpath $p=$ 
\begin{tikzpicture}\draw[rhrhombidraw] (23.095pt,8.0001pt) -- (41.571pt,8.0001pt) -- (46.19pt,0.0pt) -- (27.714pt,0.0pt) ;\draw[rhrhombidraw] (32.333pt,8.0001pt) -- (36.952pt,0.0pt) ;\fill (24.9426pt,0.0pt) circle (0.4pt);\fill (23.095pt,0.0pt) circle (0.4pt);\fill (21.2474pt,0.0pt) circle (0.4pt);;\fill (20.3236pt,8.0001pt) circle (0.4pt);\fill (18.476pt,8.0001pt) circle (0.4pt);\fill (16.6284pt,8.0001pt) circle (0.4pt);;\draw[rhrhombidraw] (13.857pt,8.0001pt) -- (-4.619pt,8.0001pt) -- (-9.238pt,0.0pt) -- (18.476pt,0.0pt) ;\draw[thin,-my] (36.952pt,8.0001pt) arc (180:240:4.619pt) arc (240:300:4.619pt);\end{tikzpicture}
occurs in $\varphi_2$ as in Figure~\ref{fig:tra-uno}.
Let $k_r$ denote the edge of $L_2$ where $p$ starts and $a$ be the corresponding side of $L$. 
Then, using Definition~\ref{def:def-inflow}, we have 
$$
 \inom {L_2} {k_r} d = \inom {L_2} {k_r} {\varphi_2} \ge 1,
$$
hence $\inflow {L_2} {k_r} d >0$. 
Lemma~\ref{obs:borderentranceexit} implies that 
$\inflow L k d >0$ for all edges~$k$ of $a$, 
see Figure~\ref{fig:tra-tre}. 
The same conclusion can be drawn when 
$p=$
\begin{tikzpicture}\draw[rhrhombidraw] (23.095pt,8.0001pt) -- (41.571pt,8.0001pt) -- (46.19pt,0.0pt) -- (27.714pt,0.0pt) ;\draw[rhrhombidraw] (32.333pt,8.0001pt) -- (36.952pt,0.0pt) ;\fill (24.9426pt,0.0pt) circle (0.4pt);\fill (23.095pt,0.0pt) circle (0.4pt);\fill (21.2474pt,0.0pt) circle (0.4pt);;\fill (20.3236pt,8.0001pt) circle (0.4pt);\fill (18.476pt,8.0001pt) circle (0.4pt);\fill (16.6284pt,8.0001pt) circle (0.4pt);;\draw[rhrhombidraw] (13.857pt,8.0001pt) -- (-4.619pt,8.0001pt) -- (-9.238pt,0.0pt) -- (18.476pt,0.0pt) ;\draw[thin,-my] (36.952pt,8.0001pt) arc (0:-60:4.619pt);\draw[thin,-my] (34.6425pt,4.0001pt) --  (-6.9285pt,4.0001pt);;\end{tikzpicture} 
occurs in $\varphi_2$. 

The clue is now that if we cut $L$ in the other possible way 
({\bf b} version, cf.\ bottom row of Figure~\ref{fig:critical-trapezoid}), 
then no critical case can occur. Indeed, otherwise, by an analogous 
reasoning as before, we had 
$\outflow L k d  >0$ for all edges~$k$ of $a$,
which contradicts $\inflow L k d >0$. 

Similarly one shows that if we start with the {\bf b} version 
of cutting~$L$ and a critical case occurs, then cutting 
with the  {\bf t} version succeeds.
\qquad\end{proof}

For later use, we note the following observation 
resulting from the above proof.

\begin{observation}\label{obs:trap-rhomb}
Let $L$ be a trapezoid of height~$1$ and $d$ be an integral hive flow on~$L$. 
If the multiset $\varphi$ of canonical turnpaths compatible with $d$ 
resulting from the proof of Proposition~\ref{pro:CTT-trapezoid} contains
the turnpath $q=$ 
\begin{tikzpicture}\draw[rhrhombidraw] (-4.619pt,-8.0001pt) -- (-9.238pt,-16.0002pt) -- (46.19pt,-16.0002pt) -- (41.571pt,-8.0001pt) -- cycle;\draw[thin,-my] (41.571pt,-16.0002pt) arc (0:60:4.619pt) arc (60:120:4.619pt) arc (300:240:4.619pt) arc (60:120:4.619pt) arc (300:240:4.619pt) arc (60:120:4.619pt) arc (300:240:4.619pt) arc (60:120:4.619pt) arc (300:240:4.619pt) arc (60:120:4.619pt) arc (120:180:4.619pt);\end{tikzpicture},
then there is a rhombus $\varrho$ and 
a turnedge~$p\in\Psi_+(\varrho)$ as in Figure~\ref{fig:trap-neu}
such that $p\subseteq\SUPP(d)$. 
\end{observation}

\begin{figure}[h]
\begin{center}
\begin{tikzpicture}\fill[rhrhombifill] (18.476pt,0.0pt) -- (13.857pt,8.0001pt) -- (23.095pt,8.0001pt) -- (27.714pt,0.0pt) -- cycle;\draw[rhrhombidraw] (-4.619pt,8.0001pt) -- (-9.238pt,0.0pt) -- (46.19pt,0.0pt) -- (41.571pt,8.0001pt) -- cycle;\draw[thin,-my] (23.095pt,0.0pt) arc (0:60:4.619pt) arc (60:120:4.619pt);\end{tikzpicture}
\hspace{1cm}
\begin{tikzpicture}\fill[rhrhombifill] (23.095pt,8.0001pt) -- (18.476pt,0.0pt) -- (9.238pt,0.0pt) -- (13.857pt,8.0001pt) -- cycle;\draw[rhrhombidraw] (-4.619pt,8.0001pt) -- (-9.238pt,0.0pt) -- (46.19pt,0.0pt) -- (41.571pt,8.0001pt) -- cycle;\draw[thin,-my] (20.7855pt,4.0001pt) arc (60:120:4.619pt) arc (120:180:4.619pt);\end{tikzpicture}
    \caption{On Observation~\ref{obs:trap-rhomb}.} 
    \nopar\label{fig:trap-neu}
\end{center}
\end{figure}

\begin{proof}
Tracing part~(B) of the inductive proof of Proposition~\ref{pro:CTT-trapezoid} shows that~$q$
results from a smaller turnpath~$\tilde{q}$ either by stretching to the right or left, 
or by appending to $\tilde{q}$ a turnedge~$p$ in the right or left rhombus~$\varrho$, 
see Figure~\ref{fig:trap-rhomb}.
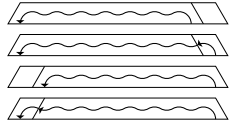
\begin{figure}[h]
\begin{center}
\begin{tikzpicture}\draw[rhrhombidraw] (-13.857pt,24.0003pt) -- (-9.238pt,32.0004pt) -- (64.666pt,32.0004pt) -- (69.285pt,24.0003pt) -- cycle;\draw[rhrhombidraw] (60.047pt,24.0003pt) -- (55.428pt,32.0004pt) ;\draw[thin,-my] (55.428pt,24.0003pt) arc (0:60:4.619pt) arc (60:120:4.619pt) arc (300:240:4.619pt) arc (60:120:4.619pt) arc (300:240:4.619pt) arc (60:120:4.619pt) arc (300:240:4.619pt) arc (60:120:4.619pt) arc (300:240:4.619pt) arc (60:120:4.619pt) arc (300:240:4.619pt) arc (60:120:4.619pt) arc (300:240:4.619pt) arc (60:120:4.619pt) arc (120:180:4.619pt);\end{tikzpicture}

\begin{tikzpicture}\draw[rhrhombidraw] (-13.857pt,8.0001pt) -- (-9.238pt,16.0002pt) -- (64.666pt,16.0002pt) -- (69.285pt,8.0001pt) -- cycle;\draw[rhrhombidraw] (60.047pt,8.0001pt) -- (55.428pt,16.0002pt) ;\draw[thin,-my] (64.666pt,8.0001pt) arc (0:60:4.619pt) arc (60:120:4.619pt);\draw[thin,-my] (57.7375pt,12.0002pt) arc (300:240:4.619pt) arc (60:120:4.619pt) arc (300:240:4.619pt) arc (60:120:4.619pt) arc (300:240:4.619pt) arc (60:120:4.619pt) arc (300:240:4.619pt) arc (60:120:4.619pt) arc (300:240:4.619pt) arc (60:120:4.619pt) arc (300:240:4.619pt) arc (60:120:4.619pt) arc (300:240:4.619pt) arc (60:120:4.619pt) arc (120:180:4.619pt);\end{tikzpicture}

\begin{tikzpicture}\draw[rhrhombidraw] (-13.857pt,-8.0001pt) -- (-9.238pt,0.0pt) -- (64.666pt,0.0pt) -- (69.285pt,-8.0001pt) -- cycle;\draw[rhrhombidraw] (0.0pt,0.0pt) -- (-4.619pt,-8.0001pt) ;\draw[thin,-my] (64.666pt,-8.0001pt) arc (0:60:4.619pt) arc (60:120:4.619pt) arc (300:240:4.619pt) arc (60:120:4.619pt) arc (300:240:4.619pt) arc (60:120:4.619pt) arc (300:240:4.619pt) arc (60:120:4.619pt) arc (300:240:4.619pt) arc (60:120:4.619pt) arc (300:240:4.619pt) arc (60:120:4.619pt) arc (300:240:4.619pt) arc (60:120:4.619pt) arc (120:180:4.619pt);\end{tikzpicture}

\begin{tikzpicture}\draw[rhrhombidraw] (-13.857pt,-24.0003pt) -- (-9.238pt,-16.0002pt) -- (64.666pt,-16.0002pt) -- (69.285pt,-24.0003pt) -- cycle;\draw[rhrhombidraw] (0.0pt,-16.0002pt) -- (-4.619pt,-24.0003pt) ;\draw[thin,-my] (64.666pt,-24.0003pt) arc (0:60:4.619pt) arc (60:120:4.619pt) arc (300:240:4.619pt) arc (60:120:4.619pt) arc (300:240:4.619pt) arc (60:120:4.619pt) arc (300:240:4.619pt) arc (60:120:4.619pt) arc (300:240:4.619pt) arc (60:120:4.619pt) arc (300:240:4.619pt) arc (60:120:4.619pt) arc (300:240:4.619pt) arc (60:120:4.619pt) arc (300:240:4.619pt);\draw[thin,-my] (-2.3095pt,-20.0003pt) arc (60:120:4.619pt) arc (120:180:4.619pt);\end{tikzpicture}

    \caption{Inductive construction of the turnpath~$q$ along the proof of Proposition~\ref{pro:CTT-trapezoid}.
                   In the first and third case stretching is necessary.}
    \nopar\label{fig:trap-rhomb}
\end{center}
\end{figure}
In the case~$p$ is appended, we know that 
$p\subseteq\SUPP(d)$ by Lemma~\ref{le:CTT-rhombus}.
Otherwise,  we conclude by the induction hypothesis.
\qquad\end{proof}

We settle now the case of pentagons. 

\begin{proposition}\label{pro:CTT-pentagon}
The assertion of the Canonical Turnpath Theorem~\ref{thm:CTT} is true 
if $L$ is a pentagon. 
\end{proposition}

\begin{proof}
We proceed by induction on the size of~$L$, cutting the pentagon by a straight line into a 
pentagon and a trapezoid, or a parallelogram and a trapezoid. 
The two critical cases, where a straightening fails, are depicted in Figure~\ref{fig:penta}.
\begin{figure}[h]
\begin{center}
\subfigure[]{\nopar\label{fig:nomergeresultpent:nomergresults}
\begin{tikzpicture}\draw[rhrhombidraw] (0.0pt,0.0pt) -- (27.714pt,0.0pt) -- (32.333pt,8.0001pt) -- (9.238pt,48.0006pt) -- (-9.238pt,16.0002pt) -- cycle;\fill (9.238pt,32.0004pt) circle (0.4pt);\fill (4.619pt,24.0003pt) circle (0.4pt);\fill (0.0pt,16.0002pt) circle (0.4pt);\fill (4.619pt,8.0001pt) circle (0.4pt);\fill (9.238pt,16.0002pt) circle (0.4pt);\fill (13.857pt,24.0003pt) circle (0.4pt);\fill (18.476pt,16.0002pt) circle (0.4pt);\fill (13.857pt,8.0001pt) circle (0.4pt);\fill (23.095pt,8.0001pt) circle (0.4pt);\draw[rhrhombidraw] (18.476pt,0.0pt) -- (27.714pt,16.0002pt) ;\draw[thin,-my] (-2.3095pt,4.0001pt) --  (20.7855pt,4.0001pt);;\draw[thin,-my] (20.7855pt,4.0001pt) arc (60:0:4.619pt);\end{tikzpicture}
\hspace{0.1cm}
\begin{tikzpicture}\draw[rhrhombidraw] (0.0pt,0.0pt) -- (27.714pt,0.0pt) -- (32.333pt,8.0001pt) -- (9.238pt,48.0006pt) -- (-9.238pt,16.0002pt) -- cycle;\fill (9.238pt,32.0004pt) circle (0.4pt);\fill (4.619pt,24.0003pt) circle (0.4pt);\fill (0.0pt,16.0002pt) circle (0.4pt);\fill (4.619pt,8.0001pt) circle (0.4pt);\fill (9.238pt,16.0002pt) circle (0.4pt);\fill (13.857pt,24.0003pt) circle (0.4pt);\fill (18.476pt,16.0002pt) circle (0.4pt);\fill (13.857pt,8.0001pt) circle (0.4pt);\fill (23.095pt,8.0001pt) circle (0.4pt);\draw[rhrhombidraw] (9.238pt,0.0pt) -- (-4.619pt,24.0003pt) ;\draw[thin,-my] (4.619pt,0.0pt) arc (180:120:4.619pt);\draw[thin,-my] (6.9285pt,4.0001pt) --  (30.0235pt,4.0001pt);;\end{tikzpicture}
}
\hspace{0.5cm}
\subfigure[]{\nopar\label{fig:nomergeresultpent:groundside}
\begin{tikzpicture}\draw[rhrhombidraw] (0.0pt,0.0pt) -- (27.714pt,0.0pt) -- (32.333pt,8.0001pt) -- (9.238pt,48.0006pt) -- (-9.238pt,16.0002pt) -- cycle;\fill (9.238pt,32.0004pt) circle (0.4pt);\fill (4.619pt,24.0003pt) circle (0.4pt);\fill (0.0pt,16.0002pt) circle (0.4pt);\fill (4.619pt,8.0001pt) circle (0.4pt);\fill (9.238pt,16.0002pt) circle (0.4pt);\fill (13.857pt,24.0003pt) circle (0.4pt);\fill (18.476pt,16.0002pt) circle (0.4pt);\fill (13.857pt,8.0001pt) circle (0.4pt);\fill (23.095pt,8.0001pt) circle (0.4pt);\draw[rhrhombidraw] (18.476pt,0.0pt) -- (27.714pt,16.0002pt) ;\draw[rhrhombithickside] (18.476pt,0.0pt) -- (27.714pt,0.0pt);\draw[->,rhrhombiarrow] (23.095pt,4.619pt) -- (23.095pt,-4.619pt);\draw[rhrhombithickside] (9.238pt,0.0pt) -- (18.476pt,0.0pt);\draw[->,rhrhombiarrow] (13.857pt,4.619pt) -- (13.857pt,-4.619pt);\draw[rhrhombithickside] (0.0pt,0.0pt) -- (9.238pt,0.0pt);\draw[->,rhrhombiarrow] (4.619pt,4.619pt) -- (4.619pt,-4.619pt);\end{tikzpicture}
\hspace{0.1cm}
\begin{tikzpicture}\draw[rhrhombidraw] (0.0pt,0.0pt) -- (27.714pt,0.0pt) -- (32.333pt,8.0001pt) -- (9.238pt,48.0006pt) -- (-9.238pt,16.0002pt) -- cycle;\fill (9.238pt,32.0004pt) circle (0.4pt);\fill (4.619pt,24.0003pt) circle (0.4pt);\fill (0.0pt,16.0002pt) circle (0.4pt);\fill (4.619pt,8.0001pt) circle (0.4pt);\fill (9.238pt,16.0002pt) circle (0.4pt);\fill (13.857pt,24.0003pt) circle (0.4pt);\fill (18.476pt,16.0002pt) circle (0.4pt);\fill (13.857pt,8.0001pt) circle (0.4pt);\fill (23.095pt,8.0001pt) circle (0.4pt);\draw[rhrhombithickside] (27.714pt,0.0pt) -- (18.476pt,0.0pt);\draw[->,rhrhombiarrow] (23.095pt,-4.619pt) -- (23.095pt,4.619pt);\draw[rhrhombithickside] (9.238pt,0.0pt) -- (0.0pt,0.0pt);\draw[->,rhrhombiarrow] (4.619pt,-4.619pt) -- (4.619pt,4.619pt);\draw[rhrhombithickside] (18.476pt,0.0pt) -- (9.238pt,0.0pt);\draw[->,rhrhombiarrow] (13.857pt,-4.619pt) -- (13.857pt,4.619pt);\draw[rhrhombidraw] (-4.619pt,24.0003pt) -- (9.238pt,0.0pt) ;\end{tikzpicture}
}
    \caption{Two ways of cutting a pentagon with the two critical cases of turnpaths 
                   that cannot be straightened and the resulting throughputs.}
    \nopar\label{fig:penta}
\end{center}
\end{figure}
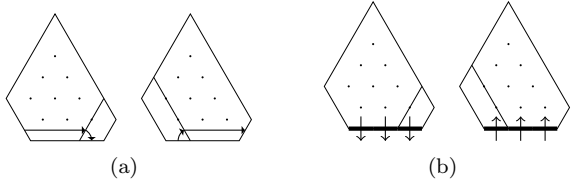
As in the proof of Proposition~\ref{pro:CTT-trapezoid}, one can show 
that in one of the two possible ways of cutting~$L$, no critical 
case can occur. The details are left to the reader.
\qquad\end{proof}

The case where $L$ is a hexagon is the simplest one of all 
and does not require an inductive argument as before.

\begin{proposition}\label{pro:CTT-hexagon}
The assertion of the Canonical Turnpath Theorem~\ref{thm:CTT} is true 
if $L$ is a hexagon.
\end{proposition}

\begin{proof}
Let $a_1,\ldots,a_6$ denote the six sides of a hexagon~$L$.
We write $a^-_i := (a_i)_{\to L}$ and $a^+_i := (a_i)_{L\to}$ 
for the entrance and exit edges of $L$, respectively. 

The essential observation is that for any pair $(i,j)$
there exists a canonical turnpath of $L$ going from $a^-_i$ to $a^+_j$. 
This is easily verified by looking at Figure~\ref{fig:tp-hexagon}.

Let $d$ be an integral flow on $L$ and put 
$\iflow(i) := \inom L {a_i} d$ and 
$\oflow(i) := \outom L {a_i} d$ 
for $1\le i\le 6$. 
The flow conservation laws imply that 
$\sum_i \iflow(i) = \sum_i \oflow(i)$. 

We form a list ${\mathcal L}^{-}$ of entrance edges 
in which $a^-_i$ occurs $\iflow(i)$ many times 
and we form a list ${\mathcal L}^{+}$of exit edges 
in which $a^+_i$ occurs $\oflow(i)$ many times.
Both lists have the same length.
We now connect, for all~$j$, the $j$th element of ${\mathcal L}^{-}$ with 
the $j$th element of ${\mathcal L}^{+}$ by a canonical turnpath~$p_j$ of $L$. 
This is possible by the observation made at the beginning of the proof.
The resulting multiset~$\varphi$ of canonical turnpaths of $L$ satisfies 
$\inom L {a_i} \varphi = \inom L {a_i} d$ and $\outom L {a_i} \varphi = \outom L {a_i} d$
by construction.
\qquad\end{proof}

\subsection{Triangles}

We need the following flow propagation lemma. 

\begin{lemma}\label{le:flow-ext}
Let $L$ be a trapezoid and $d$ be a hive flow on~$L$.

1. Let $p$ be the path in Figure~\ref{fig:flow-ext}(a)
and suppose that $p\subseteq \SUPP(d)$.
Then all the edges  of $G$ belonging to the 
turns in Figure~\ref{fig:flow-ext}(b) belong to $\SUPP(d)$ as well.
Moreover, $\SUPP(d)$ cannot contain the paths~$q_1,q_2\in\Psi_+(\varrho)$ 
in the shaded rhombi $\varrho$  
depicted in Figure~\ref{fig:flow-ext}(c). 

2. If the path~$\tilde{p}$ in Figure~\ref{fig:flow-ext}(a') 
satisfies $\tilde{p}\subseteq\SUPP(d)$, then 
a similar conclusion can be drawn, see 
Figures~\ref{fig:flow-ext}(b')-(c'). 
\end{lemma}

\begin{proof}
1. The first assertion on $\SUPP(d)$ 
follows by successively applying Lemma~\ref{lem:flowpropagation} 
on flow propagation.

The assertion $q_i\not\subseteq\SUPP(d)$ 
follows by inspecting the edges of $G$ involved in the paths appearing 
in the bottom row of the trapezoid in Figure~\ref{fig:flow-ext}(c) 
and noting that $\SUPP(d)$ cannot contain an edge $k\in E(G)$ and 
its reverse. 

2. The second case is treated similarly.
\qquad\end{proof}

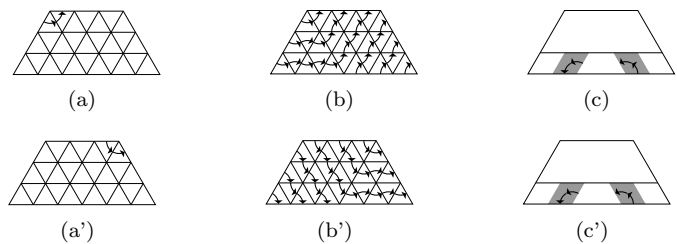
\begin{figure}[h]
\begin{center}
\subfigure[]
{
\begin{tikzpicture}\draw[rhrhombidraw] (0.0pt,0.0pt) -- (13.857pt,24.0003pt) -- (41.571pt,24.0003pt) -- (55.428pt,0.0pt) -- cycle;\draw[rhrhombidraw] (9.238pt,16.0002pt) -- (46.19pt,16.0002pt) ;\draw[rhrhombidraw] (4.619pt,8.0001pt) -- (50.809pt,8.0001pt) ;\draw[rhrhombidraw] (23.095pt,24.0003pt) -- (9.238pt,0.0pt) ;\draw[rhrhombidraw] (32.333pt,24.0003pt) -- (18.476pt,0.0pt) ;\draw[rhrhombidraw] (41.571pt,24.0003pt) -- (27.714pt,0.0pt) ;\draw[rhrhombidraw] (46.19pt,16.0002pt) -- (36.952pt,0.0pt) ;\draw[rhrhombidraw] (50.809pt,8.0001pt) -- (46.19pt,0.0pt) ;\draw[rhrhombidraw] (46.19pt,0.0pt) -- (32.333pt,24.0003pt) ;\draw[rhrhombidraw] (36.952pt,0.0pt) -- (23.095pt,24.0003pt) ;\draw[rhrhombidraw] (27.714pt,0.0pt) -- (13.857pt,24.0003pt) ;\draw[rhrhombidraw] (18.476pt,0.0pt) -- (9.238pt,16.0002pt) ;\draw[rhrhombidraw] (9.238pt,0.0pt) -- (4.619pt,8.0001pt) ;\draw[thin,-my] (11.5475pt,20.0003pt) arc (240:300:4.619pt);\draw[thin,-my] (16.1665pt,20.0003pt) arc (300:360:4.619pt);\end{tikzpicture}
}
\hspace{1cm}
\subfigure[]
{
\begin{tikzpicture}\draw[rhrhombidraw] (0.0pt,0.0pt) -- (13.857pt,24.0003pt) -- (41.571pt,24.0003pt) -- (55.428pt,0.0pt) -- cycle;\draw[rhrhombidraw] (9.238pt,16.0002pt) -- (46.19pt,16.0002pt) ;\draw[rhrhombidraw] (4.619pt,8.0001pt) -- (50.809pt,8.0001pt) ;\draw[rhrhombidraw] (23.095pt,24.0003pt) -- (9.238pt,0.0pt) ;\draw[rhrhombidraw] (32.333pt,24.0003pt) -- (18.476pt,0.0pt) ;\draw[rhrhombidraw] (41.571pt,24.0003pt) -- (27.714pt,0.0pt) ;\draw[rhrhombidraw] (46.19pt,16.0002pt) -- (36.952pt,0.0pt) ;\draw[rhrhombidraw] (50.809pt,8.0001pt) -- (46.19pt,0.0pt) ;\draw[rhrhombidraw] (46.19pt,0.0pt) -- (32.333pt,24.0003pt) ;\draw[rhrhombidraw] (36.952pt,0.0pt) -- (23.095pt,24.0003pt) ;\draw[rhrhombidraw] (27.714pt,0.0pt) -- (13.857pt,24.0003pt) ;\draw[rhrhombidraw] (18.476pt,0.0pt) -- (9.238pt,16.0002pt) ;\draw[rhrhombidraw] (9.238pt,0.0pt) -- (4.619pt,8.0001pt) ;\draw[thin,-my] (11.5475pt,20.0003pt) arc (240:300:4.619pt);\draw[thin,-my] (16.1665pt,20.0003pt) arc (300:360:4.619pt);\draw[thin,-my] (6.9285pt,12.0002pt) arc (240:300:4.619pt);\draw[thin,-my] (11.5475pt,12.0002pt) arc (120:60:4.619pt);\draw[thin,-my] (16.1665pt,12.0002pt) arc (240:300:4.619pt);\draw[thin,-my] (20.7855pt,12.0002pt) arc (300:360:4.619pt);\draw[thin,-my] (23.095pt,16.0002pt) arc (180:120:4.619pt);\draw[thin,-my] (25.4045pt,20.0003pt) arc (300:360:4.619pt);\draw[thin,-my] (2.3095pt,4.0001pt) arc (240:300:4.619pt);\draw[thin,-my] (6.9285pt,4.0001pt) arc (120:60:4.619pt);\draw[thin,-my] (11.5475pt,4.0001pt) arc (240:300:4.619pt);\draw[thin,-my] (16.1665pt,4.0001pt) arc (120:60:4.619pt);\draw[thin,-my] (20.7855pt,4.0001pt) arc (240:300:4.619pt);\draw[thin,-my] (25.4045pt,4.0001pt) arc (300:360:4.619pt);\draw[thin,-my] (27.714pt,8.0001pt) arc (180:120:4.619pt);\draw[thin,-my] (30.0235pt,12.0002pt) arc (300:360:4.619pt);\draw[thin,-my] (32.333pt,16.0002pt) arc (180:120:4.619pt);\draw[thin,-my] (34.6425pt,20.0003pt) arc (300:360:4.619pt);\draw[thin,-my] (32.333pt,0.0pt) arc (180:120:4.619pt);\draw[thin,-my] (34.6425pt,4.0001pt) arc (300:360:4.619pt);\draw[thin,-my] (36.952pt,8.0001pt) arc (180:120:4.619pt);\draw[thin,-my] (39.2615pt,12.0002pt) arc (300:360:4.619pt);\draw[thin,-my] (41.571pt,16.0002pt) arc (180:120:4.619pt);\draw[thin,-my] (41.571pt,0.0pt) arc (180:120:4.619pt);\draw[thin,-my] (43.8805pt,4.0001pt) arc (300:360:4.619pt);\draw[thin,-my] (46.19pt,8.0001pt) arc (180:120:4.619pt);\draw[thin,-my] (50.809pt,0.0pt) arc (180:120:4.619pt);\end{tikzpicture}
}
\hspace{1cm}
\subfigure[]
{
\begin{tikzpicture}\fill[rhrhombifill] (23.095pt,8.0001pt) -- (13.857pt,8.0001pt) -- (9.238pt,0.0pt) -- (18.476pt,0.0pt) -- cycle;\fill[rhrhombifill] (32.333pt,8.0001pt) -- (41.571pt,8.0001pt) -- (46.19pt,0.0pt) -- (36.952pt,0.0pt) -- cycle;\draw[rhrhombidraw] (0.0pt,0.0pt) -- (13.857pt,24.0003pt) -- (41.571pt,24.0003pt) -- (55.428pt,0.0pt) -- cycle;\draw[rhrhombidraw] (4.619pt,8.0001pt) -- (50.809pt,8.0001pt) ;\draw[thin,-my] (41.571pt,0.0pt) arc (0:60:4.619pt);\draw[thin,-my] (20.7855pt,4.0001pt) arc (60:120:4.619pt);\draw[thin,-my] (39.2615pt,4.0001pt) arc (60:120:4.619pt);\draw[thin,-my] (16.1665pt,4.0001pt) arc (120:180:4.619pt);\end{tikzpicture}
}

\setcounter{subfigure}{0}
\renewcommand*{\thesubfigure}{(\alph{subfigure}')}
\subfigure[]
{
\begin{tikzpicture}\draw[rhrhombidraw] (0.0pt,0.0pt) -- (13.857pt,24.0003pt) -- (41.571pt,24.0003pt) -- (55.428pt,0.0pt) -- cycle;\draw[rhrhombidraw] (9.238pt,16.0002pt) -- (46.19pt,16.0002pt) ;\draw[rhrhombidraw] (4.619pt,8.0001pt) -- (50.809pt,8.0001pt) ;\draw[rhrhombidraw] (23.095pt,24.0003pt) -- (9.238pt,0.0pt) ;\draw[rhrhombidraw] (32.333pt,24.0003pt) -- (18.476pt,0.0pt) ;\draw[rhrhombidraw] (41.571pt,24.0003pt) -- (27.714pt,0.0pt) ;\draw[rhrhombidraw] (46.19pt,16.0002pt) -- (36.952pt,0.0pt) ;\draw[rhrhombidraw] (50.809pt,8.0001pt) -- (46.19pt,0.0pt) ;\draw[rhrhombidraw] (46.19pt,0.0pt) -- (32.333pt,24.0003pt) ;\draw[rhrhombidraw] (36.952pt,0.0pt) -- (23.095pt,24.0003pt) ;\draw[rhrhombidraw] (27.714pt,0.0pt) -- (13.857pt,24.0003pt) ;\draw[rhrhombidraw] (18.476pt,0.0pt) -- (9.238pt,16.0002pt) ;\draw[rhrhombidraw] (9.238pt,0.0pt) -- (4.619pt,8.0001pt) ;\draw[thin,-my] (36.952pt,24.0003pt) arc (180:240:4.619pt);\draw[thin,-my] (39.2615pt,20.0003pt) arc (240:300:4.619pt);\end{tikzpicture}
}
\hspace{1cm}
\subfigure[]
{
\begin{tikzpicture}\draw[rhrhombidraw] (0.0pt,0.0pt) -- (13.857pt,24.0003pt) -- (41.571pt,24.0003pt) -- (55.428pt,0.0pt) -- cycle;\draw[rhrhombidraw] (9.238pt,16.0002pt) -- (46.19pt,16.0002pt) ;\draw[rhrhombidraw] (4.619pt,8.0001pt) -- (50.809pt,8.0001pt) ;\draw[rhrhombidraw] (23.095pt,24.0003pt) -- (9.238pt,0.0pt) ;\draw[rhrhombidraw] (32.333pt,24.0003pt) -- (18.476pt,0.0pt) ;\draw[rhrhombidraw] (41.571pt,24.0003pt) -- (27.714pt,0.0pt) ;\draw[rhrhombidraw] (46.19pt,16.0002pt) -- (36.952pt,0.0pt) ;\draw[rhrhombidraw] (50.809pt,8.0001pt) -- (46.19pt,0.0pt) ;\draw[rhrhombidraw] (46.19pt,0.0pt) -- (32.333pt,24.0003pt) ;\draw[rhrhombidraw] (36.952pt,0.0pt) -- (23.095pt,24.0003pt) ;\draw[rhrhombidraw] (27.714pt,0.0pt) -- (13.857pt,24.0003pt) ;\draw[rhrhombidraw] (18.476pt,0.0pt) -- (9.238pt,16.0002pt) ;\draw[rhrhombidraw] (9.238pt,0.0pt) -- (4.619pt,8.0001pt) ;\draw[thin,-my] (36.952pt,24.0003pt) arc (180:240:4.619pt);\draw[thin,-my] (39.2615pt,20.0003pt) arc (240:300:4.619pt);\draw[thin,-my] (27.714pt,24.0003pt) arc (180:240:4.619pt);\draw[thin,-my] (30.0235pt,20.0003pt) arc (60:0:4.619pt);\draw[thin,-my] (32.333pt,16.0002pt) arc (180:240:4.619pt);\draw[thin,-my] (34.6425pt,12.0002pt) arc (240:300:4.619pt);\draw[thin,-my] (39.2615pt,12.0002pt) arc (120:60:4.619pt);\draw[thin,-my] (43.8805pt,12.0002pt) arc (240:300:4.619pt);\draw[thin,-my] (18.476pt,24.0003pt) arc (180:240:4.619pt);\draw[thin,-my] (20.7855pt,20.0003pt) arc (60:0:4.619pt);\draw[thin,-my] (23.095pt,16.0002pt) arc (180:240:4.619pt);\draw[thin,-my] (25.4045pt,12.0002pt) arc (60:0:4.619pt);\draw[thin,-my] (27.714pt,8.0001pt) arc (180:240:4.619pt);\draw[thin,-my] (30.0235pt,4.0001pt) arc (240:300:4.619pt);\draw[thin,-my] (34.6425pt,4.0001pt) arc (120:60:4.619pt);\draw[thin,-my] (39.2615pt,4.0001pt) arc (240:300:4.619pt);\draw[thin,-my] (43.8805pt,4.0001pt) arc (120:60:4.619pt);\draw[thin,-my] (48.4995pt,4.0001pt) arc (240:300:4.619pt);\draw[thin,-my] (11.5475pt,20.0003pt) arc (60:0:4.619pt);\draw[thin,-my] (13.857pt,16.0002pt) arc (180:240:4.619pt);\draw[thin,-my] (16.1665pt,12.0002pt) arc (60:0:4.619pt);\draw[thin,-my] (18.476pt,8.0001pt) arc (180:240:4.619pt);\draw[thin,-my] (20.7855pt,4.0001pt) arc (60:0:4.619pt);\draw[thin,-my] (6.9285pt,12.0002pt) arc (60:0:4.619pt);\draw[thin,-my] (9.238pt,8.0001pt) arc (180:240:4.619pt);\draw[thin,-my] (11.5475pt,4.0001pt) arc (60:0:4.619pt);\draw[thin,-my] (2.3095pt,4.0001pt) arc (60:0:4.619pt);\end{tikzpicture}
}
\hspace{1cm}
\subfigure[]
{
\begin{tikzpicture}\fill[rhrhombifill] (23.095pt,8.0001pt) -- (13.857pt,8.0001pt) -- (9.238pt,0.0pt) -- (18.476pt,0.0pt) -- cycle;\fill[rhrhombifill] (32.333pt,8.0001pt) -- (41.571pt,8.0001pt) -- (46.19pt,0.0pt) -- (36.952pt,0.0pt) -- cycle;\draw[rhrhombidraw] (0.0pt,0.0pt) -- (13.857pt,24.0003pt) -- (41.571pt,24.0003pt) -- (55.428pt,0.0pt) -- cycle;\draw[rhrhombidraw] (4.619pt,8.0001pt) -- (50.809pt,8.0001pt) ;\draw[thin,-my] (41.571pt,0.0pt) arc (0:60:4.619pt);\draw[thin,-my] (20.7855pt,4.0001pt) arc (60:120:4.619pt);\draw[thin,-my] (39.2615pt,4.0001pt) arc (60:120:4.619pt);\draw[thin,-my] (16.1665pt,4.0001pt) arc (120:180:4.619pt);\end{tikzpicture}
}
\renewcommand*{\thesubfigure}{(\alph{subfigure})}
    \caption{If $d$ is a hive flow and the path on the left figure is contained in $\SUPP(d)$, 
                   then the turns in the middle figure are contained in $\SUPP(d)$. 
                   The two paths on the right figure cannot be contained in $\SUPP(d)$.}
    \nopar\label{fig:flow-ext}
\end{center}
\end{figure}

The following proposition completes the proof of the 
Canonical Turnpath Theorem~\ref{thm:CTT} 
for any shapes of $L$. 

\begin{proposition}\label{pro:CTT-triangle}
The assertion of the Canonical Turnpath Theorem~\ref{thm:CTT} is true 
if $L$ is a triangle.
\end{proposition}

\begin{proof}
Again we proceed by induction on the size of $L$, the start of a hive triangle being trivial. 
For the induction step, suppose that $d$ is an integral hive flow on $L$, 
and note that there are three ways of cutting~$L$ into a trapezoid~$L_1$ of height~$1$ and a  triangle~$L_2$.
We choose one as in Figure~\ref{fig:triangle}.
The induction hypothesis and Proposition~\ref{pro:CTT-trapezoid} 
yield multisets~$\varphi_i$ compatible with $d$ restricted to $L_1$ and~$L_2$, 
respectively. 
Using Figure~\ref{fig:tp-triangle} and Figure~\ref{fig:tp-trapezoid} 
showing the possible canonical turnpaths in triangles and trapezoids, 
the reader should verify that the 
procedure of concatenation and straightening, 
as explained in the proof of Proposition~\ref{pro:CTT-parallelogram},  
can only fail in the critical case where 
$\varphi_1$ contains a turnpath~$q$ as depicted in Figure~\ref{fig:triangle}. 

By Observation~\ref{obs:trap-rhomb} applied to the trapezoid $L_1$ 
we may assume that there is rhombus~$\varrho$ 
(shaded in Figure~\ref{fig:tria-flow})
and a path~$p\in\Psi_+(\varrho)$ 
such that $p\subseteq\SUPP(d)$. 
Now we can apply Lemma~\ref{le:flow-ext} as depicted in 
Figure~\ref{fig:tria-flow} and conclude that  
all the turns depicted in this figure are contained in $\SUPP(d)$. 
\begin{figure}[h]
\begin{center}
     \subfigure[Cutting a triangle into a trapezoid and a triangle with the only critical case of a turnpath~$q$ that cannot be straightened.]
       {\qquad\qquad
        \begin{tikzpicture}\draw[rhrhombidraw] (64.666pt,-16.0002pt) -- (-9.238pt,-16.0002pt) ;\draw[rhrhombidraw] (27.714pt,48.0006pt) -- (-13.857pt,-24.0003pt) -- (69.285pt,-24.0003pt) -- cycle;\draw[thin,-my] (64.666pt,-24.0003pt) arc (0:60:4.619pt) arc (60:120:4.619pt) arc (300:240:4.619pt) arc (60:120:4.619pt) arc (300:240:4.619pt) arc (60:120:4.619pt) arc (300:240:4.619pt) arc (60:120:4.619pt) arc (300:240:4.619pt) arc (60:120:4.619pt) arc (300:240:4.619pt) arc (60:120:4.619pt) arc (300:240:4.619pt) arc (60:120:4.619pt) arc (300:240:4.619pt) arc (60:120:4.619pt) arc (120:180:4.619pt);\end{tikzpicture}
        \nopar\label{fig:triangle}\qquad\qquad} \hspace{0.7cm}
       \subfigure[The turnpath in the shaded rhombus is contained in $\SUPP(d)$, 
                  implying flow propagation according to Lemma~\ref{le:flow-ext}.]
       {
\begin{tikzpicture}\fill[rhrhombifill] (32.333pt,-24.0003pt) -- (27.714pt,-16.0002pt) -- (18.476pt,-16.0002pt) -- (23.095pt,-24.0003pt) -- cycle;\draw[rhrhombidraw] (64.666pt,-16.0002pt) -- (-9.238pt,-16.0002pt) ;\draw[rhrhombidraw] (27.714pt,48.0006pt) -- (-13.857pt,-24.0003pt) -- (69.285pt,-24.0003pt) -- cycle;\draw[thin,-my] (27.714pt,-24.0003pt) arc (0:60:4.619pt) arc (60:120:4.619pt);\draw[thin,-my] (36.952pt,-24.0003pt) arc (0:60:4.619pt) arc (240:180:4.619pt) arc (0:60:4.619pt) arc (60:120:4.619pt) arc (300:240:4.619pt) arc (60:120:4.619pt);\draw[thin,-my] (46.19pt,-24.0003pt) arc (0:60:4.619pt) arc (240:180:4.619pt) arc (0:60:4.619pt) arc (240:180:4.619pt) arc (0:60:4.619pt) arc (60:120:4.619pt) arc (300:240:4.619pt) arc (60:120:4.619pt) arc (300:240:4.619pt) arc (60:120:4.619pt);\draw[thin,-my] (55.428pt,-24.0003pt) arc (0:60:4.619pt) arc (240:180:4.619pt) arc (0:60:4.619pt) arc (240:180:4.619pt) arc (0:60:4.619pt) arc (240:180:4.619pt) arc (0:60:4.619pt) arc (60:120:4.619pt) arc (300:240:4.619pt) arc (60:120:4.619pt) arc (300:240:4.619pt) arc (60:120:4.619pt) arc (300:240:4.619pt) arc (60:120:4.619pt);\draw[thin,-my] (64.666pt,-24.0003pt) arc (0:60:4.619pt) arc (240:180:4.619pt) arc (0:60:4.619pt) arc (240:180:4.619pt) arc (0:60:4.619pt) arc (240:180:4.619pt) arc (0:60:4.619pt) arc (240:180:4.619pt) arc (0:60:4.619pt) arc (60:120:4.619pt) arc (300:240:4.619pt) arc (60:120:4.619pt) arc (300:240:4.619pt) arc (60:120:4.619pt) arc (300:240:4.619pt) arc (60:120:4.619pt) arc (300:240:4.619pt);\draw[thin,-my] (43.8805pt,20.0003pt) arc (300:240:4.619pt) arc (60:120:4.619pt) arc (300:240:4.619pt) arc (60:120:4.619pt) arc (300:240:4.619pt) arc (60:120:4.619pt) arc (300:240:4.619pt);\draw[thin,-my] (39.2615pt,28.0004pt) arc (300:240:4.619pt) arc (60:120:4.619pt) arc (300:240:4.619pt) arc (60:120:4.619pt) arc (300:240:4.619pt);\draw[thin,-my] (34.6425pt,36.0005pt) arc (300:240:4.619pt) arc (60:120:4.619pt) arc (300:240:4.619pt);\draw[thin,-my] (30.0235pt,44.0006pt) arc (300:240:4.619pt);\end{tikzpicture}
\hspace{0.5cm}
\begin{tikzpicture}\fill[rhrhombifill] (23.095pt,-24.0003pt) -- (32.333pt,-24.0003pt) -- (36.952pt,-16.0002pt) -- (27.714pt,-16.0002pt) -- cycle;\draw[rhrhombidraw] (64.666pt,-16.0002pt) -- (-9.238pt,-16.0002pt) ;\draw[rhrhombidraw] (27.714pt,48.0006pt) -- (-13.857pt,-24.0003pt) -- (69.285pt,-24.0003pt) -- cycle;\draw[thin,-my] (34.6425pt,-20.0003pt) arc (60:120:4.619pt) arc (120:180:4.619pt);\draw[thin,-my] (39.2615pt,-12.0002pt) arc (60:120:4.619pt) arc (300:240:4.619pt) arc (60:120:4.619pt) arc (120:180:4.619pt) arc (0:-60:4.619pt) arc (120:180:4.619pt);\draw[thin,-my] (43.8805pt,-4.0001pt) arc (60:120:4.619pt) arc (300:240:4.619pt) arc (60:120:4.619pt) arc (300:240:4.619pt) arc (60:120:4.619pt) arc (120:180:4.619pt) arc (0:-60:4.619pt) arc (120:180:4.619pt) arc (0:-60:4.619pt) arc (120:180:4.619pt);\draw[thin,-my] (48.4995pt,4.0001pt) arc (60:120:4.619pt) arc (300:240:4.619pt) arc (60:120:4.619pt) arc (300:240:4.619pt) arc (60:120:4.619pt) arc (300:240:4.619pt) arc (60:120:4.619pt) arc (120:180:4.619pt) arc (0:-60:4.619pt) arc (120:180:4.619pt) arc (0:-60:4.619pt) arc (120:180:4.619pt) arc (0:-60:4.619pt) arc (120:180:4.619pt);\draw[thin,-my] (48.4995pt,12.0002pt) arc (300:240:4.619pt) arc (60:120:4.619pt) arc (300:240:4.619pt) arc (60:120:4.619pt) arc (300:240:4.619pt) arc (60:120:4.619pt) arc (300:240:4.619pt) arc (60:120:4.619pt) arc (120:180:4.619pt) arc (0:-60:4.619pt) arc (120:180:4.619pt) arc (0:-60:4.619pt) arc (120:180:4.619pt) arc (0:-60:4.619pt) arc (120:180:4.619pt) arc (0:-60:4.619pt) arc (120:180:4.619pt);\draw[thin,-my] (43.8805pt,20.0003pt) arc (300:240:4.619pt) arc (60:120:4.619pt) arc (300:240:4.619pt) arc (60:120:4.619pt) arc (300:240:4.619pt) arc (60:120:4.619pt) arc (300:240:4.619pt);\draw[thin,-my] (39.2615pt,28.0004pt) arc (300:240:4.619pt) arc (60:120:4.619pt) arc (300:240:4.619pt) arc (60:120:4.619pt) arc (300:240:4.619pt);\draw[thin,-my] (34.6425pt,36.0005pt) arc (300:240:4.619pt) arc (60:120:4.619pt) arc (300:240:4.619pt);\draw[thin,-my] (30.0235pt,44.0006pt) arc (300:240:4.619pt);\end{tikzpicture}
       \nopar\label{fig:tria-flow}}
   \caption{On the proof of Proposition~\ref{pro:CTT-triangle}.}
 \end{center}
\end{figure}
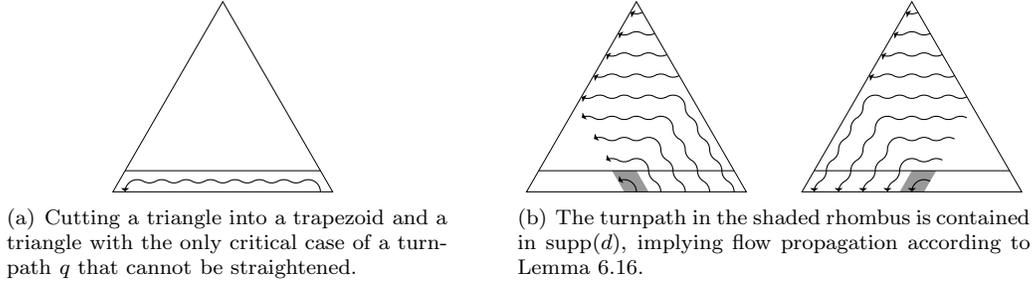
Suppose we are in the left-hand situation of Figure~\ref{fig:tria-flow}. 
Then we can decompose~$L$ into a trapezoid of height~$1$ and a triangle
by cutting along the right-hand side of~$L$. 
Lemma~\ref{le:flow-ext} implies that no critical case can arise,
so that in this situation, 
the procedure of concatenation and straightening works.
If we are in the right-hand situation of Figure~\ref{fig:tria-flow}, 
we can decompose~$L$ into a trapezoid of height~$1$ and a triangle
by cutting along the left-hand side of~$L$ and 
argue analogously.
\qquad\end{proof}


\section{Proof of the Shortest Path Theorem}
\label{sec:shortestpaththeorem}
In this section we prove Theorem~\ref{thm:shopath}.

\subsection{Special rhombi}
\label{subsec:specialrhombi}

For the whole subsection we fix a {\em shortest} \stturnpath $p$ in $\resf f$ for $f\in B_\Z$. 
Flatness shall always refer to $f$.  

We will show in several steps that the minimal length of $p$ poses severe restrictions on the way $p$ may pass a rhombus. 
Before doing so, let us verify an even simpler property of $p$ resulting from the minimality.  
Each turnvertex in $\res$ has a \emph{reverse turnvertex} in $R$ 
that points in the other direction, e.g., the reverse turnvertex of $\rhpcMl$ is $\rhpoulMr$.
We note that if $p$ contains a turnvertex~$v$ touching the boundary of $\Delta$, 
than $p$ cannot contain the reverse of~$v$
(otherwise, $p$ would use $s$ or $t$ more than once). 



Here and in the following, statements involving the pictorial
description~$\rhc$ include the possibility of a rotation by $180^\circ$. 

\begin{proposition}
\label{cla:noreverse}
The turnpath~$p$ cannot use a turnvertex and its reverse. 
\end{proposition}

\begin{proof} 
By way of contradiction, let $w$ be the {\em first}  turnvertex in $p$ whose reverse turnvertex is also used by $p$.
We already noted that $w$ cannot touch the border of $\Delta$.
Let $v$ denote the predecessor of~$w$ and let $\rhc$ stand for the rhombus that contains both $v$ and $w$.

Suppose first that $v$ is a clockwise turn: $v = \rhpoulMr$. 
Since the reverse of $w$ is in $p$, the turnpath~$p$ must use $\rhpcMl$ or $\rhpcMr$. 
But $\rhpcMl$ is excluded because of the minimal choice of $w$.
If $\rhc$ were not flat, then it is easy to check that $p$~could be rerouted via $\rhpoulMl$.
This contradicts the minimal length of~$p$.
So let us assume that $\rhc$ is flat.
Then $\rhpoulMrr \notin E(\resf f)$ by construction of $\resf f$ and thus $\rhpoulMrl \in p$.
Hence $w = \rhpcWl$, which implies $\rhpolrWr \in p$.
But since $\rhpolrWrr \notin E(\resf f)$, it follows that $\rhpolrWrl$ is used by $p$, 
in contradiction with the minimal choice of~$w$.

It remains to analyze the case where $v$ is a counterclockwise turn: $v=\rhpourMl$. 
As before we must have $\rhpcMl \in p$. 
The existence of counterclockwise turns at acute angles implies that 
$\rhrhoul$ and $\rhrhour$  are not flat. 
Hence $p$ can be rerouted via $\rhpourMr$, 
in contradiction with the minimal length of~$p$.
\qquad\end{proof}

We continue now by analyzing the possible ways $p$ may pass through a rhombus. 
Note that, due to Proposition~\ref{cla:noreverse}, the turnpath~$p$ can cross the diagonal $\rhsc$ 
of a rhombus at most twice.

\begin{lemma}
\label{cla:nonflatflow}
 If $\rhc$ is not flat, then its diagonal $\rhsc$ is crossed by $p$ at most once.
\end{lemma}

\begin{proof}
Assume by way of contradiction that both $\rhpcMl$ and $\rhpcMr$ occur in~$p$. 
Since $p$ cannot use a turnvertex twice, there are only two possibilites: 
\begin{equation}\label{eq:two-poss}
\mbox{either\quad $\rhpollWll$ and $\rhpolrWrr$ are edges of $p$\quad  
or \quad  
$\rhpollWlr$ and $\rhpolrWrl$  are edges of $p$.}
\end{equation}
In both cases, $p$~can be rerouted resulting in a shorter \stturnpath, contradicting the 
minimal length of~$p$. 
Note that the rerouting in the second case is possible since $\rhc$ is assumed to be not flat.
\qquad\end{proof}


We now focus on the rhombi in which $p$ crosses the diagonal twice. 

\begin{definition}\label{def:special}
A rhombus~$\varrho$ is called {\em special} if the turnpath~$p$ crosses its diagonal twice. 
If the crossing is in the same direction, then 
$\varrho$ is called  \emph{confluent},  
otherwise, if the crossing is in opposite directions, 
$\varrho$ is called \emph{contrafluent}. 
\end{definition}

By Lemma~\ref{cla:nonflatflow}, special rhombi are necessarily flat. 
Recall the slack contributions 
of a rhombus $\varrho$ introduced in Definition~\ref{def:slack-contr}. 

\begin{proposition}\label{pro:special}
In a special rhombus $\rhc$, the turnpath~$p$ 
uses exactly two neutral slack contributions.  
\end{proposition}

The proof proceeds by several steps. 

\begin{lemma}
\label{cla:confluentatleast}
In a confluent rhombus $\rhc$ the turnpath $p$ uses at least 
the two contributions $\rhpollWlr$ and $\rhpolrWrl$. 
\end{lemma}

\begin{proof}
Suppose that both $\rhpcMl$ and $\rhpcMr$ occur in~$p$. 
Then, as before, there are only the two possibilites of~\eqref{eq:two-poss}. 
Since $\rhc$ is flat, the first case is impossible. 
\qquad\end{proof}

We can now completely determine how $p$ passes through contrafluent rhombi. 

\begin{lemma}
\label{cla:contrafluentfixed}
In a contrafluent rhombus $\rhc$, the turnpath~$p$ 
uses the contributions $\rhpoulMrl$ and $\rhpollWlr$ 
and no other contributions in this rhombus. 
\end{lemma}

\begin{proof}
Assume first that $\rhpourMl$ and $\rhpcMl$ are in~$p$. 
Then $\rhrhoul$ and $\rhrhour$ are both not flat. 
Hence $p$ can be rerouted via $\rhpourMr$, 
contradicting the minimal length of $p$.

We are therefore left with the case where 
$\rhpoulMr$ and $\rhpcMr$ are in $p$.
By construction, $\rhpoulMrr$ and $\rhpolrWrr$ are not edges of  $\resf f$ 
and thus $\rhpoulMrl$ and $\rhpollWlr$ are both turnedges of~$p$. 
It remains to show that $p$~uses no other contribution in~$\rhc$.

Proposition~\ref{cla:noreverse} combined with the fact that $\rhc$ is flat 
easily implies that 
$\rhpourMr$ and $\rhpollWr$
are the only contributions that $p$ may possibly use.
We exclude now these two cases.

Suppose that $\rhpollWr$ occurs in~$p$. Then, as  $\rhpollWlr \in p$, 
Lemma~\ref{cla:nonflatflow} implies that $\rhrholl$ is flat. 
However, this contradicts $\rhpcWl \in p$.

We are left with the case that $\rhpourMr$ occurs in $p$.
Then $\rhrhoul$ and $\rhrhour$ 
are both contrafluent.
Applying  what we have learned so far about contrafluent rhombi,
we get the situation depicted in Figure~\ref{fig:contrafluent}.
\begin{figure}[h] 
  \begin{center}
\scalebox{2}{
\begin{tikzpicture}\draw[thin,-my] (6.9285pt,12.0002pt) --  (-6.9285pt,12.0002pt);;\draw[thin,-my] (-2.3095pt,4.0001pt) --  (4.619pt,16.0002pt);;\draw[thin,-my] (-4.619pt,16.0002pt) --  (2.3095pt,4.0001pt);;\draw[rhrhombidraw] (0.0pt,0.0pt) -- (9.238pt,16.0002pt) -- (-9.238pt,16.0002pt) -- cycle;\fill (4.619pt,8.0001pt) circle (0.4pt);\fill (0.0pt,16.0002pt) circle (0.4pt);\fill (-4.619pt,8.0001pt) circle (0.4pt);\end{tikzpicture}
\hspace{2cm}
\begin{tikzpicture}\draw[rhrhombidraw] (0.0pt,0.0pt) -- (9.238pt,16.0002pt) -- (-9.238pt,16.0002pt) -- cycle;\fill (4.619pt,8.0001pt) circle (0.4pt);\fill (0.0pt,16.0002pt) circle (0.4pt);\fill (-4.619pt,8.0001pt) circle (0.4pt);\draw[thin,-my] (-4.619pt,16.0002pt) arc (0:-60:4.619pt);\draw[thin,-my] (-2.3095pt,4.0001pt) arc (120:60:4.619pt);\draw[thin,-my] (6.9285pt,12.0002pt) arc (240:180:4.619pt);\end{tikzpicture}
}
    \caption{The situation in Lemma~\protect\ref{cla:contrafluentfixed}.
Left: The arrows represent $p$.
All three overlapping rhombi are contrafluent and thus flat. On the right: The turnvertices that can be used for rerouting.}
    \nopar\label{fig:contrafluent}
  \end{center}
\end{figure}
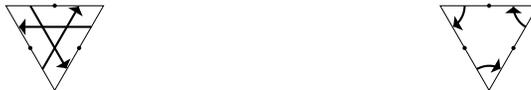
Note that the depicted triangle of side length 2 is not only contained in a flatspace,
but it is a flatspace itself: The reason is
that $p$ traverses in counterclockwise direction at the border of flatspaces, 
see Proposition~\ref{cla:outerccdir}.
This implies that $p$ can be rerouted as seen in Figure~\ref{fig:contrafluent}, 
which is a contradiction to the minimal length of~$p$.
\qquad\end{proof}

\begin{lemma}
\label{cla:contraoverlother}
1. Contrafluent rhombi cannot overlap with contrafluent or confluent rhombi. 

2. Confluent rhombi cannot overlap.
\end{lemma}

\begin{proof}
1. Assume that a contrafluent rhombus $\rhc$ overlaps with a shaded confluent or contrafluent rhombus
as in Figure~\ref{fig:overlapping} (the other cases are similar).
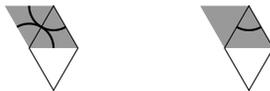
\begin{figure}[h]
  \begin{center}
\scalebox{2}{
\begin{tikzpicture}\fill[rhrhombifill] (41.571pt,8.0001pt) -- (32.333pt,8.0001pt) -- (27.714pt,16.0002pt) -- (36.952pt,16.0002pt) -- cycle;\fill[rhrhombifill] (-9.238pt,16.0002pt) -- (-4.619pt,8.0001pt) -- (4.619pt,8.0001pt) -- (0.0pt,16.0002pt) -- cycle;\draw[rhrhombidraw] (0.0pt,16.0002pt) -- (-4.619pt,8.0001pt) -- (0.0pt,0.0pt) -- (4.619pt,8.0001pt) -- cycle;\draw[thin] (2.3095pt,12.0002pt) arc (300:240:4.619pt) arc (60:120:4.619pt);\draw[thin] (-4.619pt,16.0002pt) arc (180:240:4.619pt) arc (60:0:4.619pt);\draw[rhrhombidraw] (36.952pt,0.0pt) -- (32.333pt,8.0001pt) -- (36.952pt,16.0002pt) -- (41.571pt,8.0001pt) -- cycle;\draw[thin] (39.2615pt,12.0002pt) arc (300:240:4.619pt);\end{tikzpicture}
}
\caption{The situation in the proof of Lemma~\ref{cla:contraoverlother}.}
    \nopar\label{fig:overlapping}
  \end{center}
\end{figure}
The turnpath $p$ uses at least the turnedges drawn in the left figure, where the directions are irrelevant and hence omitted. 
Hence the turnvertex in the right figure is used by $p$. 
But then, Lemma~\protect\ref{cla:contrafluentfixed} implies that 
$\rhc$~cannot be contrafluent, contradiction! 

2. Let $\rhc$ be confluent and assume that $\rhpoulMrl$ and $\rhpourMlr$ occur in~$p$. 
Assume that $\rhrhoul$ is confluent and hence $p$ contains $\rhppulWlr$ and $\rhpulWrl$. 
Then $\rhrhour$ is contrafluent and overlapping with $\rhc$, which is contradicting part one of this lemma.
The same argument works for the other three overlapping cases.  
\qquad\end{proof}



\begin{prooff}[Proof of Proposition~\textup{\ref{pro:special}}] 
By Lemma~\ref{cla:contrafluentfixed} it remains to consider the 
case of a confluent rhombus. 
We improve on Lemma~\ref{cla:confluentatleast}.
If $p$ would use any additional contribution, then
$\rhc$ would overlap with a confluent or contrafluent rhombus,
which is impossible due to Lemma~\ref{cla:contraoverlother}. 
\qquad\end{prooff}

\subsection{Rigid and critical rhombi}
\label{subsec:shopath}

Recall the polyhedron $B$ of bounded hive flows associated with chosen partitions $\la,\mu,\nu$. 
Again we fix $f\in B_\Z$ and we fix a {\em shortest} \stturnpath $p$ in $\resf f$.
We set 
\begin{equation}\label{eq:defeps}
 \varepsilon := \max \{ t \in \R \mid f + t \pi(p) \in B\}, \ \ g:=f+\varepsilon \pi(p).
\end{equation}
Then we have $\varepsilon > 0$ by Lemma~\ref{lem:direction} and $g \in B$.
For the proof of the Shortest Path Theorem~\ref{thm:shopath} 
it suffices to show that  $\varepsilon \geq 1$, since then $f+\pi(p) \in B_\Z$.

If all rhombi are $f$\dash flat, then there are only two possibilities for $p$,
going directly from the right or bottom entrance edge to the left exit edge.
In these two cases we clearly have $\varepsilon \geq 1$.

In the following we suppose that 
not all rhombi are $f$\dash flat.
We shall argue indirectly and assume that $\varepsilon<1$ 
for the rest of this subsection. 
After going through numerous detailed case distinctions, 
describing the possible local situations, 
we will finally end up with a contradiction, which then finishes the proof of 
the Shortest Path Theorem~\ref{thm:shopath}. 
Our main tools will be Proposition~\ref{pro:special} on special rhombi
and the hexagon equality~\eqref{cla:BZ}.
Unfortunately, we see no way of considerably simplifying 
the tedious arguments.

\begin{definition}\label{def:critical}
A rhombus is called \emph{critical} if it is not $f$-flat, but $g$-flat.
Moreover, we call a rhombus \emph{rigid} if it is both $f$-flat and $g$-flat.
\end{definition}


\begin{claim}\label{cla:critexists}
There exists a critical rhombus.
\end{claim}

\begin{proof}
Let $S \neq \emptyset$ denote the set of rhombi which are not $f$-flat and consider the continuous function
of $t\in\R$
$$
 F(t) := \min_{\varrho \in S} \s \varrho {f+t \pi(p)} .
$$
It is sufficient to show that $F(\varepsilon) = 0$. 

By the definition~\eqref{eq:defeps} of~$\varepsilon$ we have $F(\varepsilon) \ge 0$. 
Further, for $\varepsilon < t <1$ we have $f+t\pi(p)\not\in B$.
Since the flow $f+t\pi(p)$ satisfies the border capacity constraints by construction of~$\resf f$, 
there is a rhombus $\varrho$ with $\s \varrho {f+ t \pi(p)} < 0$. 
We must have $\varrho \in S$, since otherwise $\s \varrho p \ge 0$ (cf.\ Lemma~\ref{le:spos}),  
which would lead to the contradiction $\s \varrho {f+ t \pi(p)} \ge 0$.  
We have thus shown that $F(t) < 0$. 
Since $t$ can be arbitrarily close to $\varepsilon$, we get $F(\varepsilon) \le 0$.
Altogether, we conclude $F(\varepsilon) = 0$.
\qquad\end{proof}
 
\begin{claim}
\label{cla:verycritical}
Each critical rhombus $\varrho$ satisfies $\s \varrho p \leq -2$.
\end{claim}

\begin{proof}
We have $\s \varrho f \geq 1$ and $\s \varrho {f+\varepsilon \pi(p)} = 0$, 
hence $\s \varrho p = -\tfrac 1 \varepsilon \ \s \varrho f$.
Using $0 <\varepsilon<1$ we conclude $\s \varrho p < -1$. 
\qquad\end{proof}


\begin{lemma}\label{cor:rigidonlyneutral} 
A rhombus $\varrho$ is rigid  iff it is $f$\dash flat and $p$~uses in it
only neutral contributions. 
All special rhombi are rigid. 
\end{lemma}

\begin{proof}
The first assertion follows immediately from Lemma~\ref{lem:slackcalculation}.
The second assertion is a consequence of the first and Proposition~\ref{pro:special}.
\qquad\end{proof}




For the rest of this subsection, we denote by $\rhc$ a {\em first critical rhombus} visited by~$p$
(a~priori it might not be unique, because critical rhombi could overlap). 
By Claim~\ref{cla:verycritical},
$p$~uses at least two negative slack contributions in $\rhc$, cf.\ Lemma~\ref{lem:slackcalculation}.  
In particular, 
$p$~uses at least one turnvertex among $\rhpoulMl, \rhpoulMr, \rhpolrWl, \rhpolrWr$:
let $\rhpoulMlXoulMr \in \{\rhpoulMl, \rhpoulMr\}$ denote the {\em first one} used by $p$ in this set.
(Rotating with $180^\circ$ we may assume so without loss of generality.)
Further, let $\rhpulWrXpulWl$ denote the predecessor of $\rhpoulMlXoulMr$ in~$p$.

Our goal is to analyze the route of~$p$ through $\rhc$ and nearby rhombi.
Narrowing down the possibilities will finally lead to a contradiction.


The situation at the boundary of $\Delta$ deserves special treatment and is handled in the following claim.
\begin{claim}
\label{cla:border}$
\begin{tikzpictured}\draw[rhrhombidraw] (0.0pt,16.0002pt) -- (-4.619pt,8.0001pt) -- (0.0pt,0.0pt) -- (4.619pt,8.0001pt) -- cycle;\draw[rhrhombithickside] (-9.238pt,0.0pt) -- (0.0pt,16.0002pt);\end{tikzpictured}
$
 is not at the border of $\Delta$.
\end{claim}

{\em Proof}.
Assume the contrary.
Then $p$ enters $\Delta$ once via
\newcommand{\cI}{
\begin{tikzpictured}\draw[rhrhombidraw] (0.0pt,16.0002pt) -- (-4.619pt,8.0001pt) -- (0.0pt,0.0pt) -- (4.619pt,8.0001pt) -- cycle;\draw[rhrhombithickside] (-9.238pt,0.0pt) -- (0.0pt,16.0002pt);\draw[thin,-my] (-2.3095pt,12.0002pt) arc (240:300:4.619pt);\end{tikzpictured}
}\cI
,
\newcommand{\cIV}{
\begin{tikzpictured}\draw[rhrhombidraw] (0.0pt,16.0002pt) -- (-4.619pt,8.0001pt) -- (0.0pt,0.0pt) -- (4.619pt,8.0001pt) -- cycle;\draw[rhrhombithickside] (-9.238pt,0.0pt) -- (0.0pt,16.0002pt);\draw[thin,-my] (-2.3095pt,12.0002pt) arc (60:0:4.619pt) arc (0:-60:4.619pt);\end{tikzpictured}
}\cIV
or
\newcommand{\cV}{
\begin{tikzpictured}\draw[rhrhombidraw] (0.0pt,16.0002pt) -- (-4.619pt,8.0001pt) -- (0.0pt,0.0pt) -- (4.619pt,8.0001pt) -- cycle;\draw[rhrhombithickside] (-9.238pt,0.0pt) -- (0.0pt,16.0002pt);\draw[thin,-my] (-2.3095pt,12.0002pt) arc (60:0:4.619pt) arc (180:240:4.619pt);\end{tikzpictured}
}\cV
.
If $p$ enters over \cV, then $p$ must use both
\newcommand{\cII}{
\begin{tikzpictured}\draw[rhrhombidraw] (0.0pt,16.0002pt) -- (-4.619pt,8.0001pt) -- (0.0pt,0.0pt) -- (4.619pt,8.0001pt) -- cycle;\draw[rhrhombithickside] (-9.238pt,0.0pt) -- (0.0pt,16.0002pt);\draw[thin,-my] (2.3095pt,4.0001pt) arc (60:120:4.619pt);\end{tikzpictured}
}\cII
 and 
\newcommand{\cIII}{
\begin{tikzpictured}\draw[rhrhombidraw] (0.0pt,16.0002pt) -- (-4.619pt,8.0001pt) -- (0.0pt,0.0pt) -- (4.619pt,8.0001pt) -- cycle;\draw[rhrhombithickside] (-9.238pt,0.0pt) -- (0.0pt,16.0002pt);\draw[thin,-my] (2.3095pt,4.0001pt) arc (240:180:4.619pt) arc (180:120:4.619pt);\end{tikzpictured}
}\cIII
to satisfy $\s \rhc p \leq -2$. 
Using both
\cIII
and
\cV
is prohibited by Claim~\ref{cla:noreverse}.
Hence $p$ uses either \cI or \cIV and $p$ also uses \cII or \mbox{\cIII.}

We now make a distinction of cases, each of which leads to a contradiction.
Hereby, we heavily rely on Proposition~\ref{pro:special}. 
\begin{remunerate}

 \item[Case $\cII \cup \cIV \subseteq p$:] Here $\rhrholl$ is special and thus \cII must be continued as
\begin{tikzpictured}\draw[rhrhombidraw] (0.0pt,16.0002pt) -- (-4.619pt,8.0001pt) -- (0.0pt,0.0pt) -- (4.619pt,8.0001pt) -- cycle;\draw[rhrhombithickside] (-9.238pt,0.0pt) -- (0.0pt,16.0002pt);\draw[thin,-my] (2.3095pt,4.0001pt) arc (60:120:4.619pt) arc (300:240:4.619pt);\end{tikzpictured}
and \cIV must be continued as
\begin{tikzpictured}\draw[rhrhombidraw] (0.0pt,16.0002pt) -- (-4.619pt,8.0001pt) -- (0.0pt,0.0pt) -- (4.619pt,8.0001pt) -- cycle;\draw[rhrhombithickside] (-9.238pt,0.0pt) -- (0.0pt,16.0002pt);\draw[thin,-my] (-2.3095pt,12.0002pt) arc (60:0:4.619pt) arc (0:-60:4.619pt) arc (120:180:4.619pt);\end{tikzpictured}%
. Thus $p$ leaves and enters the same side of $\Delta$.

 \item[Case $\cIII \cup \cIV \subseteq p$:] This is impossible, because $p$ passes $\rhsc$ twice, but not with two neutral contributions 
as a special rhombus must do.

 \item[Case $\cI \cup \cIII \subseteq p$:] Here $\rhrhour$ is special and \cI must be continued as
\begin{tikzpictured}\draw[rhrhombidraw] (0.0pt,16.0002pt) -- (-4.619pt,8.0001pt) -- (0.0pt,0.0pt) -- (4.619pt,8.0001pt) -- cycle;\draw[rhrhombithickside] (-9.238pt,0.0pt) -- (0.0pt,16.0002pt);\draw[thin,-my] (-2.3095pt,12.0002pt) arc (240:300:4.619pt) arc (120:60:4.619pt);\end{tikzpictured}
and \cIII must be continued as
\begin{tikzpictured}\draw[rhrhombidraw] (0.0pt,16.0002pt) -- (-4.619pt,8.0001pt) -- (0.0pt,0.0pt) -- (4.619pt,8.0001pt) -- cycle;\draw[rhrhombithickside] (-9.238pt,0.0pt) -- (0.0pt,16.0002pt);\draw[thin,-my] (2.3095pt,4.0001pt) arc (240:180:4.619pt) arc (180:120:4.619pt) arc (300:360:4.619pt);\end{tikzpictured}%
. This enables a rerouting via
\begin{tikzpictured}\draw[rhrhombidraw] (0.0pt,16.0002pt) -- (-4.619pt,8.0001pt) -- (0.0pt,0.0pt) -- (4.619pt,8.0001pt) -- cycle;\draw[rhrhombithickside] (-9.238pt,0.0pt) -- (0.0pt,16.0002pt);\draw[thin,-my] (-2.3095pt,12.0002pt) arc (240:300:4.619pt) arc (300:360:4.619pt);\end{tikzpictured}%
, which is a contradiction to the minimal length of $p$.

 \item[Case $\cI \cup \cII \subseteq p$:]
Note that 
\begin{tikzpictured}\draw[rhrhombidraw] (0.0pt,16.0002pt) -- (-4.619pt,8.0001pt) -- (0.0pt,0.0pt) -- (4.619pt,8.0001pt) -- cycle;\draw[rhrhombithickside] (-9.238pt,0.0pt) -- (0.0pt,16.0002pt);\draw[thin,-my] (2.3095pt,4.0001pt) arc (60:120:4.619pt) arc (120:180:4.619pt);\end{tikzpictured}
$\subseteq p$, because
\begin{tikzpictured}\draw[rhrhombidraw] (0.0pt,16.0002pt) -- (-4.619pt,8.0001pt) -- (0.0pt,0.0pt) -- (4.619pt,8.0001pt) -- cycle;\draw[rhrhombithickside] (-9.238pt,0.0pt) -- (0.0pt,16.0002pt);\draw[thin,-my] (2.3095pt,4.0001pt) arc (60:120:4.619pt) arc (300:240:4.619pt);\end{tikzpictured}
$\subseteq p$ would be a contradiction, because $p$ cannot leave $\Delta$ at
\begin{tikzpictured}\draw[rhrhombidraw] (0.0pt,16.0002pt) -- (-4.619pt,8.0001pt) -- (0.0pt,0.0pt) -- (4.619pt,8.0001pt) -- cycle;\draw[rhrhombithickside] (-9.238pt,0.0pt) -- (0.0pt,16.0002pt);\end{tikzpictured}%
. The fact
\begin{tikzpictured}\draw[rhrhombidraw] (0.0pt,16.0002pt) -- (-4.619pt,8.0001pt) -- (0.0pt,0.0pt) -- (4.619pt,8.0001pt) -- cycle;\draw[rhrhombithickside] (-9.238pt,0.0pt) -- (0.0pt,16.0002pt);\draw[thin,-my] (2.3095pt,4.0001pt) arc (60:120:4.619pt) arc (120:180:4.619pt);\end{tikzpictured}
$\subseteq p$ implies that
$p$ can be rerouted via
\begin{tikzpictured}\draw[rhrhombidraw] (0.0pt,16.0002pt) -- (-4.619pt,8.0001pt) -- (0.0pt,0.0pt) -- (4.619pt,8.0001pt) -- cycle;\draw[rhrhombithickside] (-9.238pt,0.0pt) -- (0.0pt,16.0002pt);\draw[thin,-my] (-2.3095pt,12.0002pt) arc (60:0:4.619pt) arc (0:-60:4.619pt) arc (120:180:4.619pt);\end{tikzpictured}%
, which is a contradiction to the minimal length of $p$.\qquad\endproof
\end{remunerate}

We continue with the analysis of the general situation. 

\begin{claim}\label{cla:uppernotgflat}
 $\rhrholl$ is not $g$-flat.
\end{claim}

\begin{proof}
By way of contradiction, assume that $\rhrholl$ is $g$-flat. 
By Claim~\ref{cla:border} we know that
\begin{tikzpictured}\draw[rhrhombidraw] (0.0pt,16.0002pt) -- (-4.619pt,8.0001pt) -- (0.0pt,0.0pt) -- (4.619pt,8.0001pt) -- cycle;\fill[rhrhombifill] (-4.619pt,8.0001pt) -- (-9.238pt,0.0pt) -- (0.0pt,0.0pt) -- (4.619pt,8.0001pt) -- cycle;\draw[rhrhombithickside] (-9.238pt,0.0pt) -- (0.0pt,16.0002pt);\end{tikzpictured}
is not at the border of $\Delta$. 
Then $\rhrhul$ and $\rhrhpl$ exist and are both $g$-flat by Corollary~\ref{cor:BZ2} applied to~$g$.
So they are either rigid or critical.
Since $\rhc$ is not $f$-flat, it follows from Corollary~\ref{cor:BZ2} applied to~$f$ 
that not both $\rhrhul$ and $\rhrhpl$ are rigid, so at least one of them is critical. 
It remains to exclude the following two cases:

If $\rhrhul$ is critical, then the critical rhombus $\rhrhul$ is passed by $p$ before $\rhc$, 
contradicting the minimal choice of $\rhc$. 

If $\rhrhul$ is rigid, then $\rhppulWl \notin V(\resf f)$ by the definition of $\resf f$
and hence $\rhpulWrXpulWl = \rhpulWr$ and $p$ passes the critical rhombus $\rhrhpl$ before $\rhc$.
This again contradicts the minimal choice of $\rhc$. 
\qquad\end{proof}

\begin{claim}\label{cla:direct}
The turnpath $p$ goes directly from $\rhpulWrXpulWl$ to $\rhpoulMl$, 
which is the only time that $p$ leaves $\rhc$ over $\rhsour$.
Additionally, $p$ leaves $\rhc$ exactly once over $\rhsoll$.
In particular, $\s \rhc p = -2$, $\s \rhc f = 1$ and $\varepsilon = \tfrac 1 2$.
\end{claim}

\begin{proof} 
We prove the following claims, tacitly using Proposition~\ref{pro:special} on special rhombi: 
\begin{romannum}
 \item If $p$ leaves $\rhc$ at $\rhsour$, then $p$ goes directly from $\rhpulWrXpulWl$ to $\rhpoulMl$.
Proof: Otherwise $p$ could be rerouted via $\rhpoulMl$, a contradiction.
 \item $p$ does not leave $\rhc$ twice at $\rhsour$.
Proof: Otherwise $\rhrhour$ would be special and hence $p$ would use both $\rhpcMrl$ and $\rhpoulMlr$.
Then $p$ could be rerouted via $\rhpoulMll$, a contradiction.
 \item $p$ does not leave $\rhc$ twice at $\rhsoll$.
Proof: Otherwise $\rhrholl$ would be special and 
hence rigid by Lemma~\ref{cor:rigidonlyneutral}.  
However, this is prohibited by Claim~\ref{cla:uppernotgflat}.
\end{romannum}
The fact that $p$ leaves $\rhc$ over $\rhsour$ at most once and 
over $\rhsoll$ at most once 
implies $\s \rhc p = \rhaollW(p) + \rhaourM(p) \geq -2$.
On the other hand, since $\rhc$ is critical, 
we have $\s \rhc p \leq -2$ by Claim~\ref{cla:verycritical}.
Therefore, $\s \rhc p = -2$.
Hence $\s \rhc f = -\varepsilon \s \rhc p = 2 \varepsilon$,
so $\varepsilon = \frac 1 2 \s \rhc f$
and since $0<\varepsilon<1$ we obtain $\s \rhc f = 1$ and $\varepsilon = \frac 1 2$.
\qquad\end{proof}


Despite the fact that, due to Claim~\ref{cla:direct}, $p$ enters $\rhc$ at $\rhsoul$ and leaves $\rhc$ at $\rhsoll$,
$p$ cannot be rerouted via $\rhpoulMrr$ to a shorter \stturnpath,
because $p$ has minimal length. This can have two reasons, which 
leads to the following namings (compare~\eqref{eq:A-B}):
\begin{quote}
If $\rhrhoul$ is $f$-flat and $p$ uses $\rhpulWrl$, then we say that $p$ \emph{enters nonreroutably}, otherwise $p$ \emph{enters reroutably}.
If $\rhrholl$ is $f$-flat and $p$ uses $\rhpolrWlr$, then we say that $p$ \emph{leaves nonreroutably}, otherwise $p$ \emph{leaves reroutably}.
\end{quote}
An explanation of these namings can be found in the proof of the following claim.

\begin{claim}\label{cla:basicreroute}
The turnpath $p$ enters nonreroutably or leaves nonreroutably.
\end{claim}

{\em Proof}.
Assume the contrary, i.e., $p$~enters reroutably and leaves reroutably.
Recall that $\rhc$ is critical and hence not $f$\dash flat.
We make a distinction of four cases.
\begin{remunerate}
 \item Let $\rhrhoul$ be not $f$\dash flat and $\rhrholl$ not be $f$\dash flat. Then $p$ can be rerouted using $\rhpoulMrr$.
 \item Let $\rhrhoul$ be $f$\dash flat and $\rhrholl$ not be $f$\dash flat. Then $p$ uses $\rhppulWll$ by our assumption on $p$ 
at the beginning of the proof. Hence $p$ can be rerouted with~$\rhppulWlrr$.

 \item Let $\rhrhoul$ not be $f$\dash flat and $\rhrholl$ be $f$\dash flat.
Assume that $p$ uses $\rhpollMr$. Then $p$ uses $\rhpolrWl$, which is a contradiction to the fact that $p$ leaves reroutably.
Hence $p$ uses $\rhpollMl$ and $p$ can be rerouted using~$\rhpoulMrrl$.

 \item Let $\rhrhoul$ and $\rhrholl$ be both $f$\dash flat.
Then $p$ uses $\rhppulWl$ and $\rhpollMl$
by our assumption on $p$ at the beginning of the proof.
Hence $p$ can be rerouted with~\begin{tikzpictured}\draw[rhrhombidraw] (0.0pt,16.0002pt) --
  (-4.619pt,8.0001pt) -- (0.0pt,0.0pt) -- (4.619pt,8.0001pt) --
  cycle;\draw[thin,-my] (-4.619pt,16.0002pt) arc (180:240:4.619pt) arc
  (60:0:4.619pt) arc (0:-60:4.619pt) arc
  (120:180:4.619pt);\end{tikzpictured}
.\qquad\endproof 
\end{remunerate}

The possible reroutability of $p$ upon entering or leaving the critical rhombus $\rhc$ 
gives rise to a distinction of four cases, 
one of which is dealt with in Claim~\ref{cla:basicreroute}.
In the rest of this subsection we deal with the other three cases.
But first we prove the following auxiliary Claim~\ref{cla:notrigid}.

\begin{claim}
\label{cla:notrigid}
$\rhrhoul$ is not rigid.
\end{claim}

\begin{proof}
Assume the contrary. Then, according to Lemma~\ref{cor:rigidonlyneutral} and Claim~\ref{cla:direct}, $p$ uses $\rhpulWrl$. 
Hence $p$ enters nonreroutably. 
Corollary~\ref{cor:BZ2} implies that $\rhrhll$ and $\rhrhpl$ are both $g$-flat. 
The hexagon equality~\eqref{cla:BZ} and $\s \rhc f = 1$ imply that 
$\s \rhrhll f + \s \rhrhpl f = 1$.
Integrality of~$f$ implies that one of the two shaded rhombi of $\rhrhll$ and $\rhrhpl$ is critical and the other one is rigid.

The rhombus $\rhrhpl$ cannot be critical since otherwise $p$ would pass the critical $\rhrhpl$ before~$\rhc$,
which contradicts the choice of $\rhc$ as the first critical rhombus in which $p$ uses turnvertices.
Hence $\rhrhpl$ is rigid, which implies that $\rhpxulMl$ is not a turnvertex of $\resf f$.
Thus $p$ uses $\rhpplMrrl$ and hence $p$ passes the critical rhombus
$\rhrhll$ before $\rhc$. Again, this is a contradiction. 

Finally, if 
\begin{tikzpictured}\draw[rhrhombidraw] (0.0pt,0.0pt) -- (4.619pt,8.0001pt) -- (0.0pt,16.0002pt) -- (-4.619pt,8.0001pt) -- cycle;\fill[rhrhombifill] (4.619pt,8.0001pt) -- (-4.619pt,8.0001pt) -- (-9.238pt,16.0002pt) -- (0.0pt,16.0002pt) -- cycle;\draw[rhrhombithickside] (0.0pt,0.0pt) -- (-9.238pt,16.0002pt);\end{tikzpictured}
lies at the border of $\Delta$, then,
according to Claim~\ref{cla:direct},
$p$ enters and leaves~$\Delta$ over the same side, which is a contradiction to the minimal length of $p$.
\qquad\end{proof}

\begin{claim}
\label{cla:enterrerou}
$p$ enters reroutably.
\end{claim}

{\em Proof}.
We suppose the contrary, so assume that $p$ uses $\rhpulWrl$ and $\rhrhoul$ is $f$\dash flat.
Since $\rhrhoul$ is not rigid by Claim~\ref{cla:notrigid}, 
$p$ must use a positive slack contribution in $\rhrhoul$.
We claim that only $\rhppulWr$ is possible.
This is shown by the following case distinction, leading to contradictions in all three cases.
\begin{remunerate}
\item $\rhppulWll$ uses a turnvertex already used by~$p$.
\item $\rhpcMr \in p$ implies that $p$ leaves $\rhc$ over $\rhsour$ more than once, which is impossible due to Claim~\ref{cla:direct}.
\item $\rhpcMll \in p$ contradicts Proposition~\ref{cla:noreverse}. 
\end{remunerate}

The fact $\rhppulWr \in p$ implies that $\rhrhul$ is special and hence we have $\rhppulWrl \in p$.
Corollary~\ref{cor:BZ2} implies that both $\rhrholl$ and $\rhrhll$ are $f$-flat.
Moreover, $\s \rhc f = 1$ (see Claim~\ref{cla:direct}) and 
the hexagon equality~\eqref{cla:BZ} 
implies $\s \rhrhpl f = 1$.
Since $\rhpollMl \notin V(\resf f)$ and $\rhpollMrr \notin E(\resf f)$, 
$p$~must leave $\rhc$ at $\rhsoll$ via $\rhpollMrl$, 
see Figure~\ref{fig:situation:A}. 

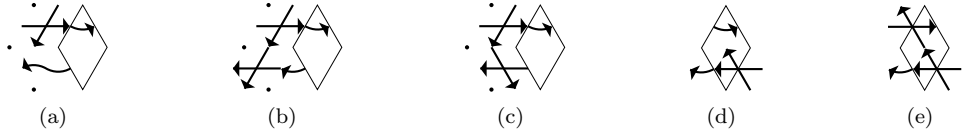
\begin{figure}[h]
  \begin{center}
      \subfigure[]
        {
\scalebox{2}{
\begin{tikzpicture}\draw[rhrhombidraw] (0.0pt,0.0pt) -- (4.619pt,8.0001pt) -- (0.0pt,16.0002pt) -- (-4.619pt,8.0001pt) -- cycle;\draw[thin,-my] (-2.3095pt,4.0001pt) arc (300:240:4.619pt) arc (60:120:4.619pt);\fill (-13.857pt,8.0001pt) circle (0.4pt);\fill (-9.238pt,16.0002pt) circle (0.4pt);\fill (-9.238pt,0.0pt) circle (0.4pt);\draw[thin,-my] (-4.619pt,16.0002pt) --  (-9.238pt,8.0001pt);;\draw[thin,-my] (-2.3095pt,12.0002pt) arc (240:300:4.619pt);\draw[thin,-my] (-11.5475pt,12.0002pt) --  (-2.3095pt,12.0002pt);;\end{tikzpicture}
}
        \nopar\label{fig:situation:A}} \hspace{0.7cm}
      \subfigure[]
        {
\scalebox{2}{
\begin{tikzpicture}\draw[rhrhombidraw] (0.0pt,0.0pt) -- (4.619pt,8.0001pt) -- (0.0pt,16.0002pt) -- (-4.619pt,8.0001pt) -- cycle;\fill (-13.857pt,8.0001pt) circle (0.4pt);\fill (-9.238pt,16.0002pt) circle (0.4pt);\fill (-9.238pt,0.0pt) circle (0.4pt);\draw[thin,-my] (-4.619pt,16.0002pt) --  (-9.238pt,8.0001pt);;\draw[thin,-my] (-2.3095pt,12.0002pt) arc (240:300:4.619pt);\draw[thin,-my] (-11.5475pt,12.0002pt) --  (-2.3095pt,12.0002pt);;\draw[thin,-my] (-2.3095pt,4.0001pt) arc (300:240:4.619pt);\draw[thin,-my] (-6.9285pt,4.0001pt) --  (-16.1665pt,4.0001pt);;\draw[thin,-my] (-9.238pt,8.0001pt) --  (-13.857pt,0.0pt);;\end{tikzpicture}
}
        \nopar\label{fig:situation:B}} \hspace{0.7cm}
      \subfigure[]
        {
\scalebox{2}{
\begin{tikzpicture}\draw[rhrhombidraw] (0.0pt,0.0pt) -- (4.619pt,8.0001pt) -- (0.0pt,16.0002pt) -- (-4.619pt,8.0001pt) -- cycle;\fill (-13.857pt,8.0001pt) circle (0.4pt);\fill (-9.238pt,16.0002pt) circle (0.4pt);\fill (-9.238pt,0.0pt) circle (0.4pt);\draw[thin,-my] (-4.619pt,16.0002pt) --  (-9.238pt,8.0001pt);;\draw[thin,-my] (-2.3095pt,12.0002pt) arc (240:300:4.619pt);\draw[thin,-my] (-11.5475pt,12.0002pt) --  (-2.3095pt,12.0002pt);;\draw[thin,-my] (-9.238pt,8.0001pt) --  (-4.619pt,0.0pt);;\draw[thin,-my] (-2.3095pt,4.0001pt) --  (-11.5475pt,4.0001pt);;\end{tikzpicture}
}
        \nopar\label{fig:situation:C}} \hspace{0.7cm}
      \subfigure[]
        {
\scalebox{2}{
\begin{tikzpicture}\draw[rhrhombidraw] (0.0pt,0.0pt) -- (4.619pt,8.0001pt) -- (0.0pt,16.0002pt) -- (-4.619pt,8.0001pt) -- cycle;\draw[thin,-my] (4.619pt,0.0pt) --  (0.0pt,8.0001pt);;\draw[thin,-my] (6.9285pt,4.0001pt) --  (-2.3095pt,4.0001pt);;\draw[thin,-my] (-2.3095pt,4.0001pt) arc (300:240:4.619pt);\draw[thin,-my] (-2.3095pt,12.0002pt) arc (240:300:4.619pt);\end{tikzpicture}
}
        \nopar\label{fig:situation:D}} \hspace{0.7cm}
      \subfigure[]
        {
\scalebox{2}{
\begin{tikzpicture}\draw[rhrhombidraw] (0.0pt,0.0pt) -- (4.619pt,8.0001pt) -- (0.0pt,16.0002pt) -- (-4.619pt,8.0001pt) -- cycle;\draw[thin,-my] (4.619pt,0.0pt) --  (0.0pt,8.0001pt);;\draw[thin,-my] (6.9285pt,4.0001pt) --  (-2.3095pt,4.0001pt);;\draw[thin,-my] (-2.3095pt,4.0001pt) arc (300:240:4.619pt);\draw[thin,-my] (0.0pt,8.0001pt) --  (-4.619pt,16.0002pt);;\draw[thin,-my] (-6.9285pt,12.0002pt) --  (2.3095pt,12.0002pt);;\end{tikzpicture}
}
        \nopar\label{fig:situation:E}}

    \caption{The situations in the proof of
      Claims~\protect\ref{cla:enterrerou}--\ref{claim:fino}. 
     We use straight arrows to depict parts of $p$ which are contained in special rhombi.}
    \nopar\label{fig:situation}
  \end{center}
\end{figure}

\begin{romannum}
\item If $p$ continues from $\rhppulWrl$ to $\rhppulWrlr$, then $\rhrhxll$ is special, see Figure~\ref{fig:situation:B}.
Slack computation shows that $\s \rhrhpl p = -2$, which implies with
$\varepsilon = \frac 1 2$ (Claim~\ref{cla:direct}) 
that $\s \rhrhpl g = 0$ and hence $\rhrhpl$ is critical.
But $p$ passes $\rhrhpl$ before $\rhc$, in contradiction with the choice of~$\rhc$.

\item If $p$ continues from $\rhppulWrl$ to $\rhppulWrll$, then $\rhrhll$ is special (see Figure~\ref{fig:situation:C}).
Note that $\rhpulMll \in p$ implies that the rhombi $\rhrhxul$ and $\rhrhxll$ are not $f$-flat.
Therefore we can reroute $p$ via $\rhpxulMrr$, which is in contradiction to the minimal length of $p$.\qquad\endproof 
\end{romannum}

Claim~\ref{cla:basicreroute} and Claim~\ref{cla:enterrerou} imply that
$p$ enters reroutably and leaves nonreroutably.
It remains to show that this leads to a contradiction.
Since $\rhrholl$ is $f$-flat and Claim~\ref{cla:uppernotgflat} ensures that $\rhrholl$ is not $g$-flat, 
$p$ must use a positive slack contribution in $\rhrholl$.

\begin{claim}\label{claim:fino}
From the positive slack contributions of $p$ in $\rhrholl$, only $\rhpolrWr$ is possible.
\end{claim}

\begin{proof}
The turnedge $\rhpolrWll$ uses a turnvertex already used by $p$.
The turnedge $\rhpllMll \subseteq p$ contradicts Proposition~\ref{cla:noreverse}. 
Now assume that $\rhpllMr \subseteq p$.
Then $\rhrhll$ is special and in particular $f$-flat.
Corollary~\ref{cor:BZ2} implies that $\rhrhoul$ and $\rhrhul$ are $f$-flat.
Since $p$ enters reroutably, $p$ uses $\rhppulWl$, which is a contradiction to $\rhppulWl \notin V(\resf f)$.
\end{proof}

It follows that $\rhrholr$ is special. The situation is depicted in Figure~\ref{fig:situation:D}.
Now $\rhpolrWr$ cannot continue to $\rhpolrWrr$, because, according to Claim~\ref{cla:direct}, $p$ leaves $\rhc$ over $\rhsour$ only once.
Hence $p$ continues to $\rhpolrWrl$ and $\rhrhoul$ is special, see Figure~\ref{fig:situation:E}.
But then, $p$ enters nonreroutably, which is a contradiction to Claim~\ref{cla:enterrerou}.

This shows that the assumption $\varepsilon<1$ is absurd.
Hence the Shortest Path Theorem~\ref{thm:shopath} is finally proved. 


\section{Proof of the Connectedness Theorem}
\label{sec:connectedness}


Recall from Section~\ref{se:honey}
the linear isomorphism $\oF(G)\simeq Z$ mapping flow classes $f$ 
to their throughput function $E(\Delta)\to\R, k\mapsto \delta(k,f)$. 
The $L_1$-norm on $Z$ induces the following {\em norm} on $\oF(G)$: 
for $f\in\oF(G)$, we set
$$
\mbox{$\|f\| := \sum_{k \in E(\Delta)} |\delta(k,f)|$.} 
$$
 Correspondingly, we define the \emph{distance} 
between $f,g \in \oF(G)$ as
$\dist(f,g) = \|g-f\|$,

In order to prove Theorem~\ref{thm:connectedness} 
we have to show that for all $f, g \in P(\la,\mu,\nu)_\Z$ there exists
a finite sequence $f=f_0, f_1, f_2, \ldots, f_\ell=g$ such that each $f_{i+1}$ equals $f_i+c_i$ 
for some $f_i$\dash secure cycle $c_i$ (cf.\ Proposition~\ref{pro:char-f-secure}).
We will construct this sequence with the additional property that $\dist(f_{i+1},g)<\dist(f_i,g)$ for all $i$.
To achieve this, it suffices to show the following Proposition~\ref{pro:connectedness:existscycle}.

\begin{proposition}
\label{pro:connectedness:existscycle}
For all distinct $f, g \in P(\la,\mu,\nu)_\Z$ there exists an $f$\dash secure cycle $c$ such that
$\dist(f+c,g) < \dist(f,g)$.
\end{proposition}

In the rest of this section we prove Proposition~\ref{pro:connectedness:existscycle}
by explicitly constructing~$c$.
We fix $f, g \in P(\la,\mu,\nu)_\Z$ and set $d := g - f$.
We can ensure the distance property $\dist(f+c,g) < \dist(f,g)$ 
for a proper cycle~$c$ by using the following result. 

\begin{lemma}
\label{cla:connectedness:morethanhalf}
If more than half of the edges of a proper cycle~$c$ on $G$ are contained in $\SUPP(d)$, 
then we have $\dist(f+c,g) < \dist(f,g)$.
\end{lemma}

\begin{proof}
Let $K \subseteq E(\Delta)$ be the set of edges of $\Delta$ crossed by $c$.
Then
\[
\dist(f,g)-\dist(f+c,g) = \sum_{k\in E(\Delta)}\big(|\delta(k,d)|-|\delta(k,d-c)|\big)
= \sum_{k\in K}\big(|\delta(k,d)|-|\delta(k,d-c)|\big).
\]
But for edges $\rhsc \in K$ we easily calculate
\[|\delta(\rhsc,d)|-|\delta(\rhsc,d-c)| = \begin{cases}
                                              1 & \text{if } \sgn\big(\rhacW(c)\big)=\sgn\big(\rhacW(d)\big) \\
                                              -1 & \text{if } \sgn\big(\rhacW(c)\big)\neq\sgn\big(\rhacW(d)\big)
                                             \end{cases}.\]
If more than half the edges of $c$ are contained in $\SUPP(d)$,
then also more than half of the summands are~1. The claim follows.
\end{proof}
In the light of Lemma~\ref{cla:connectedness:morethanhalf}, we will try to make $c$ use as many edges contained in $\SUPP(d)$ as possible.
Note that since $f$ and $g$ are both capacity achieving, we have that
\begin{equation}
\delta(k,d)=0 \text{ for all edges $k$ at the border of~$\Delta$.}
\tag{$\dag$}
\end{equation}
For the construction of~$c$ we distinguish two situations. 
The first one turns out to be considerably easier.
\smallskip

\noindent{\bf Situation 1.}
We assume that $d$ does not cross the sides of $f$\dash flatspaces, that is, 
\begin{equation}\tag{$\ast$}\label{eq:ikenstar}
\delta(k,d) = 0 \text{ for all diagonals $k$ of non\dash $f$\dash flat rhombi} 
\end{equation}

\begin{claim}
\label{cla:connectedness:noconclockcycles}
In the situation $(\ast)$, $\SUPP(d)$ does not contain a path in~$G$ consisting of two consecutive clockwise turns.
\end{claim}

\begin{proof}
Assume the contrary.
We create a contradiction by using three type of arguments:
(A) Lemma~\ref{cla:negimpliespos} on antipodal pairs, 
(B) the flow conservation laws, 
and the fact~$(\ast)$.

The following sequence of pictures shows paths contained in $\SUPP(d)$ and how rules (A), (B) and $(\ast)$ 
imply that additional paths are contained in $\SUPP(d)$. 
All rhombi that are known to be $f$\dash flat are drawn shaded:

\hspace{-1.5cm}\begin{tikzpicture}
\node[] (TOP) at (1,1.5){
\begin{tikzpicture}\fill[thick,fill=black!20] (-87.761pt,8.0001pt) -- (-83.142pt,16.0002pt) -- (-73.904pt,16.0002pt) -- (-78.523pt,8.0001pt) -- cycle;\fill[thick,fill=black!20] (138.57pt,16.0002pt) -- (147.808pt,16.0002pt) -- (152.427pt,24.0003pt) -- (157.046pt,16.0002pt) -- (166.284pt,16.0002pt) -- (157.046pt,0.0pt) -- (138.57pt,0.0pt) -- (143.189pt,8.0001pt) -- cycle;\fill[thick,fill=black!20] (106.237pt,24.0003pt) -- (110.856pt,16.0002pt) -- (120.094pt,16.0002pt) -- (110.856pt,0.0pt) -- (92.38pt,0.0pt) -- (96.999pt,8.0001pt) -- (92.38pt,16.0002pt) -- (101.618pt,16.0002pt) -- cycle;\fill[thick,fill=black!20] (-32.333pt,24.0003pt) -- (-36.952pt,16.0002pt) -- (-46.19pt,16.0002pt) -- (-36.952pt,0.0pt) -- (-32.333pt,8.0001pt) -- (-23.095pt,8.0001pt) -- cycle;\fill[thick,fill=black!20] (13.857pt,24.0003pt) -- (23.095pt,8.0001pt) -- (13.857pt,8.0001pt) -- (9.238pt,0.0pt) -- (0.0pt,16.0002pt) -- (9.238pt,16.0002pt) -- cycle;\fill[thick,fill=black!20] (46.19pt,16.0002pt) -- (50.809pt,8.0001pt) -- (46.19pt,0.0pt) -- (64.666pt,0.0pt) -- (73.904pt,16.0002pt) -- (64.666pt,16.0002pt) -- (60.047pt,24.0003pt) -- (55.428pt,16.0002pt) -- cycle;\fill[thick,fill=black!20] (-106.237pt,8.0001pt) -- (-115.475pt,8.0001pt) -- (-110.856pt,16.0002pt) -- (-101.618pt,16.0002pt) -- cycle;\draw[thin,-my] (-110.856pt,8.0001pt) arc (180:120:4.619pt);\draw[thin,-my] (-108.5465pt,12.0002pt) arc (120:60:4.619pt);\draw[thin,-my] (-36.952pt,8.0001pt) arc (180:120:4.619pt);\draw[thin,-my] (-34.6425pt,12.0002pt) arc (120:60:4.619pt);\draw[thin,-my] (-39.2615pt,12.0002pt) arc (240:300:4.619pt);\draw[thin,-my] (-34.6425pt,12.0002pt) arc (300:360:4.619pt);\draw[thin,-my] (6.9285pt,12.0002pt) arc (240:300:4.619pt);\draw[thin,-my] (11.5475pt,12.0002pt) arc (300:360:4.619pt);\draw[thin,-my] (9.238pt,8.0001pt) arc (180:120:4.619pt);\draw[thin,-my] (11.5475pt,12.0002pt) arc (120:60:4.619pt);\draw[thin,-my] (6.9285pt,4.0001pt) arc (120:60:4.619pt);\draw[thin,-my] (18.476pt,8.0001pt) arc (180:120:4.619pt);\draw[thin,-my] (53.1185pt,12.0002pt) arc (240:300:4.619pt);\draw[thin,-my] (57.7375pt,12.0002pt) arc (300:360:4.619pt);\draw[thin,-my] (55.428pt,8.0001pt) arc (180:120:4.619pt);\draw[thin,-my] (57.7375pt,12.0002pt) arc (120:60:4.619pt);\draw[thin,-my] (64.666pt,8.0001pt) arc (180:120:4.619pt);\draw[thin,-my] (53.1185pt,4.0001pt) arc (120:60:4.619pt);\draw[thin,-my] (99.3085pt,12.0002pt) arc (240:300:4.619pt);\draw[thin,-my] (103.9275pt,12.0002pt) arc (300:360:4.619pt);\draw[thin,-my] (101.618pt,8.0001pt) arc (180:120:4.619pt);\draw[thin,-my] (103.9275pt,12.0002pt) arc (120:60:4.619pt);\draw[thin,-my] (110.856pt,8.0001pt) arc (180:120:4.619pt);\draw[thin,-my] (99.3085pt,4.0001pt) arc (120:60:4.619pt);\draw[thin,-my] (99.3085pt,4.0001pt) arc (300:360:4.619pt);\draw[thin,-my] (108.5465pt,12.0002pt) arc (240:300:4.619pt);\draw[thin,-my] (145.4985pt,12.0002pt) arc (240:300:4.619pt);\draw[thin,-my] (150.1175pt,12.0002pt) arc (300:360:4.619pt);\draw[thin,-my] (147.808pt,8.0001pt) arc (180:120:4.619pt);\draw[thin,-my] (150.1175pt,12.0002pt) arc (120:60:4.619pt);\draw[thin,-my] (154.7365pt,12.0002pt) arc (240:300:4.619pt);\draw[thin,-my] (157.046pt,8.0001pt) arc (180:120:4.619pt);\draw[thin,-my] (145.4985pt,4.0001pt) arc (300:360:4.619pt);\draw[thin,-my] (145.4985pt,4.0001pt) arc (120:60:4.619pt);\draw[thin,-my] (152.427pt,0.0pt) arc (180:120:4.619pt);\draw[thin,-my] (154.7365pt,4.0001pt) arc (120:60:4.619pt);\node at (-50.809pt,8.0001pt) {{$\stackrel{\text{($\ast$)}}{\to}$}};\node at (-4.619pt,8.0001pt) {{$\stackrel{\text{\tiny(A)}}{\to}$}};\node at (41.571pt,8.0001pt) {{$\stackrel{\text{($\ast$)}}{\to}$}};\node at (87.761pt,8.0001pt) {{$\stackrel{\text{\tiny(B)}}{\to}$}};\node at (133.951pt,8.0001pt) {{$\stackrel{\text{\tiny(A)}}{\to}$}};\node at (-96.999pt,8.0001pt) {{$\stackrel{\text{\tiny(A)}}{\to}$}};\draw[thin,-my] (-83.142pt,8.0001pt) arc (180:120:4.619pt);\draw[thin,-my] (-85.4515pt,12.0002pt) arc (240:300:4.619pt);\draw[thin,-my] (-80.8325pt,12.0002pt) arc (300:360:4.619pt);\draw[thin,-my] (-80.8325pt,12.0002pt) arc (120:60:4.619pt);\end{tikzpicture}
};
\node[] (BOTTOM) at (1.57,0) {
\begin{tikzpicture}\fill[thick,fill=black!20] (415.71pt,16.0002pt) -- (424.948pt,16.0002pt) -- (429.567pt,24.0003pt) -- (434.186pt,16.0002pt) -- (443.424pt,16.0002pt) -- (438.805pt,8.0001pt) -- (448.043pt,8.0001pt) -- (438.805pt,-8.0001pt) -- (420.329pt,-8.0001pt) -- (424.948pt,0.0pt) -- (415.71pt,0.0pt) -- (420.329pt,8.0001pt) -- cycle;\fill[thick,fill=black!20] (369.52pt,16.0002pt) -- (378.758pt,16.0002pt) -- (383.377pt,24.0003pt) -- (387.996pt,16.0002pt) -- (397.234pt,16.0002pt) -- (392.615pt,8.0001pt) -- (401.853pt,8.0001pt) -- (392.615pt,-8.0001pt) -- (374.139pt,-8.0001pt) -- (378.758pt,0.0pt) -- (369.52pt,0.0pt) -- (374.139pt,8.0001pt) -- cycle;\fill[thick,fill=black!20] (277.14pt,16.0002pt) -- (286.378pt,16.0002pt) -- (290.997pt,24.0003pt) -- (295.616pt,16.0002pt) -- (304.854pt,16.0002pt) -- (300.235pt,8.0001pt) -- (304.854pt,0.0pt) -- (295.616pt,0.0pt) -- (290.997pt,-8.0001pt) -- (286.378pt,0.0pt) -- (277.14pt,0.0pt) -- (281.759pt,8.0001pt) -- cycle;\fill[thick,fill=black!20] (230.95pt,16.0002pt) -- (240.188pt,16.0002pt) -- (244.807pt,24.0003pt) -- (249.426pt,16.0002pt) -- (258.664pt,16.0002pt) -- (254.045pt,8.0001pt) -- (258.664pt,0.0pt) -- (249.426pt,0.0pt) -- (244.807pt,-8.0001pt) -- (240.188pt,0.0pt) -- (230.95pt,0.0pt) -- (235.569pt,8.0001pt) -- cycle;\fill[thick,fill=black!20] (184.76pt,16.0002pt) -- (193.998pt,16.0002pt) -- (198.617pt,24.0003pt) -- (203.236pt,16.0002pt) -- (212.474pt,16.0002pt) -- (203.236pt,0.0pt) -- (184.76pt,0.0pt) -- (189.379pt,8.0001pt) -- cycle;\fill[thick,fill=black!20] (323.33pt,16.0002pt) -- (332.568pt,16.0002pt) -- (337.187pt,24.0003pt) -- (341.806pt,16.0002pt) -- (351.044pt,16.0002pt) -- (346.425pt,8.0001pt) -- (355.663pt,8.0001pt) -- (346.425pt,-8.0001pt) -- (327.949pt,-8.0001pt) -- (332.568pt,0.0pt) -- (323.33pt,0.0pt) -- (327.949pt,8.0001pt) -- cycle;\draw[thin,-my] (191.6885pt,12.0002pt) arc (240:300:4.619pt);\draw[thin,-my] (196.3075pt,12.0002pt) arc (300:360:4.619pt);\draw[thin,-my] (193.998pt,8.0001pt) arc (180:120:4.619pt);\draw[thin,-my] (196.3075pt,12.0002pt) arc (120:60:4.619pt);\draw[thin,-my] (200.9265pt,12.0002pt) arc (240:300:4.619pt);\draw[thin,-my] (203.236pt,8.0001pt) arc (180:120:4.619pt);\draw[thin,-my] (191.6885pt,4.0001pt) arc (300:360:4.619pt);\draw[thin,-my] (191.6885pt,4.0001pt) arc (120:60:4.619pt);\draw[thin,-my] (198.617pt,0.0pt) arc (180:120:4.619pt);\draw[thin,-my] (200.9265pt,4.0001pt) arc (120:60:4.619pt);\draw[thin,-my] (196.3075pt,4.0001pt) arc (240:300:4.619pt);\draw[thin,-my] (200.9265pt,4.0001pt) arc (300:360:4.619pt);\draw[thin,-my] (237.8785pt,12.0002pt) arc (240:300:4.619pt);\draw[thin,-my] (242.4975pt,12.0002pt) arc (300:360:4.619pt);\draw[thin,-my] (240.188pt,8.0001pt) arc (180:120:4.619pt);\draw[thin,-my] (242.4975pt,12.0002pt) arc (120:60:4.619pt);\draw[thin,-my] (247.1165pt,12.0002pt) arc (240:300:4.619pt);\draw[thin,-my] (249.426pt,8.0001pt) arc (180:120:4.619pt);\draw[thin,-my] (237.8785pt,4.0001pt) arc (300:360:4.619pt);\draw[thin,-my] (237.8785pt,4.0001pt) arc (120:60:4.619pt);\draw[thin,-my] (242.4975pt,4.0001pt) arc (240:300:4.619pt);\draw[thin,-my] (244.807pt,0.0pt) arc (180:120:4.619pt);\draw[thin,-my] (247.1165pt,4.0001pt) arc (300:360:4.619pt);\draw[thin,-my] (247.1165pt,4.0001pt) arc (120:60:4.619pt);\draw[thin,-my] (284.0685pt,12.0002pt) arc (240:300:4.619pt);\draw[thin,-my] (288.6875pt,12.0002pt) arc (300:360:4.619pt);\draw[thin,-my] (286.378pt,8.0001pt) arc (180:120:4.619pt);\draw[thin,-my] (288.6875pt,12.0002pt) arc (120:60:4.619pt);\draw[thin,-my] (284.0685pt,4.0001pt) arc (300:360:4.619pt);\draw[thin,-my] (284.0685pt,4.0001pt) arc (120:60:4.619pt);\draw[thin,-my] (290.997pt,0.0pt) arc (180:120:4.619pt);\draw[thin,-my] (288.6875pt,4.0001pt) arc (240:300:4.619pt);\draw[thin,-my] (293.3065pt,4.0001pt) arc (300:360:4.619pt);\draw[thin,-my] (293.3065pt,4.0001pt) arc (120:60:4.619pt);\draw[thin,-my] (295.616pt,8.0001pt) arc (180:120:4.619pt);\draw[thin,-my] (293.3065pt,12.0002pt) arc (240:300:4.619pt);\draw[thin,-my] (288.6875pt,-4.0001pt) arc (120:60:4.619pt);\draw[thin,-my] (300.235pt,0.0pt) arc (180:120:4.619pt);\node at (364.901pt,8.0001pt) {{$\stackrel{\text{\tiny(B)}}{\to}$}};\node at (411.091pt,8.0001pt) {{$\stackrel{\text{\tiny(A)}}{\to}$}};\node at (466.519pt,8.0001pt) {{$\stackrel{\cdots}{\to}$}};\draw[thin,-my] (330.2585pt,12.0002pt) arc (240:300:4.619pt);\draw[thin,-my] (334.8775pt,12.0002pt) arc (300:360:4.619pt);\draw[thin,-my] (334.8775pt,12.0002pt) arc (120:60:4.619pt);\draw[thin,-my] (332.568pt,8.0001pt) arc (180:120:4.619pt);\draw[thin,-my] (330.2585pt,4.0001pt) arc (120:60:4.619pt);\draw[thin,-my] (330.2585pt,4.0001pt) arc (300:360:4.619pt);\draw[thin,-my] (334.8775pt,4.0001pt) arc (240:300:4.619pt);\draw[thin,-my] (337.187pt,0.0pt) arc (180:120:4.619pt);\draw[thin,-my] (339.4965pt,4.0001pt) arc (300:360:4.619pt);\draw[thin,-my] (339.4965pt,4.0001pt) arc (120:60:4.619pt);\draw[thin,-my] (346.425pt,0.0pt) arc (180:120:4.619pt);\draw[thin,-my] (334.8775pt,-4.0001pt) arc (120:60:4.619pt);\draw[thin,-my] (376.4485pt,12.0002pt) arc (240:300:4.619pt);\draw[thin,-my] (381.0675pt,12.0002pt) arc (300:360:4.619pt);\draw[thin,-my] (378.758pt,8.0001pt) arc (180:120:4.619pt);\draw[thin,-my] (381.0675pt,12.0002pt) arc (120:60:4.619pt);\draw[thin,-my] (376.4485pt,4.0001pt) arc (300:360:4.619pt);\draw[thin,-my] (376.4485pt,4.0001pt) arc (120:60:4.619pt);\draw[thin,-my] (383.377pt,0.0pt) arc (180:120:4.619pt);\draw[thin,-my] (381.0675pt,4.0001pt) arc (240:300:4.619pt);\draw[thin,-my] (385.6865pt,4.0001pt) arc (300:360:4.619pt);\draw[thin,-my] (385.6865pt,4.0001pt) arc (120:60:4.619pt);\draw[thin,-my] (392.615pt,0.0pt) arc (180:120:4.619pt);\draw[thin,-my] (381.0675pt,-4.0001pt) arc (120:60:4.619pt);\draw[thin,-my] (341.806pt,8.0001pt) arc (180:120:4.619pt);\draw[thin,-my] (339.4965pt,12.0002pt) arc (240:300:4.619pt);\draw[thin,-my] (387.996pt,8.0001pt) arc (180:120:4.619pt);\draw[thin,-my] (385.6865pt,12.0002pt) arc (240:300:4.619pt);\draw[thin,-my] (381.0675pt,-4.0001pt) arc (300:360:4.619pt);\draw[thin,-my] (390.3055pt,4.0001pt) arc (240:300:4.619pt);\draw[thin,-my] (434.186pt,-8.0001pt) arc (180:120:4.619pt);\draw[thin,-my] (436.4955pt,-4.0001pt) arc (120:60:4.619pt);\draw[thin,-my] (422.6385pt,12.0002pt) arc (240:300:4.619pt);\draw[thin,-my] (427.2575pt,12.0002pt) arc (300:360:4.619pt);\draw[thin,-my] (424.948pt,8.0001pt) arc (180:120:4.619pt);\draw[thin,-my] (427.2575pt,12.0002pt) arc (120:60:4.619pt);\draw[thin,-my] (431.8765pt,12.0002pt) arc (240:300:4.619pt);\draw[thin,-my] (434.186pt,8.0001pt) arc (180:120:4.619pt);\draw[thin,-my] (422.6385pt,4.0001pt) arc (300:360:4.619pt);\draw[thin,-my] (422.6385pt,4.0001pt) arc (120:60:4.619pt);\draw[thin,-my] (427.2575pt,4.0001pt) arc (240:300:4.619pt);\draw[thin,-my] (431.8765pt,4.0001pt) arc (300:360:4.619pt);\draw[thin,-my] (429.567pt,0.0pt) arc (180:120:4.619pt);\draw[thin,-my] (431.8765pt,4.0001pt) arc (120:60:4.619pt);\draw[thin,-my] (436.4955pt,4.0001pt) arc (240:300:4.619pt);\draw[thin,-my] (427.2575pt,-4.0001pt) arc (300:360:4.619pt);\draw[thin,-my] (427.2575pt,-4.0001pt) arc (120:60:4.619pt);\draw[thin,-my] (438.805pt,0.0pt) arc (180:120:4.619pt);\node at (318.711pt,8.0001pt) {{$\stackrel{\text{($\ast$)}}{\to}$}};\node at (272.521pt,8.0001pt) {{$\stackrel{\text{\tiny(A)}}{\to}$}};\node at (226.331pt,8.0001pt) {{$\stackrel{\text{($\ast$)}}{\to}$}};\node at (180.141pt,8.0001pt) {{$\stackrel{\text{\tiny(A)}}{\to}$}};\fill[white] (161.665pt,8.0001pt) circle (0.1pt);\end{tikzpicture}
};
\draw[-] (6,1.3) .. controls (9,1.3) and (7.25,0.8) .. (1.25,0.8);
\draw[-] (1.25,0.8) .. controls (-4.75,0.8) and (-6.5,-0.1) .. (-3.3,-0.1);
\end{tikzpicture}

The process of repeatedly applying \big[(A), $(\ast)$, (A), $(\ast)$, (B), (A)\big]
can be continued infinitely while extending the $f$\dash flat region to the lower right.
This is a contradiction to the finite size of~$\Delta$.
\end{proof}

By Lemma~\ref{le:flow-decomp} we have a decomposition 
$d=\sum_j \alpha_j c_j$ with $\alpha_j >0$ and cycles~$c_j$ in $G$ that are 
contained in $\SUPP(d)$. 
According to $(\ast)$, each $c_j$ runs in a single $f$\dash flatspace
and does not cross any $f$\dash flatspace border. 
Let $c$ be any of the cycles $c_j$ and suppose that $c$ runs inside the $f$\dash flatspace~$L$.
Claim~\ref{cla:connectedness:noconclockcycles} implies that $c$ runs in counterclockwise direction.
We will next show that $L$ is a hexagon. 

Let $\gamma$ denote the polygon (without self-intersections) 
obtained from~$c$ by linearly interpolating between the successive white vertices of~$c$.
Following the white vertices of~$c$ (in counterclockwise order) reveals that 
two consecutive counterclockwise turns lead to an angle of 120° in~$\gamma$.
Further, an alternating sequence of clockwise and counterclockwise turns in~$c$ is represented by a line segment in~$\gamma$.
By an elementary geometric argument we see that $\gamma$ must be a hexagon.


Let $\tilde c$ be the counterclockwise cycle \emph{surrounding~$c$}: more specifically, 
$\tilde c$~consists of the clockwise antipodal contributions of all counterclockwise turns in $c$
and, additionally, of the necessary counterclockwise turns in between, as illustrated below:
\begin{center}
\begin{tikzpicture}\draw[thin,-my] (2.3095pt,-4.0001pt) arc (300:360:4.619pt);\draw[thin,-my] (4.619pt,0.0pt) arc (0:60:4.619pt);\draw[thin,-my] (2.3095pt,4.0001pt) arc (240:180:4.619pt);\draw[thin,-my] (0.0pt,8.0001pt) arc (0:60:4.619pt);\draw[thin,-my] (-2.3095pt,12.0002pt) arc (60:120:4.619pt);\draw[thin,-my] (-6.9285pt,12.0002pt) arc (300:240:4.619pt);\draw[thin,-my] (-11.5475pt,12.0002pt) arc (60:120:4.619pt);\draw[thin,-my] (-16.1665pt,12.0002pt) arc (120:180:4.619pt);\draw[thin,-my] (-18.476pt,8.0001pt) arc (0:-60:4.619pt);\draw[thin,-my] (-20.7855pt,4.0001pt) arc (120:180:4.619pt);\draw[thin,-my] (-23.095pt,0.0pt) arc (180:240:4.619pt);\draw[thin,-my] (-20.7855pt,-4.0001pt) arc (240:300:4.619pt);\draw[thin,-my] (-16.1665pt,-4.0001pt) arc (120:60:4.619pt);\draw[thin,-my] (-11.5475pt,-4.0001pt) arc (240:300:4.619pt);\draw[thin,-my] (-6.9285pt,-4.0001pt) arc (120:60:4.619pt);\draw[thin,-my] (-2.3095pt,-4.0001pt) arc (240:300:4.619pt);\draw[thin,-my] (9.238pt,-8.0001pt) arc (180:120:4.619pt);\draw[thin,-my] (11.5475pt,-4.0001pt) arc (300:360:4.619pt);\draw[thin,-my] (13.857pt,0.0pt) arc (0:60:4.619pt);\draw[thin,-my] (11.5475pt,4.0001pt) arc (240:180:4.619pt);\draw[thin,-my] (9.238pt,8.0001pt) arc (0:60:4.619pt);\draw[thin,-my] (6.9285pt,12.0002pt) arc (240:180:4.619pt);\draw[thin,-my] (4.619pt,16.0002pt) arc (0:60:4.619pt);\draw[thin,-my] (2.3095pt,20.0003pt) arc (60:120:4.619pt);\draw[thin,-my] (-2.3095pt,20.0003pt) arc (300:240:4.619pt);\draw[thin,-my] (-6.9285pt,20.0003pt) arc (60:120:4.619pt);\draw[thin,-my] (-11.5475pt,20.0003pt) arc (300:240:4.619pt);\draw[thin,-my] (-16.1665pt,20.0003pt) arc (60:120:4.619pt);\draw[thin,-my] (-20.7855pt,20.0003pt) arc (120:180:4.619pt);\draw[thin,-my] (-23.095pt,16.0002pt) arc (0:-60:4.619pt);\draw[thin,-my] (-25.4045pt,12.0002pt) arc (120:180:4.619pt);\draw[thin,-my] (-27.714pt,8.0001pt) arc (0:-60:4.619pt);\draw[thin,-my] (-30.0235pt,4.0001pt) arc (120:180:4.619pt);\draw[thin,-my] (-32.333pt,0.0pt) arc (180:240:4.619pt);\draw[thin,-my] (-30.0235pt,-4.0001pt) arc (60:0:4.619pt);\draw[thin,-my] (-27.714pt,-8.0001pt) arc (180:240:4.619pt);\draw[thin,-my] (-25.4045pt,-12.0002pt) arc (240:300:4.619pt);\draw[thin,-my] (-20.7855pt,-12.0002pt) arc (120:60:4.619pt);\draw[thin,-my] (-16.1665pt,-12.0002pt) arc (240:300:4.619pt);\draw[thin,-my] (-11.5475pt,-12.0002pt) arc (120:60:4.619pt);\draw[thin,-my] (-6.9285pt,-12.0002pt) arc (240:300:4.619pt);\draw[thin,-my] (-2.3095pt,-12.0002pt) arc (120:60:4.619pt);\draw[thin,-my] (2.3095pt,-12.0002pt) arc (240:300:4.619pt);\draw[thin,-my] (6.9285pt,-12.0002pt) arc (300:360:4.619pt);\draw[rhrhombidraw] (-32.333pt,24.0003pt) -- (13.857pt,24.0003pt) ;\draw[rhrhombidraw] (-36.952pt,-16.0002pt) -- (18.476pt,-16.0002pt) ;\draw[rhrhombidraw] (-23.095pt,-24.0003pt) -- (-41.571pt,8.0001pt) ;\draw[rhrhombidraw] (-13.857pt,-24.0003pt) -- (-36.952pt,16.0002pt) ;\draw[rhrhombidraw] (-4.619pt,-24.0003pt) -- (-32.333pt,24.0003pt) ;\draw[rhrhombidraw] (4.619pt,-24.0003pt) -- (-27.714pt,32.0004pt) ;\draw[rhrhombidraw] (13.857pt,-24.0003pt) -- (-18.476pt,32.0004pt) ;\draw[rhrhombidraw] (-18.476pt,32.0004pt) -- (-41.571pt,-8.0001pt) ;\draw[rhrhombidraw] (4.619pt,-24.0003pt) -- (23.095pt,8.0001pt) ;\draw[rhrhombidraw] (23.095pt,-8.0001pt) -- (0.0pt,32.0004pt) ;\draw[rhrhombidraw] (-36.952pt,16.0002pt) -- (18.476pt,16.0002pt) ;\draw[rhrhombidraw] (-41.571pt,8.0001pt) -- (23.095pt,8.0001pt) ;\draw[rhrhombidraw] (-46.19pt,0.0pt) -- (27.714pt,0.0pt) ;\draw[rhrhombidraw] (-41.571pt,-8.0001pt) -- (23.095pt,-8.0001pt) ;\draw[rhrhombidraw] (-9.238pt,32.0004pt) -- (-36.952pt,-16.0002pt) ;\draw[rhrhombidraw] (0.0pt,32.0004pt) -- (-32.333pt,-24.0003pt) ;\draw[rhrhombidraw] (9.238pt,32.0004pt) -- (-23.095pt,-24.0003pt) ;\draw[rhrhombidraw] (13.857pt,24.0003pt) -- (-13.857pt,-24.0003pt) ;\draw[rhrhombidraw] (18.476pt,16.0002pt) -- (-4.619pt,-24.0003pt) ;\draw[rhrhombidraw] (18.476pt,-16.0002pt) -- (-9.238pt,32.0004pt) ;\end{tikzpicture}
\end{center}
The Flow Propagation Lemma~\ref{lem:flowpropagation} implies that all turns of $\tilde c$
lying inside~$L$ are contained in~$\SUPP(d)$.
Hence, by \eqref{eq:ikenstar}, $\tilde c$ cannot pass the border of~$L$.
Therefore, $\tilde c$ either lies completely inside~$L$ or completely outside~$L$.
If $\tilde c$ lies completely inside~$L$, we can form  the cycle surrounding $\tilde{c}$ 
and continue inductively, until 
we find a cycle $c' \subseteq \SUPP(d)$ which lies inside~$L$ and 
such that $\tilde{c}'$ lies outside~$L$. 
Since the polygon $\gamma'$ corresponding to $c'$ is a hexagon,
it follows that $L$ must be a hexagon. 
Summarizing, we see that the cycle $c'$ runs counterclockwise 
through the border triangles of a hexagon~$L$.
Such $c'$ is clearly $f$\dash secure.
Moreover, since $c' \subseteq \SUPP(d)$, we have $\dist(f+c',g) < \dist(f,g)$ 
by  to Lemma~\ref{cla:connectedness:morethanhalf}.
This proves Proposition~\ref{pro:connectedness:existscycle} in Situation~1.
\smallskip

\noindent{\bf Situation 2.}
We now treat the case where $d$ has nonzero throughput through some edge~$k$ 
of an $f$\dash flatspace~$L$. 
By ($\dag$), $k$ is not at the border of $\Delta$.
By Lemma~\ref{obs:borderentranceexit}, 
we can assume w.l.o.g.\ that $k$ is an $L$\dash entrance edge and $\inflow L k d >0$.
Let $p\subseteq\SUPP(d)$ be a turn in~$L$ starting at~$k$.

We will show that the following Algorithm~\ref{alg:connectedness} constructs a desired~$c$.

\begin{algorithm}[H]
\caption{}\nopar \label{alg:connectedness}
\begin{algorithmic}[1]
\REQUIRE $f,g \in P(\la,\mu,\nu)_\Z$, an edge~$k$ such that $\inflow L k d >0$, where $d:=g - f$,
                 and a turn~$p$ in $\SUPP(d)$ starting at~$k$.
\ENSURE An $f$\dash secure cycle $c$ such that $\dist(f+c,g) < \dist(f,g)$.
    \WHILE{$p$ does not contain a vertex more than once}
        \IF{one can append to $p$ a turn $\vartheta$ contained in $\SUPP(d)$ such that 
             $p$ after appending does not use a negative contribution in any $f$\dash flat rhombus\nopar\label{alg:connectedness:if}}
              \STATE Append $\vartheta$ to $p$.\nopar\label{alg:connectedness:canonical}
        \ELSE
              \STATE Append a clockwise turn followed by a counterclockwise turn to $p$.\nopar\label{alg:connectedness:swerve}
        \ENDIF
    \ENDWHILE\nopar\label{alg:connectedness:endwhile}
    \STATE Generate a cycle $c$ from the edges of~$p$ by ``truncation and concatenation''.\nopar\label{alg:connectedness:concat}
    \RETURN $c$.
\end{algorithmic}
\end{algorithm}

We postpone the definition of the procedure used in line~\ref{alg:connectedness:concat} for the construction of $c$ from~$p$.
Later on, we will give a precedence rule to determine 
what should happen in line~\ref{alg:connectedness:if} 
when both turns, clockwise and counterclockwise, are possible to append.

What is striking about Algorithm~\ref{alg:connectedness}
is that it is a priori unclear that line~\ref{alg:connectedness:swerve} can be executed (without $p$ leaving $\Delta$).
We next explain why this is the case.

To ease notation we index the intermediate results that occur during the construction of~$p$
by $p_0, p_1, \ldots$, where $p_i$ either has one or two more turns than~$p_{i-1}$,
depending on whether in the while loop there has been appended only one turn or (in case of line~\ref{alg:connectedness:swerve}) two turns
to~$p_{i-1}$.
The paths $q_i$ are defined such that each $p_{i+1}$ is the result of the concatenation of $p_i$ and $q_i$.
We denote by the term \emph{swerve} each $q_i$ that is not a single turn,
i.e.\ those $q_i$ that consist of a clockwise turn followed by a counterclockwise turn.
For a swerve $q_i$ we denote by $\varrho(q_i)$ the rhombus in which both turns of $q_i$ lie.

\begin{claim}
\label{cla:connectedness:welldefined}
For all $i$ we have the following properties:
\begin{remunerate}
 \item[(1)] Let $\rhpoulMrXourMl \in \{\rhpoulMr,\rhpourMl\}$ denote the last turn of~$p_{i}$ and
suppose that line~\ref{alg:connectedness:swerve} is about to be executed.
Then $\rhsoll \in E(\Delta)$ is not at the border of $\Delta$, which means that $q_i=\rhpcWrl$ can be appended to~$p_{i}$ in line~\ref{alg:connectedness:swerve} without leaving~$\Delta$.
 \item[(2)] The first and last edge of each $q_i$ are contained in $\SUPP(d)$.
 \item[(3)] Each $p_i$ does not use negative contributions in $f$\dash flat rhombi.
 \item[(4)] The rhombus $\varrho(q_i)$ is $f$\dash flat for each swerve~$q_i$.
\end{remunerate}
\end{claim}

Before proving Claim~\ref{cla:connectedness:welldefined} we start out with a fairly easy lemma that will prove useful.

\begin{lemma}
\label{lem:connectedness:hivepreservingcyclesonG}
Given a walk $p$ in $G$ starting with a turn at a side of an $f$\dash flatspace, 
for some fixed $f\in B$.
Further assume that $p$ does not use negative contributions in $f$\dash flat rhombi.
If the trapezoid
\begin{tikzpictured}\draw[rhrhombidraw] (0.0pt,0.0pt) -- (9.238pt,0.0pt) -- (13.857pt,8.0001pt) -- (9.238pt,16.0002pt) -- cycle;\fill (0.0pt,0.0pt) circle (1pt);\fill (9.238pt,0.0pt) circle (1pt);\fill (4.619pt,8.0001pt) circle (1pt);\fill (13.857pt,8.0001pt) circle (1pt);\fill (9.238pt,16.0002pt) circle (1pt);\end{tikzpictured}
consists of two overlapping $f$\dash flat rhombi, then $p$ does not end with one of the two turns
\begin{tikzpictured}\draw[rhrhombidraw] (0.0pt,0.0pt) -- (9.238pt,0.0pt) -- (13.857pt,8.0001pt) -- (9.238pt,16.0002pt) -- cycle;\fill (0.0pt,0.0pt) circle (1pt);\fill (9.238pt,0.0pt) circle (1pt);\fill (4.619pt,8.0001pt) circle (1pt);\fill (13.857pt,8.0001pt) circle (1pt);\fill (9.238pt,16.0002pt) circle (1pt);\draw[thin,-my] (4.619pt,16.0002pt) arc (180:240:4.619pt);\draw[thin,-my] (2.3095pt,12.0002pt) arc (120:60:4.619pt);\end{tikzpictured}%
.
\end{lemma}

{\em Proof}.
According to the hexagon inequality~(\ref{cla:BZ}), both trapezoids
\begin{tikzpictured}\draw[rhrhombidraw] (0.0pt,0.0pt) -- (9.238pt,0.0pt) -- (13.857pt,8.0001pt) -- (9.238pt,16.0002pt) -- cycle;\fill (0.0pt,0.0pt) circle (1pt);\fill (9.238pt,0.0pt) circle (1pt);\fill (4.619pt,8.0001pt) circle (1pt);\fill (13.857pt,8.0001pt) circle (1pt);\fill (9.238pt,16.0002pt) circle (1pt);\draw[rhrhombidraw] (0.0pt,0.0pt) -- (-4.619pt,8.0001pt) -- (0.0pt,16.0002pt) -- (9.238pt,16.0002pt) ;\fill (-4.619pt,8.0001pt) circle (1pt);\fill (0.0pt,16.0002pt) circle (1pt);\end{tikzpictured}
are $f$\dash flat.
Note that the following three possible cases, which could precede
\begin{tikzpictured}\draw[rhrhombidraw] (0.0pt,0.0pt) -- (9.238pt,0.0pt) -- (13.857pt,8.0001pt) -- (9.238pt,16.0002pt) -- cycle;\fill (0.0pt,0.0pt) circle (1pt);\fill (9.238pt,0.0pt) circle (1pt);\fill (4.619pt,8.0001pt) circle (1pt);\fill (13.857pt,8.0001pt) circle (1pt);\fill (9.238pt,16.0002pt) circle (1pt);\draw[thin,-my] (4.619pt,16.0002pt) arc (180:240:4.619pt);\draw[thin,-my] (2.3095pt,12.0002pt) arc (120:60:4.619pt);\end{tikzpictured}%
, all use negative contributions in $f$\dash flat rhombi, which contradicts our assumption:
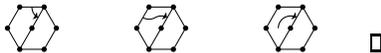
\[
\begin{tikzpicture}\draw[rhrhombidraw] (0.0pt,0.0pt) -- (9.238pt,0.0pt) -- (13.857pt,8.0001pt) -- (9.238pt,16.0002pt) -- cycle;\fill (0.0pt,0.0pt) circle (1pt);\fill (9.238pt,0.0pt) circle (1pt);\fill (4.619pt,8.0001pt) circle (1pt);\fill (13.857pt,8.0001pt) circle (1pt);\fill (9.238pt,16.0002pt) circle (1pt);\draw[rhrhombidraw] (0.0pt,0.0pt) -- (-4.619pt,8.0001pt) -- (0.0pt,16.0002pt) -- (9.238pt,16.0002pt) ;\fill (-4.619pt,8.0001pt) circle (1pt);\fill (0.0pt,16.0002pt) circle (1pt);\draw[thin,-my] (4.619pt,16.0002pt) arc (180:240:4.619pt);\end{tikzpicture}
\hspace{1cm}
\begin{tikzpicture}\draw[rhrhombidraw] (0.0pt,0.0pt) -- (9.238pt,0.0pt) -- (13.857pt,8.0001pt) -- (9.238pt,16.0002pt) -- cycle;\fill (0.0pt,0.0pt) circle (1pt);\fill (9.238pt,0.0pt) circle (1pt);\fill (4.619pt,8.0001pt) circle (1pt);\fill (13.857pt,8.0001pt) circle (1pt);\fill (9.238pt,16.0002pt) circle (1pt);\draw[rhrhombidraw] (0.0pt,0.0pt) -- (-4.619pt,8.0001pt) -- (0.0pt,16.0002pt) -- (9.238pt,16.0002pt) ;\fill (-4.619pt,8.0001pt) circle (1pt);\fill (0.0pt,16.0002pt) circle (1pt);\draw[thin,-my] (-2.3095pt,12.0002pt) arc (240:300:4.619pt) arc (120:60:4.619pt);\end{tikzpicture}
\hspace{1cm}
\begin{tikzpicture}\draw[rhrhombidraw] (0.0pt,0.0pt) -- (9.238pt,0.0pt) -- (13.857pt,8.0001pt) -- (9.238pt,16.0002pt) -- cycle;\fill (0.0pt,0.0pt) circle (1pt);\fill (9.238pt,0.0pt) circle (1pt);\fill (4.619pt,8.0001pt) circle (1pt);\fill (13.857pt,8.0001pt) circle (1pt);\fill (9.238pt,16.0002pt) circle (1pt);\draw[rhrhombidraw] (0.0pt,0.0pt) -- (-4.619pt,8.0001pt) -- (0.0pt,16.0002pt) -- (9.238pt,16.0002pt) ;\fill (-4.619pt,8.0001pt) circle (1pt);\fill (0.0pt,16.0002pt) circle (1pt);\draw[thin,-my] (0.0pt,8.0001pt) arc (180:120:4.619pt) arc (120:60:4.619pt);\end{tikzpicture}
\qquad\endproof
\]

\begin{prooff}[Proof of Claim~\textup{\ref{cla:connectedness:welldefined}}]
We prove all claims simultaneously by induction on~$i$.
If $q_{i-1}$ is a single turn, then $q_{i-1}$ is supported by~$d$
and (by definition of the if-clause in Algorithm~\ref{alg:connectedness}) $p_i$
does not use any negative contributions in $f$\dash flat rhombi, which proves (2) and (3) in this case.

It remains to consider the case where $q_{i-1}$ is a swerve, that is, line~\ref{alg:connectedness:swerve} is about to execute.
Let $\rhpoulMrXourMl \in \{\rhpoulMr,\rhpourMl\}$ be the last turn of~$p_{i-1}$.
The induction hypothesis (2) ensures $\rhacW(d)>0$.

We first show (1).
For the sake of contradiction, we assume that the edge $\rhsoll \in E(\Delta)$ is at the border of~$\Delta$.
Then, considering $(\dag)$, it follows that $\rhpcWr\not\subseteq\SUPP(d)$, 
but $\rhpcWl\subseteq\SUPP(d)$.
Thus Algorithm~\ref{alg:connectedness} uses line~\ref{alg:connectedness:canonical} and $q_{i-1} = \rhpcWl$.
This is a contradiction to the assumption that line~\ref{alg:connectedness:swerve} is about to be executed.
Hence $\rhsoll \in E(\Delta)$ is not at the border of~$\Delta$. This proves~(1).

It remains to show (2), (3) and (4).
The fact that line~\ref{alg:connectedness:swerve} is about to execute can have the following two reasons (a) and (b):
\smallskip

\noindent{\bf (a)} $\rhpcWr\subseteq\SUPP(d)$, but $\rhpcWr$ cannot be appended to $p_{i-1}$ in line~\ref{alg:connectedness:canonical}.\\
Then $\rhc$ is $f$\dash flat and $\rhpoulMrXourMl = \rhpoulMr\subseteq\SUPP(d)$
as this turn was appended in line~\ref{alg:connectedness:canonical}. 
Lemma~\ref{lem:connectedness:hivepreservingcyclesonG} applied to~%
\begin{tikzpictured}\draw[rhrhombidraw] (0.0pt,0.0pt) -- (9.238pt,0.0pt) -- (13.857pt,8.0001pt) -- (9.238pt,16.0002pt) -- cycle;\fill (0.0pt,0.0pt) circle (1pt);\fill (9.238pt,0.0pt) circle (1pt);\fill (4.619pt,8.0001pt) circle (1pt);\fill (13.857pt,8.0001pt) circle (1pt);\fill (9.238pt,16.0002pt) circle (1pt);\draw[thin,-my] (4.619pt,16.0002pt) arc (180:240:4.619pt);\draw[thin,-my] (2.3095pt,12.0002pt) arc (120:60:4.619pt);\end{tikzpictured}
yields that $\rhrholl$ is not $f$\dash flat.
Since 
$\rhpoulMrr\subseteq\SUPP(d)$ 
we have $\rhpourMll\subseteq\SUPP(d)$
 by Lemma~\ref{cla:negimpliespos}.
Therefore, 
$p_{i-1}$ can be continued via $q_{i-1} = \rhpcWl$ in line~\ref{alg:connectedness:canonical},
in contradiction to the fact that line~\ref{alg:connectedness:swerve} is about to execute.
\smallskip

\noindent{\bf (b)} $\rhpcWl\subseteq\SUPP(d)$, but $\rhpcWl$ cannot be appended to~$p_{i-1}$ in line~\ref{alg:connectedness:canonical}.\\
Then $\rhrholl$ is $f$\dash flat, which shows (4).
Lemma~\ref{cla:negimpliespos} implies that $\rhpllMr\subseteq\SUPP(d)$.
In line~\ref{alg:connectedness:swerve}, the turns $q_{i-1} = \rhpcWrl$ are appended to~$p_{i-1}$,
which shows (2).
It remains to show that appending $q_{i-1}$ does not result in negative contributions in $f$\dash flat rhombi.
But if $\rhpcWr$ leads to a negative contribution in an $f$\dash flat rhombus, then $\rhc$ is $f$\dash flat
and if $\rhpollMl$ leads to a negative contribution in an $f$\dash flat rhombus, then $\rhrhll$ is $f$\dash flat.
In both cases, this contradicts Lemma~\ref{lem:connectedness:hivepreservingcyclesonG},
for the $f$\dash flat trapezoid 
\begin{tikzpictured}\draw[rhrhombidraw] (0.0pt,0.0pt) -- (9.238pt,0.0pt) -- (13.857pt,8.0001pt) -- (9.238pt,16.0002pt) -- cycle;\fill (0.0pt,0.0pt) circle (1pt);\fill (9.238pt,0.0pt) circle (1pt);\fill (13.857pt,8.0001pt) circle (1pt);\fill (9.238pt,16.0002pt) circle (1pt);\fill (4.619pt,8.0001pt) circle (1pt);\draw[thin,-my] (2.3095pt,12.0002pt) arc (120:60:4.619pt);\draw[thin,-my] (4.619pt,16.0002pt) arc (180:240:4.619pt);\end{tikzpictured}
and 
\begin{tikzpictured}\fill (0.0pt,0.0pt) circle (1pt);\fill (9.238pt,0.0pt) circle (1pt);\fill (13.857pt,8.0001pt) circle (1pt);\fill (4.619pt,8.0001pt) circle (1pt);\draw[rhrhombidraw] (-4.619pt,8.0001pt) -- (0.0pt,0.0pt) -- (9.238pt,0.0pt) -- (13.857pt,8.0001pt) -- cycle;\fill (-4.619pt,8.0001pt) circle (1pt);\fill[white] (9.238pt,16.0002pt) circle (0.1pt);\draw[thin,-my] (6.9285pt,12.0002pt) arc (60:0:4.619pt);\draw[thin,-my] (11.5475pt,12.0002pt) arc (120:180:4.619pt);\end{tikzpictured}
, respectively. This shows (3).
\qquad\end{prooff}

We specify now the precedence rule~(\ddag) for breaking ties in line~\ref{alg:connectedness:if} 
of Algorithm~\ref{alg:connectedness}. 
\begin{equation*}\tag{\ddag} 
\begin{minipage}{12cm}
If $p_{i-1}$ ends at the diagonal of an $f$\dash flat rhombus, then Algorithm~\ref{alg:connectedness}
appends \emph{counterclockwise} turns;
if $p_{i-1}$ ends at the diagonal of a non\dash $f$\dash flat rhombus, then Algorithm~\ref{alg:connectedness}
appends \emph{clockwise} turns.
\end{minipage}
\end{equation*}

Finally, to fully specify Algorithm~\ref{alg:connectedness}, 
we now define how the cycle~$c$ is generated from~$p$ 
in line~\ref{alg:connectedness:concat}:
When line~\ref{alg:connectedness:concat} is about to execute, then $p$ has used a vertex more than once.
Let $v$ denote the first vertex of $p$ that is used more than once.
Note that $v$ is a black vertex.
Let $q$~denote the $q_i$ that was appended last.
We note that $q$ consists of either two edges or four edges.
Now we truncate everything of~$p$ previous to the first occurence of~$v$ and everything after the last occurence of~$v$, 
thus generating the cycle~$c$.
We denote by $\vartheta$ the first turn of $p$ that uses $v$ and by $\vartheta'$ the turn of $c$ that uses $v$.

For example, suppose $p$ uses the swerve $\rhpoulMrl$ and $q = \rhpollWl$.
Then $\vartheta=\rhpcWl$ and $\vartheta'=\rhpollWr$.
Note that the turn $\vartheta'$ is contained in $c$ but not contained in $p$.

Since $p$ uses no negative contributions in $f$\dash flat rhombi, 
the first assertion of Claim~\ref{cla:LAST} below is plausible, but needs proof
as $c$ may contain turns that are not contained in $p$.
In fact, we must ensure that no negative contributions exist in $c$ ``near~$v$''.

\begin{claim}\label{cla:LAST}
(\textup{1}) The cycle $c$ is $f$\dash hive preserving.

(\textup{2}) If $\vartheta'$ is not used by $p$, then $\vartheta' \subseteq \SUPP(d)$.
\end{claim}

Let us first show that once Claim~\ref{cla:LAST} is shown, we are done.

\begin{prooff}[Proof of Proposition~\textup{\ref{pro:connectedness:existscycle}}] 
We first show that Algorithm~\ref{alg:connectedness} produces an 
$f$\dash secure cycle~$c$. 
Claim~\ref{cla:LAST} already tells us that $c$ is $f$\dash hive preserving.
Assume that $c$ uses both $\rhpoulMlXolrWl$ in a rhombus~$\rhc$.
Claim~\ref{cla:connectedness:welldefined}(2) and 
Claim~\ref{cla:LAST}(2) imply
that at least the second edge of every counterclockwise turn in $c$ is contained in $\SUPP(d)$.
Hence $\rhaourW(d)>0$ and $\rhaollM(d)>0$, which implies $\s \rhc d \leq -2$.
The fact $\s \rhc {f+d} \geq 0$ implies $\s \rhc f \geq 2$ and hence $\rhc$ is not nearly $f$\dash flat.
It follows that $c$ is $f$\dash secure.

Claim~\ref{cla:LAST}(2)
combined with Claim~\ref{cla:connectedness:welldefined}(2) 
also ensures that the only turns in $c$, that are not contained in $\SUPP(d)$, are turns of swerves.
Hence at least half of the edges of $c$ are contained in $\SUPP(d)$.
This inequality is strict, because $c$ cannot consist of swerves only.
Lemma~\ref{cla:connectedness:morethanhalf} implies $\dist(f+c,g)<\dist(f,g)$.
\end{prooff}


From now on, swerves (e.g.~$\rhpoulMrl$) will be drawn as straight arrows with a filled triangular head, e.g.~%
\begin{tikzpictured}\draw[rhrhombidraw] (0.0pt,0.0pt) -- (-4.619pt,8.0001pt) -- (0.0pt,16.0002pt) -- (4.619pt,8.0001pt) -- cycle;\draw[thin,-my] (-2.3095pt,12.0002pt) --  (2.3095pt,4.0001pt);;\end{tikzpictured}%
.
(They are not to be confused with throughput arrows like $\rhacW$, which have a different head and are always drawn crossing fat edges.)

\begin{prooff}[Proof of Claim~\textup{\ref{cla:LAST}}]
Since $p$ uses no negative contributions in $f$\dash rhombi by construction, 
the proof that the cycle $c$ is $f$\dash hive preserving breaks down into the following parts:

\begin{remunerate}
 \item[(neg1)] The turn $\vartheta'$ is not counterclockwise at the acute angle of an $f$\dash flat rhombus.
 \item[(neg2)] The turn $\vartheta'$ is not both clockwise and preceded in $c$ by another clockwise turn such that both turns lie in the same $f$\dash flat rhombus.
 \item[(neg3)] The turn $\vartheta'$ is not both clockwise and succeded in $c$ by another clockwise turn such that both turns lie in the same $f$\dash flat rhombus.
\end{remunerate}

We also need to prove the following property:
\begin{remunerate}
 \item[(su)] If $\vartheta'$ is not used by $p$, then $\vartheta' \subseteq \SUPP(d)$.
\end{remunerate}
\smallskip

Recall that $q$ denotes the $q_i$ which was appended last.
Three cases can appear:
(a) $q$~is a counterclockwise turn, (b) $q$~is a clockwise turn, (c) $q$~is a swerve.
All three cases are significantly different and require careful attention to detail.
We start with the simplest one, which does not require the precedence rule (\ddag):

\noindent{\bf (a)} Assume that $q$ is a counterclockwise turn, pictorially $q=$
\begin{tikzpicture}\draw[rhrhombidraw] (0.0pt,0.0pt) -- (9.238pt,0.0pt) -- (4.619pt,8.0001pt) -- cycle;\draw[thin,-my] (2.3095pt,4.0001pt) arc (240:300:4.619pt);\end{tikzpicture}%
. Considering Algorithm~\ref{alg:connectedness}, we see that $q \subseteq \SUPP(d)$.
There are two possibilities for $\vartheta$:
$\vartheta=$
\begin{tikzpicture}\draw[rhrhombidraw] (0.0pt,0.0pt) -- (9.238pt,0.0pt) -- (4.619pt,8.0001pt) -- cycle;\draw[thin,-my] (4.619pt,0.0pt) arc (180:120:4.619pt);\end{tikzpicture}
or
$\vartheta=$
\begin{tikzpicture}\draw[rhrhombidraw] (0.0pt,0.0pt) -- (9.238pt,0.0pt) -- (4.619pt,8.0001pt) -- cycle;\draw[thin,-my] (6.9285pt,4.0001pt) arc (120:180:4.619pt);\end{tikzpicture}%
, because the other four turns lead to a contradiction to the fact that Algorithm~\ref{alg:connectedness} stops as soon as $p$ contains a vertex twice.

\noindent{\bf (a1)} Suppose first $\vartheta=$
\begin{tikzpicture}\draw[rhrhombidraw] (0.0pt,0.0pt) -- (9.238pt,0.0pt) -- (4.619pt,8.0001pt) -- cycle;\draw[thin,-my] (4.619pt,0.0pt) arc (180:120:4.619pt);\end{tikzpicture}.
In this case, we have $\vartheta'=q$.
The statement (neg1) holds, because $\vartheta'$ is used by $p$ and $p$ uses no negative contributions in $f$\dash flat rhombi.
We also see that (su)
holds in this case, because $\vartheta'$ is used by~$p$.

\noindent{\bf (a2)} Suppose now $\vartheta=$
\begin{tikzpicture}\draw[rhrhombidraw] (0.0pt,0.0pt) -- (9.238pt,0.0pt) -- (4.619pt,8.0001pt) -- cycle;\draw[thin,-my] (6.9285pt,4.0001pt) arc (120:180:4.619pt);\end{tikzpicture}. 
Since
\begin{tikzpicture}\draw[rhrhombidraw] (0.0pt,0.0pt) -- (9.238pt,0.0pt) -- (4.619pt,8.0001pt) -- cycle;\draw[rhrhombithickside] (9.238pt,0.0pt) -- (4.619pt,8.0001pt);\draw[->,rhrhombiarrow] (2.9284pt,1.6906pt) -- (10.9286pt,6.3096pt);\end{tikzpicture}%
$(d)>0$, it follows that
\begin{tikzpicture}\draw[rhrhombidraw] (0.0pt,0.0pt) -- (9.238pt,0.0pt) -- (4.619pt,8.0001pt) -- cycle;\draw[thin,-my] (6.9285pt,4.0001pt) arc (120:180:4.619pt);\end{tikzpicture}
is part of a swerve: The situation of $p$ can be depicted as
\begin{tikzpicture}\draw[rhrhombidraw] (0.0pt,0.0pt) -- (9.238pt,0.0pt) -- (4.619pt,8.0001pt) -- cycle;\draw[thin,-my] (2.3095pt,4.0001pt) arc (240:300:4.619pt);\draw[thin,-my] (9.238pt,8.0001pt) --  (4.619pt,0.0pt);;\end{tikzpicture}%
. By construction, $\vartheta'=$
\begin{tikzpicture}\draw[rhrhombidraw] (0.0pt,0.0pt) -- (9.238pt,0.0pt) -- (4.619pt,8.0001pt) -- cycle;\draw[thin,-my] (2.3095pt,4.0001pt) arc (60:0:4.619pt);\end{tikzpicture}%
, so $\vartheta'$ consists of an edge of $q$ and the last edge of a swerve.
Hence $q \subseteq \SUPP(d)$ and (su) follows in this case.
It remains to verify (neg2) and (neg3).
Note that Algorithm~\ref{alg:connectedness} ensures that the counterclockwise turn $q$ is not a negative contribution in $f$\dash flat rhombi
and hence the shaded rhombus
\begin{tikzpictured}\fill[thick,fill=black!20] (0.0pt,0.0pt) -- (4.619pt,8.0001pt) -- (0.0pt,16.0002pt) -- (-4.619pt,8.0001pt) -- cycle;\draw[rhrhombidraw] (-4.619pt,8.0001pt) -- (4.619pt,8.0001pt) -- (0.0pt,16.0002pt) -- cycle;\end{tikzpictured}
is not $f$\dash flat. This proves (neg3).
If we assume the contrary of (neg2), then the path
\begin{tikzpicture}\fill[thick,fill=black!20] (-4.619pt,8.0001pt) -- (0.0pt,0.0pt) -- (9.238pt,0.0pt) -- (4.619pt,8.0001pt) -- cycle;\draw[rhrhombidraw] (0.0pt,0.0pt) -- (9.238pt,0.0pt) -- (4.619pt,8.0001pt) -- cycle;\draw[thin,-my] (-2.3095pt,4.0001pt) arc (120:60:4.619pt) arc (60:0:4.619pt);\end{tikzpicture}
is a negative contribution in an $f$\dash flat rhombus.
But since swerves lie in $f$\dash flat rhombi according to Claim~\ref{cla:connectedness:welldefined}(4), it follows that the trapezoid
\begin{tikzpicture}\fill[thick,fill=black!20] (-4.619pt,8.0001pt) -- (0.0pt,0.0pt) -- (9.238pt,0.0pt) -- (13.857pt,8.0001pt) -- cycle;\draw[rhrhombidraw] (0.0pt,0.0pt) -- (9.238pt,0.0pt) -- (4.619pt,8.0001pt) -- cycle;\draw[thin,-my] (9.238pt,8.0001pt) --  (4.619pt,0.0pt);;\end{tikzpicture}
is $f$\dash flat, which is a contradiction to Lemma~\ref{lem:connectedness:hivepreservingcyclesonG}, 
applied to 
\begin{tikzpictured}\draw[rhrhombidraw] (-4.619pt,8.0001pt) -- (0.0pt,0.0pt) -- (9.238pt,0.0pt) -- (13.857pt,8.0001pt) -- cycle;\fill (-4.619pt,8.0001pt) circle (1pt);\fill (4.619pt,8.0001pt) circle (1pt);\fill (13.857pt,8.0001pt) circle (1pt);\fill (0.0pt,0.0pt) circle (1pt);\fill (9.238pt,0.0pt) circle (1pt);\draw[thin,-my] (11.5475pt,12.0002pt) arc (120:180:4.619pt);\draw[thin,-my] (6.9285pt,12.0002pt) arc (60:0:4.619pt);\end{tikzpictured}
. This proves (neg2).
\smallskip

\noindent{\bf (b)} Assume that $q$ is a clockwise turn, pictorially $q=$
\begin{tikzpicture}\draw[rhrhombidraw] (0.0pt,0.0pt) -- (9.238pt,0.0pt) -- (4.619pt,8.0001pt) -- cycle;\draw[thin,-my] (6.9285pt,4.0001pt) arc (300:240:4.619pt);\end{tikzpicture}%
. As in case (a), we have two possibilities: 
$\vartheta=$
\begin{tikzpicture}\draw[rhrhombidraw] (0.0pt,0.0pt) -- (9.238pt,0.0pt) -- (4.619pt,8.0001pt) -- cycle;\draw[thin,-my] (2.3095pt,4.0001pt) arc (60:0:4.619pt);\end{tikzpicture}
or
$\vartheta=$
\begin{tikzpicture}\draw[rhrhombidraw] (0.0pt,0.0pt) -- (9.238pt,0.0pt) -- (4.619pt,8.0001pt) -- cycle;\draw[thin,-my] (4.619pt,0.0pt) arc (0:60:4.619pt);\end{tikzpicture}%
.
Since $q\subseteq \SUPP(d)$, we have
\begin{tikzpicture}\draw[rhrhombidraw] (0.0pt,0.0pt) -- (9.238pt,0.0pt) -- (4.619pt,8.0001pt) -- cycle;\draw[rhrhombithickside] (4.619pt,8.0001pt) -- (0.0pt,0.0pt);\draw[->,rhrhombiarrow] (6.3096pt,1.6906pt) -- (-1.6906pt,6.3096pt);\end{tikzpicture}%
$(d)>0$.
If we had $\vartheta=$
\begin{tikzpicture}\draw[rhrhombidraw] (0.0pt,0.0pt) -- (9.238pt,0.0pt) -- (4.619pt,8.0001pt) -- cycle;\draw[thin,-my] (2.3095pt,4.0001pt) arc (60:0:4.619pt);\end{tikzpicture},
then 
\begin{tikzpicture}\draw[rhrhombidraw] (0.0pt,0.0pt) -- (9.238pt,0.0pt) -- (4.619pt,8.0001pt) -- cycle;\draw[rhrhombithickside] (0.0pt,0.0pt) -- (4.619pt,8.0001pt);\draw[->,rhrhombiarrow] (-1.6906pt,6.3096pt) -- (6.3096pt,1.6906pt);\end{tikzpicture}%
$(d)>0$, which is a contradiction.
Hence $\vartheta=$
\begin{tikzpicture}\draw[rhrhombidraw] (0.0pt,0.0pt) -- (9.238pt,0.0pt) -- (4.619pt,8.0001pt) -- cycle;\draw[thin,-my] (4.619pt,0.0pt) arc (0:60:4.619pt);\end{tikzpicture}
 and thus $\vartheta'=q$.
This proves (su) in this case. 
The fact that $\vartheta'=q$ is a part of $p$ shows (neg2).
It remains to show (neg3). Assume the contrary.
Then the rhombus
\begin{tikzpicture}\fill[thick,fill=black!20] (4.619pt,8.0001pt) -- (-4.619pt,8.0001pt) -- (-9.238pt,16.0002pt) -- (0.0pt,16.0002pt) -- cycle;\draw[rhrhombidraw] (0.0pt,16.0002pt) -- (-4.619pt,8.0001pt) -- (4.619pt,8.0001pt) -- cycle;\end{tikzpicture}
is $f$\dash flat.
Hence
\begin{tikzpicture}\draw[rhrhombidraw] (0.0pt,16.0002pt) -- (-4.619pt,8.0001pt) -- (4.619pt,8.0001pt) -- cycle;\draw[rhrhombithickside] (0.0pt,16.0002pt) -- (4.619pt,8.0001pt);\draw[->,rhrhombiarrow] (6.3096pt,14.3097pt) -- (-1.6906pt,9.6907pt);\end{tikzpicture}%
$(d)>0$
implies
\begin{tikzpicture}\draw[rhrhombidraw] (0.0pt,16.0002pt) -- (-4.619pt,8.0001pt) -- (4.619pt,8.0001pt) -- cycle;\draw[rhrhombithickside] (-9.238pt,16.0002pt) -- (-4.619pt,8.0001pt);\draw[->,rhrhombiarrow] (-2.9284pt,14.3097pt) -- (-10.9286pt,9.6907pt);\end{tikzpicture}%
$(d)>0$.
The rhombus
\begin{tikzpictured}\fill[thick,fill=black!20] (0.0pt,0.0pt) -- (4.619pt,8.0001pt) -- (0.0pt,16.0002pt) -- (-4.619pt,8.0001pt) -- cycle;\draw[rhrhombidraw] (0.0pt,0.0pt) -- (9.238pt,0.0pt) -- (4.619pt,8.0001pt) -- cycle;\end{tikzpictured}
is not $f$\dash flat by Lemma~\ref{lem:connectedness:hivepreservingcyclesonG} applied to~%
\begin{tikzpictured}\draw[rhrhombidraw] (0.0pt,0.0pt) -- (-9.238pt,0.0pt) -- (-13.857pt,8.0001pt) -- (-9.238pt,16.0002pt) -- cycle;\fill (0.0pt,0.0pt) circle (1pt);\fill (-9.238pt,0.0pt) circle (1pt);\fill (-13.857pt,8.0001pt) circle (1pt);\fill (-9.238pt,16.0002pt) circle (1pt);\fill (-4.619pt,8.0001pt) circle (1pt);\draw[thin,-my] (2.3095pt,4.0001pt) arc (60:120:4.619pt);\draw[thin,-my] (0.0pt,8.0001pt) arc (0:-60:4.619pt);\end{tikzpictured}%
.
But the precedence rule (\ddag) of Algorithm~\ref{alg:connectedness} implies that $p$ continues from
\begin{tikzpicture}\draw[rhrhombidraw] (0.0pt,0.0pt) -- (9.238pt,0.0pt) -- (4.619pt,8.0001pt) -- cycle;\draw[thin,-my] (4.619pt,0.0pt) arc (0:60:4.619pt);\end{tikzpicture}
with the counterclockwise turn
\begin{tikzpicture}\draw[rhrhombidraw] (0.0pt,0.0pt) -- (9.238pt,0.0pt) -- (4.619pt,8.0001pt) -- cycle;\draw[thin,-my] (2.3095pt,4.0001pt) arc (60:120:4.619pt);\end{tikzpicture}%
. 
This is a contradiction, proving (neg3) in this case.
\smallskip

\noindent{\bf (c)}
Assume that $q$ is a swerve, pictorially $q=$
\begin{tikzpictured}\draw[rhrhombidraw] (0.0pt,16.0002pt) -- (-4.619pt,8.0001pt) -- (0.0pt,0.0pt) -- (4.619pt,8.0001pt) -- cycle;\draw[thin,-my] (-2.3095pt,12.0002pt) --  (2.3095pt,4.0001pt);;\end{tikzpictured}
. The rhombus $\rhc$ which contains the swerve is $f$\dash flat by Claim~\ref{cla:connectedness:welldefined}(4).
Since $p$ does not use negative contributions in $f$\dash flat rhombi, we get that $\rhrholl$ is not $f$\dash flat.
The possibilities for~$\vartheta$ here are $\rhpourMl$, $\rhpcMr$, $\rhpollWr$ 
(note that $\rhpolrWl$ is ruled out, because $\rhc$ is $f$\dash flat).
We distinguish the following three cases:

\noindent{\bf (c1)} 
Suppose $\vartheta=$
\begin{tikzpictured}\draw[rhrhombidraw] (0.0pt,16.0002pt) -- (-4.619pt,8.0001pt) -- (0.0pt,0.0pt) -- (4.619pt,8.0001pt) -- cycle;\draw[thin,-my] (-2.3095pt,4.0001pt) arc (120:60:4.619pt);\end{tikzpictured}%
. Here $\vartheta'=\rhpcWl$, which is part of~$q$. This proves (su) 
in this case.
The fact that $\rhrholl$ is not $f$\dash flat implies (neg1) in this case.

\noindent{\bf (c2)} 
Suppose $\vartheta=$
\begin{tikzpictured}\draw[rhrhombidraw] (0.0pt,16.0002pt) -- (-4.619pt,8.0001pt) -- (0.0pt,0.0pt) -- (4.619pt,8.0001pt) -- cycle;\draw[thin,-my] (2.3095pt,12.0002pt) arc (120:180:4.619pt);\end{tikzpictured}%
. Here $\vartheta' = \rhpoulMr$, which is part of~$q$, which again proves (su)
in this case.
The fact (neg2) follows because $p$ uses no negative contributions in $f$\dash flat rhombi.
It remains to show (neg3).
Note that $\rhaoulM(d)>0$, because the first edge of $\rhpoulMr$ is contained in $\SUPP(d)$.
Further, 
$\rhacW(d)>0$, because the second edge of 
the counterclockwise turn $\rhpourMl$ is contained in $\SUPP(d)$
(this is always the case for counterclockwise turns by construction of~$p$).
Hence $\rhpoulMr \subseteq \SUPP(d)$.
If $\rhrhoul$ were not $f$\dash flat, then
Algorithm~\ref{alg:connectedness} would have appended the clockwise turn $\rhpoulMr$ over appending the swerve
\begin{tikzpictured}\draw[rhrhombidraw] (0.0pt,16.0002pt) -- (-4.619pt,8.0001pt) -- (0.0pt,0.0pt) -- (4.619pt,8.0001pt) -- cycle;\draw[thin,-my] (-2.3095pt,12.0002pt) --  (2.3095pt,4.0001pt);;\end{tikzpictured}. 
Hence $\rhrhoul$ is $f$\dash flat.
The hexagon equality (\ref{cla:BZ}) implies that the trapezoid
\begin{tikzpictured}\fill[rhrhombifill] (9.238pt,0.0pt) -- (0.0pt,0.0pt) -- (-4.619pt,8.0001pt) -- (0.0pt,16.0002pt) -- cycle;\draw[rhrhombidraw] (9.238pt,0.0pt) -- (13.857pt,8.0001pt) -- (9.238pt,16.0002pt) -- (4.619pt,8.0001pt) -- cycle;\end{tikzpictured}
if $f$\dash flat.
The fact (neg3) follows from Lemma~\ref{lem:connectedness:hivepreservingcyclesonG} applied to~%
\begin{tikzpictured}\draw[rhrhombidraw] (9.238pt,0.0pt) -- (0.0pt,0.0pt) -- (-4.619pt,8.0001pt) -- (0.0pt,16.0002pt) -- cycle;\fill (9.238pt,0.0pt) circle (1pt);\fill (0.0pt,0.0pt) circle (1pt);\fill (4.619pt,8.0001pt) circle (1pt);\fill (-4.619pt,8.0001pt) circle (1pt);\fill (0.0pt,16.0002pt) circle (1pt);\draw[thin,-my] (9.238pt,8.0001pt) arc (0:-60:4.619pt);\end{tikzpictured}%
.

\noindent{\bf (c3)} 
Suppose $\vartheta=\rhpcMr$. Here $\vartheta'=\rhpoulMl$, which is a negative contribution in the $f$\dash flat rhombus~$\rhc$.
Hence we need to show that this case leads to a contradiction.
Recall that $\rhrholl$ is not $f$\dash flat.
Clearly,
\begin{tikzpictured}\draw[rhrhombidraw] (0.0pt,0.0pt) -- (-4.619pt,8.0001pt) -- (0.0pt,16.0002pt) -- (4.619pt,8.0001pt) -- cycle;\draw[rhrhombithickside] (0.0pt,0.0pt) -- (4.619pt,8.0001pt);\draw[->,rhrhombiarrow] (-1.6906pt,6.3096pt) -- (6.3096pt,1.6906pt);\end{tikzpictured}%
$(d)>0$
and
\begin{tikzpictured}\draw[rhrhombidraw] (0.0pt,0.0pt) -- (-4.619pt,8.0001pt) -- (0.0pt,16.0002pt) -- (4.619pt,8.0001pt) -- cycle;\draw[rhrhombithickside] (4.619pt,8.0001pt) -- (-4.619pt,8.0001pt);\draw[->,rhrhombiarrow] (0.0pt,3.3811pt) -- (0.0pt,12.6191pt);\end{tikzpictured}%
$(d)>0$, which implies
\begin{tikzpictured}\draw[rhrhombidraw] (0.0pt,0.0pt) -- (-4.619pt,8.0001pt) -- (0.0pt,16.0002pt) -- (4.619pt,8.0001pt) -- cycle;\draw[rhrhombithickside] (0.0pt,0.0pt) -- (-4.619pt,8.0001pt);\draw[->,rhrhombiarrow] (-6.3096pt,1.6906pt) -- (1.6906pt,6.3096pt);\end{tikzpictured}%
$(d)>0$.
This means that $\vartheta$ is preceded in $p$ by the counterclockwise turn
$\rhpollWl$.
This is a contradiction to the precedence rule (\ddag), because Algorithm~\ref{alg:connectedness} would have chosen
$\rhpollWr$ instead of $\rhpollWl$.
\qquad\end{prooff}


\bibliographystyle{siam}
{\small

}
\end{document}